\documentclass[11pt,a4paper,notitlepage,oneside]{amsart}
\usepackage{etoolbox}
\usepackage{microtype}% espansione font + protusione caratteri per miglior riempimento riga
\usepackage{amsmath}
\usepackage{amsthm}
\usepackage{braket}
\usepackage{pdfpages}
\usepackage[psamsfonts]{amssymb}
\usepackage[T1]{fontenc}
\usepackage[english]{babel}
\usepackage[utf8]{inputenc}
\usepackage{enumitem}
\usepackage{mathtools}
\usepackage{caption,subfig}% per figure e didascalie
\usepackage{rotating}% per ruotare figure
\usepackage[makeroom]{cancel}
\usepackage[normalem]{ulem}
\usepackage{scalerel}
\usepackage{relsize}
\usepackage[new]{old-arrows}
\usepackage[toc,page]{appendix}

\usepackage[autostyle,italian=guillemets]{csquotes}
\usepackage[style=alphabetic,sorting=nyt,firstinits=true,maxnames=5,maxalphanames=5,backend=bibtex]{biblatex}
\allowdisplaybreaks %Long formulas can break in 2 different pages.

\usepackage{pgf,tikz,pgfplots}

\usepackage{mathrsfs}
\usetikzlibrary{arrows}
\usepackage{fancyhdr}
\usepackage{tikz-cd}
\usepackage{amscd}
\usepackage[a4paper, bottom=3cm, left=3cm,right=3cm]{geometry}
\usepackage{hyperref}
\hypersetup{hidelinks}% nasconde i quadrati colorati delle referenze.

\usepackage{subfiles}
\title{The Kudla--Millson lift of Siegel cusp forms}
\author{Paul Kiefer}
\address{
Department of Mathematics, University of Antwerp, BE-2000 Antwerp, Belgium.
}
\email{Paul.Kiefer@uantwerpen.be}
\author{Riccardo Zuffetti}
\address{
Fachbereich Mathematik, Technische Universität Darmstadt, Schlossgartenstraße 7, D–
64289 Darmstadt, Germany.
}
\email{zuffetti@mathematik.tu-darmstadt.de}

\makeatletter
\@namedef{subjclassname@2020}{\textup{2020} Mathematics Subject Classification}
\makeatother


\numberwithin{equation}{section} %numbering formulas depending on section
\theoremstyle{definition}
\newtheorem{defi}{Definition}[section]
\newtheorem{ex}[defi]{Example}

\theoremstyle{plain}
\newtheorem{thm}[defi]{Theorem}
\theoremstyle{plain}
\newtheorem{prop}[defi]{Proposition}
\theoremstyle{plain}
\newtheorem{lemma}[defi]{Lemma}
\theoremstyle{plain}
\newtheorem{cor}[defi]{Corollary}
\theoremstyle{plain}

\theoremstyle{plain}
\newtheorem*{question*}{Question}

\theoremstyle{definition}
\newtheorem{rem}[defi]{Remark}
\theoremstyle{thm}

\definecolor{ao(english)}{rgb}{0.0, 0.5, 0.0}

\DeclareMathOperator{\Log}{\rm Log}

\DeclareMathOperator{\Sp}{\rm Sp}
\DeclareMathOperator{\GL}{\rm GL}
\DeclareMathOperator{\SL}{\rm SL}
\DeclareMathOperator{\SO}{\rm SO}
\DeclareMathOperator{\UU}{\rm U}
\DeclareMathOperator{\Sym}{\rm Sym}

\newcommand{\Aut}{\text{Aut}}
\newcommand{\Mat}{\text{Mat}}

\DeclareMathOperator{\trace}{\rm tr}

\newcommand{\ZZ}{\mathbb{Z}}
\newcommand{\RR}{\mathbb{R}}
\newcommand{\QQ}{\mathbb{Q}}
\newcommand{\CC}{\mathbb{C}}

\newcommand{\HH}{\mathbb{H}}
\newcommand{\weil}[1]{\rho_{#1,2}}

\DeclareMathOperator{\kling}{C_{2,1}}
\DeclareMathOperator{\Gr}{\rm Gr}
\DeclareMathOperator{\Mp}{\rm Mp}

\DeclareMathOperator{\sgn}{\rm sgn}

\newcommand{\disc}[1]{D_{#1}} % discriminant of a lattice
\newcommand{\gendisc}{\sigma}
\newcommand{\brK}{{\textcolor{\newcolor}{K}}} % Bruinier's sublattice K in L
\newcommand{\Lor}{\brK}
\newcommand{\alphaone}{\alpha_1}
\newcommand{\alphatwo}{\alpha_2}
\newcommand{\alphagen}[1]{\alpha_{#1}}
\newcommand{\betaone}{\beta_1}
\newcommand{\betatwo}{\beta_2}
\newcommand{\betagen}[1]{\beta_{#1}}

\newcommand{\intfunctshort}{\mathcal{I}_{\vect{\alpha}}} %shorter version

\DeclareMathOperator{\KMlift}{\Lambda^{\rm KM}_2}

\newcommand{\hermdom}{\mathcal{D}}

\newcommand{\domain}{\mathcal{D}}
\newcommand{\omegaone}{\omega_{\alphaone,1}} % first standard cotangent vector at the base point
\newcommand{\omegatwo}{\omega_{\alphatwo,2}} % second standard cotangent vector at the base point
\newcommand{\omegathree}{\omega_{\betaone,1}} % third standard cotangent vector at the base point
\newcommand{\omegafour}{\omega_{\betatwo,2}} % fourth standard cotangent vector at the base point
\newcommand{\speciso}{g'} % special iso for vanishing Jacobi Pet inner product
\newcommand{\specisoborw}{\borw'} % special iso using \borw command
\newcommand{\cha}[1]{#1'}
\newcommand{\degjac}{m}
\newcommand{\genU}{u} %standard first generator of hyperbolic plane U
\newcommand{\genUU}{u'} %standard second generator of hyperbolic plane U
\newcommand{\genvec}{v} %general vector in V
\newcommand{\basevec}{e} %letter to denote standard basis vectors of \RR^{b,2}
\newcommand{\auxspace}{z}
\newcommand{\hone}{h_1} %index, it was h_1^+ at the very beginning
\newcommand{\htwo}{h_2} %index, it was h_2^+ at the very beginning
\newcommand{\hgen}[1]{h_{#1}} %index, generalization of the previous with general entry
\newcommand{\htot}{h} %index, it was h^+ at the very beginning
\newcommand{\hfunct}{\mathcal{H}}
\newcommand{\Gpol}{\mathcal{Q}} %general polynomial arysing from KM schwartz form in genus 1
\newcommand{\pol}{\mathcal{P}} %general bihomogeneous polynomial
\newcommand{\polw}[3]{\mathcal{P}_{#1,#2,#3}} %general polynomial arising from Llor
\newcommand{\polab}{\mathcal{P}_{(\alphaone,\alphatwo)}} %polynomial arising from the aux functions F_{alpha,beta}
\newcommand{\polcd}{\mathcal{P}_{(\betaone,\betatwo)}}
\newcommand{\borw}{g_{\brK}} % Borcherds w
\newcommand{\borww}{g_{\Lpos}} % Borcherds w
\newcommand{\borwmix}[1]{(#1)_{\brK}} % Borcherds w but for more complicated input
\newcommand{\boralpha}{\delta}
\newcommand{\borbeta}{\nu}
\newcommand{\vecbrlam}{\vect{\lambda}_\Lor}
\newcommand{\tauone}{\tau_1}
\newcommand{\tautwo}{\tau_2}
\newcommand{\tauthree}{\tau_3}
\newcommand{\taugen}[1]{\tau_{#1}}
\newcommand{\yone}{y_1}
\newcommand{\ytwo}{y_2}
\newcommand{\ythree}{y_3}
\newcommand{\xone}{x_1}
\newcommand{\xtwo}{x_2}
\newcommand{\xthree}{x_3}
\newcommand{\jacone}{r} % the \lambda in (\lambda,\mu) in Eischler--Zagier "Theory of Jacobi forms"
\newcommand{\jactwo}{s} % the \mu in (\lambda,\mu) in Eischler--Zagier "Theory of Jacobi forms"
\newcommand{\Lpos}{L^+}

\newcommand{\genUtwo}{\tilde{u}}
\newcommand{\genUUtwo}{\tilde{u}'}

\newcommand{\auxfunplus}{\chi}

\newcommand{\stgatwo}{\varphi_{0,2}}
\newcommand{\vect}[1]{\boldsymbol{#1}}

\newcommand{\Gpoltwo}{\mathcal{Q}}

\newcommand{\abcd}{(\alphaone,\alphatwo,\betaone,\betatwo)}
\newcommand{\halfint}{\Lambda}

\DeclareMathOperator{\bigO}{\rm O}

\def\be{\begin{equation}}
\def\ee{\end{equation}}
\def\bes{\begin{equation*}}
\def\ees{\end{equation*}}
\def\ba{\be\begin{aligned}}
\def\ea{\end{aligned}\ee}
\def\bas{\bes\begin{aligned}}
\def\eas{\end{aligned}\ees}

\makeatletter
\newcommand*\defbb[1]{
	\expandafter\newcommand\csname I#1\endcsname{\mathbb{#1}}}
\newcommand*\defbbs[1]{
	\@for\@i:=#1\do{\expandafter\defbb\expandafter{\@i}}}
\makeatother
\defbbs{A,B,C,D,E,F,G,H,I,K,L,M,N,O,P,Q,R,S,T,U,V,W,X,Y,Z} %J missing

\makeatletter
\newcommand*\deffrak[1]{
	\expandafter\newcommand\csname frak#1\endcsname{\mathfrak{#1}}}
\newcommand*\deffraks[1]{
	\@for\@i:=#1\do{\expandafter\deffrak\expandafter{\@i}}}
\makeatother
\deffraks{a,b,c,d,e,f,g,h,i,j,k,l,m,n,o,p,q,r,s,t,u,v,w,x,y,z}

\makeatletter
\newcommand*\defcal[1]{
	\expandafter\newcommand\csname cal#1\endcsname{\mathcal{#1}}}
\newcommand*\defcals[1]{
	\@for\@i:=#1\do{\expandafter\defcal\expandafter{\@i}}}
\makeatother
\defcals{A,B,C,D,E,F,G,H,I,J,K,L,M,N,O,P,Q,R,S,T,U,V,W,X,Y,Z}

\newcommand{\jacthree}{t} 
\newcommand{\jacvarone}{\tauone}
\newcommand{\jacvartwo}{\tautwo}
\newcommand{\jacvaronere}{\xone}
\newcommand{\jacvaroneim}{\yone}
\newcommand{\jacvartwore}{\xtwo}
\newcommand{\jacvartwoim}{\ytwo}
\newcommand{\jacindexvec}{\eta}
\newcommand{\jacindex}{N}
\newcommand{\weilrep}{\rho}
\newcommand{\isometry}{g}
\newcommand{\isometrypos}{z^\perp}
\newcommand{\isometryneg}{z}
\newcommand{\subspaceisometry}{{\isometry_{\sublattice}}}
\newcommand{\subspaceisometryprime}{{\isometry'_{\sublattice}}}
\newcommand{\subspaceisometrypos}{{w^\perp}}
\newcommand{\subspaceisometryneg}{w}
\newcommand{\isotropicvec}{\genU}
\newcommand{\isotropicvecprime}{\genUU}
\newcommand{\lattice}{L}
\newcommand{\sublattice}{{\textcolor{\newcolor}{K}}}
\newcommand{\heisenberg}{\calH}
\newcommand{\jacobi}{\calJ}
\newcommand{\mat}{\gamma}
\newcommand{\subpolposdeg}{h}

\newcommand{\gen}{\operatorname{gen}}
\newcommand{\tr}{\operatorname{tr}}
\newcommand{\Pet}{\operatorname{Pet}}

\newcommand{\bs}{\ensuremath{\backslash}}
\newcommand\numberthis{\addtocounter{equation}{1}\tag{\theequation}}

\newcommand{\mycolor}{black} % color to highlight the changes in 2nd version
\newcommand{\newcolor}{black} % color to highlight the changes in 3rd version

\begin{document}
	\maketitle
	%%%%%%%%%%%%%%%%%%%%%%%%%%
	\begin{abstract}
	We study the injectivity of the Kudla--Millson lift of genus~2 Siegel cusp forms, vector-valued with respect to the Weil representation associated to an even lattice~$L$.
	We prove that if~$L$ splits off two hyperbolic planes and is of sufficiently large rank, then the lift is injective.
	As an application, we deduce that the image of the lift in the degree~$4$ cohomology of the associated orthogonal Shimura variety has the same dimension as the lifted space of cusp forms.
	Our results also cover the case of moduli spaces of quasi-polarized K3 surfaces.
	To prove the injectivity, we introduce vector-valued indefinite Siegel theta functions of genus~$2$ and of Jacobi type attached to~$L$.
	We describe their behavior with respect to the split of a hyperbolic plane in~$L$.
	This generalizes results of Borcherds to genus higher than~$1$.
	\end{abstract}
	
	%%%%%%%%%%%%%%%%%%%%%%%%%%%%%%%%%%%%%
	\tableofcontents

	\section{Introduction}
	The Kudla--Millson lifts are linear maps from spaces of vector-valued Siegel cusp forms to the space of closed differential forms on some orthogonal Shimura variety~$X$.
	They were introduced by Kudla and Millson in the eighties~\cite{kumi;harmI},~\cite{kumi;harmII},~\cite{kumi;tubes}, and have been proved to be useful tools to deduce geometric and arithmetic properties of orthogonal Shimura varieties.
	
	The injectivity of the Kudla--Millson lift, of interest already in~\cite{kumi;intnum}, has been proved \emph{only in genus~$1$}, namely in the case of lifts of \emph{elliptic cusp forms}~\cite{br;borchp},~\cite{brfu},~\cite{br;converse}.
	Such a result has several useful applications.
	For instance, it implies the surjectivity of Borcherds' lift~\cite{br;converse}, it has been employed to compute the rational Picard number of the underlying Shimura variety~$X$~\cite{blmm;conj} and to study the geometric properties of cones generated by rational and cohomology classes of special cycles~\cite{brmo},~\cite{zuffetti;cones},~\cite{zuffetti;equid}.
	
	The goal of this paper is to prove that \emph{the Kudla--Millson lift of genus~$2$ Siegel cusp forms is injective}, and provide geometric applications of this result.
	
	The idea is to extend the new proof of the injectivity in genus~$1$ provided by the second author in~\cite{zuffetti;gen1} to higher genus.
	Several difficulties arise in this process, mainly due to the lack of results on Siegel theta functions of genus~$2$ associated to indefinite lattices.
	In this article we consider a vector-valued analogue of the indefinite Siegel theta functions constructed by Röhrig~\cite{roehrig}, and give their Fourier--Jacobi expansion in terms of certain indefinite Jacobi Siegel theta functions.
	The latter were not available in the literature, and are introduced here for the first time.
	
	In~\cite[Section~$5$]{bo;grass} Borcherds rewrote genus~$1$ Siegel theta functions with respect to the split of a hyperbolic plane from an even lattice.
	In this paper we generalize this procedure to higher genus and the Jacobi case.

	As an application, we calculate the Fourier expansion of the Kudla--Millson lift.
	This is given in terms of Petersson inner products of the \emph{Fourier--Jacobi coefficients} of the lifted Siegel cusp form with certain Jacobi Siegel theta functions.
	
	Thanks to the theory of Jacobi Siegel theta functions introduced in this article, we are also able to calculate these Petersson inner products, this time in terms of the \emph{Fourier coefficients} of the lifted cusp form. From a careful study of these formulas, we deduce the following result.
	\begin{thm}\label{thm;introinjeasy}
	Let $L$ be an even lattice of signature $(b, 2)$ that splits off two orthogonal hyperbolic planes, and let~$\Lpos$ be the orthogonal complement of such hyperbolic planes.
	If~$L_p^+\coloneqq\Lpos\otimes\ZZ_p$ splits off two hyperbolic planes over~$\ZZ_p$ for every prime~$p$, then the Kudla--Millson lift associated to~$L$ is injective.
	\end{thm}
	
	The moduli spaces of quasi-polarized K3 surfaces of degree~$2d$ are among the orthogonal Shimura varieties arising from lattices satisfying the hypothesis of Theorem~\ref{thm;introinjeasy}.
	In fact, we may choose lattices of signature~$(19,2)$ given by
	\be\label{eq;introeasyK3lattice}
	L = \langle 2d\rangle\oplus U^{\oplus 2} \oplus E_8^{\oplus 2}
	\ee
	for some~$d\in\ZZ_{>0}$, where~$U$ and~$E_8$ are the hyperbolic plane and the~$E_8$ root lattice respectively; see e.g.~\cite{brmo} and~\cite{blmm;conj} for details.
	
	The following result is a consequence of Theorem~\ref{thm;introinjeasy} applied to the moduli spaces of quasi-polarized~K3 surfaces.
	We denote by~$S^k_{2,L}$ the space of genus~$2$ and weight~$k$ vector-valued Siegel cusp forms with respect to the Weil representation associated to~$L$.
	\begin{cor}\label{cor;introeasy}
	Let~$X$ be a moduli space of quasi-polarized K3 surfaces of fixed degree arising from a lattice~$L$ as in~\eqref{eq;introeasyK3lattice}, and let~$k=1+\dim(X)/2$.
	The Kudla--Millson lift associated to~$L$ induces an injective map in cohomology, in particular~$\dim H^4(X,\CC)\ge \dim S^k_{2,L}$.
	\end{cor}
	\textcolor{\mycolor}{In the recent paper~\cite{br-zu} of Bruinier and the second author, a decomposition of the cohomology class of the Kudla--Millson theta function in Eisenstein, Klingen and cuspidal part is provided.
	That decomposition and Theorem~\ref{thm;introinjeasy} have been used in~\cite{br-zu} to deduce a formula for the dimension of~$H^{2,2}(X,\CC)$; see~\cite[Section~$6.2$]{br-zu} for details.}
	
	In the next sections we provide a more detailed account of the achievements of this article.
	
	\subsection{The genus~2 Kudla--Millson lift in terms of Siegel theta functions}
	Let~$L$ be an even indefinite lattice of signature~$(b,2)$.
	To simplify the exposition, in this introduction we assume~$L$ to be \emph{unimodular}, so that we may work with scalar-valued Siegel modular forms.
	In the main body of the paper this hypothesis will be dropped.
	
	Let~$k=1+b/2$ and let~$V=L\otimes\RR$.
	Note that~$k$ is an even integer, as one can easily deduce from the well-known classification of unimodular lattices.
	We denote by~$(\cdot{,}\cdot)$ and~$q(\cdot)$ respectively the bilinear form and the associated quadratic form of~$V$.
	If~$z\subseteq V$ is a subspace, we denote by~$\genvec_z$ the orthogonal projection of~$\genvec\in V$ on~$z$.
	
	The Hermitian symmetric domain~$\hermdom$ associated to the linear algebraic group~$G=\SO(V)$ may be realized as the Grassmannian~$\Gr(L)$ of negative definite planes in~$V$.
	Let~$X=\Gamma\backslash\hermdom$ be the orthogonal Shimura variety arising from a subgroup~$\Gamma$ of finite index in~$\SO(L)$.
	
	Kudla and Millson~\cite{kumi;harmI}, \cite{kumi;harmII}, \cite{kumi;intnum} constructed a~$G$-invariant Schwartz function~$\varphi_{\text{KM},2}$ on~$V^2$ with values in the space~$\mathcal{Z}^4(\hermdom)$ of closed~$4$-forms on~$\hermdom$; this is recalled in Section~\ref{sec;Schwfunct}.
	Let~$\omega_{\infty,2}$ be the Schrödinger model of the Weil representation of~$\Sp_4(\RR)$, acting on the space~$\mathcal{S}(V^2)$ of Schwartz functions on~$V^2$, associated to the standard additive character.
	
	\begin{defi}\label{def;KMthetaformgenus2}
	The \emph{Kudla--Millson theta function of genus~$2$} is defined as
	\bes
	\Theta(\tau,z,\varphi_{\text{KM},2})=\det y^{-k/2}\sum_{\vect{\lambda}\in L^2}\big(\omega_{\infty,2}(g_\tau)\varphi_{\text{KM},2}\big)(\vect{\lambda},z),
	\ees
	for every $\tau=x+iy\in\HH_2$ and $z\in\Gr(L)$, where $g_\tau
%	=\left(\begin{smallmatrix}
%	1 & x\\ 0 & 1
%	\end{smallmatrix}\right)\left(\begin{smallmatrix}
%	y^{1/2} & 0\\ 0 & (y^{1/2})^{-t}
%	\end{smallmatrix}\right)
	$ is the standard element of~$\Sp_4(\RR)$ mapping~${i\in\HH_2}$ to~$\tau$.
	\end{defi}
	In the variable~$\tau$ this theta function transforms like a (non-holomorphic) Siegel modular form of weight~$k=1+b/2$ with respect to~$\Sp_4(\ZZ)$.
	In the variable~$z$ it defines a closed~$4$-form on~$X$.
	
	Let~$S^k_2$ be the space of weight~$k$ Siegel cusp forms of genus~$2$ with respect to~$\Sp_4(\ZZ)$.
	
	\begin{defi}\label{def;KMthetaliftingenus2fd}
	\textcolor{\mycolor}{Let~$\mathcal{Z}^4(X)$ denote the space of closed $4$-forms.} The \emph{Kudla--Millson lift of genus $2$} is the linear function~$\KMlift\colon S^k_2\to\mathcal{Z}^4(X)$ defined by mapping a Siegel cusp form~$f$ to its Petersson inner product with the theta function~$\Theta(\tau,z,\varphi_{\text{KM},2})$.
	Explicitly, this means that
	\be\label{eq;liftingLambda2}
	f\longmapsto\KMlift(f)=\int_{\Sp_4(\ZZ)\backslash\HH_2}
	\det y^k
	f(\tau)
	\,
	\overline{\Theta(\tau,z,\varphi_{\text{\rm KM},2})}\,\frac{dx\, dy}{\det y^3},
	\ee
	where $dx\,dy\coloneqq\prod_{k\leq\ell}dx_{k,\ell}\,dy_{k,\ell}$ is the Euclidean volume element, and~$\frac{dx\, dy}{\det y^3}$ is the standard~\textcolor{\newcolor}{$\Sp_4(\RR)$}-invariant volume element of~$\HH_2$.
	\end{defi}
	Inspired by the construction of indefinite theta functions provided by the recent article~\cite{roehrig}, we define certain genus~$2$ Siegel theta functions~$\Theta_{L,2}$ associated to the lattice~$L$.
	In the more general case where~$L$ is not unimodular, they are vector-valued with respect to the Weil representation associated to~$L$.
	These may be considered as a generalization of the theta functions defined by Borcherds in~\cite[Section~$4$]{bo;grass} to higher genus.
	
	The theta functions~$\Theta_{L,2}$ depend on the variables~$\tau\in\HH_2$ and~$g\in G$, and are attached to \emph{very homogeneous polynomials~$\pol$ of degree~$d$} on the space of matrices~$\RR^{(b+2)\times 2}$, the latter property meaning that
	\bes
	\pol(\vect{x}\cdot N)=\det N^d\cdot\pol(\vect{x}),\qquad\text{for every~$N\in\RR^{2\times 2}$ and~$\vect{x}\in\RR^{({b+2})\times 2}$}.
	\ees
	We denote these theta functions by~$\Theta_{L,2}(\tau,g,\pol)$ and refer to Section~\ref{sec;vvsiegeltheta2} for further details.
	
	We show that there exist very homogeneous polynomials~$\pol_{\vect{\alpha}}$ of degree~$2$, depending on some tuples of indices~$\vect{\alpha}=(\alphaone,\alphatwo,\betaone,\betatwo)$, such that the lift~$\KMlift(f)$ may be rewritten in terms of Siegel theta integrals as
	\ba\label{eq;KMdeg2liftmoreexplintro}
	\KMlift(f)&=
	\sum_{\substack{\alphaone,\betaone=1\\ \alphaone<\betaone}}^b
	\sum_{\substack{\alphatwo,\betatwo=1\\ \alphatwo<\betatwo}}^b
	\Big(\underbrace{\int_{\Sp_4(\ZZ)\backslash\HH_2}
	\det y^{k}
	f(\tau)
	\,
	\overline{\Theta_{L,2}(\tau,g,\pol_{\vect{\alpha}})} \,\frac{dx\,dy}{\det y^3}}_{\intfunctshort(g)}\Big)
	\\
	&\quad\times g^*\big(\omegaone\wedge\omegatwo\wedge\omegathree\wedge\omegafour\big).
	\ea
	
	In~\eqref{eq;KMdeg2liftmoreexplintro} the differential form~$\KMlift(f)$ is rewritten  over the point~$z\in\hermdom$ by choosing any~$g\in G$ mapping~$z\in\hermdom$ to a fixed base point~$z_0$, and~$\omegaone\wedge\dots\wedge\omegafour$ is an explicit vector of~$\bigwedge^4 T_{z_0}^* \hermdom$ coming from the definition of the Kudla--Millson Schwartz function~$\varphi_{\text{KM},2}$; see Section~\ref{sec;deg2kmtform} for further information.
	
	We refer to the integral functions~$\intfunctshort\colon G\to\CC$ appearing in~\eqref{eq;KMdeg2liftmoreexplintro} as the \emph{defining integrals} of the Kudla--Millson lift~$\KMlift(f)$.
	
	The first step to prove the injectivity of the lift is to compute the Fourier expansion of~$\intfunctshort$ with respect to the split of a hyperbolic plane~$U$ in~$L$.
	To do so, we generalize Borcherds' formalism~\cite[Section~$5$]{bo;grass} to the genus~$2$ Siegel theta functions~$\Theta_{L,2}$.
	More precisely, we choose a split~$L=\brK\oplus U$ for some Lorentzian sublattice~$\brK$ and rewrite~$\Theta_{L,2}$ as a combination of \emph{Jacobi Siegel theta functions}~$\Theta_{\brK,\lambda}$ associated to~$\brK$ and certain lattice vectors~$\lambda\in \brK$.
	
	The above indefinite Siegel theta functions of Jacobi type were not available in the literature.
	Their development is an ancillary achievement of this article, and may play an important role in generalizations of higher genus theta lifts.
	
	The Fourier expansion of~$\intfunctshort$ is illustrated in the following result.
	We do not provide here the definitions of the twist~$\borw$ of the isometry~$g\in G$, nor the polynomials~$\polw{\vect{\alpha},\borw}{\hone}{\htwo}$ decomposing~$\pol_{\vect{\alpha}}$, and instead refer to Section~\ref{sec;splitgen2theta}.
	We write~$\tau\in\HH_2$ in a matrix form as~$\tau=\big(\begin{smallmatrix}
	\tauone & \tautwo\\
	\tautwo & \tauthree
	\end{smallmatrix}\big)$, same for its real and imaginary part.
	Recall that an integer~$t\ge1$ is said to divide~$\lambda\in\brK$, in short~$t|\lambda$, if~$\lambda/t$ is still a lattice vector in~$\brK$.
	\begin{thm}\label{thm;introscvalver1unf}
	Let~$f(\tau)=\textcolor{\mycolor}{\sum_{m = 1}^\infty} \phi_m(\tauone,\tautwo)e^{2\pi i \textcolor{\mycolor}{m}\tauthree}$ be the Fourier--Jacobi expansion of~$f\in S^k_2$.
	The defining integral~$\intfunctshort$ of the Kudla--Millson lift of~$f$ admits a Fourier expansion with respect to the split~$L=U\oplus\brK$ of the form
	\[
	\intfunctshort(g)=\sum_{\lambda\in \brK} c_\lambda(g)\cdot e^{2\pi i (\lambda,\mu)},
	\]
	where~$\mu=-\genUU+\genU_{z^\perp}/2\genU_{z^\perp}^2 + \genU_z/2\genU_z^2$, with~$\genU$ and~$\genUU$ the standard generators of~$U$, and Fourier coefficients given as follows.
	The constant term of the Fourier expansion is
	\[
	c_0(g)
	=
	\int_{\Sp_4(\ZZ)\backslash\HH_2}\frac{\det y^{k+1/2}}{2 u_{z^\perp}^2}
	f(\tau) \cdot \overline{\Theta_{\brK,2}}(\tau,\borw,\polw{\vect{\alpha},\borw}{0}{0})
	\,\frac{dx\,dy}{\det y^3}.
	\]
	If~$\lambda\in\brK$ is of positive norm, then
	\bas
	c_\lambda(g)
	&=
	\sum_{t\ge 1,\, t|\lambda}
	\sum_{\htot=0}^2
	\Big(\frac{t}{2i}\Big)^{h}
	\int_{(\tauone,\tautwo)\in\Gamma^J\backslash\HH\times\CC}
	\frac{\yone^{\htot+5/2}}{\genU_{z^\perp}^2}
	\int_{\ythree=\ytwo^2/\yone}^\infty
	\det y^{k-5/2-\htot}
	\phi_{q(\lambda)/t^2}(\tauone,\tautwo)
	\\
	&\quad\times
	\overline{\Theta_{\brK,\lambda/t}}\big(\tauone,\tautwo,\borw,\exp\big(-\yone\Delta_2\det y^{-1}\big)\polw{\vect{\alpha},\borw}{0}{\htot}(\cdot,\borw(\lambda)/t)\big)
		 \\
		 &\quad\times
		 \exp\Big(-\frac{\pi t^2\yone}{2\genU_{z^\perp}^2\det y}
		 -\frac{2\pi}{t^2}\Big(
		 \lambda_{w^\perp}^2\ythree + \lambda_w^2\frac{\ytwo^2}{\yone}
		 \Big)
		 \Big)
		 \,d\ythree
		 \,\frac{d\xone\,d\yone\,d\xtwo\,d\ytwo}{\yone^3},
	\eas
	where~$\Delta_2$ is the Laplacian on the second copy of~$\RR^{b+2}$ in~$\RR^{(b+2)\times2}$.
	Here we denote by~$w$ the orthogonal complement of~$\genU_z$ in~$z$, and by~$w^\perp$ the orthogonal complement of~$\genU_{z^\perp}$ in~$z^\perp$.
	In all remaining cases the Fourier coefficients~$c_\lambda$ vanish.
	\end{thm}
	The knowledge of the Fourier coefficients of~$\intfunctshort$ does not immediately imply the injectivity of the lift.
	This makes the situation very different from the genus~$1$ case considered in~\cite{zuffetti;gen1}.
	In this article we illustrate how to further unfold the integral and deduce a Fourier expansion of~$\intfunctshort$ in terms of the \emph{Fourier coefficients} of~$f$.
	This leads us to prove injectivity results of Jacobi theta integrals, from which we deduce Theorem~\ref{thm;introinjeasy}.
	
	Before the explanation of the second unfolding, we remark an interesting behavior of~$\Theta_{L,2}$ with respect to the split~$L=\brK\oplus U$, which makes the situation of genus~$2$ more subtle with respect to the genus~$1$ case considered by Borcherds in~\cite{bo;grass}.
	Since the polynomials~$\pol_{\vect{\alpha}}$ are very homogeneous, the associated genus~$2$ Siegel theta functions~$\Theta_{L,2}(\tau,g,\pol_{\vect{\alpha}})$ are (non-holomorphic) Siegel modular forms with respect to the full modular group~$\Sp_4(\ZZ)$.
	Many of the genus~$2$ theta functions~$\Theta_{\brK,2}$ attached to~$\polw{\vect{\alpha},\borw}{\hone}{\htwo}$ \emph{fail to be modular} with respect to the whole~$\Sp_4(\ZZ)$, due to the fact that many of the latter polynomials are not very homogeneous.
	This makes the case of higher genus more subtle than the genus~$1$ case.
		
	However, we show that a suitable combination of the functions~$\Theta_{\brK,2}$ has a Fourier--Jacobi type expansion in terms of the newly defined Jacobi Siegel theta functions~$\Theta_{\brK,\lambda}$, which behave as (non-holomorphic) Jacobi forms with respect to the Jacobi subgroup~$\Gamma^J$ of~$\Sp_4(\ZZ)$.
	
	Summarizing, even if the genus~$2$ Siegel theta functions arising from~$\Theta_{L,2}(\tau,g,\pol_{\vect{\alpha}})$ in terms of a split~$L=\brK\oplus U$ are not modular with respect to~$\Sp_4(\ZZ)$, we show that it is possible to gather them and recover a modular behavior with respect to the Jacobi subgroup of~$\Sp_4(\ZZ)$.
	
	\subsection{The double unfolding of the lift and its injectivity}
	We now illustrate the idea of the proof of Theorem~\ref{thm;introinjeasy}.
	Let~$f\in S^k_2$ be a cusp form with Fourier expansion~\textcolor{\mycolor}{${f=\sum_{\substack{0\leq T\in\halfint_2}} a(T) q^T}$, where~$\halfint_2$ denotes the set of symmetric half-integral $2 \times 2$-matrices}.
	To prove the injectivity of the Kudla--Millson lift, we show that if~$\KMlift(f)=0$, then~$a(T)=0$ for every~$T$.
	
	Since the vectors~$\omegaone\wedge\dots\wedge\omegafour$ of~$\bigwedge^4 T_{z_0}^* \hermdom$ appearing in~\eqref{eq;KMdeg2liftmoreexplintro} are linearly independent, we deduce that if~$\KMlift(f)=0$, then the defining integrals~$\intfunctshort$ vanish.
	This implies that the Fourier coefficients of~$\intfunctshort$ provided by Theorem~\ref{thm;introscvalver1unf} are zero for every~$g\in G$.
	
	Let~$\lambda\in \brK$ be of positive norm.
	We construct special isometries~$\speciso\in G$ such that the vanishing of~$c_\lambda(\speciso)$ implies the vanishing of Petersson inner products of Jacobi type, as follows\textcolor{\mycolor}{; see Section~\ref{sec;finalunfvvcase} for details}.
	\begin{cor}%[\textcolor{\mycolor}{see the discussion before Theorem \ref{thm:JacobiThetaLiftInjectivity}}]
	\label{cor;introvanpetjac}
	Let~$\vect{\alpha}$ be such that~$\alphaone \neq \alphatwo$, $\betaone \neq \betatwo$ and~$\alpha_j < \beta_j$, \textcolor{\mycolor}{and let~$\lambda\in \brK$ be of positive norm}.
	There exists an isometry~$\speciso\in G$ such that if the lift~$\KMlift(f)$ vanishes, then
	\be\label{eq;inprodforinj}
	\big\langle
	\phi_{q(\lambda)} , \Theta_{\brK,\lambda}\big(\cdot,\cdot,\specisoborw,\polw{\vect{\alpha},\specisoborw}{0}{\htot}(\cdot,\specisoborw(\lambda)\big)
	\big\rangle_{\Pet}=0,
	\ee
	where~$\langle\cdot{,}\cdot\rangle_{\Pet}$ is the Petersson inner product for Jacobi forms \textcolor{\mycolor}{and $\phi_{q(\lambda)}$ denotes the Fourier-Jacobi coefficient of $f$ of index $q(\lambda)$.}
	\end{cor}
	
	This inner product is an integral over~$\Gamma^J \backslash \HH\times\CC$.
	We extend the unfolding method of Borcherds~\cite{bo;grass} to the Jacobi Siegel theta functions~$\Theta_{\brK,\lambda}$.
	We then apply a second unfolding to deduce a Fourier expansion of~\eqref{eq;inprodforinj} for general~$g'$.
	Since the formulas are similar in spirit to the ones of Theorem~\ref{thm;introscvalver1unf}, we avoid to write them here and instead refer to Theorem~\ref{thm:JacobiThetaUnfolding} for details.
	
	The coefficients of the latter Fourier expansion are given in terms of the Fourier coefficients of~$f$.
	We eventually check that if the assumptions of Theorem~\ref{thm;introinjeasy} are satisfied, then the vanishing of the Fourier coefficients of~\eqref{eq;inprodforinj} implies the vanishing of the ones of~$f$, concluding the proof of the injectivity.
	
	\subsection{Some applications}
	
	As recalled above, the Kudla--Millson lift associates to every cusp form~$f\in S^k_2$ a \emph{closed} differential form~$\KMlift(f)$ of degree~$4$ on~$X$.
	Hence the lift induces a map~$S^k_2\to H^4(X,\CC)$ to the fourth cohomology group of~$X$.
	
	Let~$H^4_{(2)}(X,\CC)$ be the cohomology of square-integrable forms on~$X$.
	The inclusion of the space of square-integrable closed~$4$-forms to the standard de Rham space of closed~$4$-forms on~$X$ induces a map~$H^4_{(2)}(X,\CC)\to H^4(X,\CC)$.
	It is known that such a map is an isomorphism if~$\dim X>5$.
	This follows from isomorphisms between the (singular) cohomology groups of~$X$ with the intersection cohomology of the Baily--Borel compactification of~$X$, known for low degrees, as well as Zucker's conjecture proved by Looijenga~\cite{looijgengaL2} and Saper--Stern~\cite{SapSt;L2coho}; see~\cite[Example~$3.4$]{blmm;conj} and~\cite[Chapter~$5$]{HarrisZucker} for further details.
	
	Let~$\mathcal{H}^4_{(2)}(X)$ be the space of square-integrable \emph{harmonic} $4$-forms on~$X$.
	By the $L^2$-version of the Hodge theorem, the natural map~$\mathcal{H}^4_{(2)}(X)\to H^4_{(2)}(X,\CC)$ is an isomorphism.
	The lift of a cusp form is a \emph{harmonic} form by~\cite[Theorem~$4.1$]{kumi;tubes}.
	Moreover, following the argument of~\cite[Proposition~$4.1$]{brfu} it is easy to see that the genus~$2$ Kudla--Millson lift gives \emph{square-integrable} forms.
	Hence, the main theorems of this article imply the following result. 
	\begin{cor}\label{cor;introapplic}
	Let~$\dim X > 9$.
	The Kudla--Millson lift induces an injective map in cohomology.
	In particular~$\dim H^4(X,\CC)\ge \dim S^k_2$, where~$k=1+\dim X/2$.
	\end{cor}
	\textcolor{\mycolor}{In~\cite{br-zu}, the injectivity result of the present paper has been used to improve Corollary~\ref{cor;introapplic} to~$\dim H^{2,2}(X,\CC)=\dim M^k_2$, where~$M^k_2$ is the space of genus~$2$ and weight~$k$ Siegel modular forms.}
	
%	Bruinier and Raum proved that the generating series of rational classes of special cycles in~$\CH^2(X)\otimes\RR$ is a Siegel modular form, where~$\CH^2(X)$ is the Chow group of codimension~$2$ cycles on~$X$; see~\cite{brra;modconj}.
%	This and Corollary~\ref{cor;introapplic} imply that also the dimension of the span of special cycles in~$\CH^2(X)\otimes\RR$ is at least~$\dim S^k_2$.
	
	\subsection{Outline of the paper}
	Section~\ref{sec;somebackground} contains some background on local-to-global principles for lattices and on orthogonal Shimura varieties.
%	Section~\ref{vvmodforms} is a quick introduction to vector-valued Siegel and Jacobi modular forms.
	
	In Section~\ref{vvmodforms} we develop the theory of vector-valued Siegel \textcolor{\mycolor}{theta functions} of genus~$2$ \textcolor{\mycolor}{and of Jacobi type} associated to a lattice~$L$.
%	Here we explain how to rewrite these theta functions with respect to the split of a hyperbolic plane from~$L$.
	
	In Sections~\textcolor{\mycolor}{\ref{sec;vvsiegeltheta2} and \ref{sec;somegenred}} we illustrate how to rewrite these theta functions with respect to the split of a hyperbolic plane from~$L$.
%	In Section~\ref{sec;somegen} we introduce Jacobi Siegel theta functions and illustrate how to unfold Jacobi theta integrals.
	
	Section~\ref{sec;deg2kmtform} begins with a quick recall of the Kudla--Millson Schwartz function and the Kudla--Millson theta function and show how to rewrite it in terms of genus~$2$ Siegel theta functions associated to certain very homogeneous polynomials of degree~$2$.
	
	In Section~\ref{sec;deg2KMlift} we apply the machinery developed in Section~\ref{sec;vvsiegeltheta2} to the Kudla--Millson lift.
	We unfold the lift and compute its Fourier expansion; see Theorem~\ref{thm;Fexpgen2}.
	
	The second unfolding is carried out in Section~\ref{sec;finalunfvvcase}, where we apply the theory of indefinite Jacobi Siegel theta functions introduced in Section~\textcolor{\mycolor}{\ref{vvmodforms}}.
	The injectivity is eventually proved with Theorem~\ref{thm:KMInjectivity}.
	
	\textcolor{\mycolor}{Appendix~\ref{sec;appendixtot}} contains ancillary technical results regarding Fourier transforms and decompositions of the polynomials~$\pol_{\vect{\alpha}}$ on subspaces of~$\RR^{(b+2)\times 2}$.
	These play a role in Sections~\ref{sec;splitgen2theta} and~\ref{sec;theunfoldingofKMliftgen2new}, in particular in the proof of Theorem~\ref{thm;unfongen2}, which enables us to unfold the defining integrals of the Kudla--Millson lift.
	
	\subsection*{Acknowledgments}
	We would like to thank Claudia Alfes-Neumann, Jan Bruinier, Jens Funke, Martin Möller and Christina Röhrig for useful discussions on the topic of the present article.
	This work started as one of the PhD projects of the second author~\cite{zuffetti;thesis}, who is grateful to Martin Möller for his patience and advise. 
	The first author is partially funded by the Research Foundation – Flanders (FWO) within the framework of the Odysseus program project number G0D9323N, and by the Deutsche Forschungsgemeinschaft (DFG, German Research Foundation) -- SFB-TRR 358/1 2023 -- 491392403.
	% and by the Collaborative Research Centre TRR 326 ``Geometry and Arithmetic of Uniformized Structures'', project number 444845124.
	The second author is partially supported by the Loewe research unit ``Uniformized Structures in Arithmetic and Geometry''.
	Both authors are partially funded by the Collaborative Research Centre TRR 326 ``Geometry and Arithmetic of Uniformized Structures'', project number 444845124.
	%%%%%%%%%%%%%%%%%%%%%%%%%%%%%%%%%%%%%%%%%%%%%%%%%%%%%%%%%%%%%%	
	\section{Lattices and orthogonal Shimura varieties}\label{sec;somebackground}
	In this short section we provide the necessary background on lattices and orthogonal Shimura varieties.
	The notation introduced here will be used in the rest of the paper.
	
	\subsection{Lattices and local-global principles}
	
	Throughout this section let $R$ be a principal ideal domain, $K$ its field of fractions and $L$ a finitely generated \textcolor{\mycolor}{free} $R$-module.
	
	\begin{defi}
		A \emph{quadratic form} on $L$ is a map $q : L \to K$ which satisfies the following two properties.
		\begin{enumerate}[label=(\roman*), leftmargin=*]
		\item If~$r \in R$ and~$x \in L$, then~$q(rx) = r^2 q(x)$.
\item The map $(x, y) \coloneqq q(x + y) - q(x) - q(y)$ for $x, y \in L$ is a symmetric bilinear form.
		\end{enumerate}
	The pair $(L, q)$ is called a \textcolor{\mycolor}{\emph{lattice}} over $R$. If $R = K$ is a field, we call it \emph{quadratic space}.
	\end{defi}
	
	\textcolor{\mycolor}{In this paper we will consider only lattices over $\IZ$ and~$\IZ_p$. If we do not specify the ring~$R$ of a lattice, we implicitly assume that~$R=\ZZ$.}

	For $x \in L$ we write $x^\perp \coloneqq \{y \in L : (x, y) = 0 \}$ for the \emph{orthogonal complement} of~$x$.
	We call~$q$ \emph{non-degenerate} if~$x^\perp \neq L$ for every~$x \in L \setminus \{0\}$. From now on we will assume that every \textcolor{\mycolor}{lattice} is non-degenerate.
	
	\begin{defi}
		Let $(L, q_L)$ and~$(M, q_M)$ be \textcolor{\mycolor}{lattices} over~$R$. An $R$-linear map~${\isometry : M \to L}$ is called \emph{isometric embedding} or a representation of $M$ by $L$ if it is injective and satisfies $q_L(\isometry(x)) = q_M(x)$ for all $x \in M$. In this case we say that~$L$ \emph{represents}~$M$. An isometric embedding that is also surjective is called \emph{isometry}. The orthogonal group $O(L)$ is the group of all isometries $\isometry : L \to L$. \textcolor{\mycolor}{It can be naturally embedded into $\GL(L \otimes K)$.}
	\end{defi}

	It is well-known that every quadratic space $V$ over the real numbers $\IR$ of dimension $n$ is isometric to a quadratic space $\IR^n$ with quadratic form
	$$q(x) = x_1^2 + \ldots + x_{b^+}^2 - x_{b^+ + 1}^2 - \ldots - x_{n}^2.$$
	We denote this quadratic space by $\IR^{b^+, b^-}$, where $ b^- = n -b^+$.
	The tuple~$(b^+, b^-)$ is uniquely determined and is called the signature of~$V$.
	
	Given a \textcolor{\mycolor}{lattice} $(L, q)$ over a ring $R$ and a ring extension~$R \to R'$, one obtains a \textcolor{\mycolor}{lattice}~$(L \otimes R', q)$.
	Here~$q(x \otimes r') = r'^2 q(x)$ for every~$x \in L$ and~$r' \in R'$.
	In particular, if~$V$ is a rational quadratic space, i.e.\ a quadratic space over~$\IQ$, we define its signature to be the signature of~$V \otimes \IR$.
	
	\begin{defi}
		\textcolor{\mycolor}{An $R$-lattice is called \emph{integral} if the bilinear form~$( \cdot{,} \cdot)$ takes values in~$R$.
			We say that~$L$ is \emph{even} if the quadratic form takes values in~$R$.}
	\end{defi}
	
	For an $R$-lattice $L$ we define its dual lattice $L'$ as
	$$L' \coloneqq \{ x \in L \otimes K : (x, y) \in R \text{ for all } y \in L \}.$$
	If $L$ is integral, then~$L \subseteq L'$.
	The quotient group~$\disc{L} \coloneqq L' / L$ is called the \emph{discriminant group} of~$L$.
	If $L$ is even, then the quadratic form on~$L$ induces a map~$\disc{L} \to K / R$, which we also denote by~$q$. \textcolor{\mycolor}{Given two discriminant groups $\disc{L}$ and $\disc{M}$, a bijective map $\isometry : \disc{L} \to \disc{M}$ is called \emph{isometry}, if $q(\isometry(x)) = q(x)$ for all $x \in \disc{L}$.}
	
	An even lattice is called \emph{unimodular} if $\disc{L}$ \textcolor{\mycolor}{is trivial}, and we call~$L$ \emph{maximal} if~$\disc{L}$ is anisotropic, i.e.\ if~$q(x) = 0$ for~$x \in \disc{L}$, then~$x = 0 \in \disc{L}$.
	
	Let~$L$ be a~$\ZZ$-lattice and~$v$ be a place of~$\IQ$, i.e.\ either~$v$ is (a non-archimedean place associated to) some prime number or (the archimedean place)~$v = \infty$.
	We define $L_v \coloneqq L \otimes \IZ_p$ if~${v = p}$, and~${L_v \coloneqq L \otimes \IR}$ if~${v = \infty}$.
	
	\begin{defi}
		We say that two $\IZ$-lattices $L$ and~$M$ are in the same genus if they are locally isometric, i.e.~$L_{v} \simeq M_{v}$ for all places $v$.
	\end{defi}

	Isometric lattices are obviously in the same genus and we write~$\gen(L)$ for the set of isometry classes of lattices in the same genus of~$L$.
	This set is always finite; see~\cite[Satz~$21.3$]{Kneser}.
	If~$L$ represents another lattice~$M$, then~$L$ also represents $M$ locally, i.e.~$L_v$ represents~$M_v$ for every place~$v$.
	The converse fails in general, but the local-global-principle asserts that if~$L$ represents~$M$ locally, then there exists a lattice~$\tilde{L} \in \gen{(L)}$ such that~$\tilde{L}$ represents~$M$; see \cite[Satz~$30.9$]{Kneser} for a slightly weaker statement.
	The following lemma is a slight refinement of this.
	
	\begin{lemma}\label{thm:LocalGlobalPrinciple}
		Let $L$ be an even lattice with dual lattice $L'$ and let $\gamma \in D_L^r$. Let $M$ be a (not necessarily integral) lattice with a fixed basis $(e_i)_{i = 1, \ldots, r}$ such that for all places $v$ there exists an isometric embedding $\isometry_v : M_v \to L_v'$ with $(\isometry_v(e_i))_{i = 1, \ldots, r} \in \gamma + L_v^r$. Then there exists an even lattice $\tilde{L} \in \gen(L)$, an isometric embedding $\isometry : M \to \tilde{L}'$ and an isometry~${\isometry' : D_L \to D_{\tilde{L}}}$ such that $(\isometry(e_i))_{i = 1, \ldots, r} \in \isometry'(\gamma) + \tilde{L}^r$.
	\end{lemma}
	
	\begin{proof}
		We follow the proof of \cite[Satz 30.9]{Kneser}. By Hasse-Minkowski \cite[Satz~19.1]{Kneser} there exists an isometric embedding~$M \otimes \IQ \to L \otimes \IQ$, so that we can assume that~$M$ is contained in~$L \otimes \IQ$. By \cite[Satz~21.5]{Kneser} for almost all places~$v$ we have~${M \subseteq L_v'}$ with~$(e_i)_{i = 1, \ldots, r} \in \gamma + L_v^r$, in fact~$L_v' = L_v$ for almost all~$v$.
		In particular, the set~$S$ of places~$v$ that do not satisfy these properties is finite. By assumption, for every~${v \in S}$ there exists a (not necessarily integral) sublattice~$N_v \subseteq L_v'$ and an isometry~$\isometry'_v : N_v \to M_v$ with~$(e_i)_{i = 1, \ldots, r} \in \isometry'_v(\gamma + L_v)$.
		This isometry can be extended to an isometry~$\isometry'_v : L_v \otimes \IQ_v \to L_v \otimes \IQ_v$ with~$\phi_v(N_v) = M_v$.
		By~\cite[Satz~21.5]{Kneser} there exists a lattice~$\tilde{L} \subseteq L\otimes \IQ$ with~$\tilde{L}_v = L_v$ for $v \notin S$ and~$\tilde{L}_v = \isometry'_v(L_v)$ for~$v \in S$.
		Then~$\tilde{L} \in \gen(L)$ and~$M_v \subseteq \tilde{L}_v'$, so that~$M \subseteq \tilde{L}'$ and by construction we have~$(e_i)_{i = 1, \ldots, r} \in \isometry'(\gamma) + \tilde{L}^r$.
	\end{proof}
	
	\begin{lemma}\label{lem:LocalRepresentability}
		Let $p$ be a prime and $L$ be an even $\IZ_p$-lattice and assume that the lattice $L$ splits $r$ hyperbolic planes. Let $\gamma + L^r \in D_L^r$ and let $M$ be a $\IZ_p$-lattice with basis $(e_i)_{i = 1, \ldots, r}$ such that $q((e_i)_{i = 1, \ldots, r}) \in q(\gamma + L^r)$. Then there exists a representation~$\isometry \colon M \to L'$ with~$(\isometry(e_i))_{i = 1, \ldots, r} \in \gamma + L^r$.
	\end{lemma}
	
	\begin{proof}
		We have $L = U^r \oplus \tilde{L}$, where $U$ is a hyperbolic plane and $\tilde{L}$ is an even lattice.
		We write~${f_i, f_i'}$ for a standard basis of the $i$th hyperbolic plane with~$(f_i, f_i') = 1$ and~$q(f_i) = q(f_i') = 0$.
		Since $D_L = D_{\tilde{L}}$, we may assume $\gamma \in \tilde{L}'$ and define
		$$\tilde{\gamma}_i \coloneqq \gamma_i + f_i + \frac{1}{2}\sum_{j = 1}^{r} \big((e_i, e_j) - (\gamma_i, \gamma_j)\big) f_j' \in \gamma_i + L.$$
		We obtain the isometric embedding~$\isometry$ by mapping the basis element $e_i$ to $\tilde{\gamma}_i$ for~${i = 1, \ldots, r}$.
	\end{proof}
	\begin{rem}\label{rem:NikulinApplication}
		\textcolor{\mycolor}{For an even lattice $L$ let $l(L)$ be the minimal number of generators of $L' / L$. If $L$ splits a hyperbolic plane we have that the rank of $L$ is at least $l(L) + 2$, so that the conditions of~\cite[Theorem 1.14.2]{Nikulin} are satisfied. Therefore we have $\lvert \gen(L) \rvert = 1$ and the map $O(L) \to O(\disc{L})$ is surjective.}
	\end{rem}
		
	\begin{cor}\label{cor:Representations}
		Let $L^+$ be an even positive definite lattice with dual lattice $L^{+'}$ and assume that for all primes~$p$ the lattice $L^+_p$ splits~$r$ hyperbolic planes.
		 Let~$L = U \oplus L^+$.
		 Then for every~$\gamma \in D_L^r$ and positive definite symmetric matrix~$T \in q(\gamma)$ there exist a splitting~$L = U \oplus \tilde{L}^+$ with~$\tilde{L}^+ \in \gen(L^+)$ and~$\lambda \in (\tilde{L}^{+'})^r$ with~$\lambda \in \gamma + L^r$ and~$q(\lambda) = T$.
	\end{cor}
	
	\begin{proof}
		The positive definite symmetric matrix~$T \in q(\gamma)$ corresponds to a positive definite lattice~$M$ with basis~$(e_i)_{i = 1, \ldots, r}$ such that~$q((e_i)_{i = 1, \ldots, r}) = T$.
		By Lemma~\ref{lem:LocalRepresentability}, the lattice~$L^+$ represents~$M$ locally and thus, by Lemma \ref{thm:LocalGlobalPrinciple}, there exist an even lattice~$\tilde{L}^+ \in \gen(L^+)$, an isometric embedding~$\isometry \colon M \to \tilde{L}^{+'}$ and an isometry~$\isometry' : D_{L^{+}}  \to D_{\tilde{L}^{+}}$ with~$(\isometry(e_i))_{i = 1, \ldots, r}$ in~$\isometry'(\gamma) + (\tilde{L}^+)^r$.
		\textcolor{\mycolor}{By Remark~\ref{rem:NikulinApplication}} we have an isometry
		$$U \oplus \tilde{L}^+ \to U \oplus L^+ = L,$$
		so that we can assume $U \oplus \tilde{L}^+ = L$ and $\isometry' \colon D_L \to D_L$. \textcolor{\mycolor}{Again by Remark~\ref{rem:NikulinApplication}} the map $O(L) \to O(D_L)$ is surjective and thus there exists some $\tilde{\isometry}' \in O(L)$ mapping to $\isometry'$. Hence, replacing $\isometry$ by $\tilde{\isometry}'^{-1} \circ \isometry$ and $\tilde{L}^+$ by $\tilde{\isometry}'^{-1}(\tilde{L}^+)$, we obtain
		$$\isometry((e_i)_{i = 1, \ldots, r}) \in \gamma + L^r$$
		and $\isometry((e_i)_{i = 1, \ldots, r}) \in (\tilde{L}^{+'})^r$. Moreover, $\lambda = \isometry((e_i)_{i = 1, \ldots, r})$ satisfies $q(\lambda) = T$.
	\end{proof}

%				
%	\begin{lemma}\label{lem:unimodularsrepresentation}
	%		For every positive definite half-integral $2 \times 2$ matrix $T$ there exist $\lambda, \jacindexvec \in L^+$ such that
	%		$$T = q(\jacindexvec, \lambda) = \begin{pmatrix}q(\jacindexvec) & (\jacindexvec,\lambda)/2\\
		%			(\jacindexvec,\lambda)/2 & q(\lambda)\end{pmatrix}.$$
	%	\end{lemma}
%	
%	\begin{proof}
	%		We will prove this for $L^+ = E_8$. The general case then follows since we assumed that $L^+$ is a direct sum of $E_8$. The existence of $\lambda, \jacindexvec \in L^+$ with the above properties for a positive definite half-integral $T$ is equivalent to the existence of an embedding of the positive definite even lattice $\mathbb{Z}^2$ with gram matrix $2T$. According to \cite[Theorem 1.12.4]{Nikulin}, there exists a primitive embedding of every even lattice of signature $(t_+, t_-)$ into some even unimodular lattice of signature $(l_+, l_-)$ if and only if
	%		$$l_+ - l_- \equiv 0 \mod{8}, \quad t_+ \leq l_+, t_- \leq l_-, \quad t_+ + t_- \leq \frac{l_+ + l_-}{2}.$$
	%		In our case we have $t_+ = 2, l_+ = 8, t_- = l_- = 0$, so that all these inequalities are satisfied and thus for every even lattice of signature $(2, 0)$ there exists a primitive embedding into some even unimodular lattice of signature $(8, 0)$. Moreover, the lattice $L^+ = E_8$ is the only even unimodular lattice of signature $(8, 0)$. This shows the assertion.
	%	\end{proof}
	
	\subsection{Orthogonal Shimura varieties}\label{sec;OSV}
	Let~$L$ be a (non-degenerate) even lattice of signature~$(b,2)$, for some~$b>0$.
	The Grassmannian associated to~$L$ is the set of negative definite planes in~$V=L\otimes\RR$, namely
	\bes
	\Gr(L)=\{z\subset V : \text{$\dim z=2$ and $(\cdot{,}\cdot)|_z<0$}\}.
	\ees
	The Hermitian symmetric space~$\hermdom$ attached to~$V$ may be identified with~$\Gr(L)$; see~\cite[Part~$2$, Section~$2.4$]{1-2-3}.
	From now on, we write~$\hermdom$ and~$\Gr(L)$ interchangeably.
	
	Let~$\widetilde{\SO}(L)$ be the discriminant kernel of~$\SO(L)$, namely the kernel of the natural homomorphism~$\SO(L)\to\Aut(\disc{L})$.
	An \emph{orthogonal Shimura variety} is a quotient of the form~$X=\Gamma\backslash\domain$ for some \textcolor{\mycolor}{finite index} subgroup~$\Gamma\subseteq\widetilde{\SO}(L)$.
	
	By the Theorem of Baily and Borel, the locally symmetric space~$X$ admits a unique algebraic structure that makes it a quasi-projective variety of dimension~$b$.
	This variety may be possibly compact only if~$b\le 2$.
	%%%%%%%%%%%%%%%%%%%%%%%%%%%%%%%%%%%%%%%%%%%%%%%%%%%%%%%%%%%%%%
	\section{\textcolor{\mycolor}{Vector-valued Siegel and Jacobi modular forms}}\label{vvmodforms}
	Let~$L$ be an even indefinite lattice of signature $(b,2)$.
	In this section we introduce certain vector-valued genus~$2$ Siegel theta functions and Jacobi Siegel theta functions attached to~$L$.
	These may be regarded as generalizations of the Siegel theta functions~$\Theta_L$ introduced by Borcherds in~\cite[Section~4]{bo;grass}.
	\textcolor{\mycolor}{We also recall vector-valued holomorphic Siegel and Jacobi modular forms.
	For simplicity, we restrict here to the case of genus~$2$ Siegel modular forms, since this is the only case needed in the rest of the present paper.}
	\\
	
	To simplify the notation, we \textcolor{\mycolor}{put~$e(t)=\exp(2\pi i \tr(t))$ for every~$t\in\CC^{n\times n}$, and for $t \in \CC$} denote by~${\sqrt{t}=t^{1/2}}$ the principal branch of the square root, so that~$\arg(\sqrt{t})\in(-\pi/2,\pi/2]$.
	If~$s\in\CC$, we define~$t^s=e^{s\Log(t)}$, where~$\Log(t)$ is the principal branch of the logarithm.
	If~$M$ is a matrix, we denote by~$M^t$ its transpose, and whenever~$M$ is invertible, we denote by~$M^{-t}$ the inverse of~$M^t$.
	
	Let~$V=L\otimes\RR$.
	We fix once and for all an orthogonal basis~$(\basevec_j)_j$ of~$V$ such that~$(e_j,e_j)=1$, for every~$j=1,\dots,b$, and~$(e_j,e_j)=-1$ for~$j=b+1,b+2$.
%	We denote the Grassmannian~$\Gr(V)$ also by~$\Gr(L)$.
	For every $\vect{\genvec}=(\genvec_1,\genvec_2)\in V^2$ we denote by~$x_{i,j}$ the coordinate of~$\genvec_j$ with respect to~$\basevec_i$, where~$j=1,2$ and~$i=1,\dots,b+2$.
	Note that we consider the elements of~$V^2$ as row vectors.
	
	We denote by $g_0\colon L\otimes\RR\to\RR^{b,2}$ the standard isometry induced by the choice of the basis~$(\basevec_j)_j$, and by~$G$ the isometry group $\SO(V)$.
	By a slight abuse of notation, we denote by $g_0$ also the isometry applied componentwise on~$V^2$ as~$g_0\colon(\genvec_1,\genvec_2)\mapsto(g_0(\genvec_1),g_0(\genvec_2))$.
	We use the same notation also for the isometries~${g\in G}$ acting on Cartesian products of~$V$.
	We consider the image of~$\vect{v}\in V^2$ under~$g_0$ as a~$(b+2)\times 2$ matrix, writing it as
	\be\label{eq;stgatwomat}
	g_0(\vect{\genvec})=\left(\begin{smallmatrix}
	x_{1,1} & x_{1,2}\\
	\vdots & \vdots\\
	x_{b+2,1} & x_{b+2,2}
	\end{smallmatrix}\right)\in(\RR^{b,2})^2.
	\ee
	
%	The Grassmannian associated to $V$ is the set of negative definite planes in~$V$, namely
%	\bes
%	\Gr(V)=\{z\subset V : \text{$\dim z=2$ and $(\cdot{,}\cdot)|_z<0$}\}.
%	\ees
%	The subspace~$z_0$ spanned by~$\basevec_{b+1}$ and~$\basevec_{b+2}$ is the \emph{base point} of~$\Gr(V)$.
%	The Hermitian symmetric space~$\hermdom$ attached to~$V$ may be identified with~$\Gr(V)$; see~\cite[Part~$2$, Section~$2.4$]{1-2-3}.
%	From now on, we write~$\hermdom$ and~$\Gr(V)$ interchangeably.
	
	For every $\vect{\genvec}=(\genvec_1,\genvec_2)\in V^2$, we define the projection of $\vect{\genvec}$ with respect to $z\in\Gr(L)$ by
	\bes
	\vect{\genvec}_z=\big((\genvec_1)_z,(\genvec_2)_z\big),
	\ees
	that is, the projection is considered componentwise.
	Moreover, we write
	\bes
	\vect{\genvec}^2=(\vect{\genvec},\vect{\genvec})=\begin{pmatrix}
	\genvec_1^2 & (\genvec_1,\genvec_2)\\
	(\genvec_1,\genvec_2) & \genvec_2^2
	\end{pmatrix}
	\ees
	to denote the \emph{matrix of inner products} of the entries of~$\vect{\genvec}$, and analogously~$q(\vect{\genvec})=\frac{1}{2}\vect{\genvec}^2$.
	
	Lastly, for fixed~$g\in G$, we denote by~$z=g^{-1}(z_0)\in\Gr(L)$ the negative-definite plane mapping to~$z_0$ under~$g$

	\subsection{Very homogeneous polynomials and differential operators}\label{sec;veryhompol}
	\textcolor{\mycolor}{The Siegel theta series we will employ to prove the injectivity of the Kudla--Millson theta lift are special theta series arising from certain homogeneous polynomials on~$(\RR^{b,2})^2$.}
	To illustrate such notion of homogeneity, we need to introduce another piece of notation.
	Let $(g_0(\basevec_j))_j$ be the standard basis of the quadratic space~$\RR^{b,2}$.
	For every vector~${x=\sum_{j=1}^{b+2}x_jg_0(\basevec_j)\in \RR^{b,2}}$, we define~${x^+=\sum_{j=1}^bx_jg_0(\basevec_j)}$ and~${x^-=\sum_{j=b+1}^{b+2}x_jg_0(\basevec_j)}$.
	For every $\vect{x}=(x_1,x_2)\in (\RR^{b,2})^2$, we define~${\vect{x}^+=(x_1^+,x_2^+)}$ and~$\vect{x}^-=(x_1^-,x_2^-)$.
	
	\begin{defi}\label{def;homogforrb22}
	We say that a polynomial $\pol\colon(\RR^{b,2})^2\to\CC$ is \emph{very homogeneous of degree~$(m^+,m^-)$} if it splits as a product of two polynomials $\pol(\vect{x})=\pol_b(\vect{x}^+)\pol_2(\vect{x}^-)$ such that
	\bes
	\pol_b(\vect{x}^+N)=(\det N)^{m^+}\pol_b(\vect{x}^+)\qquad\text{and}\qquad\pol_2(\vect{x}^-N)=(\det N)^{m^-}\pol_2(\vect{x}^-),
	\ees
	for every $N\in\CC^{2\times2}$.
	\end{defi}
	This homogeneity property is the same as the one introduced in~\cite{roehrig}.
	Very homogeneous polynomials are a (not necessarily harmonic) generalization of the ``harmonic forms'' defined by Freitag~\cite[Definition~$3.5$]{freitag} and Maass~\cite{maassharmfo} to indefinite quadratic spaces.
	To avoid confusion with the harmonic differential forms on the Hermitian domain~$\hermdom$, we decided to refer to such polynomials with a different terminology.
	
	\textcolor{\mycolor}{
	An example of a very homogeneous harmonic polynomial of degree~$(1,0)$ for~$b\geq 2$ is~$\pol(\vect{x})=\det\big(\begin{smallmatrix}
	x_{1,1} & x_{1,2} \\
	x_{2,1} & x_{2,2}
	\end{smallmatrix}\big)$.
	Large powers of it provide examples of very homogeneous non-harmonic polynomials; see Lemma~\ref{lemma;homogpolabcddeg2} and~\cite[Remark~$3.3$]{roehrig} for further examples.}
	
	\begin{rem}\label{rem;veryhomogpolprop}
	Let $\pol$ be a very homogeneous polynomial on $(\RR^{b,2})^2$ of degree $(m^+,m^-)$,
	and let $N=\left(\begin{smallmatrix} \lambda & 0\\ 0 & \lambda\end{smallmatrix}\right)$, for some $\lambda\in\RR\setminus\{0\}$.
	For every $\vect{x}=(x_1,x_2)\in(\RR^{b,2})^2$, we have
	\bes
	\pol_b\big(\lambda \vect{x}^+\big)=\pol_b\big(\vect{x}^+N\big)
	=
	\det N^{m^+}\pol_b(\vect{x}^+)=\lambda^{2m^+}\pol_b(\vect{x}).
	\ees
	The case of~$\pol_2(\vect{x}^-)$ is analogous.
	We have just shown that the polynomials~$\pol_b$ and~$\pol_2$ are homogeneous of even degree in the classical sense, if considered as polynomials on~${(\RR^{b,0})^2\cong\RR^{2b}}$ and~${(\RR^{0,2})^2\cong\RR^4}$ respectively.
	An analogous procedure shows that~$\pol_b$ is homogeneous of degree~$m^+$ on each copy of~$\RR^{b,2}$ in~$(\RR^{b,2})^2$.
	A similar statement holds for~$\pol_2$ as well.
	\end{rem}
	
	Let~$\Delta$ be the standard Laplacian on~$(\RR^{b,2})^2$ defined as
	\be\label{eq;strdlapl}
	\Delta=\left(\frac{\partial}{\partial\vect{x}}\right)^t\cdot\frac{\partial}{\partial\vect{x}},\qquad\text{where}\quad \frac{\partial}{\partial\vect{x}}=\left(
	\frac{\partial}{\partial x_{i,j}}
	\right)_{1\le i\le b+2,\, 1\le j\le 2}.
	\ee
	We consider
	\be\label{eq;deftrdandexpop}
	\trace\Delta=\sum_{i=1}^{b+2}\sum_{j=1}^{2}\frac{\partial^2}{\partial x_{i,j}^2}\qquad\text{and}\qquad\exp\Big(-\frac{1}{8\pi}\trace(\Delta y^{-1})\Big)=\sum_{m=0}^\infty\frac{1}{m!}\Big(\frac{-\trace (\Delta y^{-1})}{8\pi}\Big)^m,
	\ee
	for any symmetric positive-definite matrix $y\in\RR^{2\times 2}$, as operators acting on the space of smooth~$\CC$-valued functions on~$(\RR^{b,2})^2$\textcolor{\mycolor}{; see~\eqref{eq;explicitformforexpopgen2bor} for an explicit expression of~$\trace (\Delta y^{-1})$}.
	We say that a smooth function~$f\colon(\RR^{b,2})^2\to\RR$ is \emph{harmonic} if~$\trace\Delta f=0$.
	
	\begin{rem}\label{rem;fromFreitagGerman}
	If a very homogeneous polynomial~$\pol$ is harmonic, then~$\Delta\pol=0$, namely
	\bes
	\sum_{j=1}^{b+2}\frac{\partial^2 \pol}{\partial x_{j,\eta} \partial x_{j,\xi}}=0,\qquad\text{for every~$1\le\eta$, $\xi\le 2$};
	\ees
	see e.g.~\cite[Bemerkung~$3.3$]{freitag}.
	This implies that
	\bes
	\exp \Big(
	-\frac{1}{8\pi}\trace(\Delta y^{-1})
	\Big)(\pol)=\pol.
	\ees
	This is analogous to the case of homogeneous harmonic polynomials in the genus~$1$ case; see~\cite[Remark~$3.2$]{zuffetti;gen1}.
	\end{rem}

	\subsection{\textcolor{\mycolor}{The Weil representation and theta functions}}
	
	We recall here the Schrödinger model~$\omega_{\infty,2}$ of the genus~$2$ Weil representation on the space of Schwartz functions on a quadratic space of general signature; see~\cite[Section~4]{fm;cycleshyp} and~\cite[Section~7]{fm;cycleswith} for further details.
	We then use it to construct Siegel theta functions of genus~$2$.\\
	
	Let~$V$ be a real quadratic space of signature~$(b^+,b^-)$.
	Only at a later point we will restrict to the signature of interest for the present paper, i.e.~$b^-=2$.
	We write~$\Mp_4(\RR)$ for the metaplectic group of order~$4$, namely the double cover of~$\Sp_4(\RR)$ realized by the two choices of holomorphic square roots of~$\tau\mapsto\det(C\tau+D)$ for every~$\big(\begin{smallmatrix}
	A & B\\
	C & D
	\end{smallmatrix}\big)\in\Sp_4(\RR)$.
	Hence, it consists of elements of the form
	\[
	\Big(\big(\begin{smallmatrix}
	A & B\\ C & D
	\end{smallmatrix}\big), \pm\det(C\tau + D)^{1/2}
	\Big),\qquad
	\text{where~$\big(\begin{smallmatrix}
	A & B\\ C & D
	\end{smallmatrix}\big)\in\Sp_4(\RR)$.}
	\]
	Occasionally, we just write~$\big(\begin{smallmatrix}
	A & B\\ C & D
	\end{smallmatrix}\big)$ for~$\big(\big(\begin{smallmatrix}
	A & B\\ C & D
	\end{smallmatrix}\big),\det(C\tau + D)^{1/2}\big)$ if it is clear from the context that we are working with the metaplectic cover.
	The group operation in~$\Mp_4(\RR)$ is given by~$\big(M,f(\cdot)\big)\cdot \big(N,g(\cdot)\big) =\big(MN , f(N(\cdot))g(\cdot)\big)$; see~\cite[Section~$2.1$]{zh;phd} for further information.
	The group~$\Mp_4(\ZZ)$ is the inverse image of~$\Sp_4(\ZZ)$ under the covering map.
	For simplicity we denote by~$S$ the element~$S=\big(\big(\begin{smallmatrix}
	0 & -I_2\\
	I_2 & 0
	\end{smallmatrix}\big),\det\tau^{1/2}\big)$ of~$\Mp_4(\ZZ)$.
	Note that~$S^2=Z$, where~$Z=(-I_4,-1)$ lies in the center of~$\Mp_4(\ZZ)$.
	
	\begin{defi}\label{def;schrmoddeg2def}
		The \emph{Schrödinger model}~$\omega_{\infty,2}$ of the Weil representation provides an action of~$\Mp_4(\RR)\times\bigO(V)$ on the space~$\mathcal{S}(V^2)$ of Schwartz functions on~$V^2$ as follows.
		The action of~$\bigO(V)$ is given by
		\begin{align*}
			\omega_{\infty,2}(g)\varphi(\vect{v})=\varphi\big(g^{-1}(\vect{v})\big),
		\end{align*}
		for every~$\varphi\in\mathcal{S}(V^2)$ and~$g\in\bigO(V)$.
		The action of~$\Mp_4(\RR)$ is given by
		\ba\label{eq;schrmoddeg2def}
			\omega_{\infty,2}&\big(\begin{smallmatrix}
				A & 0\\
				0 & {A}^{-t}
			\end{smallmatrix}\big)
			\varphi(\vect{\genvec})=(\det A)^{(b^++b^-)/2}\varphi(\vect{\genvec}A),\qquad\text{for every
%			$m(A)\coloneqq\big(\begin{smallmatrix}
%				A & 0\\
%				0 & {A}^{-t}
%			\end{smallmatrix}\big),\det A^{1/2}$ with
			$A\in\GL^+_2(\RR)$,}\\
			\omega_{\infty,2}&\big(\begin{smallmatrix}
				I_2 & B\\
				0 & I_2
			\end{smallmatrix}\big)\varphi(\vect{\genvec})=e\big(\trace B q(\vect{\genvec})\big)\varphi(\vect{\genvec}),\qquad\text{for every $B\in\Sym_2(\RR)$,}\\
			\omega_{\infty,2}&(S)\varphi(\vect{\genvec})=i^{b^--b^+} \widehat{\varphi}(\vect{v}),
		\ea
		where~$\widehat{\varphi}$ is the Fourier transform of~$\varphi$ \textcolor{\newcolor}{normalized as in Appendix \ref{sec;Ftransfgen2}}.
	\end{defi}
%	\begin{ex}
%	The \emph{standard Gaussian}~$\stgatwo$ of~$(\RR^{b,2})^2$ is defined as
%	\bes
%	\stgatwo(\vect{x})=\exp\Big(
%	-\pi\sum_{i=1}^{b+2}\sum_{j=1}^2 x_{i,j}^2
%	\Big),\qquad\text{for every~$\vect{x}=(x_1,x_2)\in(\RR^{b,2})^2$},
%	\ees
%	where~$x_j=(x_{1,j},\dots,x_{b+2,j})^t\in \RR^{b,2}$.
%	The standard Gaussian of $V^2$ is the composition of~${\stgatwo}$ with~$g_0$, where the latter is as in~\eqref{eq;stgatwomat}.
%	For simplicity we will denote the standard Gaussian of~$V^2$ with~$\stgatwo$ in place of~$\stgatwo\circ g_0$.
%	\end{ex}
	
	Let~$\vect{\boralpha},\vect{\borbeta}\in V^2$.
	For any~$\varphi \in \calS(V^2)$ we define a new the Schwartz function~$\cha{\varphi}$ with characteristics $(\vect{\boralpha},\vect{\borbeta})$ by\textcolor{\newcolor}{
	$$
	\cha{\varphi}(\vect{v}, \vect{\boralpha},\vect{\borbeta})
	\coloneqq
	\varphi(\vect{v} + \vect{\borbeta}) e((-\vect{v} + \vect{\borbeta}/2, \vect{\boralpha})).
	$$}
	Then a straight forward calculation using Lemma~\ref{lemma;somepropFtr} yields\textcolor{\newcolor}{
		\begin{align}\label{eq;strfwo}
			\omega_{\infty,2}
			(\widetilde{M})(\cha{\varphi})(\vect{\genvec},\vect{\boralpha},\vect{\borbeta})= \cha{(\omega_{\infty,2}(\widetilde{M})\varphi)}(\vect{\genvec},(\vect{\boralpha},\vect{\borbeta})\widetilde{M}^{t})
		\end{align}
	for all~$\widetilde{M}=(M,\phi)\in\Mp_4(\RR)$, where we simply denote by~$(\vect{\boralpha},\vect{\borbeta})\widetilde{M}^{t}$ the tuple~$(\vect{\boralpha},\vect{\borbeta})M^t$. From now on, we will write $M \in \Mp_4(\RR)$ instead of $\widetilde{M}$.}

	We now construct theta series with characteristics~$(\vect{\boralpha},\vect{\borbeta}) \in V^2$ attached to lattice co-sets of~$L^2$, where~$L$ is an even lattice such that~$V=L\otimes\RR$.
	\begin{defi}
	The \emph{Siegel theta function of index~$\sigma\in\disc{L}^2$ and characteristics~$(\vect{\boralpha},\vect{\borbeta})\in V^2$} is defined as
	\[
	\theta_{L, 2}^\sigma(M, \vect{\boralpha}, \vect{\borbeta}, g, \varphi)
	\coloneqq
	\sum_{\vect{\lambda} \in L^2 + \sigma} (\omega_{\infty,2}(M) \cha{\varphi})\big(g(\vect{\lambda}, (\vect{\boralpha},\vect{\borbeta})M^{-t})\big)
	\]
	for every $M \in \Mp_4(\IR)$, $ g \in G$ and $\varphi \in \calS(V^2)$.
	\end{defi}
%	\textcolor{red}{\textbf{Problem:} I think that the~$g^{-1}\vect{\lambda}$ appearing above should be replaced by~$g\vect{\lambda}$. In fact, in our definition of~$\Theta_{L,2}$ the general summand is of the form
%	\begin{align*}
%	\exp \Big(
%	-\frac{1}{8\pi}\trace(\Delta y^{-1})
%	\Big)(\pol)\big(g(\vect{\lambda}+\vect{\borbeta})\big)
%	\cdot e\Big(
%	\trace\big( q((\vect{\lambda}+\vect{\borbeta})_{z^\perp})\tau\big)+
%	\trace\big( q((\vect{\lambda}+\vect{\borbeta})_z)\bar{\tau}\big)-
%	\trace(\vect{\lambda}+\vect{\borbeta}/2,\vect{\boralpha})
%	\Big),
%	\end{align*}
%	and in the argument of~$\exp \Big(
%	-\frac{1}{8\pi}\trace(\Delta y^{-1})
%	\Big)(\pol)$ we have~$g(\vect{\lambda}+\vect{\borbeta})$, not~$g^{-1}(\vect{\lambda}+\vect{\borbeta})$.}\\
%	The definition above is stated in signature~$(b,2)$, which is the case of interest in the present paper, but one can easily generalize it to arbitrary signatures.
	The real analytic functions~$\theta_{L, 2}^\sigma$ behave with respect to the action of the generators of~$\Mp_4(\ZZ)$ as follows.
%	Recall that the quadratic form~$q$ of~$L$ induces a~$\QQ/\ZZ$-valued quadratic form on the discriminant group~$\disc{L}$, which we still denote by~$q$.
	If~$A\in\SL_2(\ZZ)$, then
	\bas
	\theta_{L, 2}^\sigma\Big(\big(\begin{smallmatrix}
			A & 0\\
			0 & {A}^{-t}
		\end{smallmatrix}\big)M, (\vect{\boralpha}, \vect{\borbeta})\big(\begin{smallmatrix}
		A & 0\\
		0 & {A}^{-t}
		\end{smallmatrix}\big)^t, g, \varphi\Big)
		&=
		\theta_{L, 2}^{\sigma A}(M, \vect{\boralpha}, \vect{\borbeta}, g, \varphi).
	\eas
	If $B\in\Sym_2(\ZZ)$, then
	\bas
	\theta_{L, 2}^\sigma\Big(\big(\begin{smallmatrix}
			I_2 & B\\
			0 & I_2
		\end{smallmatrix}\big)M, (\vect{\boralpha}, \vect{\borbeta})\big(\begin{smallmatrix}
		I_2 & B\\
		0 & I_2
		\end{smallmatrix}\big)^t, g, \varphi\Big)
		=
		e(q(\sigma) B) \cdot \theta_{L, 2}^\sigma\big(M, \vect{\boralpha}, \vect{\borbeta}, g, \varphi\big).
	\eas
	Furthermore, by Poisson summation formula we have that
	\begin{align*}
		\theta_{L, 2}^\sigma\Big(\big(\begin{smallmatrix}
			0 & -I_2\\
			I_2 & 0
		\end{smallmatrix}\big)M, (\vect{\boralpha}, \vect{\borbeta}) \big(\begin{smallmatrix}
		0 & -I_2\\
		I_2 & 0
		\end{smallmatrix}\big)^t, g, \varphi\Big)
		&=
		\frac{i^{b^- - b^+}}{\lvert D_L \rvert} \sum_{\sigma' \in D_L^2} e(-\sigma', \sigma) \cdot \theta_{L, 2}^{\sigma'}\big(M, \vect{\boralpha}, \vect{\borbeta}, g, \varphi\big).
	\end{align*}
	These equations define a representation~$\rho_{L,2}$ of $\Mp_4(\IZ)$ on the group algebra $\IC[D_L^2]$.
	Let~$(\mathfrak{e}_{\gendisc})_{{\gendisc}\in\disc{L}^2}$ be the standard basis of the group algebra~$\CC[\disc{L}^2]$, and let~$\langle\cdot{,}\cdot\rangle$ be the standard Hermitian scalar product on~$\CC[\disc{L}^2]$ defined as
	\[
	\Big\langle
	\sum_{{\gendisc}\in\disc{L}^2}\lambda_{\gendisc}\mathfrak{e}_{\gendisc},\sum_{{\gendisc}\in\disc{L}^2}\mu_{\gendisc}\mathfrak{e}_\gendisc
	\Big\rangle
	=
	\sum_{{\gendisc}\in\disc{L}^2}\lambda_{\gendisc}\overline{\mu_{\gendisc}}.
	\]
	The Weil representation~$\rho_{L,2}$ is unitary with respect to such scalar product; see~\cite{zh;phd} and~\cite{br-zu} for further details on~$\rho_{L,2}$ and generalizations.
	
	With the notation introduced above, we may rephrase the behavior of the functions~$\theta^\sigma_{L,2}$ with respect to~$\Mp_4(\ZZ)$ by saying that the vector-valued theta function
	$$
	\Theta_{L, 2}(M, \vect{\boralpha}, \vect{\borbeta}, g, \varphi) 
	\coloneqq
	\sum_{\sigma \in D_L^2} \theta_{L, 2}^\sigma(M, \vect{\boralpha}, \vect{\borbeta}, g, \varphi) \frake_\sigma
	$$
	satisfies
	\be\label{eq;thetawithgrvar}
	\Theta_{L,2}(\gamma M, (\vect{\boralpha}, \vect{\borbeta})\gamma^t, g, \varphi) = \rho_{L,2}(\gamma) \Theta_{L, 2}(M, \vect{\boralpha}, \vect{\borbeta}, g, \varphi)
	\ee
	for all $\gamma \in \Mp_4(\IZ)$.
	
	\subsection{Siegel modular forms of genus~2}\label{sec;siegmodg2aft}
	In this section we employ the representation~$\rho_{L,2}$ of~$\Mp_4(\ZZ)$ on~$\CC[\disc{L}^2]$ to construct vector-valued holomorphic Siegel modular forms of genus~$2$.

%	
%	
%	\begin{defi}
%	The genus~$2$ Weil representation associated to~$L$ is the representation
%	\[
%	\weil{L}\colon \Mp_4(\ZZ)\longrightarrow\Aut(\CC[\disc{L}^2])
%	\]
%	defined on the standard generators of~$\Mp_4(\ZZ)$ as
%	\begin{align*}
%	\weil{L}\Big(\Big(\big(\begin{smallmatrix}
%	A & 0\\
%	0 & A^{-t}
%	\end{smallmatrix}\big),
%	\det(A)^{1/2}\Big)\Big)\mathfrak{e}_{\gendisc}
%	&= \det(A)^{(b^--b^+)/2} \mathfrak{e}_{{\gendisc} A^{-1}}
%	\qquad\text{for every~$A\in\GL_2(\ZZ)$,}\\
%	\weil{L}\Big(\Big(\big(\begin{smallmatrix}
%	I_2 & B\\ 0 & I_2
%	\end{smallmatrix}\big), 1 \Big)\Big)\mathfrak{e}_{\gendisc}
%	& =
%	\mathfrak{e}_{\gendisc}\big(
%	\trace(q({\gendisc})B)
%	\big)
%	\qquad\text{for every~$B=B^t\in\ZZ^{2\times 2}$,}
%	\\
%	\weil{L}\Big(\Big(\big(\begin{smallmatrix}
%	0 & -I_2\\ I_2 & 0
%	\end{smallmatrix}\big), \det\tau^{1/2}\Big)\Big)\mathfrak{e}_{\gendisc}
%	& =
%	\frac{i^{b^- - b^+}}{|\disc{L}|}
%	\sum_{{\gendisc}'\in\disc{L}^2}
%	\mathfrak{e}_{{\gendisc}'}\big(-\trace( {\gendisc}',{\gendisc})\big).
%	\end{align*}
%	This is a unitary representation with respect to the standard scalar product of~$\CC[\disc{L}^2]$.
%	\end{defi}

	Let~$k\in\frac{1}{2}\ZZ$.
	A \emph{weight~$k$ Siegel modular form of genus~$2$ with respect to the Weil representation~$\weil{L}$} is a holomorphic function~$f\colon\HH_2\to\CC[\disc{L}^2]$ such that
	\[
	f(\gamma\cdot\tau)=\phi(\tau)^{2k}\weil{L}(\gamma)f(\tau)\qquad\text{for all~$\gamma=\big(\big(\begin{smallmatrix}
	A & B\\ C & D
	\end{smallmatrix}\big),\phi(\tau)\big)\in\Mp_4(\ZZ)$},
	\]
	where~$\gamma\cdot\tau=(A\tau+B)(C\tau+D)^{-1}$.
	We denote the scalar components of~$f$ by~$f_{\gendisc}$, so that~$f=\sum_{{\gendisc}\in\disc{L}^2} f_{\gendisc} \mathfrak{e}_{\gendisc}$.
	
	The invariance of~$f$ with respect to translations in~$\Mp_4(\ZZ)$ implies that the functions~${e\big(- q({\gendisc})\tau\big)f_{\gendisc}(\tau)}$ are periodic with respect to integral symmetric~$2\times 2$-matrices, for every~${\gendisc}\in\disc{L}^2$.
	Hence~$f$ admits a Fourier expansion of the form
	\be\label{eq;FexpofSiegmod}
	f(\tau)=\sum_{{\gendisc}\in\disc{L}^2}\sum_{\substack{T\in\halfint_2+q({\gendisc})\\ T\ge 0}} a_f(\gendisc,T) \cdot \mathfrak{e}_{\gendisc}(T\tau),
	\ee
	where~$\halfint_2$ is the set of symmetric half-integral~$2\times 2$-matrices and~$\mathfrak{e}_\gendisc(t)\coloneqq e(t)\mathfrak{e}_\gendisc$.
	If it is clear that the Fourier coefficients are referred to a given modular form~$f$, then we drop the index~$f$ and denote the coefficients simply by~$a(\gendisc,T)$.
	
	We say that a Siegel modular form is a \emph{cusp form} if all its Fourier coefficients indexed over degenerate matrices vanish.
	We define~$M^k_{2,L}$, resp.~$S^k_{2,L}$, to be the space of weight~$k$ and genus~$2$ Siegel modular forms, resp.\ cusp forms, with respect to the Weil representation~$\weil{L}$.

	\subsection{\textcolor{\mycolor}{Siegel theta functions of genus 2 \textcolor{\newcolor}{\`a} la Borcherds}}\label{sec;Siegthealabo}
	Let $\textcolor{\newcolor}{K_\infty'} \subseteq \Mp_4(\IR)$ be the preimage of~$\UU(2) \coloneqq \big\{ \big(\begin{smallmatrix}
		A & B \\ -B & A
	\end{smallmatrix}\big) : AA^t + BB^t = 1 \big\}$ under the metaplectic cover.
	We define a character~$\chi_{1/2}$ on~$\textcolor{\newcolor}{K_\infty'}$ as
	$$\chi_{1/2}\Big(\big(\begin{smallmatrix}
		A & B \\ -B & A
	\end{smallmatrix}\big), \pm \det(-B\tau + A)^{1/2}\Big) = \pm \det(-Bi + A)^{-1/2}.$$
	
	Let~$k\in\frac{1}{2}\ZZ$.
	If a Schwartz function $\varphi \in \calS(V^2)$ satisfies $\omega_{\infty,2}(M') \varphi = \chi_{1/2}^{2k}(M') \varphi$ for all $M' \in \textcolor{\newcolor}{K_\infty'}$, then
	$$\Theta_{L,2}(M M', (\vect{\boralpha}, \vect{\borbeta}), g, \varphi) = \chi_{1/2}^{2k}(M') \Theta_{L, 2}(M, \vect{\boralpha}, \vect{\borbeta}, g, \varphi)
	\qquad
	\text{for all $M' \in \textcolor{\newcolor}{K_\infty'}$.}$$
	In this case, we may associate to~$\Theta_{L, 2}(M, \vect{\boralpha}, \vect{\borbeta}, g, \varphi)$ a smooth function on~$\HH_2$ that behaves similarly as the Siegel modular forms constructed in Section~\ref{sec;siegmodg2aft}, with the following procedure.
	For every~$\tau \in \IH_2$ we denote by~$M_\tau\coloneqq(\begin{smallmatrix}
	I_2 & x\\ 0 & I_2
	\end{smallmatrix})\Big(\begin{smallmatrix}
	y^{1/2} & 0\\ 0 & (y^{1/2})^{-t}
	\end{smallmatrix}\Big)$ the standard element of~$\Sp_4(\RR)$ mapping~$iI_2$ to~$\tau$.
	Then, the theta function
	$$\Theta_{L, 2}(\tau, \vect{\boralpha}, \vect{\borbeta}, g, \varphi) \coloneqq 
	\det y^{-k/2}
	\Theta_{L, 2}(M_\tau, \vect{\boralpha}, \vect{\borbeta}, g, \varphi)$$
	satisfies
	$$\Theta_{L,2}(\gamma \cdot \tau, (\vect{\boralpha}, \vect{\borbeta})\gamma^t, g, \varphi) = \phi(\tau)^{2k} \rho_{L,2}(\gamma) \Theta_{L, 2}(\tau, \vect{\boralpha}, \vect{\borbeta}, g, \varphi)$$
	for all $\gamma \in \Mp_4(\IZ)$.
	
	\begin{ex}
	Let~$(b^+,b^-)=(b,2)$.
	The \emph{standard Gaussian}~$\stgatwo$ of~$(\RR^{b,2})^2$ is defined as
	\bes
	\stgatwo(\vect{x})=\exp\Big(
	-\pi\sum_{i=1}^{b+2}\sum_{j=1}^2 x_{i,j}^2
	\Big),\qquad\text{for every~$\vect{x}=(x_1,x_2)\in(\RR^{b,2})^2$},
	\ees
	where~$x_j=(x_{1,j},\dots,x_{b+2,j})^t\in \RR^{b,2}$.
	The standard Gaussian of $V^2$ is the composition of~${\stgatwo}$ with~$g_0$, where the latter is as in~\eqref{eq;stgatwomat}.
	For simplicity we will denote the standard Gaussian of~$V^2$ with~$\stgatwo$ in place of~$\stgatwo\circ g_0$.
	It is well-known that~$\stgatwo$ is an eigenfunction under the action of~$\textcolor{\newcolor}{K_\infty'}$ of weight~$k=b/2-1$, in the sense that~$\omega_{\infty,2}(M')\stgatwo = \chi_{1/2}^{b-2}(M')\stgatwo$;
	see e.g.~\cite[Chapter~VIII, Section~$1$]{bowa}.
	\end{ex}
	
	From now on, \textbf{we assume that the signature of~$V$ is~$(b,2)$.}
	Let $\pol$ be a very homogeneous polynomial of degree~$(m^+,m^-)$ on~$(\RR^{b,2})^2$.
	For simplicity we write~$\pol(\vect{\genvec})$ for the value~$\pol\big(g_0(\vect{\genvec})\big)$ where~$\vect{\genvec}\in V^2$, i.e.\ we consider~$\pol$ as a polynomial in the coordinates of~$\vect{\genvec}$ with respect to the basis~$(\basevec_j)_j$ of~$V$. 
	It is easy to see using Lemma~\ref{lemma;onFtransfgenus2} that the Schwartz function
	\be\label{eq;ourchoiceofphi}
	\varphi(\vect{v}) = \exp(-\tr\Delta / 8 \pi)(\calP)(\vect{v}) \stgatwo(\vect{v})
	\ee
	is an eigenfunction of~$\textcolor{\newcolor}{K_\infty'}$ of weight~$b/2-1+m^+-m^-$, namely
	\be\label{eq;ourchoiceofphiwbe}
	\omega(M')\varphi = \chi_{1/2}^{b-2+2m^+-2m^-}(M') \varphi
	\qquad\text{for all~$M'\in \textcolor{\newcolor}{K_\infty'}$.}
	\ee
%	\textcolor{red}{\textbf{Doublecheck:} Let
%	\[
%	\widetilde{\varphi}(\vect{v},\tau)
%	\coloneqq
%	\exp(-\tr(\Delta y^{-1}) / 8 \pi)(\calP)(\vect{v}) \cdot e\big(
%	\tr q(\vect{v}^+)\tau + \tr q(\vect{v}^-\overline{\tau}
%	\big),
%	\]
%	where~$\vect{v}^\pm$ are the projection to the standard positive/negative definite subspaces of~$(\RR^{b,2})^2$.
%	It is easy to see that if~$M=\big(\begin{smallmatrix}
%	I_2 & B \\
%	0 & I_2
%	\end{smallmatrix}\big)$, then
%	\[
%	\omega_{\infty,2}(M)(\widetilde{\varphi})(\vect{v},\tau)
%	=
%	\widetilde{\varphi}(\vect{v},\tau+B)
%	=
%	\widetilde{\varphi}(\vect{v},M\cdot \tau).
%	\]
%	Also, one can show that
%	\[
%	\omega_{\infty,2}(S)(\widetilde{\varphi})(\vect{v},\tau)=
%	(-1)^{2-b}
%	\det(S\cdot\tau)^{b/2+m^+}
%	\cdot\overline{\det(S\cdot\tau)}^{1+m^-}
%	\widetilde{\varphi}(\vect{v},S\cdot\tau).
%	\]
%	}\\
	This will be relevant in Section~\ref{sec;deg2kmtform} only for special cases of~$\calP$, for which~\eqref{eq;ourchoiceofphiwbe} follows immediately from the similar behavior of the Kudla--Millson Schwartz function; see~\cite[Theorem~$3.1$, (ii)]{kumi;harmI} and Section~\ref{sec;Schwfunct} for further details.
	
	The theta functions arising from~$\varphi$ as in~\eqref{eq;ourchoiceofphi} are generalizations in genus~$2$ of Borcherds' Siegel theta functions~\cite[Section~$4$]{bo;grass}.
	They can be also considered as vector-valued analogues of the theta functions introduced by Roehrig in~\cite{roehrig}.
	These considerations can be easily checked using the following explicit rewriting of~$\theta_{L,2}^{\gendisc}$ \textcolor{\newcolor}{\`a} la Borcherds.
	Recall that for fixed~$g\in G$, we define~$z=g^{-1}(z_0)\in\Gr(L)$ to be the negative-definite plane mapping to~$z_0$ under~$g$.
	\begin{lemma}\label{lemma;fronv2tov1}
	If~$\varphi(\vect{v}) = \exp(-\tr\Delta / 8 \pi)(\calP)(\vect{v}) \stgatwo(\vect{v})$ for some very homogeneous polynomial~$\calP$ of degree~$(m^+,m^-)$ on~$(\RR^{b,2})^2$, then
	\ba\label{eq;Ithinkggenborindeg2thetanofact}
	\theta_{L,2}^{\gendisc}(\tau,\vect{\boralpha},\vect{\borbeta},g,\varphi)
	&=
	\det y^{1 + m^-} \sum_{\vect{\lambda}\in L^2+{\gendisc}}\exp \Big(
	-\frac{1}{8\pi}\trace(\Delta y^{-1})
	\Big)(\pol)\big(g(\vect{\lambda}+\vect{\borbeta})\big)
	\\
	&\quad\times e\Big(
	q((\vect{\lambda}+\vect{\borbeta})_{z^\perp})\tau+
	q((\vect{\lambda}+\vect{\borbeta})_z)\bar{\tau}-
	(\vect{\lambda}+\vect{\borbeta}/2,\vect{\boralpha})
	\Big).
	\ea
	\end{lemma}
	The theta functions considered in Lemma~\ref{lemma;fronv2tov1} play a central role in the present paper.
	For this reason, we reserve a special notation for them, given as follows.
	\begin{defi}\label{def;thetasergen2}
	Let~$\calP$ be a very homogeneous polynomial on~$(\RR^{b,2})^2$.
	We denote by $\theta_{L,2}^{\gendisc}(\tau,\vect{\boralpha},\vect{\borbeta},g,\pol)$ the theta function arising from~$\varphi=\exp(-\tr\Delta / 8 \pi)(\calP)(\vect{v}) \stgatwo(\vect{v})$; see \eqref{eq;Ithinkggenborindeg2thetanofact} for an explicit formulation.
%	We define
%	\ba\label{eq;Ithinkggenborindeg2thetanofact}
%	\theta_{L,2}^{\gendisc}(\tau,\vect{\boralpha},\vect{\borbeta},g,\pol)
%	&\coloneqq \theta_{L,2}^{\gendisc}(\tau,\vect{\boralpha},\vect{\borbeta},g,\varphi),
%	\ea
%	for every $\tau\in\HH_2$, $g\in G$.
	The~$\CC[\disc{L}^2]$-valued \emph{Siegel theta function associated to~$L$ and the polynomial~$\pol$} is
	\[
	\Theta_{L,2}(\tau,\vect{\boralpha},\vect{\borbeta},g,\pol)
	\coloneqq
	\sum_{{\gendisc}\in\disc{L}^2}\theta_{L,2}^\gendisc(\tau,\vect{\boralpha},\vect{\borbeta},g,\pol) \mathfrak{e}_{\gendisc}.
	\]
	If $\vect{\boralpha},\vect{\borbeta}=0$, we usually drop them from the notation and simply write~$\Theta_{L,2}(\tau,g,\pol)$.
	\end{defi}
	\begin{proof}[Proof of Lemma~\ref{lemma;fronv2tov1}]
	By~\eqref{eq;strfwo} we may rewrite
	\begin{align*}
	\theta_{L,2}^\sigma(\tau, \vect{\boralpha}, \vect{\borbeta}, g, \varphi)
%	&=
%	\det y^{-k/2}
%	\theta_{L,2}^\sigma(M_\tau, \vect{\boralpha}, \vect{\borbeta}, g, \varphi)
%	\\
%	&=
%	\det y^{-b/2+1-m^++m^-}
%	\sum_{\vect{\lambda}\in L^2+\sigma}
%	(\omega_{\infty,2}(M_\tau)\cha{\varphi})\big(g(\vect{\lambda},(\vect{\delta},\vect{\nu})M_\tau^{-t})\big)
%	\\
	&=
	\det y^{(-b/2+1-m^++m^-)/2}
	\sum_{\vect{\lambda}\in L^2+\sigma}
	\cha{(\omega_{\infty,2}(M_\tau)\varphi)}\big(g(\vect{\lambda},\vect{\delta},\vect{\nu})\big).
	\end{align*}
	For any~$\vect{v}\in V^2$ we may compute
	\bas
%	\omega_{\infty,2}(M_\tau)\varphi
%	=
	(\omega_{\infty,2}(M_\tau)\varphi)(\vect{v})
%	&=
%	\omega_{\infty,2}\big(\begin{smallmatrix}
%	I_2 & x \\
%	0 & I_2
%	\end{smallmatrix}\big)
%	\omega_{\infty,2}\Big(\begin{smallmatrix}
%	y^{1/2} & 0 \\
%	0 & (y^{1/2})^{-t}
%	\end{smallmatrix}\Big)
%	(\varphi)(\vect{v})
%	\\
%	&=
%	e(\tr x q(\vect{v}))\cdot \omega_{\infty,2}\Big(\begin{smallmatrix}
%	y^{1/2} & 0 \\
%	0 & (y^{1/2})^{-t}
%	\end{smallmatrix}\Big)
%	(\varphi)(\vect{v})
%	\\
	&=
	\det y^{(b+2)/4} \cdot e( x q(\vect{v}))\cdot\varphi(\vect{v}y^{1/2}),
	\eas
	where~$\varphi(\vect{v}y^{1/2})
	=
	\exp(-\tr\Delta / 8 \pi)(\calP)(\vect{v}y^{1/2}) \stgatwo(\vect{v}y^{1/2})$.
	By~\cite[Lemma~$4.4$, ($4.5$)]{roehrig} and the very homogeneity of~$\calP$ we deduce that
	\[
	\exp(-\tr\Delta/8\pi)(\calP)(\vect{v}y^{1/2})
	=
	\det y^{(m^++m^-)/2}\exp(-\tr(\Delta y^{-1})/8\pi)(\calP)(\vect{v}).
	\]
	This and the rewriting~$e\big( xq(\vect{\genvec})\big)\cdot\stgatwo\big(g(\vect{\genvec}y^{1/2})\big)
	=
	e\big(
	\trace \big(q(\vect{\genvec}_{z^\perp})\tau\big)+\trace\big(q(\vect{\genvec}_z)\bar{\tau}\big)\big)$ imply the claimed formula for~$\theta_{L,2}^{\gendisc}(\tau,\vect{\boralpha},\vect{\borbeta},g,\varphi)$.
	\end{proof}
	
	\begin{rem}\label{rem;trformsiegthv2}
	The Siegel theta functions of Definition~\ref{def;thetasergen2} satisfy
	\be
	\Theta_{L,2}(\gamma\cdot \tau,(\vect{\boralpha},\vect{\borbeta})\gamma^t,g,\pol)
	=
	\phi(\tau)^{b-2+2(m^++m^-)}
	\rho_{L,2}(\gamma)
	\Theta_{L,2}(\tau,\vect{\boralpha},\vect{\borbeta},g,\pol)
	\ee
	for all~$\gamma\in\Mp_4(\ZZ)$.
	\end{rem}
	
	\subsection{\textcolor{\mycolor}{Jacobi forms and Jacobi Siegel theta functions}}\label{subsec:JacobiForms}
	
	Let $\heisenberg(\IR) \coloneqq \IR^3$ be the Heisenberg group, where the multiplication is given by
	$$(\jacone, \jactwo, \jacthree) \cdot (\jacone', \jactwo', \jacthree')
	\coloneqq
	(\jacone + \jacone', \jactwo + \jactwo', \jacthree + \jacthree' + \jacone \jactwo' - \jacone'\jactwo)$$
	and define the metaplectic Jacobi group by~$\jacobi(\IR) \coloneqq \heisenberg(\IR) \rtimes \Mp_2(\IR)$, where the action of $\Mp_2(\IR)$ on $\heisenberg(\IR)$ is given by 
	$$(\jacone, \jactwo, \jacthree)
	\cdot
	\left(\begin{smallmatrix} a & b \\ c & d\end{smallmatrix}\right) = (a \jacone + c \jactwo, b \jacone + d \jactwo, \jacthree).$$
	The Jacobi group $\jacobi(\IR)$ acts on $\IH \times \IC$ under
	$$(\jacone, \jactwo, \jacthree) (\jacvarone, \jacvartwo)
	\coloneqq
	(\jacvarone, \jacvartwo + \jacone + \jactwo \jacvarone)
	\qquad\text{and}\qquad
	\big(\big(\begin{smallmatrix} a & b \\ c & d\end{smallmatrix}\big), \phi\big) (\jacvarone, \jacvartwo)
	\coloneqq
	\Big(\frac{a \jacvarone + b}{c \jacvarone + d}, \frac{\jacvartwo}{c \jacvarone + d}\Big).$$
	The metaplectic Jacobi group $\jacobi(\IR)$ can be embedded into $\Mp_4(\IR)$ as
	\begin{align*}
		\left(\jacone, \jactwo, \jacthree, \left(\begin{smallmatrix} a & b \\ c & d\end{smallmatrix}\right)\right) \mapsto \left(\begin{smallmatrix}
			a & 0 & b & a \jactwo - b \jacone \\
			\jacone & 1 & \jactwo & \jacthree \\
			c & 0 & d & c \jactwo - d \jacone \\
			0 & 0 & 0 & 1
		\end{smallmatrix}\right)
	\end{align*}
	and by mapping the square root $\phi(\jacvarone)$ of $c \jacvarone + d$ to~$\tilde{\phi}\left(\big(\begin{smallmatrix} \tauone & \tautwo \\ \tautwo & \tauthree \end{smallmatrix}\big)\right)  = \phi(\tauone)$. Its image is the Klingen parabolic $C_{1,2}(\IR)$, which is the stabilizer of the standard $1$-dimensional boundary component of $\IH_2$. Moreover, the action of $\Mp_4(\IR)$ on $\big(\begin{smallmatrix} \tauone & \tautwo \\ \tautwo & \tauthree \end{smallmatrix}\big) \in \IH_2$ restricts to the above action of $\jacobi(\IR)$ on $(\jacvarone, \jacvartwo) \in \IH \times \IC$.
	
	For $M \in \jacobi(\IR), \eta \in L', g \in G$ and a Schwartz function $\varphi \in \calS(V^2)$ we define the \emph{Jacobi theta function with characteristics $\delta, \nu \in V$} by
	$$\Theta_{L, \eta}^\sigma(M, \boralpha, \borbeta, g, \varphi) = \sum_{\lambda \in L + \sigma} (\omega_{\infty,2}(M) \varphi')(g(\lambda, \eta, ((\boralpha, 0),(\borbeta, 0))M^{-t})).$$
	Then for $b \in \IZ$ we have
	$$\Theta_{L, \eta}^\sigma\bigg(\big(\begin{smallmatrix}
		1 & b\\
		0 & 1
	\end{smallmatrix}\big)M, (\boralpha, \borbeta)\big(\begin{smallmatrix}
		1 & b\\
		0 & 1
	\end{smallmatrix}\big)^t, g, \varphi\bigg) = e(q(\sigma) b) \Theta_{L, \eta}^\sigma\big(M, \boralpha, \borbeta, g, \varphi\big)$$
	and
	$$\Theta_{L, \eta}^\sigma\bigg(\big(\begin{smallmatrix}
		0 & -1\\
		1 & 0
	\end{smallmatrix}\big)M, (\boralpha, \borbeta) \big(\begin{smallmatrix}
		0 & -1\\
		1 & 0
	\end{smallmatrix}\big)^t, g, \varphi\bigg) = \frac{\sqrt{i}^{b^- - b^+}}{\sqrt{\lvert D_L \rvert}} \sum_{\sigma' \in D_L} e(-\sigma', \sigma) \Theta_{L,\eta}^{\sigma'}\big(M, \boralpha, \borbeta, g, \varphi\big).$$
	Both equations can be seen by observing that the restriction of the Schr\"odinger model of the Weil representation of $\Mp_4(\IR)$ to $\Mp_2(\IR) \subseteq \jacobi(\IR)$ yields the Schr\"odinger model of the Weil representation of $\Mp_2(\IR)$ on $\calS(V)$. Moreover, a direct calculation yields for every $(r, s, t) \in \calH(\IZ)$.
	\begin{align*}
		\Theta_{L, \eta}^\sigma\bigg((r, s, t) M, \boralpha, \borbeta, g, \varphi\bigg) = e(r(\sigma, \eta) + (t - rs) q(\eta)) \Theta_{L,\eta}^{\sigma - s \eta}\big(M, \boralpha, \borbeta, g, \varphi\big).
	\end{align*}
	These equations define a representation $\rho_{L, \eta}$ of $\calJ(\IZ)$ on the group ring $\IC[D_L]$ such that
	$$\rho_{L, 2}(M, \tilde{\phi}) \frake_{(\sigma, \eta)} = (\rho_{L, \eta}(M, \phi) \frake_\sigma) \otimes \frake_\eta$$
	for all $(M, \phi) \in \calJ(\IZ)$, where we identify~$\IC[D_L] \otimes \IC[D_L]$ with~$\IC[D_L^2]$ using the isomorphism \be\label{eq;inclusionofgroupalg}
	\IC[D_L] \otimes \IC[D_L] \to \IC[D_L^2], \quad \frake_{\sigma_1} \otimes \frake_{\sigma_2} \mapsto \frake_{(\sigma_1, \sigma_2)}.
	\ee
	\begin{defi}
		A \emph{vector-valued Jacobi form of weight $k$ and index $\jacindex$ with respect to the Weil representation $\weilrep_{\lattice, \sigma_2}$} is a holomorphic function $\phi : \IH \times \IC \to \IC[D_L]$ with
		$$\phi(\jacvarone, \jacvartwo + \jacone + \jactwo \jacvarone) = e(-\jacindex \jactwo^2 \jacvarone - 2 \jacindex \jactwo \jacvartwo - \jacone \jactwo \jacindex - \jacthree \jacindex) \weilrep_{\lattice, \sigma_2}(\jacone, \jactwo, \jacthree) \phi(\jacvarone, \jacvartwo)$$
		for $(\jacone, \jactwo, \jacthree) \in \heisenberg(\IZ)$ and
		$$\phi(\mat (\jacvarone, \jacvartwo)) = (c \jacvarone + d)^{k}
		e\left(\frac{\jacindex c \jacvartwo^2}{c \jacvarone + d}\right) \weilrep_{\lattice, \sigma_2}(\mat) \phi(\jacvarone, \jacvartwo)$$
		for $\mat = \left(\left(\begin{smallmatrix}a & b \\ c & d\end{smallmatrix}\right), \phi\right) \in \Mp_2(\IZ)$. If $\jacindex = q(\jacindexvec)$ for $\jacindexvec \in \sigma_2 + L$, we say that $f$ is a vector-valued Jacobi form of weight $k$ and index $\jacindexvec$.
	\end{defi}
	
	\begin{ex}\label{exmpl:JacobiThetaPoly}
		Let $\varphi \in \calS(V)$ be a Schwartz function such that $\omega(M') \varphi = \chi_{1/2}^{2k}(M') \varphi$ for all $M' \in \Mp_2(\IR) \cap \textcolor{\newcolor}{K_\infty'}$ and some $k \in \frac{1}{2}\IZ$. Let $M_{\tau_1, \tau_2} \in \calJ(\IR)$ be given by $M_{\tau_1} \cdot \left(\frac{x_2}{\sqrt{y_1}}, \frac{y_2}{\sqrt{y_1}}, 0\right)$, where $M_{\tau_1} \coloneqq(\begin{smallmatrix}
			1 & x_1 \\ 0 & 1
		\end{smallmatrix})\Big(\begin{smallmatrix}
			y_1^{1/2} & 0\\ 0 & y_1^{-1/2}
		\end{smallmatrix}\Big)$, so that $M_{\tau_1, \tau_2} (i, 0) = (\tau_1, \tau_2)$. Then
		$$\Theta_{L, \eta}(\tau_1, \tau_2, \delta, \nu, g, \varphi) \coloneqq y_1^{k/2} \Theta_{L, \eta}(M_{\tau_1, \tau_2}, \delta, \nu, g, \varphi)$$ satisfies
		\begin{align*}
		&\Theta_{L, \eta}(\tau_1, \tau_2 + r + s \tau_1, \delta, \nu, g, \varphi) \\
		&= e(-q(\eta) s^2 \tau_1 - 2q(\eta) s \tau_2 - rs q(\eta) - tq(\eta)) \rho_{L, \eta}(r, s, t) \Theta_{L, \eta}(\tau_1, \tau_2, \delta, \nu, g, \varphi)
		\end{align*}
		and
		$$\Theta_{L, \eta}(M(\tau_1, \tau_2), (\delta, \nu)M^t, g, \varphi) = \phi(\tau)^{2k} e\left(\frac{q(\eta) c \tau_2^2}{c \tau_1 + d}\right) \rho_{L, \eta}(M) \Theta_{L, \eta}(\tau_1, \tau_2, \delta, \nu, g, \varphi).$$
		Assume that $L$ has signature $(b^+, b^-)$ and fix an isometry $g_0 : L \otimes \IR \to \IR^{b^+, b^-}$. Let
		$$\Delta = \sum_{i = 1}^{b^+ + b^-} \frac{\partial^2}{\partial x_i^2}$$
		be the Laplace operator on $\IR^{b^+ + b^-}$. If $\calP : \IR^{b^+, b^-} \to \IC$ is a homogeneous polynomial, then
		$$\varphi(\lambda) = \exp(- \Delta / 8 \pi)(\calP)(g_0(\lambda)) \varphi_0(\lambda)$$
		is a Schwartz function that satisfies
		$$\omega(M') \varphi = \chi_{1/2}^{2k}(M') \varphi$$
		for all $M' \in \Mp_2(\IR) \cap K'$ where $k = b^+/2 + m^+ - b^- / 2$. Write $z_0^\perp$, resp.~$z_0$, for the preimage of $\IR^{b^+, 0}$, resp. $\IR^{0, b^-}$, under $g_0$. for an isometry $g \colon L \otimes \IR \to L \otimes \IR$ we define $z^\perp$, resp.~$z$, to be the preimage of $z_0^\perp$, resp. $z_0$ under $g$. Then $z_0$ and $z$ are negative definite subspaces of dimension $b^-$ in $L \otimes \IR$, so that they lie in the Grassmannian $\Gr(L)$.
		If we denote by~$\lambda_{z^\perp}$ and~$\lambda_z$ the projections of $\lambda$ to the subspace $z^\perp$ and $z$ respectively, for $\lambda \in L \otimes \IR$, then
		\begin{align*}
			&\theta^\gendisc_{\lattice, \jacindexvec}(\jacvarone, \jacvartwo, \boralpha, \borbeta, \isometry, \calP) := \theta^\gendisc_{\lattice, \jacindexvec}(\jacvarone, \jacvartwo, \boralpha, \borbeta, \isometry, \varphi)
			\\
			&=
			\jacvaroneim^{b^-/2} \exp\big(4 \pi q(\jacindexvec_{\isometryneg}) \jacvartwoim^2 / \jacvaroneim\big)
			\sum_{\lambda \in \sigma + \lattice} \exp(-\Delta / 8 \pi \jacvaroneim)(\calP)\big(\isometry(\lambda + \jacvartwoim\jacindexvec/\jacvaroneim  + \borbeta)\big)
			\\
			&\quad\times e\Big(q\big((\lambda + \borbeta)_{\isometrypos}\big) \jacvarone + q\big((\lambda + \borbeta)_{\isometryneg}\big) \overline{\jacvarone} + \jacvartwo (\lambda + \borbeta, \jacindexvec_{\isometrypos}) + \overline{\jacvartwo} (\lambda + \borbeta, \jacindexvec_{\isometryneg}) - (\lambda + \borbeta / 2, \boralpha)\Big)
		\end{align*}
		and
		$$\Theta_{L, \eta}(\tau_1, \tau_2, \delta, \nu, g, \calP) := \Theta_{L, \eta}(\tau_1, \tau_2, \delta, \nu, g, \varphi) = \sum_{\sigma \in D_L} \theta_{L, \eta}^\sigma(\tau_1, \tau_2, \delta, \nu, g, \calP) \frake_\sigma.$$
	\end{ex}
	
	If $\phi$ is a vector-valued Jacobi form of weight $k$ and index $\jacindexvec \in L'$, the transformation properties imply that $\phi$ has a Fourier expansion of the form
	$$\phi(\jacvarone, \jacvartwo) = \sum_{\substack{\sigma \in D_\lattice \\ r \in \IZ + q(\sigma) \\ s \in \IZ + (\sigma, \jacindexvec)}} a_\phi(\sigma, r, s) \frake_\lambda(r \jacvarone + s \jacvartwo).$$
	We will drop the index of the Fourier coefficients if the Jacobi form is clear from the context. Since $\phi(\jacvarone, \jacvartwo + \jacvarone) = e(-q(\jacindexvec) \jacvarone - 2 q(\jacindexvec) \jacvartwo) \weilrep_{\lattice, \jacindexvec}(0, 1, 0) \phi(\jacvarone, \jacvartwo)$, we deduce that
	$$a(\sigma, r, s) = a(\sigma + \jacindexvec, r + s + q(\jacindexvec), s + 2 q(\jacindexvec)),$$
	hence
	\begin{align*}
		a(\sigma + \jacindexvec, q(\sigma + \jacindexvec), (\sigma + \jacindexvec, \jacindexvec))
		&=
		a(\sigma + \jacindexvec, q(\sigma) + (\sigma, \jacindexvec) + q(\jacindexvec), (\sigma, \jacindexvec) + 2q(\jacindexvec))
		\\
		&=
		a(\sigma, q(\sigma), (\sigma, \jacindexvec)).
	\end{align*}
	A Jacobi form $\phi$ is said to be a \emph{Jacobi cusp form} if~$a(\sigma, r, s) \neq 0$ implies $4 q(\eta) r > s^2$.
	
	Jacobi forms occur as the Fourier--Jacobi coefficients of Siegel modular forms. That is, if~${f \in M_{2, L}^k}$, then there exist Jacobi forms $\phi_{\sigma_2, m}(\jacvarone, \jacvartwo)$ of weight $k$ and index $m$ with respect to $\weilrep_{\lattice, \sigma_2}$ such that
	$$f(\tau) = \sum_{\sigma_2 \in D_\lattice} \sum_{m \in \IZ + q(\sigma_2)} \phi_{\sigma_2, m}(\jacvarone, \jacvartwo) \otimes \frake_{\sigma_2}(m \tauthree).$$
	If $f$ has the Fourier expansion as in~\eqref{eq;FexpofSiegmod},
	%	$$f(\tau) = \sum_{\sigma \in D_L^2} \sum_{\substack{T \in \Lambda_2 + q(\sigma) \\ T \geq 0}} a_f(\sigma, T)e(\tr(T \tau)) \frake_\sigma,$$
	then the Fourier coefficients $a_{\sigma_2, m}$ of~$\phi_{\sigma_2, m}$ are related to the ones of~$f$ as
	\begin{align}
		a_{\sigma_2, m}(\sigma_1, r, s)
		=
		a_f\Big((\sigma_1, \sigma_2), \big(\begin{smallmatrix} r & s / 2 \\ s / 2 & m\end{smallmatrix}\big)\Big). \label{eq:FouerCoeffFourierJacExpansion}
	\end{align}
	
%	\begin{rem}\label{rem;onFJexpofsiegtheta2}
		Let~$L$ be an even indefinite lattice of signature~$\left(b,2\right)$.
		We illustrate here the \emph{Fourier--Jacobi} expansion of~$\Theta_{L, 2}$ in terms of the Jacobi Siegel theta functions introduced in the current section.
		
		Let~$\Theta_{L, 2}(\tau,\vect{\boralpha},\vect{\borbeta},g,\pol)$ be the Siegel theta function of genus~$2$ associated to some very homogeneous polynomial~$\pol$ of degree~$(m,0)$; the reader may recall them from Section~\ref{sec;vvsiegeltheta2}.
		The polynomial~$\pol$ is homogeneous of degree $\degjac$ in the first variable; see Remark~\ref{rem;veryhomogpolprop}.
		Since~$\pol\big((\lambda,\eta)\cdot N\big)=\det N^m\cdot\pol(\lambda,\eta)$ for every~$N\in\CC^{2\times 2}$, if we choose~$N=\big(\begin{smallmatrix}
			1 & 0\\
			C & 1
		\end{smallmatrix}\big)$ for some~$C\in\CC$, then we have that
		\[
		(\lambda,\eta)\cdot N =(\lambda + C\eta, \eta)\qquad\text{and}\qquad\det N=1.
		\]
		We then deduce that~$\pol(\lambda, \jacindexvec) = \pol(\lambda + C\jacindexvec, \jacindexvec)$ for every~$C\in\CC$.
		This property simplifies the construction of Jacobi Siegel theta functions arising from polynomials of the form~$\pol(\cdot,\eta)$, for some fixed~$\eta$.
		Moreover, it plays a key role in proving the injectivity result on Jacobi Petersson inner products that we will consider in Section~\ref{sec:JacobiThetaInnerProducts}.
		
		For simplicity, we assume here that~$\vect{\borbeta}=0$.
		This is sufficient for the main goal of the present paper, namely to prove the injectivity of the genus~$2$ Kudla--Millson lift.
		
		Recall from~\eqref{eq;inclusionofgroupalg} that we may regard~$\CC[\disc{L}^2]$ as the product~$\CC[\disc{L}]\otimes\CC[\disc{L}]$.
		Then the
		\emph{Fourier--Jacobi} expansion of~$\Theta_{L, 2}(\tau,\vect{\boralpha},0,g,\pol)$ is
		\begin{align*}
			&\Theta_{L,2}(\tau,\vect{\boralpha},0,g,\pol)
			\\
			&=
			\Big(\frac{\det y}{\yone}\Big)^{1/2} \sum_{\sigma \in D_\lattice}
			\sum_{\jacindexvec \in \sigma + L}
			\Theta_{L, \jacindexvec}\Big(\jacvarone, \jacvartwo, \delta_1, 0, g, \exp\Big(-\frac{\jacvaroneim}{\det y} \Delta_2\Big) \pol\big(\cdot, g(\jacindexvec)\big)\Big)
			\\
			&\quad\times
			e\Big(-2 i q\big(\jacindexvec_w\big) (\det y) / \jacvaroneim - (\jacindexvec, \delta_2)\Big)
			\otimes \frake_{\sigma}\big(q(\jacindexvec) \tauthree\big)
			\\
			&=
			\Big(\frac{\det y}{\yone}\Big)^{1/2}
			\sum_{\sigma \in D_\lattice}
			\sum_{m \in \IZ+ q(\sigma)}
			\sum_{\substack{\jacindexvec \in \sigma + L \\ q(\jacindexvec) = m}}
			\Theta_{L, \jacindexvec}\Big(\jacvarone, \jacvartwo, \delta_1, 0, g, \exp\Big(-\frac{\jacvaroneim}{\det y} \Delta_2\Big) \pol\big(\cdot, g(\jacindexvec)\big)\Big)
			\\
			&\quad\times
			e\Big(-2 i q\big(\jacindexvec_w\big) (\det  y) / \jacvaroneim - (\jacindexvec, \delta_2)\Big) \otimes \frake_{\sigma}(m \tauthree),
		\end{align*}
		where $\Delta_2$ is the Laplace operator acting on the second variable of~$\pol$.
%	\end{rem}
	
	The following lemma gathers some properties of the elements of~$\HH\times\CC$.
	The proof is a straightforward calculation.
	\begin{lemma}\label{lem:ImaginaryPartTransformation}
		Let~$(\tauone,\tautwo)\in\HH\times\CC$.
		\begin{enumerate}[label=(\roman*), leftmargin=*]
			\item $\frac{\Im(\jacvartwo + \jacone + \jactwo \jacvarone)}{\Im(\jacvarone)} = \frac{\Im(\jacvartwo)}{\Im(\jacvarone)} + \jactwo$ for every $(\jacone, \jactwo) \in \IR^2$.
			\item $\frac{\Im(\jacvartwo/\jacvarone)}{\Im(-1/\jacvarone)} = \jacvarone \big(\frac{\Im(\jacvartwo)}{\Im(\jacvarone)} - \frac{\jacvartwo}{\jacvarone}\big)$.
			\item If $(\jacone, \jactwo) \in \IR^2$ and $\jacindex \in \IZ$, then
			\begin{align*}
				\exp\big(4 \pi \jacindex \Im(\jacvartwo + \jactwo \jacvarone + \jacone)^2 / \Im(\jacvarone)\big)
				%				\\
				%				&=
				=
				\exp\big(4 \pi \jactwo^2 \jacvaroneim \jacindex + 8 \pi \jactwo \jacvartwoim \jacindex\big)
				\cdot
				\exp\big(4 \pi \jacindex \jacvartwoim^2 / \jacvaroneim\big).
			\end{align*}
			\item For $\jacindex \in \IZ$ we have
			\begin{align*}
				\exp\big(4 \pi \jacindex
				\Im(\jacvartwo / \jacvarone)^2 / \Im(- 1 / \jacvarone)\big)
				=
				e\big(\jacindex \jacvartwo^2 / \jacvarone - \jacindex \overline{\jacvartwo}^2 / \overline{\jacvarone}\big)
				\cdot
				\exp(4 \pi \jacindex \jacvartwoim^2 / \jacvaroneim).
			\end{align*}
		\end{enumerate}
	\end{lemma}
	%%%%%%%%%%%%%%%%%%%%%%%%%%%%%%%%%%%%%%%%%%%%%%%%%%%%%%%%%%%%%%
	\section{Reduction of Siegel theta functions to sublattices}\label{sec;vvsiegeltheta2}
In~\cite{zuffetti;gen1} the second author explained how to unfold the defining integrals of the genus~$1$ Kudla--Millson lift.
	The idea was to apply Borcherds' formalism~\cite[Section 5]{bo;grass} to rewrite the genus~$1$ theta function~$\Theta_L$ with respect to the splitting of a hyperbolic plane in~$L$.
	
	Many difficulties arise in the genus~$2$ generalization of this idea.
	One of those is the lack of results on how to rewrite the theta function~$\Theta_{L,2}$ with respect to the splitting of a hyperbolic plane.
	In fact, Borcherds' work~\cite{bo;grass} covers only the genus~$1$ case.
	The goal of the following sections is to fill this gap.
	
\subsection{On auxiliary polynomials defined on subspaces}\label{sec;splitgen2theta}
	We begin here with illustrating how to rewrite very homogeneous polynomials as combinations of polynomials defined on certain subspaces of~$(\RR^{b,2})^2$, in analogy with~\cite[Section~$5$]{bo;grass}, and study their homogeneity.
	Examples of such decomposition are collected in the \textcolor{\mycolor}{appendix} of the present article, namely Section~\ref{sec;appendix}.
	
	In fact, we will see that those polynomials on subspaces \emph{are not always very homogeneous}, in contrast with the case of genus~$1$; see Lemma~\ref{lemma;nonhomogofderpol}.
	This implies that the associated genus~$2$ theta functions are not always modular with respect to~$\Mp_4(\ZZ)$.
	Such unexpected behavior will be further investigated in Section~\ref{sec;deg2KMlift} using Lemma~\ref{lemma;someFtransfofabcdh12lor}, which provides the Fourier transforms of the general summands of such theta functions.
	\\

	For simplicity, \emph{we restrict to the case of~$L$ splitting off (orthogonally) a hyperbolic plane~$U$}, hence~$L=\brK\oplus U$ for some Lorentzian sublattice~$\brK$.
	We denote by~$\genU,\genUU$ the standard hyperbolic basis of~$U$ such that
	\[
	\genU^2=\genUU^2=0 \qquad\text{and}\qquad (\genU,\genUU)=1.
	\]
	The orthogonal projection from~$L$ to~$\brK$ induces a projection~$p\colon \disc{L}\to\disc{\brK}$, which is an isomorphism.
	We denote by~$p$ also the componentwise projection induced on~$\disc{L}^2$.
%	Let~$\brN\in\ZZ_{>0}$ be such that~$|\disc{L}|=\brN^2|\disc{\brK}|$.
%	The value~$\brN$ is the smallest positive value that the inner product of~$\genU$ with some element of~$L$ can assume.

	Without loss of generality we may assume that the orthogonal basis~$(\basevec_j)_j$ of~$L\otimes\RR$ is such that
	\ba\label{eq;choiceofuu'gen2}
	\text{$\brK\otimes\RR$ and~$U\otimes\RR$ are generated by resp.~${\basevec_1,\dots,\basevec_{b-1},\basevec_{b+1}}$ and~$\basevec_b,\basevec_{b+2}$}
	,\\
	\genU=\frac{\basevec_b+\basevec_{b+2}}{\sqrt{2}}\qquad\text{and}\qquad\genUU=\frac{\basevec_b-\basevec_{b+2}}{\sqrt{2}}.\qquad\qquad\qquad
	\ea
	This assumption will simplify several of the following computations.
	
	\begin{ex}\label{ex;Lunimodhtcuu'}
	If~$L$ is \emph{unimodular}, then so is~$\brK$, and both split off a hyperbolic plane.
	In fact~$\brK$ is isomorphic to an orthogonal direct sum of the form~$\brK= E_8\oplus\dots\oplus E_8\oplus U$, where~$E_8$ is the~$8$-th root lattice.
	\end{ex}
	
	The following definitions recall some notions introduced in~\cite[Section~$5$]{bo;grass}.
	\begin{defi}\label{defi;not&deffromborwgen2}
	Let~$z\in \Gr(L)$, and let~$g\in G$ be such that $g\colon z\mapsto z_0$. 
	We denote by~$w$ the orthogonal complement of $\genU_z$ in $z$, and by $w^\perp$ the orthogonal complement of $\genU_{z^\perp}$ in $z^\perp$.
	The linear map~$\borw \colon L\otimes\RR\to L\otimes\RR$ is defined as~$\borw(\genvec)=g(\genvec_{w^\perp}+\genvec_w)$.
	\end{defi}
	
	We now define certain polynomials on subspaces of $(\RR^{b,2})^2$, to be considered as the analogue in genus $2$ of the polynomials defined by Borcherds in~\cite[Section~$5$, p.~$508$]{bo;grass}.
	Since the very homogeneous polynomials on~$(\RR^{b,2})^2$ we will work with in the next sections, namely~$\pol_{\vect{\alpha}}$ defined in~\eqref{eq;defpolPabcd}, are of degree~$(2,0)$, we restrict our attention to very homogeneous polynomials of degree~$(m^+,0)$.%see Example~\ref{ex;plabcdarevhom}.
	\begin{defi}\label{def;genborpolgenus2}
	Let $z\in\Gr(L)$, and let $g\in G$ be such that $g$ maps $z$ to $z_0$.
	For every very homogeneous polynomial $\pol$ of degree $(m^+,0)$ on $(\RR^{b,2})^2$, we define the polynomials~$\pol_{\borw,\hone,\htwo}$
%	,
%	of degrees respectively~$(m^+-\hone-\htwo,0)$
	on $g_0\circ \borw (L\otimes\RR)^2\cong (\RR^{b-1,1})^2$ by
	\be\label{eq;gengen2borfupol}
	\pol\big( g(\vect{\genvec})\big)=\sum_{\hone,\htwo}(\genvec_1,u_{z^\perp})^{\hone}(\genvec_2,u_{z^\perp})^{\htwo} \pol_{\borw,\hone,\htwo}\big(\borw (\vect{\genvec})\big)
	\ee
	for every~$\vect{\genvec}=(\genvec_1,\genvec_2)\in(\RR^{b,2})^2$, where as usual we drop the isometry~$g_0$ from the notation, and write~$\pol\big( g(\vect{\genvec})\big)$ and~$\pol_{\borw,\hone,\htwo}\big(\borw (\vect{\genvec})\big)$ in place of respectively~$\pol\big( g_0\circ g(\vect{\genvec})\big)$ and~$\pol_{\borw,\hone,\htwo}\big(g_0\circ \borw (\vect{\genvec})\big)$.
	Furthermore, we define
	\[
	\pol_{\borw, \htot}(\borw (\vect{\genvec}), \jacvartwoim / \jacvaroneim)
	\coloneqq
	\sum_{\hone} \Big(-\frac{\jacvartwoim}{\jacvaroneim}\Big)^{\hone} \pol_{\borw, \hone, \htot - \hone}(\borw (\vect{\genvec})).
	\]
	\end{defi}
	Although~$\pol$ is very homogeneous, the auxiliary polynomials~$\pol_{\borw,\hone,\htwo}$ are not necessarily very homogeneous.
	This will be shown in the case of~$\pol=\pol_{\vect{\alpha}}$ in Lemma~\ref{lemma;nonhomogofderpol}.
	%We will need the following result.
	
	\begin{lemma}\label{lem:SublatticePolynomialVeryHomogeneous}
		We have
		\[
		\pol_{\borw, \htot}\big(\borw (\vect{\genvec}), \jacvartwoim / \jacvaroneim\big)
		=
		\pol_{\borw, 0, \htot}\big(\borw (\genvec_1 + \jacvartwoim\genvec_2/\jacvaroneim , \genvec_2)\big).
		\]
	\end{lemma}
	
	\begin{proof}
		Let $a, c \in \IR$. Then we have
		\begin{align*}
			& a^{m^+} \sum_{\hone, \htwo} (\genvec_1,u_{z^\perp})^{\hone}(\genvec_2,u_{z^\perp})^{\htwo} \pol_{\borw, \hone, \htwo}(\borw (\vect{\genvec}))
			=
			a^{m^+} \pol( g(\vect{\genvec})) = \pol(g(a \genvec_1 + c \genvec_2, \genvec_2))
%			\\
%			&=
%			a^{m^+} \pol( g(\vect{\genvec})) = \pol(g(a \genvec_1 + c \genvec_2, \genvec_2)) 
			\\
			&= \sum_{\hone, \htwo} (a \genvec_1 + c \genvec_2,u_{z^\perp})^{\hone}(\genvec_2,u_{z^\perp})^{\htwo} \pol_{\borw, \hone, \htwo}( \borw(a \genvec_1 + c \genvec_2, \genvec_2)) \\
			&= \sum_{\hone, \htwo} \sum_{j} \binom{\hone}{j} a^{j} (\genvec_1, u_{z^\perp})^{j} c^{\hone - j} (\genvec_2, u_{z^\perp})^{\hone + \htwo - j} \pol_{\borw, \hone, \htwo}( \borw(a \genvec_1 + c \genvec_2, \genvec_2)) \\
			&= \sum_{\hone, \htwo} \sum_{j} \binom{j}{\hone} a^{\hone} (\genvec_1, u_{z^\perp})^{\hone} c^{j - \hone} (\genvec_2, u_{z^\perp})^{j + \htwo - \hone} \pol_{\borw, j, \htwo}(\borw(a \genvec_1 + c \genvec_2, \genvec_2)) \\
			&= \sum_{\hone, \htwo} (\genvec_1, u_{z^\perp})^{\hone} (\genvec_2, u_{z^\perp})^{\htwo} a^{\hone} \sum_{j} \binom{j}{\hone} c^{j - \hone}  \pol_{\borw, j, \hone + \htwo - j}(\borw(a \genvec_1 + c \genvec_2, \genvec_2)).
		\end{align*}
		Hence we have
		\begin{align*}
			a^{m^+} \pol_{\borw, \hone, \htwo}(\borw(\genvec)) = a^{\hone} \sum_{j} \binom{j}{\hone} c^{j - \hone}  \pol_{\borw, j, \hone + \htwo - j}(\borw(a \genvec_1 + c \genvec_2, \genvec_2)).
		\end{align*}
		Apply this with $a = 1, c = \frac{\jacvartwoim}{\jacvaroneim}$ to obtain
		\begin{align*}
			&\pol_{\borw, \htot}\Big(\borw(\genvec_1, \genvec_2), \frac{\jacvartwoim }{ \jacvaroneim}\Big) 
%			\\
%			&= 
			=
			\sum_{\hone} \left(-\frac{\jacvartwoim}{\jacvaroneim}\right)^{\hone} \pol_{\borw, \hone, \htot - \hone}(\borw(\genvec_1, \genvec_2)) \\
			&\qquad=
			\sum_{\hone} \left(-\frac{\jacvartwoim}{\jacvaroneim}\right)^{\hone} \sum_{j} \binom{j}{\hone} \left(\frac{\jacvartwoim}{\jacvaroneim}\right)^{j - \hone}
			 \pol_{\borw, j, \htot - j}\Big(\borw\Big(\genvec_1 + \frac{\jacvartwoim}{\jacvaroneim} \genvec_2, \genvec_2\Big)\Big)
			\\
			&\qquad=
			\sum_{j} \sum_{\hone} \binom{j}{\hone} \left(-\frac{\jacvartwoim}{\jacvaroneim}\right)^{\hone} \left(\frac{\jacvartwoim}{\jacvaroneim}\right)^{j - \hone} \pol_{\borw, j, \htot - j}\Big(\borw\Big(\genvec_1 + \frac{\jacvartwoim}{\jacvaroneim} \genvec_2, \genvec_2\Big)\Big)
			\\
			&\qquad=
			\pol_{\borw, 0, \htot}\Big(\borw\Big(\genvec_1 + \frac{\jacvartwoim}{\jacvaroneim} \genvec_2, \genvec_2\Big)\Big).\qedhere
		\end{align*}
	\end{proof}
	
	In what follows we will be interested in the auxiliary polynomials~$\polw{\vect{\alpha},\borw}{\hone}{\htwo}$ arising as in Definition~\ref{def;genborpolgenus2} with~$\pol=\pol_{\vect{\alpha}}$.
	Since the related proofs are rather technical, we collect in Section~\ref{sec;appendix} the explicit formulas of these polynomials, together with some of their properties.
	We remark here only that~$\polw{\vect{\alpha},\borw}{\hone}{\htwo}$ are in general not very homogeneous; see Lemma~\ref{lemma;nonhomogofderpol} for details.

	\subsection{Reduction of the Siegel theta series~$\boldsymbol{\Theta_{L,2}}$ to smaller lattices}\label{sec;rewThetaL2wrtsplit}
	In this section we explain how to rewrite the theta function~$\Theta_{L,2}$, introduced in Section~\textcolor{\mycolor}{\ref{sec;Siegthealabo}}, with respect to theta functions associated to the Lorentzian sublattice~$\brK$.
	Recall that we assume~$L$ splits off a hyperbolic plane and that we chose~$\genU$, $\genUU$ and a basis of~$L\otimes\RR$ as in~\eqref{eq;choiceofuu'gen2}.
	\begin{lemma}\label{lemma;formwplitThetaL2Janbor}
	Let~$\pol$ be a very homogeneous polynomial of degree $(m^+,0)$ on $(\RR^{b,2})^2$, and let~$\gendisc\in\disc{L}^2$.
	We have
	\ba\label{eq;formwplitThetaL2Janbor}
	\,&\Theta_{L,2}^{\gendisc}(\tau,g,\pol)
	\\
	 &=\frac{\sqrt{\det y}}{2 u_{z^\perp}^2}\sum_{\vect{\lambda}\in \gendisc + (L/\ZZ\genU)^2}\sum_{n\in\ZZ^{1\times 2}}
	\sum_{\hone,\htwo}\frac{\big[\big(n+(u,\vect{\lambda})\bar{\tau}\big)y^{-1}\big]_1^{\hone}\big[\big(n+(u,\vect{\lambda})\bar{\tau}\big)y^{-1}\big]_2^{\htwo}}{(-2i)^{\hone+\htwo}}
	\\
	 &\quad\times\exp\Big(-\frac{1}{8\pi}\trace(\Delta y^{-1})\Big)\big(\pol_{\borw,\hone,\htwo}\big)\big(\borw(\vect{\lambda})\big) \cdot
	e\Big(
	 q(\vect{\lambda}_{w^\perp})\tau+ q(\vect{\lambda}_w)\bar{\tau}
	\Big)
	\\
	&\quad\times \exp\Big(
	-\frac{\pi}{2 u_{z^\perp}^2}\trace \big(n+(u,\vect{\lambda})\tau\big)^t\big(n+(u,\vect{\lambda})\bar{\tau}\big)y^{-1}-\frac{\pi i}{u_{z^\perp}^2}\trace\big((\vect{\lambda},u_{z^\perp}-u_z)n\big)
	\Big),
	\ea
	where we denote by $[\,\cdot\,]_j$ the extraction of the $j$-th entry of the given tuple.
	\end{lemma}
	\begin{proof}
	We follow the wording of~\cite[Proof of Lemma~$5.1$]{bo;grass}, that is, we apply the Poisson summation formula on~$\Theta_{L,2}^\gendisc(\tau,g,\pol)$ with respect to an isotropic line in each subspace~${V=L\otimes\RR}$ of~$V^2$.
	
	We may rewrite any element of~$L^2+\gendisc$ as $\vect{\lambda}+nu$, for some $\vect{\lambda}\in \gendisc + (L/\ZZ\genU)^2$ and some row-vector $n\in\ZZ^{1\times 2}$.
% Note that $L/\ZZ\genU = \brK\oplus\ZZ u'$.
	To simplify the notation, we write~$q(\vect{\lambda}+nu)_z$ instead of~$q((\vect{\lambda}+nu)_z)$, and the same for~$z^\perp$ in place of~$z$.	
	We define the auxiliary function $f(\vect{\lambda},g;n)$ as
	\ba\label{eq;defggenus2bor}
	f(\vect{\lambda},g;n) &= \exp \Big(
	-\frac{1}{8\pi}\trace(\Delta y^{-1})
	\Big)(\pol)\big(g(\vect{\lambda}+nu)\big)
	\\
	&\quad\times e\Big(
	q(\vect{\lambda}+nu)_{z^\perp}\tau
	 +	q(\vect{\lambda}+nu)_z\bar{\tau}
	\Big),
	\ea
	for every $\vect{\lambda}\in \gendisc + (L/\ZZ\genU)^2$, $g\in G$, and $n\in\RR^{1\times 2}$, where $z=g^{-1}(z_0)$.
	We may then rewrite~$\Theta_{L,2}^\gendisc$ using the Poisson summation formula as
	\be\label{eq;aftposumfoplanebor}
	\Theta_{L,2}^\gendisc(\tau,g,\pol)=\sqrt{\det y}\sum_{\vect{\lambda}\in \gendisc + (L/\ZZ\genU)^2}\sum_{n\in\ZZ^{1\times 2}}f(\vect{\lambda},g;n)=\sqrt{\det y}\sum_{\vect{\lambda}\in \gendisc + (L/\ZZ \genU)^2}\sum_{n\in\ZZ^{1\times 2}}\widehat{f}(\vect{\lambda},g;n),
	\ee
	where~$\widehat{f}(\vect{\lambda},g;n)$ is the Fourier transform of~$f$ with respect to the vector~$n$.
	
	Let $\vect{\lambda}=(\lambda_1,\lambda_2)\in \gendisc + (L/\ZZ\genU)^2$, $n=(n_1,n_2)\in\RR^{1\times 2}$, and $\tau=\left(\begin{smallmatrix}
	\tau_1 & \tau_2\\ \tau_2 & \tau_3
	\end{smallmatrix}\right)\in\HH_2$, with analogous notation for the real part $x$ and imaginary part $y$ of $\tau$.
%	Since the Fourier transform of $f$ is with respect to the variables $n_1$ and $n_2$, we simplify the computation gathering all terms appearing in the polynomial factor and exponential defining~$f$ in~\eqref{eq;defggenus2bor} which do not depend on $n_1$ and $n_2$.
	It is easy to see that
	\bes
	q(\vect{\lambda}+nu)_z=q(\vect{\lambda}_z)+(\textcolor{\mycolor}{\vect{\lambda}_z},nu_z)+q(nu_z)\qquad\text{and}\qquad
	q(\vect{\lambda}_z)=q(\vect{\lambda}_w)+(\vect{\lambda},u_z)(\vect{\lambda},u_z)^t/2u_z^2,
	\ees
	same with $z^\perp$ in place of $z$.
	We use such relations to rewrite $f(\vect{\lambda},g;n)$ as
	\ba\label{eq;somesiplofggenus2bor}
	f(\vect{\lambda},g;n)
	&=
	\exp \Big(
	-\frac{1}{8\pi}\trace(\Delta y^{-1})
	\Big)(\pol)\big(g(\vect{\lambda}+nu)\big)
	\cdot
	e\Big(
	q(\vect{\lambda}_{z^\perp})\tau+g(q(\vect{\lambda}_z)\overline{\tau}
	\Big)
	\\
	&\quad\times e\Big(
	(\vect{\lambda}_{z^\perp},nu_{z^\perp})\tau+q(nu_{z^\perp})\tau+(\vect{\lambda}_z,nu_z)\overline{\tau}+q(nu_z)\overline{\tau}
	\Big)
	\ea
	The second factor on the right-hand side of~\eqref{eq;somesiplofggenus2bor} may be computed as
	\bas
	e\Big(
	q(\vect{\lambda}_{z^\perp})\tau+q(\vect{\lambda}_z)\overline{\tau}
	\Big)
	&=
	e\Big(
	q(\vect{\lambda}_{w^\perp})\tau + q(\vect{\lambda}_w)\overline{\tau}
	\Big)
	\\
	&\quad\times  e\Big(
	\frac{(\vect{\lambda},u_{z^\perp})(\vect{\lambda},u_{z^\perp})^t\tau}{2u_{z^\perp}^2} + \frac{(\vect{\lambda},u_z)(\vect{\lambda},u_z)^t\overline{\tau}}{2u_z^2}
	\Big).
	\eas
	Let $h(\vect{\lambda},g;n)$ be the auxiliary function defined as the product between the first and the last factor on the right-hand side of~\eqref{eq;somesiplofggenus2bor}, that is, the part of $f(\vect{\lambda},g;n)$ which depends on the entries~$n_1$ and~$n_2$ of~$n$.
	Using the relation $q(u_{z^\perp})+q(u_z)=0$ and~$n^t\cdot n=\big(\begin{smallmatrix}
	n_1^2 & n_1n_2\\ n_1n_2 & n_2^2
	\end{smallmatrix}\big)$, it is easy to see that
	\ba\label{eq;somerwritofauxhbor}
	h(\vect{\lambda},g;n)
	&=
	\exp \Big(
	-\frac{1}{8\pi}\trace(\Delta y^{-1})
	\Big)(\pol)\big(g(\vect{\lambda}+nu)\big)
	\\
	&\quad \times e\Big(
	\big[x(\vect{\lambda},u) + i y(\vect{\lambda},u_{z^\perp}-u_z)\big]n+iu_{z^\perp}^2yn^t n
	\Big).
	\ea
	We then rewrite~\eqref{eq;aftposumfoplanebor} as
	\ba\label{eq;aftposumfoplanewithhbor}
	\Theta_{L,2}^\gendisc(\tau,g,\pol)
	&=
	\sqrt{\det y} \sum_{\vect{\lambda}\in \gendisc + (L/\ZZ\genU)^2}
	e\Big(
	\frac{(\vect{\lambda},u_{z^\perp})(\vect{\lambda},u_{z^\perp})^t\tau}{2u_{z^\perp}^2} - \frac{(\vect{\lambda},u_z)(\vect{\lambda},u_z)^t\overline{\tau}}{2u_{z^\perp}^2}
	\Big)
	\\
	&\quad\times 
	e\Big(
	q(\vect{\lambda}_{w^\perp})\tau + q(\vect{\lambda}_w)\overline{\tau}
	\Big)
	\sum_{n\in\ZZ^{1\times 2}}\widehat{h}(\vect{\lambda},g;n),
	\ea
	The remaining part of the proof is devoted to the computation of~$\widehat{h}(\vect{\lambda},g;n)$.
	To simplify the notation, we define 
	\ba\label{eq;valAandBgen2bor}
	A=iu_{z^\perp}^2y\qquad\text{and}\qquad B=x(\vect{\lambda},u) + i y(\vect{\lambda},u_{z^\perp}-u_z)=\tau (\vect{\lambda},u_{z^\perp})+\bar{\tau}(\vect{\lambda},u_z),
	\ea
	so that we may rewrite $h$ as
	\ba\label{eq;hdecompinABbor}
	h(\vect{\lambda},g;n)= \exp \Big(
	-\frac{1}{8\pi}\trace (\Delta y^{-1})
	\Big)(\pol)\big(g(\vect{\lambda}+un)\big)	\cdot e\Big(
	An^tn+ Bn
	\Big).
	\ea
	
	We want to make the dependence of $\exp \big(
	-\frac{1}{8\pi}\trace (\Delta y^{-1})
	\big)(\pol)\big(g(\vect{\lambda}+un)\big)$ with respect to the variables $n_1$ and $n_2$ explicit.
	Recall that we split the polynomial~$\pol$ as
	\be\label{eq;splitPolgengen2bor}
	\pol\big(g(\vect{\genvec})\big)=\sum_{\hone,\htwo}(v_1,u_{z^\perp})^{\hone}\cdot (v_2,u_{z^\perp})^{\htwo}\cdot \pol_{\borw,\hone,\htwo}\big(\borw(\vect{\genvec})\big),
	\ee
	and that the operator $-\trace(\Delta y^{-1})/8\pi$ may be rewritten as
	\be\label{eq;explicitformforexpopgen2bor}
	-\frac{\trace(\Delta y^{-1})}{8\pi} = -\frac{1}{8\pi \det y}\Big(
	\ythree\sum_{j=1}^{b+2}\frac{\partial ^2}{\partial x_{j,1}^2}-2\ytwo\sum_{j=1}^{b+2}\frac{\partial}{\partial x_{j,1}}\frac{\partial}{\partial x_{j,2}} + \yone\sum_{j=1}^{b+2}\frac{\partial^2}{\partial x_{j,2}^2}
	\Big).
	\ee
	Since the three factors appearing as the summand on the right-hand side of~\eqref{eq;splitPolgengen2bor} are defined on linearly independent subspaces of $(\RR^{b,2})^2$, we deduce\footnote{If $f$ and $g$ are smooth maps defined on~$\RR^{p,q}$, such that the variables of dependence of $f$ are pairwise different to the ones of $g$, then $\Delta(fg)=\Delta(f)g + f\Delta(g)$.
	This implies that ${\exp(\Delta)(fg)=(\exp(\Delta)(f)){\cdot}(\exp(\Delta)(g))}$.} that
	\ba\label{eq;splitexpoponcasePhomposdef}
	\exp\Big(&-\frac{1}{8\pi}\trace (\Delta y^{-1})\Big)(\pol)\big(g(\vect{\lambda}+nu)\big)
	\\
	&=
	\sum_{\hone,\htwo}\exp\Big(-\frac{1}{8\pi}\trace(\Delta y^{-1})\Big)(\pol_{\borw,\hone,\htwo})\big(\borw (\vect{\lambda})\big)
	\\
	&\quad\times\exp\Big(-\frac{1}{8\pi}\trace(\Delta y^{-1})\Big)\Big(
	(\lambda_1+n_1u,u_{z^\perp})^{\hone}(\lambda_2+n_2u,u_{z^\perp})^{\htwo}
	\Big).
	\ea
	By~\eqref{eq;explicitformforexpopgen2bor}, we may rewrite the summands of the right-hand side of~\eqref{eq;splitexpoponcasePhomposdef} as
	\bas
	\exp\Big(&-\frac{1}{8\pi}\trace(\Delta y^{-1})\Big)(\pol_{\borw,\hone,\htwo})\big(\borw (\vect{\lambda})\big)
	\\
	&\times
	\exp\Big(-\frac{1}{8\pi \genU_{z^\perp}^2 \det y}\Big(
	y_{2,2}\frac{\partial ^2}{\partial n_1^2}-2y_{1,2}\frac{\partial}{\partial n_1}\frac{\partial}{\partial n_2} + y_{1,1}\frac{\partial^2}{\partial n_2^2}
	\Big)\Big)
	\\
	&
	\qquad
	\Big(
	(\lambda_1+n_1u,u_{z^\perp})^{\hone}(\lambda_2+n_2u,u_{z^\perp})^{\htwo}
	\Big).
	\eas	
	We may then rewrite $h$ as
	\bas
	h(\vect{\lambda},g,n)
	&=
	\sum_{\hone,\htwo}
	\exp\Big(-\frac{1}{8\pi}\trace(\Delta y^{-1})\Big)(\pol_{\borw,\hone,\htwo})\big(\borw (\vect{\lambda})\big)\cdot 	e\Big(
	An^tn+Bn
	\Big)
	\\
	&\quad\times
	\exp\Big(-\frac{1}{8\pi \genU_{z^\perp}^2 \det y}\Big(
	y_{2,2}\frac{\partial ^2}{\partial n_1^2}-2y_{1,2}\frac{\partial}{\partial n_1}\frac{\partial}{\partial n_2} + y_{1,1}\frac{\partial^2}{\partial n_2^2}
	\Big)\Big)
	\\
	&\qquad\Big(
	(\lambda_1+n_1u,u_{z^\perp})^{\hone}(\lambda_2+n_2u,u_{z^\perp})^{\htwo}
	\Big).
	\eas
	We compute the Fourier transform of $h$, as a function of $n$, via Lemma~\ref{lemma;onFtransfgenus2}~\ref{item;3onFtransfgenus2}.
	In fact, if we denote by~$N_j$ the~$j$-th entry of $(-n-B^t)A^{-1}/2$, we may compute
	\begin{align}\label{eq;somecompfinforhtransff}
	\widehat{h}(\vect{\lambda},g,n)
	&=
	\det(-2iA)^{-1/2}\sum_{\hone,\htwo}\exp\Big(-\frac{1}{8\pi}\trace(\Delta y^{-1})\Big)(\pol_{\borw,\hone,\htwo})\big(\borw(\vect{\lambda})\big)
	\\
	&\quad\times \exp\Big(\frac{1}{8\pi u_{z^\perp}^2\det y}\Big( \nonumber
	y_{2,2}\frac{\partial ^2}{\partial N_1^2}-2y_{1,2}\frac{\partial}{\partial N_1}\frac{\partial}{\partial N_2} + y_{1,1}\frac{\partial^2}{\partial N_2^2}
	\Big)\Big)
	\\
	&\qquad\exp\Big(-\frac{1}{8\pi u_{z^\perp}^2\det y}\Big( \nonumber
	y_{2,2}\frac{\partial ^2}{\partial N_1^2}-2y_{1,2}\frac{\partial}{\partial N_1}\frac{\partial}{\partial N_2} + y_{1,1}\frac{\partial^2}{\partial N_2^2}
	\Big)\Big)
	\\
	&\qquad\Big( \nonumber
	(\lambda_1+N_1 u,u_{z^\perp})^{\hone}(\lambda_2+N_2 u, u_{z^\perp})^{\htwo}
	\Big)
	\\
	&\quad\times \nonumber
	e\Big(
	-\frac{1}{4}nn^tA^{-1} -\frac{1}{2} BnA^{-1} -\frac{1}{4} BB^tA^{-1}
	\Big)
	\\
	&=\det(-2iA)^{-1/2} \nonumber
	\sum_{\hone,\htwo}\exp\Big(-\frac{1}{8\pi}\trace(\Delta y^{-1})\Big)(\pol_{\borw,\hone,\htwo})\big(\borw(\vect{\lambda})\big)
	\\
	&\quad\times (\lambda_1+N_1 u,u_{z^\perp})^{\hone} (\lambda_2+N_2 u, u_{z^\perp})^{\htwo} \nonumber
	\\
	&\quad\times\nonumber
	e\Big(
	-\frac{1}{4} nn^tA^{-1} -\frac{1}{2} BnA^{-1} -\frac{1}{4} BB^tA^{-1}
	\Big).
	\end{align}
	We may compute $(\lambda_j+N_j u, u_{z^\perp})$, for $j=1,2$, as the $j$-th entry of the vector
	\bas
	(u_{z^\perp},\vect{\lambda}) + \left(\begin{smallmatrix}N_1\\ N_2\end{smallmatrix}\right)u_{z^\perp}^2
	&=
	(u_{z^\perp},\vect{\lambda}) -\frac{1}{2i}\big(n+(u_{z^\perp},\vect{\lambda})\tau+(u_z,\vect{\lambda})\bar{\tau}\big)y^{-1}
	\\
	&=-\frac{1}{2i}\Big(
	n+(u,\vect{\lambda})x-i(u,\vect{\lambda})y
	\Big)y^{-1})=
	-\frac{1}{2i}\big(n+(u,\vect{\lambda})\bar{\tau}\big)y^{-1}.
	\eas
	We replace the values of $A$ and $B$ into~\eqref{eq;somecompfinforhtransff} using that~${\det A=-u_{z^\perp}^4\det y}$, and then replace~$\widehat{h}(\vect{\lambda},g;n)$ in~\eqref{eq;aftposumfoplanewithhbor}.
	The resulting formula may be further simplified rewriting
	\bas
	\trace\Big(
	(n^t ny^{-1})+2 (\vect{\lambda},u)nxy^{-1} + (\vect{\lambda},u)(\vect{\lambda},u)^tx^2y^{-1}
	+ (\vect{\lambda},u)(\vect{\lambda},u)^ty
	\Big)
	\\
	= \trace \big(n+(u,\vect{\lambda})\tau\big)^t \big(n+(u,\vect{\lambda})\bar{\tau}\big)y^{-1},
	\eas
	to eventually obtain~\eqref{eq;formwplitThetaL2Janbor}.
	\end{proof}
%	Recall that we denote by~$\brN$ the positive integer such that~$|\disc{L}|=\brN^2|\disc{\brK}|$.
	\begin{thm}\label{thm;genborthmsplitthetaL}
	Let $\mu\in(\brK\otimes\RR)\oplus\RR u$ be the vector
	\be\label{eq;defmu}
	\mu=-\genUU+\genU_{z^\perp}/2\genU_{z^\perp}^2 + \genU_z/2\genU_z^2.
	\ee
	The theta function~$\Theta_{L,2}(\tau,g,\pol)$ satisfies
	\ba\label{eq;genborthmsplitthetaLmp}
	\Theta_{L,2}(\tau,g,\pol)
	&=
	\frac{1}{2 u_{z^\perp}^2}
	\sum_{c,d\in\ZZ^{1\times 2}}
	\sum_{\hone,\htwo}
	\exp\Big(-\frac{\pi}{2 u_{z^\perp}^2}\trace \big(c\tau + d\big)^t\big(c\bar{\tau}+d\big)y^{-1}\Big)
	\\
	&\quad\times \frac{\big[\big(c\bar{\tau} + d\big)y^{-1}\big]_1^{\hone}\big[\big(c\bar{\tau}+d\big)y^{-1}\big]_2^{\htwo}}{(-2i)^{\hone+\htwo}} \Theta_{\brK,2}(\tau,\mu d,-\mu c,\borw,\pol_{\borw,\hone,\htwo}).
	\ea
	\end{thm}
	In Theorem~\ref{thm;genborthmsplitthetaL}, the theta functions~$\Theta_{\brK,2}$ are attached to a lattice~$M$ of signature~$(b - 1, 1)$ and some polynomials~$\pol_{\borw,\hone,\htwo}$, which may be non-very homogeneous.
	Such theta functions are constructed as in Definition~\ref{def;thetasergen2}.
	They are absolutely convergent, as illustrated in~\cite[p.~$2$]{roehrig}.
	\begin{rem}\label{rem;explofident}
	The isomorphic projection~$p\colon\disc{L}^2\to\disc{\brK}^2$ induces an isomorphism~$\CC[\disc{L}^2]\to\CC[\disc{\brK}^2]$.
	The latter is implicitly used in~\eqref{eq;genborthmsplitthetaLmp} to identify values of theta functions on isomorphic group algebras.	
	Moreover, among the entries of~$\Theta_{\brK,2}$ in Theorem~\ref{thm;genborthmsplitthetaL} we should write~$\mu_{\brK}$ as argument, namely the projection of~$\mu$ to~$\brK\otimes\RR$, instead of~$\mu$.
	However, since $\mu_{\brK}=\mu-(\mu,\genUU)\genU$, we have
	\bas
	\mu_w&=(\mu_{\brK})_w=-\genUU_w,\\
	\mu_{w^\perp}&=(\mu_{\brK})_{w^\perp}=-\genUU_{w^\perp},\\
	(\mu,\genU)&=(\mu_{\brK},\genU).
	\eas
	This explain why we may use such an abuse of notation.
	Note also that the orthogonal projection~$L\otimes\RR\to\brK\otimes\RR$ induces an isometric isomorphism~$w^\perp\oplus w\to w_{\Lor}^\perp\oplus w_{\Lor}=\brK\otimes\RR$.
	This implies that we may identify~$w$ with~$w_\Lor$ and consider~$w$ as an element of~$\Gr(\brK)$; see~\cite[p.~$42$]{br;borchp}.
	Analogously, we may regard~$\borw|_{\brK\otimes\RR}$ as an element of~$\SO(\brK\otimes\RR)$.
	\end{rem}
	\begin{proof}[Proof of Theorem~\ref{thm;genborthmsplitthetaL}]
	Recall that~$p\colon\disc{L}^2\to\disc{\brK}^2$ is the isomorphism induced by the orthogonal projection~$L\to\brK$.
	Every~$\vect{\lambda}\in\gendisc + (L/\ZZ \genU)^2$ can be rewritten as~${\vect{\lambda}=\vecbrlam+c\genUU}$ in a unique way, where~$\vecbrlam\in p(\gendisc) + \brK^2$ and~$c\in\ZZ^{1\times 2}$.
	Since~$\genU$ and~$\genUU$ are orthogonal to~$\brK$, we may rewrite the formula provided by Lemma~\ref{lemma;formwplitThetaL2Janbor} as	
	\begin{align*}%\label{eq;formwplitThetaL2Janborinproof}
	\Theta_{L,2}^{\gendisc}&(\tau,g, \pol)
	\\
	&=\frac{\sqrt{\det y}}{2 u_{z^\perp}^2}
	\sum_{c,d\in\ZZ^{1\times 2}}
	\sum_{\hone,\htwo}
	\sum_{\vecbrlam\in p(\gendisc)+\brK^2} \exp\Big(-\frac{\pi}{2 u_{z^\perp}^2}\trace \big(c\tau + d\big)^t\big(c\bar{\tau}+d\big)y^{-1}\Big)
	\\
	&\quad\times \frac{\big[\big(c\bar{\tau} + d\big)y^{-1}\big]_1^{\hone}\big[\big(c\bar{\tau}+d\big)y^{-1}\big]_2^{\htwo}}{(-2i)^{\hone+\htwo}}
	\cdot
	\exp\Big(
	-\frac{\pi i}{u_{z^\perp}^2}\trace\big((\vecbrlam+c\genUU,u_{z^\perp}-u_z)d\big)
	\Big)
	\\
	&\quad\times\exp\Big(-\frac{1}{8\pi}\trace(\Delta y^{-1})\Big)(\pol_{\borw,\hone,\htwo})\big(\borw(\vecbrlam-c\mu)\big)
	\\
	&\quad\times e\Big(
	q(\vecbrlam-c\mu)_{w^\perp}\tau+ q(\vecbrlam-c\mu)_w\bar{\tau}
	\Big) ,
	\end{align*}
	where we denote by $[\,\cdot\,]_j$ the extraction of the $j$-th entry, and we write $q(\vect{\genvec})_w$ instead of~$q((\vect{\genvec})_w)$, for every $\vect{\genvec}\in (L\otimes\RR)^2$, same for $w^\perp$ in place of $w$.
	To conclude the proof, it is enough to check that
	\bas
	\exp\Big(
	-\frac{\pi i}{u_{z^\perp}^2}\trace\big((\vecbrlam+c\genUU,u_{z^\perp}-u_z)d\big)
	\Big)=e\Big(-(\vecbrlam-\mu c/2,\mu d)\Big).
	\eas
	This may be proved as in~\cite[End of the proof of Theorem~$5.2$]{bo;grass}.
	\end{proof}
	The following results illustrate how to rewrite the formula provided by Theorem~\ref{thm;genborthmsplitthetaL} in terms of vectors in~$\ZZ^{1\times2}$ with \emph{coprime} entries, as well as in terms of the Klingen parabolic subgroup of~$\Sp_4(\ZZ)$.
	\begin{cor}\label{cor;genborthmsplitthetaL}
	The vector-valued theta function $\Theta_{L,2}(\tau,g,\pol)$ satisfies
	\ba\label{eq;corgenborthmsplitthetaL}
	\Theta_{L,2}&(\tau,g,\pol)= \frac{1}{2 u_{z^\perp}^2}\Theta_{\brK,2}(\tau,\borw,\pol_{\borw,0,0})
	\\
	&+ \frac{1}{2 u_{z^\perp}^2}
	\sum_{\substack{c,d\in\ZZ^{1\times 2}\\ \gcd(c,d)=1}}
	\sum_{r\ge 1}
	\sum_{\hone,\htwo}
	\Big(-\frac{r}{2i}\Big)^{\hone+\htwo}\big[\big(c\bar{\tau} + d\big)y^{-1}\big]_1^{\hone}\big[\big(c\bar{\tau}+d\big)y^{-1}\big]_2^{\htwo}
	\\
	&\times \exp\Big(-\frac{\pi r^2}{2 u_{z^\perp}^2}\trace \big(c\tau + d\big)^t\big(c\bar{\tau}+d\big)y^{-1}\Big)
	 \Theta_{\brK,2}(\tau,r\mu d,-r\mu c,\borw,\pol_{\borw,\hone,\htwo}).
	\ea
	\end{cor}
	\begin{proof}
	This is a direct consequence of Theorem~\ref{thm;genborthmsplitthetaL}.
	The first summand on the right-hand side of~\eqref{eq;corgenborthmsplitthetaL} arises from the couple~${(c,d)=(0,0)}$, which is not taken into account in the second summand on the right-hand side of~\eqref{eq;corgenborthmsplitthetaL}.
	\end{proof}
	
	\begin{defi}
	The Klingen parabolic subgroup $\kling$ is the subgroup of matrices in~$\Sp_4(\ZZ)$ whose last row equals~$\left(\begin{smallmatrix}
	0 & 0 & 0 & 1
	\end{smallmatrix}\right)$, namely
	\bes
	\kling=\left\{
	\left(\begin{smallmatrix}
	* & *\\
	0_{1,3} & \textcolor{\mycolor}{1}
	\end{smallmatrix}\right)\in\Sp_4(\ZZ)
	\right\}.
	\ees
	\end{defi}
	\begin{cor}\label{cor;genborthmsplitthetaLklin}
	The vector-valued theta function $\Theta_{L,2}(\tau,g,\pol)$ satisfies
	\ba\label{eq;corgenborthmsplitthetaLklin}
	\,&\Theta_{L,2}(\tau,g,\pol)
	= \frac{1}{2 u_{z^\perp}^2}\Theta_{\brK,2}(\tau,\borw,\pol_{\borw,0,0})
	\\
	&+ \frac{1}{2 u_{z^\perp}^2}
	\sum_{\left(\begin{smallmatrix}* & *\\ c & d\end{smallmatrix}\right)\in\kling\backslash\Sp_4(\ZZ)}
	\sum_{r\ge 1}
	\sum_{\hone,\htwo}
	\Big(-\frac{r}{2i}\Big)^{\hone+\htwo}\big[\big(c\bar{\tau} + d\big)y^{-1}\big]_1^{\hone}\big[\big(c\bar{\tau}+d\big)y^{-1}\big]_2^{\htwo}
	\\
	&\times \exp\Big(-\frac{\pi r^2}{2 u_{z^\perp}^2}\trace \big(c\tau + d\big)^t\big(c\bar{\tau}+d\big)y^{-1}\Big)
	\Theta_{\brK,2}(\tau,r\mu d,-r\mu c,\borw,\pol_{\borw,\hone,\htwo}),
	\ea
	where $(c\,\, d)$ is the last row of~$\left(\begin{smallmatrix}* & *\\ c & d\end{smallmatrix}\right)\in\kling\backslash\Sp_4(\ZZ)$.
	\end{cor}
	\begin{proof}
	It is well-known that the function mapping a matrix in $\Sp_4(\ZZ)$ to its last row induces a bijection between $\kling\backslash\Sp_4(\ZZ)$ and the set of vectors in $\ZZ^4$ with coprime entries.
	We may use such result to rewrite the formula provided by Corollary~\ref{cor;genborthmsplitthetaL} as in~\eqref{eq;corgenborthmsplitthetaLklin}.
	\end{proof}
	%%%%%%%%%%%%%%%%%%%%%%%%%%%%%%%%%%%%%%%%%%%%%%%%%%%%%%%%%%%%%%
	\section{Reduction of Jacobi Siegel theta functions to sublattices}\label{sec;somegenred}
	In Section~\ref{sec:ReductionToSmallerLattices} we use a method of \cite{bo;grass} to rewrite the Jacobi theta function of Example~\ref{exmpl:JacobiThetaPoly} as a Poincar\'e series. This will then be used in Section~\ref{sec:JacobiThetaInnerProducts} to calculate Petersson inner products of a Jacobi form with the Jacobi theta function via the unfolding trick.

\subsection{Reduction to smaller lattices}\label{sec:ReductionToSmallerLattices}

In this section we will derive a formula for the Jacobi theta function as a Poincar\'e series. Therefore, let $\isotropicvec \in \lattice$ be primitive isotropic of level $N_\isotropicvec$, i.e.~$(\isotropicvec, L) = N_\isotropicvec \IZ$, and $\isotropicvecprime \in \lattice'$ with $(\isotropicvec, \isotropicvecprime) = 1$.
We denote by~$\sublattice$ the orthogonal complement of the span of~$\isotropicvec$ and~$\isotropicvecprime$.

\textcolor{\mycolor}{The idea is the same as in \cite[Section 5]{bo;grass}. We first make a partial Poisson summation along $\IZ u$. The sum over the primitive elements of the sublattice $\IZ u \oplus \IZ u'$ can be identified with a sum over the cosets $\Gamma_\infty \bs \SL_2(\IZ)$, where $\Gamma_\infty$ is the standard parabolic in $\SL_2(\IZ)$, which leads to a Poincar\'e series. We start with calculating the partial Fourier transform along the sublattice $\IZ u$.}

Let~$\calP$ be a homogeneous polynomial of degree $\degjac$.
We define the polynomials $\calP_{\subspaceisometry, \subpolposdeg}$ of degree $\degjac - \subpolposdeg$, for every~$\subpolposdeg \in \IN$, by
$$\calP(\isometry(\lambda)) = \sum_{\subpolposdeg} (\lambda, \isotropicvec_{\isometrypos})^{\subpolposdeg} \calP_{\subspaceisometry, \subpolposdeg}(\subspaceisometry(\lambda)).$$
Here~$\subspaceisometry\colon\lattice \otimes \IR \to \lattice \otimes \IR$ is the linear map defined as~$\lambda \mapsto \isometry(\lambda_{\subspaceisometrypos} + \lambda_{\subspaceisometryneg})$, where $\subspaceisometrypos$ and~$\subspaceisometryneg$ are the orthogonal complements of $\isotropicvec_{\isometrypos}$ and~$\isotropicvec_{\isometryneg}$ in respectively~$\isometrypos$ and~$\isometryneg$.
This is the analogue of its homonym in Section~\ref{sec;splitgen2theta}.

Let
\be\label{eq;defmuinsecjac}
\mu = - \isotropicvecprime + \frac{\isotropicvec_{\isometrypos}}{2 \isotropicvec_{\isometrypos}^2} + \frac{\isotropicvec_{\isometryneg}}{2 \isotropicvec_{\isometryneg}^2}.
\ee

\textcolor{\mycolor}{The next lemma calculates the partial Fourier transform along $\IZ u$.}

\begin{lemma}\label{lem:partialFT}
	\textcolor{\mycolor}{Let $\sigma \in L'$. For $\lambda \in \sigma + L / u$ and $n \in \IZ$, let~$g(\lambda, n)$ be defined as
	\begin{align*}
%		&g(\lambda, n)
%		\\
		&
%		=
		e\big(q(\lambda_{\isometrypos}) \jacvarone + q(\lambda_{\isometryneg}) \overline{\jacvarone} + \jacvartwo (\lambda , \jacindexvec_{\isometrypos}) + \overline{\jacvartwo} (\lambda, \jacindexvec_{\isometryneg})\big)
		%		\\
		%		&\quad\times
		\cdot
		\exp(-\Delta / 8 \pi \jacvaroneim)(\calP)\Big(\isometry\Big(\lambda + \frac{\jacvartwoim}{\jacvaroneim} \jacindexvec + n \isotropicvec\Big)\Big)
		\\
		&\quad\times
		e\Big(n^2 \big(\jacvarone q(\isotropicvec_{\isometrypos}) + \overline{\jacvarone} q(\isotropicvec_{\isometryneg})\big) + n \big(\jacvarone (\lambda, \isotropicvec_{\isometrypos}) + \overline{\jacvarone} (\lambda, \isotropicvec_{\isometryneg}) + \jacvartwo (\jacindexvec, \isotropicvec_{\isometrypos}) + \overline{\jacvartwo} (\jacindexvec, \isotropicvec_{\isometryneg})\big)\Big).
	\end{align*}
	Then for $\lambda \in K + \sigma - c u'$ and $c \in \IZ$ with $c \equiv (\sigma, u) \pmod{N_u}$ we have
	\begin{align*}
		&\hat{g}\left(\lambda + c \isotropicvecprime, n\right)
		\\
		&= \frac{1}{\sqrt{2 \jacvaroneim \isotropicvec_{\isometrypos}^2}} \sum_{\subpolposdeg} \frac{\left(c \overline{\jacvarone} + n\right)^{\subpolposdeg}}{\left(-2 i \jacvaroneim\right)^{\subpolposdeg}}
		e\bigg(-i\frac{\left(\jacvartwoim \left(\jacindexvec, \isotropicvec_{\isometrypos}\right)\right)^2}{\jacvaroneim \isotropicvec_{\isometrypos}^2} - \frac{n \jacvartwoim \left(\jacindexvec, \isotropicvec_{\isometrypos}\right)}{\jacvaroneim \isotropicvec_{\isometrypos}^2} \bigg)
		\\
		&\quad \times
		\exp\left(- \Delta / 8 \pi \jacvaroneim\right)\left(\calP_{\subspaceisometry, \subpolposdeg}\right)\left(\subspaceisometry\Big(\lambda + \frac{\jacvartwoim}{\jacvaroneim} \jacindexvec\Big)\right) e\left(- n (\sigma, \isotropicvecprime) + cn q(\isotropicvecprime)\right) 
		\\
		&\quad\times
		e\bigg(-\frac{\lvert n + c \jacvarone \rvert^2 + 2 c \left(\jacvarone \overline{\jacvartwo} \left(\jacindexvec, \isotropicvec_{\isometrypos}\right) + \overline{\jacvarone} \jacvartwo \left(\jacindexvec, \isotropicvec_{\isometryneg}\right)\right)}{4 i \jacvaroneim \isotropicvec_{\isometrypos}^2}\bigg)
		\cdot
		e\Big(\jacvarone q\big(\left(\lambda - c \mu\right)_{\subspaceisometrypos}\big)
		\\
		&\quad+
		\overline{\jacvarone} q\big(\left(\lambda - c \mu\right)_{\subspaceisometryneg}\big) + \jacvartwo \left(\lambda - c \mu, \jacindexvec_{\subspaceisometrypos}\right) + \overline{\jacvartwo} \left(\lambda - c \mu, \jacindexvec_{\subspaceisometryneg}\right) - \left(\lambda - c \mu / 2, n \mu\right)\Big),
	\end{align*}}
\end{lemma}

\begin{proof}
	\textcolor{\mycolor}{Let
	\begin{align*}
		A &= \jacvarone q(\isotropicvec_{\isometrypos}) + \overline{\jacvarone} q(\isotropicvec_{\isometryneg}), \\
		B &= \jacvarone (\lambda, \isotropicvec_{\isometrypos}) + \overline{\jacvarone} (\lambda, \isotropicvec_{\isometryneg}) + \jacvartwo (\jacindexvec, \isotropicvec_{\isometrypos}) + \overline{\jacvartwo} (\jacindexvec, \isotropicvec_{\isometryneg}).
		\end{align*}}
	We remark that $(\jacindexvec, \isotropicvec) = 0$ implies that~$B = \jacvarone (\lambda, \isotropicvec_{\isometrypos}) + \overline{\jacvarone} (\lambda, \isotropicvec_{\isometryneg}) + 2i\jacvartwoim (\jacindexvec, \isotropicvec_{\isometrypos})$.
	Using
	\begin{align*}
		\exp(- \Delta / 8 \pi \jacvaroneim)(\calP)\big(\isometry(\lambda + n \isotropicvec)\big)
		&=
		\sum_{\subpolposdeg}
		\exp\Big(- \frac{1}{8 \pi \jacvaroneim \isotropicvec_{\isometrypos}^2} \frac{d^2}{dn^2}\Big)\left(\left(\lambda + n \isotropicvec, \isotropicvec_{\isometrypos}\right)^{\subpolposdeg}\right)
		\\
		&\quad\times
		\exp(- \Delta / 8 \pi \jacvaroneim)(\calP_{\subspaceisometry, \subpolposdeg})\big(\subspaceisometry(\lambda)\big)
	\end{align*}
	and \cite[Corollary 3.3]{bo;grass}, we see that the Fourier transform of $g$ with respect to $n$ is given by
	\begin{align*}
		& \frac{1}{\sqrt{2 \jacvaroneim \isotropicvec_{\isometrypos}^2}}
		\sum_{\subpolposdeg}
		\exp(- \Delta / 8 \pi \jacvaroneim)(\calP_{\subspaceisometry, \subpolposdeg})
		\bigg(\subspaceisometry\Big(\lambda + \frac{\jacvartwoim}{\jacvaroneim} \jacindexvec\Big)\bigg)
		\\
		&\quad\times
		\exp\Big(\frac{1}{8 \pi \jacvaroneim \isotropicvec_{\isometrypos}^2} \frac{d^2}{d N^2}\Big)
		\bigg(\exp\Big(- \frac{1}{8 \pi \jacvaroneim \isotropicvec_{\isometrypos}^2} \frac{d^2}{dN^2}\Big)
		\Big(\Big(\lambda + \frac{\jacvartwoim}{\jacvaroneim} \jacindexvec + N \isotropicvec, \isotropicvec_{\isometrypos}\Big)^{\subpolposdeg}\Big)\bigg) \\
		&\quad\times
		e\bigg(- \frac{\big(n + \jacvarone(\lambda, \isotropicvec_{\isometrypos}) + \overline{\jacvarone} (\lambda, \isotropicvec_{\isometryneg}) + 2 i \jacvartwoim (\jacindexvec, \isotropicvec_{\isometrypos})\big)^2}{4 i \jacvaroneim \isotropicvec_{\isometrypos}^2}\bigg)
		\\
		&\quad\times
		e\Big(q(\lambda_{\isometrypos}) \jacvarone + q(\lambda_{\isometryneg}) \overline{\jacvarone} + \jacvartwo (\lambda , \jacindexvec_{\isometrypos}) + \overline{\jacvartwo} (\lambda, \jacindexvec_{\isometryneg})\Big), \numberthis \label{eq:ReductionSummand}
	\end{align*}
	where $N = -\frac{n + B}{2 A}$.
	Since
	$$N = - \frac{n + \jacvarone (\lambda, \isotropicvec_{\isometrypos}) + \overline{\jacvarone} (\lambda, \isotropicvec_{\isometryneg}) + 2i \jacvartwoim (\jacindexvec, \isotropicvec_{\isometrypos})}{2 i \jacvaroneim \isotropicvec_{\isometrypos}^2},$$
	we deduce that
	\begin{align*}
		& \exp\Big(\frac{1}{8 \pi \jacvaroneim \isotropicvec_{\isometrypos}^2} \frac{d^2}{d N^2}\Big)
		\bigg(\exp\Big(- \frac{1}{8 \pi \jacvaroneim \isotropicvec_{\isometrypos}^2} \frac{d^2}{dN^2}\Big)\Big(\Big(\lambda + \frac{\jacvartwoim}{\jacvaroneim} \jacindexvec + N \isotropicvec, \isotropicvec_{\isometrypos}\Big)^{\subpolposdeg}\Big)\bigg)
		\\
		%		&=
		%		\Big(\lambda + \frac{\jacvartwoim}{\jacvaroneim} \jacindexvec + N \isotropicvec, \isotropicvec_{\isometrypos}\Big)^{\subpolposdeg}
		%		\\
		&= \bigg(\lambda + \frac{\jacvartwoim}{\jacvaroneim} \jacindexvec - \frac{n + \jacvarone (\lambda, \isotropicvec_{\isometrypos}) + \overline{\jacvarone} (\lambda, \isotropicvec_{\isometryneg}) + 2i \jacvartwoim (\jacindexvec, \isotropicvec_{\isometrypos})}{2 i \jacvaroneim \isotropicvec_{\isometrypos}^2} \isotropicvec, \isotropicvec_{\isometrypos}\bigg)^{\subpolposdeg}.
	\end{align*}
	A simple computation shows that this equals~$\big(-((\lambda, \isotropicvec) \overline{\jacvarone} + n)/2i \jacvaroneim\big)^{\subpolposdeg}$.
	%		\bas
	%		&= \left(\frac{2 i \jacvaroneim \left(\lambda, \isotropicvec_{\isometrypos}\right) + 2 i \jacvartwoim \left(\jacindexvec, \isotropicvec_{\isometrypos}\right) - n - \jacvarone \left(\lambda, \isotropicvec_{\isometrypos}\right) - \overline{\jacvarone} \left(\lambda, \isotropicvec_{\isometryneg}\right) - 2i \jacvartwoim \left(\jacindexvec, \isotropicvec_{\isometrypos}\right)}{2 i \jacvaroneim}\right)^{\subpolposdeg} \\
	%		&=
	%		\frac{\left(\left(\lambda, \isotropicvec\right) \overline{\jacvarone} + n\right)^{\subpolposdeg}}{\left(-2i \jacvaroneim\right)^{\subpolposdeg}}.
	%		\eas
	On the other hand, the exponential term in \eqref{eq:ReductionSummand} may be rewritten as
	\begin{align*}
		& e\bigg(- \frac{\left(n + \jacvarone\left(\lambda, \isotropicvec_{\isometrypos}\right) + \overline{\jacvarone} \left(\lambda, \isotropicvec_{\isometryneg}\right) + 2 i \jacvartwoim \left(\jacindexvec, \isotropicvec_{\isometrypos}\right)\right)^2}{4 i \jacvaroneim \isotropicvec_{\isometrypos}^2}\bigg)
		\\
		&\quad\times e\Big(q\left(\lambda_{\isometrypos}\right) \jacvarone + q\left(\lambda_{\isometryneg}\right) \overline{\jacvarone} + \jacvartwo \left(\lambda , \jacindexvec_{\isometrypos}\right) + \overline{\jacvartwo} \left(\lambda, \jacindexvec_{\isometryneg}\right)\Big) \\
		&=
		e\bigg(-i\frac{\left(\jacvartwoim \left(\jacindexvec, \isotropicvec_{\isometrypos}\right)\right)^2}{\jacvaroneim \isotropicvec_{\isometrypos}^2} - \frac{n \jacvartwoim \left(\jacindexvec, \isotropicvec_{\isometrypos}\right)}{\jacvaroneim \isotropicvec_{\isometrypos}^2} \bigg)
		\\
		&\quad\times e\Big(\jacvarone q\left(\lambda_{\subspaceisometrypos}\right) + \overline{\jacvarone} q\left(\lambda_{\subspaceisometryneg}\right) + \jacvartwo \left(\lambda, \jacindexvec_{\subspaceisometrypos}\right) + \overline{\jacvartwo} \left(\lambda, \jacindexvec_{\subspaceisometryneg}\right)\Big) \\
		&\quad\times e\bigg(-\frac{\lvert n + \left(\lambda, \isotropicvec\right) \jacvarone \rvert^2 + 2 \left(\lambda, \isotropicvec\right) \left(\jacvarone \overline{\jacvartwo} \left(\jacindexvec, \isotropicvec_{\isometrypos}\right) + \overline{\jacvarone} \jacvartwo \left(\jacindexvec, \isotropicvec_{\isometryneg}\right)\right)}{4 i \jacvaroneim \isotropicvec_{\isometrypos}^2} - \frac{n \left(\lambda, \isotropicvec_{\isometrypos} - \isotropicvec_{\isometryneg}\right)}{2 \isotropicvec_{\isometrypos}^2}\bigg).
	\end{align*}
	Every element in~$\sigma + \lattice / \isotropicvec$ can be written as~$\lambda + c \isotropicvecprime$ for some $\lambda \in \sublattice + \sigma - c\isotropicvecprime$ and $c \in \IZ$ such that~$c \equiv \left(\sigma, \isotropicvec\right)$ mod~$N_\isotropicvec$.
	We may then compute
	\begin{align*}
		&\hat{g}\left(\lambda + c \isotropicvecprime, n\right)
		\\
		&=
		\frac{1}{\sqrt{2 \jacvaroneim \isotropicvec_{\isometrypos}^2}} \sum_{\subpolposdeg}
		\exp\left(- \Delta / 8 \pi \jacvaroneim\right)\left(\calP_{\subspaceisometry, \subpolposdeg}\right)
		\left(\subspaceisometry\Big(\lambda + \frac{\jacvartwoim}{\jacvaroneim} \jacindexvec\Big)\right)
		\\
		&\quad\times
		\frac{\left(c \overline{\jacvarone} + n\right)^{\subpolposdeg}}{\left(-2i \jacvaroneim\right)^{\subpolposdeg}}
		e\bigg(-i\frac{\left(\jacvartwoim \left(\jacindexvec, \isotropicvec_{\isometrypos}\right)\right)^2}{\jacvaroneim \isotropicvec_{\isometrypos}^2} - \frac{n \jacvartwoim \left(\jacindexvec, \isotropicvec_{\isometrypos}\right)}{\jacvaroneim \isotropicvec_{\isometrypos}^2} \bigg)
		\\
		&\quad\times e\Big(\jacvarone q\left(\left(\lambda + c \isotropicvecprime\right)_{\subspaceisometrypos}\right) + \overline{\jacvarone} q\left(\left(\lambda + c \isotropicvecprime\right)_{\subspaceisometryneg}\right) + \jacvartwo \left(\lambda + c \isotropicvecprime, \jacindexvec_{\subspaceisometrypos}\right) + \overline{\jacvartwo} \left(\lambda + c \isotropicvecprime, \jacindexvec_{\subspaceisometryneg}\right)\Big) 
		\\
		&\quad\times
		e\bigg(-\frac{\lvert n + c \jacvarone \rvert^2 + 2 c \left(\jacvarone \overline{\jacvartwo} \left(\jacindexvec, \isotropicvec_{\isometrypos}\right) + \overline{\jacvarone} \jacvartwo \left(\jacindexvec, \isotropicvec_{\isometryneg}\right)\right)}{4 i \jacvaroneim \isotropicvec_{\isometrypos}^2} - \frac{n \left(\lambda + c \isotropicvecprime, \isotropicvec_{\isometrypos} - \isotropicvec_{\isometryneg}\right)}{2 \isotropicvec_{\isometrypos}^2}\bigg)
		\\
		&= \frac{1}{\sqrt{2 \jacvaroneim \isotropicvec_{\isometrypos}^2}} \sum_{\subpolposdeg} \frac{\left(c \overline{\jacvarone} + n\right)^{\subpolposdeg}}{\left(-2 i \jacvaroneim\right)^{\subpolposdeg}}
		e\bigg(-i\frac{\left(\jacvartwoim \left(\jacindexvec, \isotropicvec_{\isometrypos}\right)\right)^2}{\jacvaroneim \isotropicvec_{\isometrypos}^2} - \frac{n \jacvartwoim \left(\jacindexvec, \isotropicvec_{\isometrypos}\right)}{\jacvaroneim \isotropicvec_{\isometrypos}^2} \bigg)
		\\
		&\quad \times
		\exp\left(- \Delta / 8 \pi \jacvaroneim\right)\left(\calP_{\subspaceisometry, \subpolposdeg}\right)\left(\subspaceisometry\Big(\lambda + \frac{\jacvartwoim}{\jacvaroneim} \jacindexvec\Big)\right) e\left(- n (\sigma, \isotropicvecprime) + cn q(\isotropicvecprime)\right) 
		\\
		&\quad\times
		e\bigg(-\frac{\lvert n + c \jacvarone \rvert^2 + 2 c \left(\jacvarone \overline{\jacvartwo} \left(\jacindexvec, \isotropicvec_{\isometrypos}\right) + \overline{\jacvarone} \jacvartwo \left(\jacindexvec, \isotropicvec_{\isometryneg}\right)\right)}{4 i \jacvaroneim \isotropicvec_{\isometrypos}^2}\bigg)
		\cdot
		e\Big(\jacvarone q\big(\left(\lambda - c \mu\right)_{\subspaceisometrypos}\big)
		\\
		&\quad+
		\overline{\jacvarone} q\big(\left(\lambda - c \mu\right)_{\subspaceisometryneg}\big) + \jacvartwo \left(\lambda - c \mu, \jacindexvec_{\subspaceisometrypos}\right) + \overline{\jacvartwo} \left(\lambda - c \mu, \jacindexvec_{\subspaceisometryneg}\right) - \left(\lambda - c \mu / 2, n \mu\right)\Big),
	\end{align*}
	where we used that~$\left(\lambda, \isotropicvecprime\right) = \left(\lambda + c \isotropicvecprime - c \isotropicvecprime, \isotropicvecprime\right) \equiv \left(\sigma - c \isotropicvecprime, \isotropicvecprime\right)$ mod~$1$ and that
	\begin{align*}
		\frac{\left(\lambda + c \isotropicvecprime, \isotropicvec_{\isometrypos} - \isotropicvec_{\isometryneg}\right)}{2 \isotropicvec_{\isometrypos}^2} = \left(\lambda - c \mu / 2, \mu\right) + c q\left(\isotropicvecprime\right) + \left(\lambda, \isotropicvecprime\right).
	\end{align*}
\end{proof}

We will need the following result.
\begin{lemma}[{\cite[Lemma 6.3]{kiefer}}]\label{lem:kiefer}
Let $\sigma \in D_\sublattice, \mat = \left(\left(\begin{smallmatrix} a & b \\ c & d\end{smallmatrix}\right), \phi\right) \in \Mp_2\left(\IZ\right)$ and $n \in \IZ$. Then
\begin{align*}
\weilrep_\lattice(\mat)
\sum_{m \in \IZ / N_\isotropicvec \IZ}
\frake_{\sigma + \frac{m \isotropicvec}{N_{\isotropicvec}}}\Big(-\frac{mn}{N_{\isotropicvec}}\Big)
=
\big(\weilrep_\sublattice(\mat) \frake_{\sigma}\big)
\sum_{m \in \IZ / N_\isotropicvec \IZ} \frake_{\frac{m \isotropicvec}{N_{\isotropicvec}} - nc \isotropicvecprime}\Big(-\frac{amn}{N_{\isotropicvec}} + q(\isotropicvecprime) ac n^2\Big).
\end{align*}
\end{lemma}
Let $\Gamma_\infty = \left\{\left(\begin{smallmatrix}1 & n \\ 0 & 1\end{smallmatrix}\right) : n \in \IZ\right\}$ be the standard parabolic subgroup in~$\SL_2(\ZZ)$.
For every~${\mat 
%= \left(\begin{smallmatrix} a & b \\ c & d\end{smallmatrix}\right)
\in \SL_2(\IZ)}$ we denote by~$\tilde{\mat}
%= (\mat, \sqrt{c \tau + d})
\in \Mp_2(\IZ)$ the standard preimage of~$\gamma$ under the metaplectic double cover.
The following result is an analogue of~\cite[Theorem~$5.2$]{bo;grass} for the Jacobi Siegel theta functions introduced in \ref{exmpl:JacobiThetaPoly}.
It provides a way to rewrite~$\Theta_{\lattice, \jacindexvec}(\tau,g,\pol)$ as a Poincaré series in terms of the theta functions associated to~$\brK$. 
\begin{thm}\label{thm;rewritingthetajacaspoincare}
	If~$(\jacindexvec, \isotropicvec) = 0$, then $\Theta_{\lattice, \jacindexvec}(\tau,g,\pol)$ may be rewritten as
	\begin{align*}
		&\frac{1}{\sqrt{2 \isotropicvec_{\isometrypos}^2}} \Theta_{\sublattice, \jacindexvec}(\jacvarone, \jacvartwo, \subspaceisometry, \calP_{\subspaceisometry, 0}) \sum_{m \in \IZ / N_\isotropicvec \IZ} \frake_{\frac{m \isotropicvec}{N_\isotropicvec}}
		\\
		&\quad +
		\frac{1}{\sqrt{2 \isotropicvec_{\isometrypos}^2}}
		\sum_{\mat = \big(\begin{smallmatrix} a & b \\ c & d\end{smallmatrix}\big) \in \Gamma_\infty \bs \SL_2(\IZ)}
		\sum_{n = 1}^\infty \sum_{\subpolposdeg} e\left(-\frac{q(\jacindexvec) c \jacvartwo^2}{c \jacvarone + d}\right)
		\\
		&\quad\times
		\frac{n^{\subpolposdeg}(c \jacvarone + d)^{ - \degjac + (b^- - b^+)/2}}{\left(-2 i \Im\left(\mat \jacvarone\right)\right)^{\subpolposdeg}}
		e\bigg(-\frac{n^2 + 4 i n  \Im(\mat \jacvartwo) (\jacindexvec, \isotropicvec_{\isometrypos})}{4 i \Im(\mat \jacvarone) \isotropicvec_{\isometrypos}^2}\bigg)
		\\
		&\quad\times
		\weilrep_{\lattice, \jacindexvec}^{-1}(\tilde{\mat}) \bigg(\Theta_{\sublattice, \jacindexvec}(\mat\jacvarone, \mat\jacvartwo, n\mu, 0, \subspaceisometry, \calP_{\subspaceisometry, \subpolposdeg})
		\sum_{m \in \IZ / N_\isotropicvec \IZ}
		\frake_{\frac{m \isotropicvec}{N_\isotropicvec}}\Big(- \frac{mn}{N_\isotropicvec}\Big) \bigg).
	\end{align*}
\end{thm}

\begin{proof}
	Let $\sigma \in \lattice'$.
	We write the elements in~$\sigma + L$ as~$\lambda + n \isotropicvec$ for some~$\lambda \in \sigma + L / u$ and~$n \in \IZ$, so that we may compute
	\begin{align*}
		&
		\theta_{\sigma + \lattice, \jacindexvec}
		(\jacvarone, \jacvartwo, \isometry, \calP)
		\\
		&= \jacvaroneim^{-b^+/2}
		\exp\big(-4 \pi q(\jacindexvec_{\isometryneg}) \jacvartwoim^2 / \jacvaroneim\big)
		\sum_{\lambda \in \sigma + \lattice / \isotropicvec}
		\sum_{n \in \IZ}
		\exp(-\Delta / 8 \pi \jacvaroneim)(\calP)
		\bigg(\isometry\Big(\lambda + \frac{\jacvartwoim}{\jacvaroneim} \jacindexvec + n \isotropicvec\Big)\bigg)
		\\
		&\quad\times
		e\Big(q\big((\lambda + n \isotropicvec)_{\isometrypos}\big) \jacvarone + q\big((\lambda + n \isotropicvec)_{\isometryneg}\big) \overline{\jacvarone} + \jacvartwo (\lambda + n \isotropicvec, \jacindexvec_{\isometrypos}) + \overline{\jacvartwo} (\lambda + n \isotropicvec, \jacindexvec_{\isometryneg})\Big)
		\\
		&= \jacvaroneim^{-b^+/2}
		\exp\big(-4 \pi q(\jacindexvec_{\isometryneg}) \jacvartwoim^2 / \jacvaroneim\big)
		\sum_{\lambda \in \sigma + \lattice / \isotropicvec} \sum_{n \in \IZ} g(\lambda, n).
	\end{align*}
	\textcolor{\mycolor}{The factor in front of the theta function can be rewritten as}
	$$\jacvaroneim^{b^-/2}
	\exp\big(4 \pi q\left(\jacindexvec_{\isometryneg}\right) \jacvartwoim^2 / \jacvaroneim\big)
	=
	\sqrt{\jacvaroneim} \cdot e\bigg(-i \frac{\left(\jacindexvec, \isotropicvec_{\isometryneg}\right)^2 \jacvartwoim^2}{\jacvaroneim \isotropicvec_{\isometryneg}^2}\bigg)
	\jacvaroneim^{(b^- - 1)/2}
	\exp\big(4 \pi q\left(\jacindexvec_{\subspaceisometryneg}\right) \jacvartwoim^2 / \jacvaroneim\big).$$
	\textcolor{\mycolor}{Applying a partial Poisson summation in $n$ and using Lemma \ref{lem:partialFT}, we may rewrite the theta function as}
	\begin{align*}
		& \theta_{\sigma + \lattice, \jacindexvec}\left(\jacvarone, \jacvartwo, \isometry, \calP\right)
		\\
		&=
		\frac{1}{\sqrt{2 \isotropicvec_{\isometrypos}^2}}
		e\bigg(-i\frac{\left(\jacvartwoim \left(\jacindexvec, \isotropicvec_{\isometrypos}\right)\right)^2}{\jacvaroneim \isotropicvec_{\isometrypos}^2} \bigg)
		\cdot
		e\bigg(- i \frac{\left(\jacindexvec, \isotropicvec_{\isometryneg}\right)^2 \jacvartwoim^2}{\jacvaroneim \isotropicvec_{\isometryneg}^2}\bigg)
		\\
		&\quad\times
		\sum_{\substack{c, d \in \IZ \\ c \equiv \left(\sigma, \isotropicvec\right) \pmod N}}
		\sum_{\subpolposdeg} \frac{\left(c \overline{\jacvarone} + d\right)^{\subpolposdeg}}{\left(-2 i \jacvaroneim\right)^{\subpolposdeg}}
		e\bigg(- \frac{d \jacvartwoim \left(\jacindexvec, \isotropicvec_{\isometrypos}\right)}{\jacvaroneim \isotropicvec_{\isometrypos}^2} \bigg)
		\cdot
		e\left(- d (\sigma, \isotropicvecprime) + cd q(\isotropicvecprime)\right)
		\\
		&\quad\times
		e\bigg(-\frac{\lvert c \jacvarone + d \rvert^2 + 2 c \left(\jacvarone \overline{\jacvartwo} \left(\jacindexvec, \isotropicvec_{\isometrypos}\right) + \overline{\jacvarone} \jacvartwo \left(\jacindexvec, \isotropicvec_{\isometryneg}\right)\right)}{4 i \jacvaroneim \isotropicvec_{\isometrypos}^2}\bigg)
		\\
		&\quad\times \theta_{\sigma - c\isotropicvecprime + \sublattice, \jacindexvec}\left(\jacvarone, \jacvartwo, d \mu, -c \mu, \subspaceisometry, \calP_{\subspaceisometry, \subpolposdeg}\right)
		\\
		&= \frac{1}{\sqrt{2 \isotropicvec_{\isometrypos}^2}} \sum_{\substack{c, d \in \IZ \\ c \equiv \left(\sigma, \isotropicvec\right) \pmod N}} \sum_{\subpolposdeg} \frac{\left(c \overline{\jacvarone} + d\right)^{\subpolposdeg}}{\left(-2 i \jacvaroneim\right)^{\subpolposdeg}}
		e\left(- d (\sigma, \isotropicvecprime) + cd q(\isotropicvecprime)\right)
		\\
		&\quad\times
		e\bigg(-\frac{\lvert c \jacvarone + d \rvert^2 + 2 \left(c\jacvarone + d\right) \overline{\jacvartwo} \left(\jacindexvec, \isotropicvec_{\isometrypos}\right) + 2 \left(c\overline{\jacvarone} + d\right) \jacvartwo \left(\jacindexvec, \isotropicvec_{\isometryneg}\right)}{4 i \jacvaroneim \isotropicvec_{\isometrypos}^2}\bigg) \\
		&\quad\times\theta_{\sigma - c\isotropicvecprime + \sublattice, \jacindexvec}\left(\jacvarone, \jacvartwo, d \mu, -c \mu, \subspaceisometry, \calP_{\subspaceisometry, \subpolposdeg}\right).
	\end{align*}
	Summing over $\sigma \in D_\lattice$ and using the correspondence between coprime integers $\left(c, d\right)$ and elements in $\Gamma_\infty \bs \SL_2\left(\IZ\right)$ yields
	\begin{align*}
		&\Theta_{\lattice, \jacindexvec}\left(\jacvarone, \jacvartwo, \isometry, \calP\right) \\
		&= \frac{1}{\sqrt{2 \isotropicvec_{\isometrypos}^2}} \sum_{\substack{\sigma \in D_\lattice \\ \left(\sigma, \isotropicvec\right) \equiv 0 \pmod{N_\isotropicvec}}} \theta_{\sigma + \sublattice, \jacindexvec}\left(\jacvarone, \jacvartwo, \subspaceisometry, \calP_{\subspaceisometry, 0}\right) \frake_{\sigma} \\
		&\quad+ 
		\frac{1}{\sqrt{2 \isotropicvec_{\isometrypos}^2}} \sum_{\substack{c, d \in \IZ \\ \left(c, d\right) = 1}} \sum_{\substack{\sigma \in D_\lattice \\ c \equiv \left(\sigma, \isotropicvec\right) \pmod{N_\isotropicvec}}} \sum_{n = 1}^\infty \sum_{\subpolposdeg} \frac{n^{\subpolposdeg}\left(c \overline{\jacvarone} + d\right)^{\subpolposdeg}}{\left(-2 i \jacvaroneim\right)^{\subpolposdeg}} e\left(- nd (\sigma, \isotropicvecprime) + n^2 cd q(\isotropicvecprime)\right) \\
		&\quad\times e\bigg(-\frac{n^2 \lvert c \jacvarone + d \rvert^2 + 2 n \left(c\jacvarone + d\right) \overline{\jacvartwo} \left(\jacindexvec, \isotropicvec_{\isometrypos}\right) + 2n \left(c\overline{\jacvarone} + d\right) \jacvartwo \left(\jacindexvec, \isotropicvec_{\isometryneg}\right)}{4 i \jacvaroneim \isotropicvec_{\isometrypos}^2}\bigg)
		\\
		&\quad\times \theta_{\sigma - nc\isotropicvecprime + \sublattice, \jacindexvec}\left(\jacvarone, \jacvartwo, nd \mu, -nc \mu, \subspaceisometry, \calP_{\subspaceisometry, \subpolposdeg}\right) \frake_{\sigma}
		\\
		&= \frac{1}{\sqrt{2 \isotropicvec_{\isometrypos}^2}} \sum_{\substack{\sigma \in D_\lattice \\ \left(\sigma, \isotropicvec\right) \equiv 0 \pmod{N_\isotropicvec}}} \theta_{\sigma + \sublattice, \jacindexvec}\left(\jacvarone, \jacvartwo, \subspaceisometry, \calP_{\subspaceisometry, 0}\right) \frake_{\sigma} \\
		&\quad+ \frac{1}{\sqrt{2 \isotropicvec_{\isometrypos}^2}} \sum_{\mat = \left(\begin{smallmatrix} a & b \\ c & d\end{smallmatrix}\right) \in \Gamma_\infty \bs \SL_2\left(\IZ\right)} \sum_{n = 1}^\infty \sum_{\subpolposdeg} \frac{n^{\subpolposdeg}\left(c \overline{\jacvarone} + d\right)^{\subpolposdeg}}{\left(-2 i \jacvaroneim\right)^{\subpolposdeg}}
		\\
		&\quad\times
		e\bigg(-\frac{n^2 + 2 n \overline{\mat \jacvartwo} \left(\jacindexvec, \isotropicvec_{\isometrypos}\right) + 2n \mat \jacvartwo \left(\jacindexvec, \isotropicvec_{\isometryneg}\right)}{4 i \Im\left(\mat \jacvarone\right) \isotropicvec_{\isometrypos}^2}\bigg)
		\\
		&\quad\times \sum_{\substack{\sigma \in D_\lattice \\ c \equiv \left(\sigma, \isotropicvec\right) \pmod{N_\isotropicvec}}} \theta_{\sigma - nc\isotropicvecprime + \sublattice, \jacindexvec}\left(\jacvarone, \jacvartwo, n d \mu, -n c \mu, \subspaceisometry, \calP_{\subspaceisometry, \subpolposdeg}\right) \frake_{\sigma}\left(- nd (\sigma, \isotropicvecprime) + n^2 cd q(\isotropicvecprime)\right)
		\\
		&= \frac{1}{\sqrt{2 \isotropicvec_{\isometrypos}^2}} \sum_{\substack{\sigma \in D_\lattice \\ \left(\sigma, \isotropicvec\right) \equiv 0 \pmod{N_\isotropicvec}}} \theta_{\sigma + \sublattice, \jacindexvec}\left(\jacvarone, \jacvartwo, \subspaceisometry, \calP_{\subspaceisometry, 0}\right) \frake_{\sigma} \\
		&\quad+ \frac{1}{\sqrt{2 \isotropicvec_{\isometrypos}^2}} \sum_{\mat = \left(\begin{smallmatrix} a & b \\ c & d\end{smallmatrix}\right) \in \Gamma_\infty \bs \SL_2\left(\IZ\right)} \sum_{n = 1}^\infty \sum_{\subpolposdeg} \frac{n^{\subpolposdeg}\left(c \jacvarone + d\right)^{-\subpolposdeg}}{\left(-2 i \Im\left(\mat \jacvarone\right)\right)^{\subpolposdeg}}
		e\bigg(-\frac{n^2 + 4 i n  \Im\left(\mat \jacvartwo\right) \left(\jacindexvec, \isotropicvec_{\isometrypos}\right)}{4 i \Im\left(\mat \jacvarone\right) \isotropicvec_{\isometrypos}^2}\bigg)
		\\
		&\quad\times \sum_{\substack{\sigma \in D_\lattice \\ c \equiv \left(\sigma, \isotropicvec\right) \pmod{N_\isotropicvec}}} \theta_{\sigma - nc\isotropicvecprime + \sublattice, \jacindexvec}\left(\jacvarone, \jacvartwo, n d \mu, -n c \mu, \subspaceisometry, \calP_{\subspaceisometry, \subpolposdeg}\right) \frake_{\sigma}\left(- nd (\sigma, \isotropicvecprime) + n^2 cd q(\isotropicvecprime)\right).
	\end{align*}
	We apply the change of variables $\sigma \mapsto \sigma + nc\isotropicvecprime$ and obtain that
	\begin{align*}
		&\frac{1}{\sqrt{2 \isotropicvec_{\isometrypos}^2}} \sum_{\substack{\sigma \in D_\lattice \\ \left(\sigma, \isotropicvec\right) \equiv 0 \pmod{N_\isotropicvec}}} \theta_{\sigma + \sublattice, \jacindexvec}\left(\jacvarone, \jacvartwo, \subspaceisometry, \calP_{\subspaceisometry, 0}\right) \frake_{\sigma}
		\\
		&\quad+ \frac{1}{\sqrt{2 \isotropicvec_{\isometrypos}^2}} \sum_{\mat = \left(\begin{smallmatrix} a & b \\ c & d\end{smallmatrix}\right) \in \Gamma_\infty \bs \SL_2\left(\IZ\right)}
		\sum_{n = 1}^\infty \sum_{\subpolposdeg}
		\frac{n^{\subpolposdeg}\left(c \jacvarone + d\right)^{-\subpolposdeg}}{\left(-2 i \Im\left(\mat \jacvarone\right)\right)^{\subpolposdeg}}
		e\bigg(-\frac{n^2 + 4 i n  \Im\left(\mat \jacvartwo\right) \left(\jacindexvec, \isotropicvec_{\isometrypos}\right)}{4 i \Im\left(\mat \jacvarone\right) \isotropicvec_{\isometrypos}^2}\bigg)
		\\
		&\quad\times \sum_{\substack{\sigma \in D_\lattice \\ \left(\sigma, \isotropicvec\right) \equiv 0 \pmod{N_\isotropicvec}}} \theta_{\sigma + \sublattice, \jacindexvec}\left(\jacvarone, \jacvartwo, n d \mu, -n c \mu, \subspaceisometry, \calP_{\subspaceisometry, \subpolposdeg}\right) \\
		&\quad\times \frake_{\sigma + nc\isotropicvecprime}\left(- nd \left(\sigma, \isotropicvecprime\right) - n^2 cd q\left(\isotropicvecprime\right)\right).
	\end{align*}
	The elements $\sigma \in D_\lattice$ with $\left(\sigma, \isotropicvec\right) \equiv 0 \pmod{N_\isotropicvec}$ are represented by $\sigma + \frac{m \isotropicvec}{N_\isotropicvec}$ with $\sigma \in D_\sublattice$ and $m \in \IZ / N_\isotropicvec \IZ$.
	We may then rewrite~$\Theta_{\sublattice, \jacindexvec}$ as
	\begin{align*}
		&\frac{1}{\sqrt{2 \isotropicvec_{\isometrypos}^2}} \sum_{\sigma \in D_\sublattice} \theta_{\sigma + \sublattice, \jacindexvec}\left(\jacvarone, \jacvartwo, \subspaceisometry, \calP_{\subspaceisometry, 0}\right) \frake_{\sigma} \sum_{m \in \IZ / N_\isotropicvec \IZ} \frake_{\frac{m \isotropicvec}{N_\isotropicvec}}
		\\
		&\quad+
		\frac{1}{\sqrt{2 \isotropicvec_{\isometrypos}^2}} \sum_{\mat = \left(\begin{smallmatrix} a & b \\ c & d\end{smallmatrix}\right) \in \Gamma_\infty \bs \SL_2\left(\IZ\right)} \sum_{n = 1}^\infty \sum_{\subpolposdeg} \frac{n^{\subpolposdeg}\left(c \jacvarone + d\right)^{-\subpolposdeg}}{\left(-2 i \Im\left(\mat \jacvarone\right)\right)^{\subpolposdeg}}
		e\bigg(-\frac{n^2 + 4 i n  \Im\left(\mat \jacvartwo\right) \left(\jacindexvec, \isotropicvec_{\isometrypos}\right)}{4 i \Im\left(\mat \jacvarone\right) \isotropicvec_{\isometrypos}^2}\bigg)
		\\
		&\quad\times
		\sum_{\sigma \in D_\sublattice} \theta_{\sigma + \sublattice, \jacindexvec}\left(\jacvarone, \jacvartwo, n d \mu, -n c \mu, \subspaceisometry, \calP_{\subspaceisometry, \subpolposdeg}\right) \frake_{\sigma}
		\\
		&\quad\times
		\sum_{m \in \IZ / N_\isotropicvec \IZ} \frake_{\frac{m \isotropicvec}{N_\isotropicvec} + nc\isotropicvecprime}\left(- \frac{mnd}{N_\isotropicvec} - n^2 cd q\left(\isotropicvecprime\right)\right)
		\\
		&= \frac{1}{\sqrt{2 \isotropicvec_{\isometrypos}^2}} \Theta_{\sublattice, \jacindexvec}\left(\jacvarone, \jacvartwo, \subspaceisometry, \calP_{\subspaceisometry, 0}\right) \sum_{m \in \IZ / N_\isotropicvec \IZ} \frake_{\frac{m \isotropicvec}{N_\isotropicvec}} \\
		&\quad+
		\frac{1}{\sqrt{2 \isotropicvec_{\isometrypos}^2}} \sum_{\mat = \left(\begin{smallmatrix} a & b \\ c & d\end{smallmatrix}\right) \in \Gamma_\infty \bs \SL_2\left(\IZ\right)} \sum_{n = 1}^\infty \sum_{\subpolposdeg} \frac{n^{\subpolposdeg}\left(c \jacvarone + d\right)^{-\subpolposdeg}}{\left(-2 i \Im\left(\mat \jacvarone\right)\right)^{\subpolposdeg}} e\bigg(-\frac{n^2 + 4 i n  \Im\left(\mat \jacvartwo\right) \left(\jacindexvec, \isotropicvec_{\isometrypos}\right)}{4 i \Im\left(\mat \jacvarone\right) \isotropicvec_{\isometrypos}^2}\bigg)
		\\
		&\quad\times
		\Theta_{\sublattice, \jacindexvec}\left(\jacvarone, \jacvartwo, n d \mu, -n c \mu, \subspaceisometry, \calP_{\subspaceisometry, \subpolposdeg}\right) \sum_{m \in \IZ / N_\isotropicvec \IZ} \frake_{\frac{m \isotropicvec}{N_\isotropicvec} + nc\isotropicvecprime}\Big(- \frac{mnd}{N_\isotropicvec} - n^2 cd q(\isotropicvecprime)\Big).
	\end{align*}
	The transformation formula of $\Theta_{\sublattice, \jacindexvec}$ provided by Example~\ref{exmpl:JacobiThetaPoly} implies that the formula above equals
	\begin{align*}
		& \frac{1}{\sqrt{2 \isotropicvec_{\isometrypos}^2}} \Theta_{\sublattice, \jacindexvec}\left(\jacvarone, \jacvartwo, \subspaceisometry, \calP_{\subspaceisometry, 0}\right) \sum_{m \in \IZ / N_\isotropicvec \IZ} \frake_{\frac{m \isotropicvec}{N_\isotropicvec}}
		\\
		&\quad+
		\frac{1}{\sqrt{2 \isotropicvec_{\isometrypos}^2}} \sum_{\mat = \left(\begin{smallmatrix} a & b \\ c & d\end{smallmatrix}\right) \in \Gamma_\infty \bs \SL_2\left(\IZ\right)} \sum_{n = 1}^\infty \sum_{\subpolposdeg} e\left(-\frac{q\left(\jacindexvec\right) c \jacvartwo^2}{c \jacvarone + d}\right)
		\\
		&\quad\times
		\frac{n^{\subpolposdeg}\left(c \jacvarone + d\right)^{\frac{b^-}{2} - \frac{b^+}{2} - \degjac}}{\left(-2 i \Im\left(\mat \jacvarone\right)\right)^{\subpolposdeg}} e\bigg(-\frac{n^2 + 4 i n  \Im\left(\mat \jacvartwo\right) \left(\jacindexvec, \isotropicvec_{\isometrypos}\right)}{4 i \Im\left(\mat \jacvarone\right) \isotropicvec_{\isometrypos}^2}\bigg) \\
		&\quad\times \left(\weilrep_{\sublattice, \jacindexvec}^{-1}\left(\mat\right) \Theta_{\sublattice, \jacindexvec}\left(\mat\jacvarone, \mat\jacvartwo, n\mu, 0, \subspaceisometry, \calP_{\subspaceisometry, \subpolposdeg}\right)\right) \sum_{m \in \IZ / N_\isotropicvec \IZ} \frake_{\frac{m \isotropicvec}{N_\isotropicvec} + nc\isotropicvecprime}\left(- \frac{mnd}{N_\isotropicvec} - n^2 cd q\left(\isotropicvecprime\right)\right)
		\\
		&=
		\frac{1}{\sqrt{2 \isotropicvec_{\isometrypos}^2}} \Theta_{\sublattice, \jacindexvec}\left(\jacvarone, \jacvartwo, \subspaceisometry, \calP_{\subspaceisometry, 0}\right) \sum_{m \in \IZ / N_\isotropicvec \IZ} \frake_{\frac{m \isotropicvec}{N_\isotropicvec}}
		\\
		&\quad+
		\frac{1}{\sqrt{2 \isotropicvec_{\isometrypos}^2}}
		\sum_{\mat = \left(\begin{smallmatrix} a & b \\ c & d\end{smallmatrix}\right) \in \Gamma_\infty \bs \SL_2\left(\IZ\right)} \sum_{n = 1}^\infty \sum_{\subpolposdeg} e\left(-\frac{q\left(\jacindexvec\right) c \jacvartwo^2}{c \jacvarone + d}\right)
		\\
		&\quad\times
		\frac{n^{\subpolposdeg}\left(c \jacvarone + d\right)^{\frac{b^-}{2} - \frac{b^+}{2} - \degjac}}{\left(-2 i \Im\left(\mat \jacvarone\right)\right)^{\subpolposdeg}}
		e\bigg(-\frac{n^2 + 4 i n  \Im\left(\mat \jacvartwo\right) \left(\jacindexvec, \isotropicvec_{\isometrypos}\right)}{4 i \Im\left(\mat \jacvarone\right) \isotropicvec_{\isometrypos}^2}\bigg) \\
		&\quad\times
		\weilrep_{\lattice, \jacindexvec}^{-1}
		\left(\tilde{\mat}\right)
		\bigg(\Theta_{\sublattice, \jacindexvec}\left(\mat\jacvarone, \mat\jacvartwo, n\mu, 0, \subspaceisometry, \calP_{\subspaceisometry, \subpolposdeg}\right) \sum_{m \in \IZ / N_\isotropicvec \IZ} \frake_{\frac{m \isotropicvec}{N_\isotropicvec}}\Big(- \frac{mn}{N_\isotropicvec}\Big) \bigg),
	\end{align*}
	where in the last equality we used Lemma~\ref{lem:kiefer}.
\end{proof}

\subsection{Evaluating Petersson inner products}\label{sec:JacobiThetaInnerProducts}

Let~$\calP : \IR^{b^+, 0} \to \IC$ be a homogeneous polynomial of degree~$\degjac$ and let $\jacindexvec \in L'$.
Let $\phi \colon \IH \times \IC \to \IC[D_\lattice]$ be a Jacobi form of weight~${k = \degjac + (b^+ - b^-)/2}$ and index~$\jacindexvec \in \lattice'$.
\textcolor{\mycolor}{Let~$\langle\cdot{,}\cdot\rangle$ be the standard Hermitian product of~$\CC[D_L]$.}
We define the~\emph{Jacobi Petersson inner product} of~$\phi$ and~$\Theta_{\lattice, \jacindexvec}$ as
\bas\langle \phi, &\, \Theta_{\lattice, \jacindexvec}\left(\cdot, \cdot, g, \pol\right) \rangle_{\Pet}
\\
&= \int_{\jacobi \bs \IH \times \IC} \big\langle
\phi(\jacvarone, \jacvartwo), \Theta_{\lattice, \jacindexvec}\left(\jacvarone, \jacvartwo, \isometry, \calP\right) \big\rangle \jacvaroneim^k \exp\big(-4 \pi q\left(\jacindexvec\right) \jacvartwoim^2 / \jacvaroneim\big) \frac{d \jacvartwo\, d \jacvarone}{\jacvaroneim^3}.
\eas

In this section we deduce from the rewriting of $\Theta_{\lattice, \jacindexvec}$ as a Poincar\'e series given by Theorem~\ref{thm;rewritingthetajacaspoincare} an unfolding of the Jacobi Petersson inner products defined above.
To do so, we need to recall that a fundamental domain for the action of $\heisenberg\left(\IZ\right) \rtimes \Gamma_\infty$ on~$\IH \times \IC$ is given by
$$\{ \left(\jacvarone, \jacvartwo\right) \in \IH \times \IC \mid -1/2 \leq \Re\left(\jacvarone\right) \leq 1/2, \jacvartwo \in T\left(\jacvarone\right) \},$$
where we denote by~$T\left(\jacvarone\right)$ the torus
$$T\left(\jacvarone\right) = \{\jacvartwo \in \IC \mid 0 \leq \jacvartwoim \leq \jacvaroneim,\, \jacvartwoim \jacvaronere / \jacvaroneim \leq \jacvartwore \leq 1 + \jacvartwoim \jacvaronere / \jacvaroneim \}.$$

Let $\isotropicvec \in \lattice$ be primitive isotropic of level $N_\isotropicvec$, and let~$\isotropicvecprime \in \lattice'$ be such that~$\left(\isotropicvec, \isotropicvecprime\right) = 1$.
\textcolor{\mycolor}{The notation~$\sublattice$, $\subspaceisometry$, $\calP_{\subspaceisometry, \subpolposdeg}$ and $\mu$ is as in Section~\ref{sec:ReductionToSmallerLattices}.}

\begin{lemma}\label{lem:JacobiThetaUnfolding}
	Let $\jacindexvec \in \isotropicvec^\perp$ and consider the \textcolor{\mycolor}{linear} map
	$$\sublattice' \to (\jacindexvec^\perp \cap \sublattice)', \quad \lambda \mapsto \lambda_{\jacindexvec} = \lambda - \frac{\left(\lambda, \jacindexvec\right)}{\jacindexvec^2} \jacindexvec_{\sublattice}.$$
	Then
	\begin{align*}
		&\big\langle \phi, \Theta_{\lattice, \jacindexvec}\left(\cdot, \cdot, g, \pol\right) \big\rangle_{\Pet}
		\\
		&=
		\frac{1}{\sqrt{2 \isotropicvec_{\isometrypos}^2}} \Big\langle \phi, \Theta_{\sublattice, \jacindexvec}\left(\cdot, \cdot, g_K, \pol_{g_K, 0}\right)
		\sum_{m \in \IZ / N_\isotropicvec \IZ} \frake_{\frac{m \isotropicvec}{N_\isotropicvec}} \Big\rangle_{\Pet}
		+
		\frac{\sqrt{2}}{\lvert \isotropicvec_{\isometrypos} \rvert} \sum_{n = 1}^\infty \sum_{\subpolposdeg} \Big(\frac{n}{2i}\Big)^{\subpolposdeg}
		\\
		&\quad\times
		\int_0^\infty
		\jacvaroneim^{b^- + k - \subpolposdeg - 5/2}
		\exp\bigg(-\frac{\pi n^2}{2 \jacvaroneim \isotropicvec_{\isometrypos}^2}\bigg)
		\sum_{\tilde{\lambda} \in (\jacindexvec^\perp \cap \sublattice)'} 
		\sum_{\substack{\lambda \in \sublattice' / \jacindexvec_{\sublattice}
		\\ \lambda_{\jacindexvec} = \tilde{\lambda}}} \int_{-\infty}^{\infty}
		e\bigg(\frac{n  \jacvartwoim \left(\jacindexvec, \isotropicvec_{\isometrypos}\right)}{\isotropicvec_{\isometrypos}^2}\bigg)
		\\
		&\quad\times \exp\left(-\Delta / 8 \pi \jacvaroneim\right)\left(\overline{\calP}_{\subspaceisometry, \subpolposdeg}\right)\left(\subspaceisometry\left(\lambda_{\jacindexvec} + \jacvartwoim \jacindexvec\right)\right)
		\cdot
		\exp\Big(-4 \pi q\left(\left(\lambda_{\jacindexvec} + \jacvartwoim \jacindexvec\right)_{\subspaceisometrypos}\right) \jacvaroneim\Big) d \jacvartwoim\, d \jacvaroneim \\
		&\quad\times \sum_{m \in \IZ / N_\isotropicvec \IZ} e\left(\frac{mn}{N_\isotropicvec} + \left(n \lambda_{\jacindexvec}, \mu\right) + \frac{\left(n\lambda, \jacindexvec\right)\left(\jacindexvec, \isotropicvecprime\right)}{\jacindexvec^2}\right)
		\cdot
		a\Big(\lambda + \frac{m \isotropicvec}{N_{\isotropicvec}}, q\left(\lambda\right), \left(\lambda, \jacindexvec\right)\Big),
	\end{align*}
	where $a\left(\cdot, \cdot, \cdot\right)$ are the Fourier coefficients of $\phi$.
\end{lemma}

\begin{proof}
	We write $\Theta_{\lattice, \jacindexvec}$ as
	\begin{align*}
		&\frac{1}{\sqrt{2 \isotropicvec_{\isometrypos}^2}} \Theta_{\sublattice, \jacindexvec}\left(\jacvarone, \jacvartwo, \subspaceisometry, \calP_{\subspaceisometry, 0}\right) \sum_{m \in \IZ / N_\isotropicvec \IZ} \frake_{\frac{m \isotropicvec}{N_\isotropicvec}} \\
		&\quad+ \frac{1}{\sqrt{2 \isotropicvec_{\isometrypos}^2}} \sum_{\mat = \left(\begin{smallmatrix} a & b \\ c & d\end{smallmatrix}\right) \in \Gamma_\infty \bs \SL_2\left(\IZ\right)} \sum_{n = 1}^\infty \sum_{\subpolposdeg} e\left(-\frac{q\left(\jacindexvec\right) c \jacvartwo^2}{c \jacvarone + d}\right)
		\\
		&\quad\times \frac{n^{\subpolposdeg}
		\left(c \jacvarone + d\right)^{(b^--b^+)/2 - \degjac}}{\left(-2 i \Im\left(\mat \jacvarone\right)\right)^{\subpolposdeg}}
		e\bigg(-\frac{n^2 + 4 i n  \Im\left(\mat \jacvartwo\right) \left(\jacindexvec, \isotropicvec_{\isometrypos}\right)}{4 i \Im\left(\mat \jacvarone\right) \isotropicvec_{\isometrypos}^2}\bigg)
		\\
		&\quad\times \weilrep_{\lattice, \jacindexvec}^{-1}\left(\tilde{\mat}\right) 
		\bigg(\Theta_{\sublattice, \jacindexvec}\left(\mat\jacvarone, \mat\jacvartwo, n\mu, 0, \subspaceisometry, \calP_{\subspaceisometry, \subpolposdeg}\right) \sum_{m \in \IZ / N_\isotropicvec \IZ} \frake_{\frac{m \isotropicvec}{N_\isotropicvec}}\left(- \frac{mn}{N_\isotropicvec}\right) \bigg)
	\end{align*}
	and unfold the integral against this Poincar\'e series to obtain
	\ba\label{eq:jacobithetalift1}
		& \frac{1}{\sqrt{2 \isotropicvec_{\isometrypos}^2}}
		\Big\langle \phi, \Theta_{\sublattice, \jacindexvec}(\cdot, \cdot, g_K, \pol_{g_K, 0}) \sum_{m \in \IZ / N_\isotropicvec \IZ} \frake_{\frac{m \isotropicvec}{N_\isotropicvec}} \Big\rangle_{\Pet}
		+
		\frac{\sqrt{2}}{\lvert \isotropicvec_{\isometrypos} \rvert} \sum_{n = 1}^\infty \sum_{\subpolposdeg} \Big(\frac{n}{2i} \Big)^{\subpolposdeg}
		\\
		&\quad\times
		\int_0^\infty \exp\bigg(-\frac{\pi n^2}{2\jacvaroneim \isotropicvec_{\isometrypos}^2}\bigg) \jacvaroneim^{k - \subpolposdeg} \int_{-1/2}^{1/2} \int_{\jacvartwo \in T\left(\jacvarone\right)}
		e\bigg(\frac{n  \jacvartwoim \left(\jacindexvec, \isotropicvec_{\isometrypos}\right)}{\jacvaroneim \isotropicvec_{\isometrypos}^2}\bigg) 
		\\
		&\quad\times
		\Big\langle \phi\left(\jacvarone, \jacvartwo\right), \Theta_{\sublattice, \jacindexvec}(\jacvarone, \jacvartwo, n\mu, 0, \subspaceisometry, \calP_{\subspaceisometry, \subpolposdeg}) \sum_{m \in \IZ / N_\isotropicvec \IZ} \frake_{\frac{m \isotropicvec}{N_\isotropicvec}}\bigg(- \frac{mn}{N_\isotropicvec}\bigg) \Big\rangle
		\\
		&\quad\times
		\exp\big(-4 \pi q(\jacindexvec) \jacvartwoim^2 / \jacvaroneim\big) d \jacvartwo \,\frac{d \jacvarone}{\jacvaroneim^3}.
	\ea
	The factor of $2$ in the second summand of the formula above comes from the trivial action of $\big(\begin{smallmatrix}-1 & 0 \\ 0 & -1\end{smallmatrix}\big) \in \Gamma_\infty \bs \SL_2\left(\IZ\right)$. The first summand has the correct form so that we will only consider the second summand of \eqref{eq:jacobithetalift1}. We apply the change of variables
	$$\left(\jacvartwore, \jacvartwoim\right) \mapsto \left(\jacvartwore + \jacvartwoim \jacvaronere, \jacvaroneim \jacvartwoim\right)$$
	which transforms the rectangle $[0, 1]^2$ to the torus $T\left(\jacvarone\right)$. Hence, the second summand of~\eqref{eq:jacobithetalift1} is given by
	\begin{align}\label{eq:jacobithetalift2}
		& \frac{\sqrt{2}}{\lvert \isotropicvec_{\isometrypos} \rvert}
		\sum_{n = 1}^\infty
		\sum_{\subpolposdeg}
		\Big(\frac{n}{2i}\Big)^{\subpolposdeg} \int_0^\infty
		\exp\bigg(-\frac{\pi n^2}{2 \jacvaroneim \isotropicvec_{\isometrypos}^2}\bigg)
		\jacvaroneim^{k - \subpolposdeg}
		\int_{-1/2}^{1/2} \int_{0}^1 \int_0^1
		e\bigg(\frac{n  \jacvartwoim \left(\jacindexvec, \isotropicvec_{\isometrypos}\right)}{\isotropicvec_{\isometrypos}^2}\bigg)
		\\
		&\quad\times
		\Big\langle
		\phi\left(\jacvarone, \jacvartwore + \jacvartwoim \jacvarone\right), \Theta_{\sublattice, \jacindexvec}\left(\jacvarone, \jacvartwore + \jacvartwoim \jacvarone, n\mu, 0, \subspaceisometry, \calP_{\subspaceisometry, \subpolposdeg}\right) \sum_{m \in \IZ / N_\isotropicvec \IZ} \frake_{\frac{m \isotropicvec}{N_\isotropicvec}}\Big(- \frac{mn}{N_\isotropicvec}\Big) \Big\rangle
		\nonumber
		\\
		&\quad\times \exp\left(-4 \pi q\left(\jacindexvec\right) \jacvartwoim^2 \jacvaroneim\right) d \jacvartwore \,d \jacvartwoim \,\frac{d \jacvaronere \,d \jacvaroneim}{\jacvaroneim^2}
		\nonumber
		\\
		&=
		\frac{\sqrt{2}}{\lvert \isotropicvec_{\isometrypos} \rvert}
		\sum_{n = 1}^\infty
		\sum_{\subpolposdeg}
		\frac{n^{\subpolposdeg}}{\left(2i\right)^{\subpolposdeg}}
		\int_0^\infty \jacvaroneim^{k - \subpolposdeg - 2}
		\int_{0}^1
		\exp\bigg(-\frac{\pi n^2}{2 \jacvaroneim \isotropicvec_{\isometrypos}^2}\bigg)
		\nonumber
		\\
		&\quad\times
		\exp\left(-4 \pi q\left(\jacindexvec\right) \jacvartwoim^2 \jacvaroneim\right) 
		\cdot
		e\bigg(\frac{n  \jacvartwoim \left(\jacindexvec, \isotropicvec_{\isometrypos}\right)}{\isotropicvec_{\isometrypos}^2}\bigg)
		\int_{-1/2}^{1/2} \int_0^1 \bigg\langle \phi\left(\jacvarone, \jacvartwore + \jacvartwoim \jacvarone\right),
		\nonumber
		\\
		&\qquad \Theta_{\sublattice, \jacindexvec}\left(\jacvarone, \jacvartwore + \jacvartwoim \jacvarone, n\mu, 0, \subspaceisometry, \calP_{\subspaceisometry, \subpolposdeg}\right) \sum_{m \in \IZ / N_\isotropicvec \IZ} \frake_{\frac{m \isotropicvec}{N_\isotropicvec}}\left(- \frac{mn}{N_\isotropicvec}\right) \bigg\rangle d \jacvartwore \, d \jacvaronere \, d \jacvartwoim \, d \jacvaroneim.
		\nonumber
	\end{align}
	We insert the Fourier expansion of the Jacobi theta function
	\begin{align*}
		&\Theta_{\sublattice, \jacindexvec}\left(\jacvarone, \jacvaronere + \jacvartwoim \jacvarone, n \mu, 0, \subspaceisometry, \calP_{\subspaceisometry, \subpolposdeg}\right) \\
		&= \jacvaroneim^{\frac{b^- - 1}{2}} \exp\left(4 \pi q\left(\jacindexvec_{\subspaceisometryneg}\right) \jacvartwoim^2 \jacvaroneim\right) \sum_{\lambda \in \sublattice'} \exp\left(-\Delta / 8 \pi \jacvaroneim\right)\left(\pol_{\subspaceisometry, \subpolposdeg}\right)\left(\subspaceisometry\left(\lambda + \jacvartwoim \jacindexvec\right)\right) \\
		&\quad\times \frake_{\lambda}\Big(q\left(\lambda\right) \jacvaronere + q_{\subspaceisometrypos}\left(\lambda\right) i \jacvaroneim + \left(\jacvartwore + \jacvartwoim \jacvaronere\right) \left(\lambda, \jacindexvec\right) + i \jacvartwoim \jacvaroneim \left(\lambda, \jacindexvec\right)_{\subspaceisometrypos} - \left(\lambda, n \mu\right)\Big)
	\end{align*}
	to obtain for the integrals over $\jacvaronere, \jacvartwore$ appearing on the right-hand side of~\eqref{eq:jacobithetalift2}
	\begin{align*}
		&\int_{-1/2}^{1/2} \int_0^1
		\bigg\langle \phi\left(\jacvarone, \jacvartwore + \jacvartwoim \jacvarone\right), \Theta_{\sublattice, \jacindexvec}\left(\jacvarone, \jacvartwore + \jacvartwoim \jacvarone, n\mu, 0, \subspaceisometry, \calP_{\subspaceisometry, \subpolposdeg}\right)
		\\
		&\quad\times\sum_{m \in \IZ / N_\isotropicvec \IZ} \frake_{\frac{m \isotropicvec}{N_\isotropicvec}}\left(- \frac{mn}{N_\isotropicvec}\right) \bigg\rangle d \jacvartwore \, d \jacvaronere \\
		&= \jacvaroneim^{\frac{b^- - 1}{2}} \exp\left(4 \pi q\left(\jacindexvec_{\subspaceisometryneg}\right) \jacvartwoim^2 \jacvaroneim\right) \sum_{\lambda \in \sublattice'} \exp\left(-\Delta / 8 \pi \jacvaroneim\right)\left(\overline{\calP}_{\subspaceisometry, \subpolposdeg}\right)\left(\subspaceisometry\left(\lambda + \jacvartwoim \jacindexvec\right)\right) \\
		&\quad\times \exp\Big(-2\pi\big(q_{\subspaceisometrypos}\left(\lambda\right) \jacvaroneim + \jacvartwoim \jacvaroneim \left(\lambda, \jacindexvec\right)_{\subspaceisometrypos}\big)\Big) \sum_{m \in \IZ / N_\isotropicvec \IZ} e\left(\frac{mn}{N_\isotropicvec} + \left(\lambda, n \mu\right)\right) \\
		&\quad\times \int_{-1/2}^{1/2} \int_0^1 \phi_{\lambda + \frac{m \isotropicvec}{N_{\isotropicvec}}}\left(\jacvarone, \jacvartwore + \jacvartwoim \jacvarone\right) e\Big(-q\left(\lambda\right) \jacvaronere - \left(\jacvartwore + \jacvartwoim \jacvaronere\right) \left(\lambda, \jacindexvec\right)\Big) d\jacvartwore\, d \jacvaronere \\
		&= \jacvaroneim^{\frac{b^- - 1}{2}} \exp\left(2 \pi q\left(\jacindexvec\right) \jacvartwoim^2 \jacvaroneim\right) \sum_{\lambda \in \sublattice'} \exp\left(-\Delta / 8 \pi \jacvaroneim\right)\left(\overline{\calP}_{\subspaceisometry, \subpolposdeg}\right)\left(\subspaceisometry\left(\lambda + \jacvartwoim \jacindexvec\right)\right) \\
		&\quad\times \exp\Big(-2 \pi q_{\subspaceisometrypos}\left(\lambda + \jacvartwoim \jacindexvec\right) \jacvaroneim\Big) \sum_{m \in \IZ / N_\isotropicvec \IZ} e\left(\frac{mn}{N_\isotropicvec} + \left(\lambda, n \mu\right)\right) \\
		&\quad\times \int_{-1/2}^{1/2} \int_0^1 \phi_{\lambda + \frac{m \isotropicvec}{N_{\isotropicvec}}}\left(\jacvarone, \jacvartwore + \jacvartwoim \jacvarone\right) e\Big(-q\left(\lambda\right) \jacvaronere - \left(\jacvartwore + \jacvartwoim \jacvaronere\right) \left(\lambda, \jacindexvec\right)\Big) d \jacvartwore\, d \jacvaronere
	\end{align*}
	Next, we insert the Fourier expansion
	$$\phi_{\lambda + \frac{m \isotropicvec}{N_{\isotropicvec}}}\left(\jacvarone, \jacvartwore + \jacvartwoim \jacvarone\right) = \sum_{\substack{r \in \IZ + q\left(\lambda\right) \\ s \in \IZ + \left(\lambda + \frac{m \isotropicvec}{N_{\isotropicvec}}, \jacindexvec\right)}} a\left(\lambda + \frac{m \isotropicvec}{N_{\isotropicvec}}, r, s\right)
	\cdot
	e\Big(r \jacvarone + s \left(\jacvartwore + \jacvartwoim \jacvarone\right)\Big),$$
	calculate the integrals and obtain
	\begin{align*}
		&\jacvaroneim^{(b^- - 1)/2} \exp\left(2 \pi q\left(\jacindexvec\right) \jacvartwoim^2 \jacvaroneim\right) \sum_{\lambda \in \sublattice'} \exp\left(-\Delta / 8 \pi \jacvaroneim\right)\left(\overline{\calP}_{\subspaceisometry, \subpolposdeg}\right)\left(\subspaceisometry\left(\lambda + \jacvartwoim \jacindexvec\right)\right) \\
		&\quad\times \exp\Big(-2 \pi q_{\subspaceisometrypos}\left(\lambda + \jacvartwoim \jacindexvec\right) \jacvaroneim\Big) \exp\Big(-2 \pi \left(\left(\lambda, \jacindexvec\right)\jacvartwoim \jacvaroneim + q\left(\lambda\right) \jacvaroneim\right)\Big)
		\\
		&\quad\times
		\sum_{m \in \IZ / N_\isotropicvec \IZ} e\bigg(\frac{mn}{N_\isotropicvec} + \left(\lambda, n \mu\right)\bigg) 
		\cdot
		a\left(\lambda + \frac{m \isotropicvec}{N_{\isotropicvec}}, q\left(\lambda\right), \left(\lambda, \jacindexvec\right)\right)
		\\
		&= \jacvaroneim^{(b^- - 1)/2} \exp\left(4 \pi q\left(\jacindexvec\right) \jacvartwoim^2 \jacvaroneim\right) \sum_{\lambda \in \sublattice'} \exp\left(-\Delta / 8 \pi \jacvaroneim\right)\left(\overline{\calP}_{\subspaceisometry, \subpolposdeg}\right)\left(\subspaceisometry\left(\lambda + \jacvartwoim \jacindexvec\right)\right) 
		\\
		&\quad\times
		\exp\Big(-2 \pi q_{\subspaceisometrypos}\left(\lambda + \jacvartwoim \jacindexvec\right) \jacvaroneim\Big) 
		\cdot
		\exp\Big(-2 \pi q\left(\lambda + \jacvartwoim \jacindexvec\right) \jacvaroneim\Big)
		\\
		&\quad\times
		\sum_{m \in \IZ / N_\isotropicvec \IZ} e\left(\frac{mn}{N_\isotropicvec} + \left(\lambda, n \mu\right)\right)
		\cdot
		a\left(\lambda + \frac{m \isotropicvec}{N_{\isotropicvec}}, q\left(\lambda\right), \left(\lambda, \jacindexvec\right)\right).
	\end{align*}
	Hence, the second summand of \eqref{eq:jacobithetalift1} is given by
	\begin{align*}
		& \frac{\sqrt{2}}{\lvert \isotropicvec_{\isometrypos} \rvert} \sum_{n = 1}^\infty 
		\sum_{\subpolposdeg} 
		\Big(\frac{n}{2i}\Big)^{\subpolposdeg}
		\int_0^\infty \jacvaroneim^{b^-/2 + k - \subpolposdeg - 5/2}
		\int_{0}^1 
		\exp\bigg(-\frac{\pi n^2}{2 \jacvaroneim \isotropicvec_{\isometrypos}^2}\bigg)
		\cdot
		e\bigg(\frac{n  \jacvartwoim \left(\jacindexvec, \isotropicvec_{\isometrypos}\right)}{\isotropicvec_{\isometrypos}^2}\bigg)
		\\
		&\quad\times
		\sum_{\lambda \in \sublattice'} \exp\left(-\Delta / 8 \pi \jacvaroneim\right)\left(\overline{\calP}_{\subspaceisometry, \subpolposdeg}\right)\left(\subspaceisometry\left(\lambda + \jacvartwoim \jacindexvec\right)\right)
		\cdot
		\exp\Big(-4 \pi q\left(\left(\lambda + \jacvartwoim \jacindexvec\right)_{\subspaceisometrypos}\right) \jacvaroneim\Big) d \jacvartwoim\, d \jacvaroneim
		\\
		&\quad\times
		\sum_{m \in \IZ / N_\isotropicvec \IZ}
		e\bigg(\frac{mn}{N_\isotropicvec} + \left(\lambda, n \mu\right)\bigg)
		\cdot
		a\left(\lambda + \frac{m \isotropicvec}{N_{\isotropicvec}}, q\left(\lambda\right), \left(\lambda, \jacindexvec\right)\right).
	\end{align*}
	By summing over $\lambda + l \jacindexvec_\sublattice$, where~$\lambda \in \sublattice' / \jacindexvec_\sublattice$ and $l \in \IZ$, and by using that
	$$a\big(\lambda + \jacindexvec, q\left(\lambda + \jacindexvec\right), \left(\lambda + \jacindexvec, \jacindexvec\right)\big) 
	=
	 a\big(\lambda, q\left(\lambda\right), \left(\lambda, \jacindexvec\right)\big),$$
	we obtain
	\begin{align*}
		& \frac{\sqrt{2}}{\lvert \isotropicvec_{\isometrypos} \rvert}
		\sum_{n = 1}^\infty
		\sum_{\subpolposdeg}
		\Big(\frac{n}{2i}\Big)^{\subpolposdeg}
		\int_0^\infty \jacvaroneim^{b^-/2 + k - \subpolposdeg - 5/2}
		\exp\bigg(-\frac{\pi n^2}{2 \jacvaroneim \isotropicvec_{\isometrypos}^2}\bigg)
		\\
		&\quad\times
		\sum_{\lambda \in \sublattice' / \jacindexvec_{\sublattice}}
		\sum_{l \in \IZ} \int_{0}^1
		e\bigg(\frac{n  \left(\jacvartwoim + l\right) \left(\jacindexvec, \isotropicvec_{\isometrypos}\right)}{\isotropicvec_{\isometrypos}^2}\bigg)
		\cdot
		\exp\Big(-4 \pi q\big(\left(\lambda + \left(\jacvartwoim + l\right) \jacindexvec\right)_{\subspaceisometrypos}\big) \jacvaroneim\Big) \\
		&\quad\times \exp\left(-\Delta / 8 \pi \jacvaroneim\right)\left(\overline{\calP}_{\subspaceisometry, \subpolposdeg}\right)\left(\subspaceisometry\left(\lambda + \left(\jacvartwoim + l\right) \jacindexvec\right)\right) d \jacvartwoim \,d \jacvaroneim \\
		&\quad\times \sum_{m \in \IZ / N_\isotropicvec \IZ} e\left(\frac{mn}{N_\isotropicvec} - n l \left(\jacindexvec, \isotropicvecprime\right) + \left(\lambda, n \mu\right)\right) 
		\cdot
		a\left(\lambda + \frac{m \isotropicvec}{N_{\isotropicvec}} - l \left(\jacindexvec, \isotropicvecprime\right) \isotropicvec, q\left(\lambda\right), \left(\lambda, \jacindexvec\right)\right) \\
		&=
		\frac{\sqrt{2}}{\lvert \isotropicvec_{\isometrypos} \rvert}
		\sum_{n = 1}^\infty 
		\sum_{\subpolposdeg} 
		\Big(\frac{n}{2i}\Big)^{\subpolposdeg}
		\int_0^\infty 
		\jacvaroneim^{b^-/2 + k - \subpolposdeg - 5/2} 
		\exp\bigg(-\frac{\pi n^2}{2 \jacvaroneim \isotropicvec_{\isometrypos}^2}\bigg) 
		\\
		&\quad\times \sum_{\lambda \in \sublattice' / \jacindexvec_{\sublattice}} \int_{-\infty}^{\infty}
		e\bigg(\frac{n  \jacvartwoim \left(\jacindexvec, \isotropicvec_{\isometrypos}\right)}{\isotropicvec_{\isometrypos}^2}\bigg)
		\cdot
		\exp\Big(-4 \pi q\left(\left(\lambda + \jacvartwoim \jacindexvec\right)_{\subspaceisometrypos}\right) \jacvaroneim\Big)
		\\
		&\quad\times
		\exp\left(-\Delta / 8 \pi \jacvaroneim\right)\left(\overline{\calP}_{\subspaceisometry, \subpolposdeg}\right)\left(\subspaceisometry\left(\lambda + \jacvartwoim \jacindexvec\right)\right) d \jacvartwoim \,d \jacvaroneim \\
		&\quad\times \sum_{m \in \IZ / N_\isotropicvec \IZ} e\left(\frac{mn}{N_\isotropicvec} + \left(\lambda, n \mu\right)\right) 
		\cdot
		a\left(\lambda + \frac{m \isotropicvec}{N_{\isotropicvec}}, q\left(\lambda\right), \left(\lambda, \jacindexvec\right)\right).
	\end{align*}
	We remark that~$\left(\jacindexvec, \isotropicvec_{\isometrypos}\right)/\isotropicvec_{\isometrypos}^2 = (\jacindexvec, \mu + \isotropicvecprime)$.
	We use~$\lambda_{\jacindexvec} = \lambda - \frac{\left(\lambda, \jacindexvec\right)}{\jacindexvec^2} \jacindexvec_\sublattice$ and make the substitution $\jacvartwoim \mapsto \jacvartwoim - \frac{\left(\lambda, \jacindexvec\right)}{\jacindexvec^2}$ to rewrite this as
	\begin{align*}
		&\frac{\sqrt{2}}{\lvert \isotropicvec_{\isometrypos} \rvert} \sum_{n = 1}^\infty 
		\sum_{\subpolposdeg} 
		\Big(\frac{n}{2i}\Big)^{\subpolposdeg} 
		\int_0^\infty 
		\jacvaroneim^{b^-/2 + k - \subpolposdeg - 5/2} 
		\exp\bigg(-\frac{\pi n^2}{2 \jacvaroneim \isotropicvec_{\isometrypos}^2}\bigg)
		\sum_{\lambda \in \sublattice' / \jacindexvec_{\sublattice}} \int_{-\infty}^{\infty}
		e\bigg(\frac{n  \jacvartwoim \left(\jacindexvec, \isotropicvec_{\isometrypos}\right)}{\isotropicvec_{\isometrypos}^2}\bigg)
		\\
		&\quad\times \exp\left(-\Delta / 8 \pi \jacvaroneim\right)\left(\overline{\calP}_{\subspaceisometry, \subpolposdeg}\right)\left(\subspaceisometry\left(\lambda_{\jacindexvec} + \jacvartwoim \jacindexvec\right)\right) 
		\cdot
		\exp\Big(-4 \pi q\left(\left(\lambda_{\jacindexvec} + \jacvartwoim \jacindexvec\right)_{\subspaceisometrypos}\right) \jacvaroneim\Big) d \jacvartwoim \,d \jacvaroneim \\
		&\quad\times
		\sum_{m \in \IZ / N_\isotropicvec \IZ}
		e\left(\frac{mn}{N_\isotropicvec} + \frac{\left(n\lambda, \jacindexvec\right)\left(\jacindexvec, \isotropicvecprime\right)}{\jacindexvec^2}\right)
		\cdot
		a\left(\lambda + \frac{m \isotropicvec}{N_{\isotropicvec}}, q\left(\lambda\right), \left(\lambda, \jacindexvec\right)\right) 
		\cdot
		e\left(n \lambda_{\jacindexvec}, \mu\right) \\
		&=
		\frac{\sqrt{2}}{\lvert \isotropicvec_{\isometrypos} \rvert}
		\sum_{n = 1}^\infty 
		\sum_{\subpolposdeg} \Big(\frac{n}{2i}\Big)^{\subpolposdeg} \int_0^\infty \jacvaroneim^{b^-/2 + k - \subpolposdeg - 5/2} 
		\exp\bigg(-\frac{\pi n^2}{2 \jacvaroneim \isotropicvec_{\isometrypos}^2}\bigg) 
		\\
		&\quad\times 
		\sum_{\tilde{\lambda} \in \left(\jacindexvec^\perp \cap \sublattice\right)'} 
		\sum_{\substack{\lambda \in \sublattice' / \jacindexvec_K \\ \lambda_{\jacindexvec} = \tilde{\lambda}}} 
		\int_{-\infty}^{\infty} 
		e\bigg(\frac{n  \jacvartwoim \left(\jacindexvec, \isotropicvec_{\isometrypos}\right)}{\isotropicvec_{\isometrypos}^2}\bigg) 
		\\
		&\quad\times 
		\exp\left(-\Delta / 8 \pi \jacvaroneim\right)
		\left(\overline{\calP}_{\subspaceisometry, \subpolposdeg}\right)
		\left(\subspaceisometry\big(\tilde{\lambda} + \jacvartwoim \jacindexvec\big)\right) 
		\cdot
		\exp\left(-4 \pi q\big((\tilde{\lambda} + \jacvartwoim \jacindexvec)_{\subspaceisometrypos}\big) \jacvaroneim\right) d \jacvartwoim \,d \jacvaroneim \\
		&\quad\times 
		\sum_{m \in \IZ / N_\isotropicvec \IZ} e\left(\frac{mn}{N_\isotropicvec} + \frac{\left(n\lambda, \jacindexvec\right)\left(\jacindexvec, \isotropicvecprime\right)}{\jacindexvec^2}\right) 
		\cdot
		a\left(\lambda + \frac{m \isotropicvec}{N_{\isotropicvec}}, q\left(\lambda\right), \left(\lambda, \jacindexvec\right)\right)
		\cdot 
		e(n \tilde{\lambda}, \mu).
	\end{align*}
This shows the assertion.
\end{proof}

\begin{thm}\label{thm:JacobiThetaUnfolding}
Let $\jacindexvec \in \isotropicvec^\perp$ and assume $N_{\isotropicvec} = 1$, so that $\lattice$ is the orthogonal sum of a hyperbolic plane and~$\sublattice$.
Moreover, assume that $\calP_{\subspaceisometry, \subpolposdeg}$ is harmonic and satisfies $\calP_{\subspaceisometry, \subpolposdeg}\left(\lambda + C \jacindexvec\right) = \calP_{\subspaceisometry, \subpolposdeg}\left(\lambda\right)$ for every~$C \in \IC$ and~$\lambda \in \sublattice$. Then~$\langle \phi, \Theta_{\lattice, \jacindexvec}(\cdot, \cdot, g, \pol) \rangle_{\Pet}$, as a function of~$G$, admits a Fourier expansion of the form
$$\big\langle \phi, \Theta_{\lattice, \jacindexvec}(\cdot, \cdot, g, \pol) \big\rangle_{\Pet} 
= 
\sum_{\tilde{\lambda} \in (\jacindexvec^\perp \cap K)'} c_{\tilde{\lambda}}(g) \cdot e(\tilde{\lambda}, \mu).$$
The constant term of the expansion is the coefficient
$$c_0(g) = \frac{1}{\sqrt{2} \lvert \isotropicvec_{\isometrypos}\rvert}
\big\langle \phi, \Theta_{\sublattice, \jacindexvec}(\cdot, \cdot, g_K, \pol_{g_K, 0}) \big\rangle_{\Pet}.$$
If~$\tilde\lambda$ has positive norm, then the Fourier coefficient associated to~$\tilde\lambda$ is
\textcolor{\newcolor}{\begin{align*}
c_{\tilde{\lambda}}(g) &= \frac{2}{\lvert \isotropicvec_{\isometrypos} \rvert \lvert \eta_{w^\perp} \rvert} 
\sum_{n \mid \tilde{\lambda}}
\sum_{\subpolposdeg}
\frac{n^{b^- + h + k - 4}}{(2i)^{-\subpolposdeg}}
\Bigg(
\frac{\lvert \eta_{z^\perp} \rvert}{2 \lvert u_{z^\perp} \rvert \big(\tilde{\lambda}_{w^\perp}^2 \eta_{w^\perp}^2 - (\tilde{\lambda}_{w^\perp}, \eta_{w^\perp})^2\big)^{1/2}}
\Bigg)^{b^- / 2 + k - h - 2}
\\
&\times 
\exp\bigg(
-\frac{\pi i (\eta, u_{z^\perp}) (\tilde{\lambda}_{w^\perp}, \eta_{w^\perp})}{u_{z^\perp}^2 \eta_{w^\perp}^2}\bigg)
K_{b^-/2 + k - h - 2}\Bigg(
2 \pi \frac{\lvert \eta_{z^\perp} \rvert \big(\tilde{\lambda}_{w^\perp}^2 \eta_{w^\perp}^2 - (\tilde{\lambda}_{w^\perp}, \eta_{w^\perp})^2\big)^{1/2}}{\lvert u_{z^\perp} \rvert \eta_{w^\perp}^2}\Bigg) \numberthis \label{eq;unfjacint}\\
&\times 
\overline{\pol}_{\subspaceisometry, \subpolposdeg}\big(\subspaceisometry(\tilde{\lambda})\big) \sum_{\substack{\lambda \in \sublattice' / \jacindexvec_K \\ \lambda_{\jacindexvec} = \tilde{\lambda}}} e\bigg(\frac{\left(\lambda, \jacindexvec\right)\left(\jacindexvec, \isotropicvecprime\right)}{\jacindexvec^2}\bigg) 
\cdot
a\big(\lambda / n, q\left(\lambda\right) / n^2, \left(\lambda, \jacindexvec\right) / n\big),
\end{align*}
where $K_s(x)$ is the K-Bessel function.}
%\ba
%c_{\tilde{\lambda}}(g) 
%&= 
%\frac{1}{\lvert \isotropicvec_{\isometrypos} \rvert \lvert \jacindexvec \rvert}
%\sum_{\subpolposdeg} 
%\sum_{n \mid \tilde{\lambda}} 
%\frac{n^{2\subpolposdeg - \degjac}}{\left(2i\right)^{\subpolposdeg}} \sum_{\substack{\lambda \in \sublattice' / \jacindexvec_K \\ \lambda_{\jacindexvec} = \tilde{\lambda}}} 
%\overline{\calP}_{\subspaceisometry, \subpolposdeg}\big(\subspaceisometry(\tilde{\lambda})\big) 
%\\
%&\quad\times 
%\int_0^\infty 
%\jacvaroneim^{b^-/2 + k - \subpolposdeg - 3} 
%\exp\bigg(-\frac{\pi n^2}{2 \jacvaroneim \isotropicvec_{\isometrypos}^2} - \frac{\pi n^2 \left(\jacindexvec, \isotropicvec_{\isometrypos}\right)^2}{4 q\left(\jacindexvec\right) \isotropicvec_{\isometrypos}^4 \jacvaroneim}\bigg) d \jacvaroneim  
%\\
%&\quad\times 
%e\left(\frac{\left(\lambda, \jacindexvec\right)\left(\jacindexvec, \isotropicvecprime\right)}{\jacindexvec^2}\right)
%\cdot
%a\Big(\lambda / n, q\left(\lambda\right) / n^2, \left(\lambda, \jacindexvec\right) / n\Big).
%\ea
In all remaining cases, the Fourier coefficients vanish.
\end{thm}

The \emph{Eichler transformation} associated to the isotropic lattice vector~$\isotropicvec$ and $\lambda \in \jacindexvec_{\sublattice}^\perp \cap \sublattice \otimes \IR$ is defined as
$$E(\isotropicvec, \lambda)(v) \coloneqq v - (v, \isotropicvec) \lambda + (v, \lambda) \isotropicvec - q(\lambda) (v, \isotropicvec) \isotropicvec.$$
The fact that~$\langle \phi, \Theta_{\lattice, \jacindexvec} \rangle_{\Pet} \colon G\to\CC$ admits a Fourier expansion of the form
\bes
\big\langle \phi, \Theta_{\lattice, \jacindexvec}(\cdot, \cdot, g, \pol) \big\rangle_{\Pet}
=
\sum_{\tilde{\lambda} \in (\jacindexvec^\perp \cap \sublattice)'}
c_{\tilde{\lambda}}(g)\cdot e(\tilde{\lambda}, \mu),
\ees
for some Fourier coefficients~$c_{\tilde{\lambda}}$, follows from its invariance under the Eichler transformations~$E\left(\isotropicvec, \lambda\right)$, where $\lambda \in \jacindexvec^\perp \cap K$.
This invariance can be easily checked using the following result, which is known to experts and whose proof is left to the reader.
\begin{lemma}\label{lemma;auxlemmaforchfe}
	Let~$g\in G$ be an isometry mapping~$z\in\Gr\left(L\right)$ to the base point~$z_0$, and let~$\lambda\in\eta^\perp\cap\brK\otimes\RR$.
	Let~$z'$ be the negative definite plane that maps to~$z$ under~$E\left(\genU,\lambda\right)$, and let~$w'$ (resp.~$w'^\perp$) be the orthogonal complement of~$\genU_{z'}$ (resp.~$\genU_{z'^\perp}$) in~$z'$ (resp.~$z'^\perp$).
	Let~$\mu$ and~$\mu'$ be constructed respectively from~$z$ and~$z'$ as in~\eqref{eq;defmuinsecjac}.
	\begin{enumerate}[label=(\roman*), leftmargin=*]
	\item We have that~$E(\genU,\lambda)(w')=w$ and that~$E(\genU,\lambda)(w'^\perp)=w^\perp$.
	\label{auxlemma1,4}
	\item We have that~$E(\genU,\lambda)(\mu')=\mu + \lambda + q(\lambda)\genU$.
	\label{auxlemma1,5}
	\item We have that~$\borw(\genvec)=\borwmix{g\circ E(\genU,\lambda)}(\genvec)$ for every~$\genvec\in\brK\otimes\RR$.
	\label{auxlemma2,1}
	\item We have~$\genU^2_{z^\perp} = \genU^2_{E(\genU,\lambda)(z^\perp)}$.
	\label{auxlemma2,2}
	\item
	We have~$\genvec^2_{w^\perp} = \genvec^2_{w'^\perp}$ for every~$\genvec\in\brK\otimes\RR$.
	\label{auxlemma2,3}
	\end{enumerate}
	\end{lemma}

To check that the Fourier coefficients~$c_{\tilde{\lambda}}$ are as provided by Lemma~\ref{lem:JacobiThetaUnfolding}, it is enough to check that~$(2 \isotropicvec_{\isometrypos}^2)^{-1/2} \big\langle f, \Theta_{\sublattice, \jacindexvec}\left(\cdot, \cdot, g_K, \pol_{g_K, 0}\right) \big\rangle_{\Pet}$ \textcolor{\newcolor}{and the coefficients $c_{\tilde{\lambda}}$}
%\ba\label{eq:FourierCoeff}
%\frac{1}{\lvert \isotropicvec_{\isometrypos} \rvert \lvert \jacindexvec \rvert}
%&
%\sum_{\tilde{\lambda} \in (\jacindexvec^\perp \cap \sublattice)'}
%\sum_{\subpolposdeg} \sum_{n \mid \tilde{\lambda}} \frac{n^{2\subpolposdeg - \degjac}}{\left(2i\right)^{\subpolposdeg}} 
%\sum_{\substack{\lambda \in \sublattice' / \jacindexvec_K \\ \lambda_{\jacindexvec} = \tilde{\lambda}}} 
%\overline{\calP}_{\subspaceisometry, \subpolposdeg}\big(\subspaceisometry(\tilde{\lambda})\big)
%\\
%&\quad\times
%\int_0^\infty \jacvaroneim^{b^-/2 + k - \subpolposdeg - 3} 
%\exp\bigg(-\frac{\pi n^2}{2 \jacvaroneim \isotropicvec_{\isometrypos}^2} - \frac{\pi n^2 \left(\jacindexvec, \isotropicvec_{\isometrypos}\right)^2}{4 q\left(\jacindexvec\right) \isotropicvec_{\isometrypos}^4 \jacvaroneim}\bigg) d \jacvaroneim 
%\\
%&\quad\times 
%e\left(\frac{\left(\lambda, \jacindexvec\right)\left(\jacindexvec, \isotropicvecprime\right)}{\jacindexvec^2}\right) 
%\cdot
%a\big(\lambda / n, q\left(\lambda\right) / n^2, \left(\lambda, \jacindexvec\right) / n\big)
%\ea
are invariant with respect to Eichler transformations~$E\left(\isotropicvec, \lambda\right)$, where~$\lambda \in \jacindexvec^\perp \cap K\otimes\RR$.
This follows easily from Lemma~\ref{lemma;auxlemmaforchfe}.

We conclude the section with the proof of Theorem~\ref{thm:JacobiThetaUnfolding}.
\begin{proof}[Proof of Theorem~\ref{thm:JacobiThetaUnfolding}]
	We use Lemma~\ref{lem:JacobiThetaUnfolding} and only consider the second summand. It is given by
	\ba\label{eq:jacobithetalift3}
		&\frac{\sqrt{2}}{\lvert \isotropicvec_{\isometrypos} \rvert} 
		\sum_{n = 1}^\infty 
		\sum_{\subpolposdeg} \Big(\frac{n}{2i}\Big)^{\subpolposdeg} 
		\int_0^\infty \jacvaroneim^{b^-/2 + k - \subpolposdeg - 5/2}
		\\
		&\quad\times 
		\exp\bigg(-\frac{\pi n^2}{2 \jacvaroneim \isotropicvec_{\isometrypos}^2}\bigg) 
		\sum_{\tilde{\lambda} \in (\jacindexvec^\perp \cap \sublattice)'} \sum_{\substack{\lambda \in \sublattice' / \jacindexvec_K \\ \lambda_{\jacindexvec} = \tilde{\lambda}}} \overline{\pol}_{\subspaceisometry, \subpolposdeg}\big(\subspaceisometry(\tilde{\lambda})\big)
		\\
		&\quad\times 
		\int_{-\infty}^{\infty} e\bigg(\frac{n  \jacvartwoim \left(\jacindexvec, \isotropicvec_{\isometrypos}\right)}{\isotropicvec_{\isometrypos}^2}\bigg)
		\cdot
		 \exp\left(-4 \pi q\big((\tilde{\lambda} + \jacvartwoim \jacindexvec)_{\subspaceisometrypos}\big) \jacvaroneim\right) d \jacvartwoim\, d \jacvaroneim \\
		&\quad\times e\left(\frac{\left(n\lambda, \jacindexvec\right)\left(\jacindexvec, \isotropicvecprime\right)}{\jacindexvec^2}\right) 
		\cdot
		a\big(\lambda, q\left(\lambda\right), \left(\lambda, \jacindexvec\right)\big) 
		\cdot
		e(n \tilde{\lambda}, \mu).
	\ea
	\textcolor{\newcolor}{The inner integral of~\eqref{eq:jacobithetalift3} is a Fourier transform of a Gaussian and given by
		\begin{align*}
			&\exp\bigg(-\frac{\pi n^2}{2 \jacvaroneim \isotropicvec_{\isometrypos}^2}\bigg)  \int_{-\infty}^{\infty} 
			\exp\bigg(2 \pi i \frac{n  \jacvartwoim \left(\jacindexvec, \isotropicvec_{\isometrypos}\right)}{\isotropicvec_{\isometrypos}^2}\bigg) 
			\exp\left(-4 \pi q\big((\tilde{\lambda} + \jacvartwoim \jacindexvec)_{\subspaceisometrypos}\big) \jacvaroneim\right) d \jacvartwoim
			\\ \\
			&= (2 \eta_{w^\perp}^2 y_1)^{-1/2} \exp\bigg(
			- \frac{\pi n^2 \eta_{z^\perp}^2}{2 y_1u_{z^\perp}^2 \eta_{w^\perp}^2} - 2 \pi y_1 \frac{\tilde{\lambda}_{w^\perp}^2 \eta_{w^\perp}^2 - (\tilde{\lambda}_{w^\perp}, \eta_{w^\perp})^2}{\eta_{w^\perp}^2}\bigg)
			\\
			&\quad\times \exp\bigg(-\frac{\pi i n (\eta, u_{z^\perp}) (\tilde{\lambda}_{w^\perp}, \eta_{w^\perp})}{u_{z^\perp}^2 \eta_{w^\perp}^2}\bigg).
		\end{align*}
		Hence, the integral over $y_1$ of~\eqref{eq:jacobithetalift3} is given by
		$$\int_0^\infty y_1^{b^-/2 + k - h - 3} \exp\bigg(
		- \frac{\pi n^2 \eta_{z^\perp}^2}{2 y_1u_{z^\perp}^2 \eta_{w^\perp}^2} - 2 \pi y_1 \frac{\tilde{\lambda}_{w^\perp}^2 \eta_{w^\perp}^2 - (\tilde{\lambda}_{w^\perp}, \eta_{w^\perp})^2}{\eta_{w^\perp}^2}\bigg).$$
		Using the formula \cite[p. 313, 6.3(17)]{ErdelyiIntegrals}
		$$\int_{y_1 = 0}^\infty \exp\bigg(-\alpha y_1 - \frac{\beta}{y_1} \bigg) y_1^\gamma \frac{d y_1}{y_1} = 2 \Big(\frac{\beta}{\alpha}\Big)^{\gamma / 2} K_\gamma(2 (\alpha \beta)^{1/2})$$
		shows that the integral over $y_1$ of~\eqref{eq:jacobithetalift3} is equal to
		\begin{align*}
			&2 \Bigg(
			\frac{n \lvert \eta_{z^\perp} \rvert}{2 \lvert u_{z^\perp} \rvert \big(\tilde{\lambda}_{w^\perp}^2 \eta_{w^\perp}^2 - (\tilde{\lambda}_{w^\perp}, \eta_{w^\perp})^2\big)^{1/2}}\Bigg)^{b^- / 2 + k - h - 2}
			\\
			&\quad\times
			K_{b^-/2 + k - h - 2}
			\Bigg(
			2 \pi n \frac{\lvert \eta_{z^\perp} \rvert \big(\tilde{\lambda}_{w^\perp}^2 \eta_{w^\perp}^2 - (\tilde{\lambda}_{w^\perp}, \eta_{w^\perp})^2\big)^{1/2}}{\lvert u_{z^\perp} \rvert \eta_{w^\perp}^2}
			\Bigg).
		\end{align*}
		Hence, \eqref{eq:jacobithetalift3} is given by
		\begin{align*}
			&\frac{2}{\lvert \isotropicvec_{\isometrypos} \rvert \lvert \eta_{w^\perp} \rvert} 
			\sum_{n = 1}^\infty n^{b^- /2 + k - 2}
			\sum_{\subpolposdeg} (2i)^{-\subpolposdeg}
			\\
			&\times 
			\sum_{\tilde{\lambda} \in (\jacindexvec^\perp \cap \sublattice)'} \sum_{\substack{\lambda \in \sublattice' / \jacindexvec_K \\ \lambda_{\jacindexvec} = \tilde{\lambda}}}
			\overline{\pol}_{\subspaceisometry, \subpolposdeg}\big(\subspaceisometry(\tilde{\lambda})\big)
			\Bigg(\frac{\lvert \eta_{z^\perp} \rvert}{2 \lvert u_{z^\perp} \rvert \big(\tilde{\lambda}_{w^\perp}^2 \eta_{w^\perp}^2 - (\tilde{\lambda}_{w^\perp}, \eta_{w^\perp})^2\big)^{1/2}}\Bigg)^{b^- / 2 + k - h - 2}
			\\
			&\times 
			\exp\bigg(-\frac{\pi i n (\eta, u_{z^\perp}) (\tilde{\lambda}_{w^\perp}, \eta_{w^\perp})}{u_{z^\perp}^2 \eta_{w^\perp}^2}\bigg)
			K_{b^-/2 + k - h - 2}\Bigg(
			2 \pi n \frac{\lvert \eta_{z^\perp} \rvert \big(\tilde{\lambda}_{w^\perp}^2 \eta_{w^\perp}^2 - (\tilde{\lambda}_{w^\perp}, \eta_{w^\perp})^2\big)^{1/2}}{\lvert u_{z^\perp} \rvert \eta_{w^\perp}^2}\Bigg)
			\\
			&\times e\left(\frac{\left(n\lambda, \jacindexvec\right)\left(\jacindexvec, \isotropicvecprime\right)}{\jacindexvec^2}\right) 
			\cdot
			a\big(\lambda, q\left(\lambda\right), \left(\lambda, \jacindexvec\right)\big) 
			\cdot
			e(n \tilde{\lambda}, \mu).
		\end{align*}
		Rewriting this as a divisor sum yields the result.
	}
\end{proof}
	%%%%%%%%%%%%%%%%%%%%%%%%%%%%%%%%%%%%%%%%%%%%%%%%%%%%%%%%%%%%%%
	\section{The Kudla--Millson theta function}\label{sec;deg2kmtform}
	In this section we introduce the Kudla--Millson Schwartz function~$\varphi_{\text{\rm KM},2}$ and the Kudla--Millson theta function~$\Theta(\tau,z,\varphi_{\text{\rm KM},2})$ of genus~$2$, following the wording of~\cite[Section~5]{kumi;intnum} and~\cite[Section~5.2]{fm;cycleswith}.
	We then illustrate how to rewrite~$\Theta(\tau,z,\varphi_{\text{\rm KM},2})$ in terms of the Siegel theta functions introduced in Section~\ref{sec;vvsiegeltheta2}.
	\\
	
	Let~$L$ be a (non-degenerate) even lattice of signature~$(b,2)$, for some~$b>0$.
	Recall that we denote by~$(\cdot{,}\cdot)$, resp.~$q(\cdot)$, the associated bilinear form, resp.\ quadratic form, and by~$\Gr(L)$ the Grassmannian associated to~$L$.
	The latter is the set of negative definite planes in~$V=L\otimes\RR$, and is identified with the Hermitian symmetric space~$\hermdom$ attached to~$V$.
	From now on, we write~$\hermdom$ and~$\Gr(L)$ interchangeably.
	We denote by~$X=\Gamma\backslash\hermdom$ an orthogonal Shimura variety arising from some finite index subgroup~$\Gamma\subseteq\widetilde{\SO}(L)$; see Section~\ref{sec;OSV} for further information.
	
	\subsection{The Kudla--Millson Schwartz function}\label{sec;Schwfunct}
	We keep the same notation introduced in Section~\ref{sec;vvsiegeltheta2}.
	In particular, we fix once and for all an orthogonal basis~$(\basevec_j)_j$ of~${V}$ such that~$(\basevec_\alpha,\basevec_\alpha)=1$ for every~$\alpha=1,\dots,b$, and~$(\basevec_\mu,\basevec_\mu)=-1$ for~$\mu=b+1,b+2$.
	We define the standard majorant~$(\cdot{,}\cdot)_z$ of~$V^2$ with respect to~$z\in\Gr(L)$ as
	\be\label{eq;stmajorantgen2}
	(\vect{\genvec},\vect{\genvec})_z=(\vect{\genvec}_{z^\perp},\vect{\genvec}_{z^\perp})-(\vect{\genvec}_z,\vect{\genvec}_z),\qquad\text{for every~$\vect{v}\in V^2$.}
	\ee
	
	Let $\mathfrak{g}$ be the Lie algebra of $G$, and let $\mathfrak{g}=\mathfrak{p}+\mathfrak{k}$ be its Cartan decomposition.
	It is well-known that~$\mathfrak{p}\cong\mathfrak{g}/\mathfrak{k}$ is isomorphic to the tangent space of~$\hermdom$ at any base point~$z_0$.
	We may choose~$z_0$ to be the negative-definite plane spanned by~$\basevec_{b+1}$ and~$\basevec_{b+2}$.
	
	With respect to the basis of $V$ chosen above, we have
	\be\label{eq;mathfracpisogen2}
	\mathfrak{p}\cong\left\{\left(
	\begin{smallmatrix}
	0 & X\\
	X^t & 0
	\end{smallmatrix}
	\right)|X\in\Mat_{b,2}(\RR)\right\}\cong \Mat_{b,2}(\RR).
	\ee
	We may assume that the chosen isomorphism is such that the complex structure on~$\mathfrak{p}$ is given as the right-multiplication by~$J=\big(\begin{smallmatrix}0 & 1\\ -1 & 0\end{smallmatrix}\big)\in\GL_2(\RR)$ on~$\Mat_{b,2}(\RR)$.
	
%	\begin{defi}\label{def;stdgausgen2}
%	The \emph{standard Gaussian}~$\stgatwo$ of~$(\RR^{b,2})^2$ is defined as
%	\bes
%	\stgatwo(\vect{x})=\exp\Big(
%	-\pi\sum_{i=1}^{b+2}\sum_{j=1}^2 x_{i,j}^2
%	\Big),\qquad\text{for every~$\vect{x}=(x_1,x_2)\in(\RR^{b,2})^2$},
%	\ees
%	where~$x_j=(x_{1,j},\dots,x_{b+2,j})^t\in \RR^{b,2}$.
%	The standard Gaussian of $V^2$ is the composition of~${\stgatwo}$ with~$g_0$, where the latter is as in~\eqref{eq;stgatwomat}.
%	For simplicity we will denote the standard Gaussian of~$V^2$ with~$\stgatwo$ in place of~$\stgatwo\circ g_0$.
%	It is~$K$ invariant with respect to the action given by the Schrödinger model, where~$K$ is the standard compact maximal of~$G$ stabilizing the base point~$z_0\in\Gr(V)$.
%	\end{defi}
	
	The Kudla--Millson Schwartz function~$\varphi_{\text{\rm KM},2}$ is a~$G$-invariant element of~$\mathcal{S}(V^2)\otimes\mathcal{A}^4(\hermdom)$, where~$\mathcal{A}^4(\hermdom)$ is the space of differential~$4$-forms on~$\hermdom$.
	Recall that
	\bes
	\big[\mathcal{S}(V^2)\otimes\mathcal{A}^4(\hermdom)\big]^G
	\cong
	\big[
	\mathcal{S}(V^2)\otimes{\bigwedge}^4(\mathfrak{p}^*)
	\big]^K,
	\ees
	where the isomorphism is given by evaluating at the base point~$z_0$ of~$\hermdom$, and~\textcolor{\newcolor}{$K_\infty$} is the standard maximal compact of~$G$ stabilizing the base point~$z_0\in\Gr(V)$.
	We may then define~$\varphi_{\text{\rm KM},2}$ firstly as an element of~$[ \mathcal{S}(V^2)\otimes{\bigwedge}^4(\mathfrak{p}^*)]^{\textcolor{\newcolor}{K_\infty}}$, and then spread it to the whole~$\hermdom$ via the action of~$G$.
	
	\begin{defi}\label{defi;KMfunctgen2}
	We denote by $X_{\alpha,\mu}$, with $1\le\alpha\le b$ and~$1\le\mu\le 2$, the basis elements of~$\Mat_{b,2}(\RR)$ given by matrices with~$1$ at the~$(\alpha,\mu)$-th entry and zero otherwise.
	These elements give a basis of~$\mathfrak{p}$ via the isomorphism~\eqref{eq;mathfracpisogen2}.
	Let~$\omega_{\alpha,\mu}$ be the element of the dual basis which extracts the $(\alpha,\mu)$-th coordinate of elements in~$\mathfrak{p}$, and let~$A_{\alpha,\mu}$ be the left multiplication by~$\omega_{\alpha,\mu}$.
	The function~$\varphi_{\text{\rm KM},2}$ is defined applying the operator
	\bes
	\mathcal{D}^{b,2}_2=\frac{1}{4}\prod_{j,\mu=1}^2
	%\prod_{\mu=1}^{2}
	\sum_{\alpha=1}^b\Big(
	x_{\alpha,j}-\frac{1}{2\pi}\frac{\partial}{\partial x_{\alpha,j}}
	\Big)\otimes A_{\alpha,\mu}
	\ees
	to the standard Gaussian~$\stgatwo\otimes 1\in[\mathcal{S}(V^2)\otimes{\bigwedge}^4(\mathfrak{p}^*)]^{\textcolor{\newcolor}{K_\infty}}$, namely
	\bes
	\varphi_{\text{\rm KM},2}=\mathcal{D}^{b,2}_2\big(\stgatwo\otimes 1\big).
	\ees
	\end{defi}
	As a differential form,~$\varphi_{\text{\rm KM},2}$ is \emph{closed}; see~\cite{kumi;harmI}.	
	The following result provides an explicit formula of~$\varphi_{\text{\rm KM},2}$.
	The idea of the proof is analogous to the one in~\cite[Section~$2$]{zuffetti;gen1}, where we illustrated how to rewrite the Kudla--Millson Schwartz function of genus~$1$ in terms of certain polynomials~$\Gpol_{(\alpha,\beta)}$.
	Recall that the latter are defined on~$\RR^{b,2}$ as
	\be\label{eq;recallQabgen2}
	\Gpol_{(\alpha,\beta)}(x)\coloneqq\begin{cases}
	\pol_{(\alpha,\beta)}(x), & \text{if $\alpha\neq\beta$,}\\
	\pol_{(\alpha,\beta)}(x)-\frac{1}{2\pi}, & \text{otherwise,}
	\end{cases}
	\qquad\text{where}\qquad
	\pol_{(\alpha,\beta)}(x)\coloneqq 2x_\alpha x_\beta,
	\ee
	for every $x=(x_1,\dots,x_{b+2})^t\in\RR^{b,2}$.
	
	As in Definition~\ref{def;thetasergen2}, for simplicity we consider~$\Gpol_{(\alpha,\beta)}$ also as a polynomial in the coordinates of~$V$ with respect to the basis~$(\basevec_j)_j$ chosen above.
	Hence, we drop~$g_0$ from the notation and write~$\Gpol_{(\alpha,\beta)}(\genvec)$ instead of~$\Gpol_{(\alpha,\beta)}\big( g_0(\genvec)\big)$.
	This convention will be used for all polynomials defined on~$V$ and~$V^2$.
	\begin{prop}\label{prop;comexpphiKMdeg2}
	The Kudla--Millson Schwartz function $\varphi_{\text{\rm KM},2}\in\big[
	\mathcal{S}(V^2)\otimes{\bigwedge}^4(\mathfrak{p}^*)
	\big]^{\textcolor{\newcolor}{K_\infty}}$ may be rewritten as
	\ba\label{eq;schwKMexpldeg2}
	\varphi_{\text{\rm KM},2}(\vect{\genvec},z_0)
	&=
	\sum_{\substack{\alphaone,\betaone=1\\ \alphaone<\betaone}}^b
	\sum_{\substack{\alphatwo,\betatwo=1\\ \alphatwo<\betatwo}}^b \big(\Gpol_{\abcd}\cdot\stgatwo\big)(\vect{v})
	\\
	&\quad\times\omegaone\wedge\omegatwo\wedge\omegathree\wedge\omegafour,
	\ea
	where $\Gpoltwo_{\abcd}$ is the polynomial on~$V^2$ defined as
	\bes
	\Gpoltwo_{\abcd}(\vect{\genvec})=\sum_{\sigma,\sigma'\in S_2}\sgn(\sigma)\sgn(\sigma') \Gpol_{(\sigma(\alphaone),\sigma'(\alphatwo))}(\genvec_1) \cdot \Gpol_{(\sigma(\betaone),\sigma'(\betatwo))}(\genvec_2),
	\ees
%	\bes
%	\Gpoltwo_{\abcd}\big(g_0(\vect{\genvec})\big)=\sum_{\sigma,\sigma'\in S_2}\sgn(\sigma)\sgn(\sigma') \Gpol_{(\sigma(\alpha),\sigma'(\beta))}\big(g_0(\genvec_1)\big) \cdot \Gpol_{(\sigma(\gamma),\sigma'(\delta))}\big(g_0(\genvec_2)\big),
%	\ees
	and where we denote by $\sigma$, resp.\ $\sigma'$, a permutation of the indexes $\{\alphaone,\betaone\}$, resp.\ $\{\alphatwo,\betatwo\}$.
	\end{prop}
	To simplify the notation, we will frequently replace~$\abcd$ by a vector of indices~$\vect{\alpha}$.
	\begin{rem}
	From~\eqref{eq;recallQabgen2}, we may rewrite the product of polynomials appearing in the summand of the defining sum of~$\Gpoltwo_{\vect{\alpha}}$ more explicitly as
	\bas
	\Gpol_{(\alphaone,\alphatwo)}(\genvec_1)\cdot \Gpol_{(\betaone,\betatwo)}(\genvec_2)=\begin{cases}
	4 x_{\alphaone,1} \cdot x_{\alphatwo,1} \cdot x_{\betaone,2} \cdot x_{\betatwo,2}, & \text{if $\alphaone\neq\alphatwo$ and $\betaone\neq\betatwo$,}\\
	2 x_{\alphaone,1} \cdot x_{\alphatwo,1} \cdot \big( 2 x_{\betaone,2}^2-\frac{1}{2\pi}\big), & \text{if $\alphaone\neq\alphatwo$ and $\betaone=\betatwo$,}\\
	\big(2 x_{\alphaone,1}^2-\frac{1}{2\pi}\big) \cdot 2 x_{\betaone,2} \cdot x_{\betatwo,2}, & \text{if $\alphaone=\alphatwo$ and $\betaone\neq\betatwo$,}\\
	\big(2 x_{\alphaone,1}^2-\frac{1}{2\pi}\big)\big( 2x_{\betaone,2}^2-\frac{1}{2\pi}\big), & \text{if $\alphaone=\alphatwo$ and $\betaone=\betatwo$.}
	\end{cases}
	\eas
	Moreover, an easy calculation shows that
	$$\Gpol_{(\alphaone,\alphatwo)}(\genvec_1)\cdot \Gpol_{(\betaone,\betatwo)}(\genvec_2) = \exp\Big(-\frac{\trace\Delta}{8 \pi}\Big)\polab(\genvec_1) \polcd(\genvec_2),$$
	where~$\Delta$ is the Laplacian on~$(\RR^{b,2})^2$ defined in~\eqref{eq;strdlapl}.
	\end{rem}
	
%	OLD VERSION OF THE PREVIOUS RESULT ---- BAD NOTATION!
%	\begin{prop}
%	The Kudla--Millson Schwartz form $\varphi^{b,2}_{\text{\rm KM},2}\in\big[
%	\mathcal{S}(V^2)\otimes{\bigwedge}^4(\mathfrak{p}^*)
%	\big]^K$ can be rewritten as
%	\be\label{eq;schwKMexpldeg2}
%	\varphi^{b,2}_{\text{\rm KM},2}(\vect{\genvec},z_0)=\sum_{\substack{\alpha,\gamma=1\\ \alpha<\gamma}}^b
%	\sum_{\substack{\beta,\delta=1\\ \beta<\delta}}^b \big(\polabcd\stgatwo\big)(g_0(\vect{v}))\otimes\omega_{\alpha,b+1}\wedge\omega_{\beta,b+2}\wedge\omega_{\gamma,b+1}\wedge\omega_{\delta,b+2},
%	\ee
%	with $\polabcd$ defined as
%	\bes
%	\polabcd=\sum_{\sigma,\sigma'\in S_2}\sgn(\sigma)\sgn(\sigma') \widetilde{\pol}_{(\sigma(\alpha),\sigma'(\beta),\sigma(\gamma),\sigma'(\delta))},
%	\ees
%	where $\polabcdti$ is the polynomial on $(\RR^{b,2})^2$ defined as
%	\bas
%	\polabcdti\big(g_0(\vect{v})\big)=\begin{cases}
%	4 x_{\alpha,1} \cdot x_{\beta,1} \cdot x_{\gamma,2} \cdot x_{\delta,2}, & \text{if $\alpha\neq\beta$ and $\gamma\neq\delta$,}\\
%	2 x_{\alpha,1} \cdot x_{\beta,1} \cdot \big( 2 x_{\gamma,2}^2-\frac{1}{2\pi}\big), & \text{if $\alpha\neq\beta$ and $\gamma=\delta$,}\\
%	\big(2 x_{\alpha,1}^2-\frac{1}{2\pi}\big) 2 x_{\gamma,2} \cdot x_{\delta,2}, & \text{if $\alpha=\beta$ and $\gamma\neq\delta$,}\\
%	\big(2 x_{\alpha,1}^2-\frac{1}{2\pi}\big)\big( 2x_{\gamma,2}^2-\frac{1}{2\pi}\big), & \text{if $\alpha=\beta$ and $\gamma=\delta$,}
%	\end{cases}
%	\eas
%	and where we denote by $\sigma$, resp.\ $\sigma'$, a permutation of $\{\alpha,\gamma\}$, resp.\ $\{\beta,\delta\}$.
%	\end{prop}
	\begin{proof}[Proof of Proposition~\ref{prop;comexpphiKMdeg2}]
	For simplicity, we write $\mathcal{F}_{\alpha,j}=x_{\alpha,j}-\frac{1}{2\pi}\frac{\partial}{\partial x_{\alpha,j}}$, for every~${j=1,2}$ and~${\alpha=1,\dots,b}$.
	We may use such operators to rewrite
	\ba\label{eq;dim;comexpphiKMdeg2}
	\varphi_{\text{\rm KM},2}
	&=
%	\mathcal{D}^{b,2}_2\big(\stgatwo\otimes 1\big)
%	&=
	\sum_{\substack{\alphaone,\betaone=1\\ \alphaone\neq\betaone}}^b\sum_{\substack{\alphatwo,\betatwo=1\\ \alphatwo\neq\betatwo}}^b \frac{1}{4}\mathcal{F}_{\alphaone,1}\mathcal{F}_{\alphatwo,1}\mathcal{F}_{\betaone,2}\mathcal{F}_{\betatwo,2}(\stgatwo)
	\cdot
%	\\
%	&\quad\times
	\omegaone\wedge\omegatwo\wedge \omegathree\wedge \omegafour,
	\ea
	where we deleted all summands associated to wedge products containing two functionals which are equal.
	We may compute
	\ba\label{eq2;dim;comexpphiKMdeg2}
	\frac{1}{4}\mathcal{F}_{\alphaone,1}\mathcal{F}_{\alphatwo,1}\mathcal{F}_{\betaone,2}\mathcal{F}_{\betatwo,2}(\stgatwo)
	&
	=
	\frac{1}{4}\mathcal{F}_{\alphaone,1}\mathcal{F}_{\alphatwo,1}\mathcal{F}_{\betaone,2}\big(2 x_{\betatwo,2}\cdot\stgatwo\big)
	\\
	&=
	\begin{cases}
	\frac{1}{2}\mathcal{F}_{\alphaone,1}\mathcal{F}_{\alphatwo,1}\big(2x_{\betaone,2}x_{\betatwo,2}\cdot\stgatwo \big), &\text{if $\betaone\neq\betatwo$,}\\
	\frac{1}{2}\mathcal{F}_{\alphaone,1}\mathcal{F}_{\alphatwo,1}\big((2x_{\betaone,2}^2-\frac{1}{2\pi})\stgatwo\big), &\text{if $\betaone=\betatwo$.}
	\end{cases}
	\ea
	Since the entries $x_1$ and $x_2$ of~$g_0(\vect{\genvec})$ are independent to each others, we may repeat an analogous procedure to compute the action of the operator $\frac{1}{2}\mathcal{F}_{1,\alphaone}\mathcal{F}_{1,\alphatwo}$ on the right-hand side of~\eqref{eq2;dim;comexpphiKMdeg2}, to deduce that
	\ba\label{eq3;dim;comexpphiKMdeg2}
	\varphi_{\text{\rm KM},2}(\vect{\genvec},z_0)
	&=
	\sum_{\substack{\alphaone,\betaone=1\\ \alphaone\neq\betaone}}^b\sum_{\substack{\alphatwo,\betatwo=1\\ \alphatwo\neq\betatwo}}^b\Big(\Gpol_{(\alphaone,\alphatwo)}(\genvec_1)\cdot \Gpol_{(\betaone,\betatwo)}(\genvec_2)\cdot\stgatwo(\vect{\genvec})\Big)
	\\
	&\quad\times
	\omegaone\wedge\omegatwo\wedge \omegathree\wedge \omegafour.
	\ea
	The wedge products appearing on the right-hand side of~\eqref{eq3;dim;comexpphiKMdeg2} are linearly dependent in the vector space~${\bigwedge}^4(\mathfrak{p}^*)$.
%	For instance
%	\bes
%	\omega_{\alpha,b+1}\wedge\omega_{\beta,b+2}\wedge \omega_{\gamma,b+1}\wedge \omega_{\delta,b+2}=
%	\omega_{\gamma,b+1}\wedge\omega_{\delta,b+2}\wedge \omega_{\alpha,b+1}\wedge \omega_{\beta,b+2},
%	\qquad\text{for every $\alpha,\beta,\gamma,\delta$.}
%	\ees
	A set of linearly independent wedge products, with respect to which we can write all the ones appearing in~\eqref{eq3;dim;comexpphiKMdeg2}, is
	\bes
	\{\omegaone\wedge\omegatwo\wedge \omegathree\wedge \omegafour\text{ : such that $\alphaone<\betaone$ and $\alphatwo<\betatwo$}\}.
	\ees
	If we rewrite~\eqref{eq3;dim;comexpphiKMdeg2} with respect to such set, taking into account permutations of the indexes~$\{\alphaone,\betaone\}$ and~$\{\alphatwo,\betatwo\}$, then we obtain~\eqref{eq;schwKMexpldeg2}.
	\end{proof}
	\begin{cor}\label{cor;spreadofgen2KMfunc}
	The \textcolor{\mycolor}{extension} of $\varphi_{\text{\rm KM},2}\in\big[\mathcal{S}(V^2)\otimes{\bigwedge}^4(\mathfrak{p}^*)\big]^{\textcolor{\newcolor}{K_\infty}}$ to the whole~$\hermdom$ is
	\bes
	\varphi_{\text{\rm KM},2}(\vect{\genvec},z)
	=
	\sum_{\substack{\alphaone,\betaone=1\\ \alphaone<\betaone}}^b
	\sum_{\substack{\alphatwo,\betatwo=1\\ \alphatwo<\betatwo}}^b \big(\Gpol_{\vect{\alpha}}\stgatwo\big)\big(g(\vect{\genvec})\big)
	\cdot
	g^*(\omegaone\wedge\omegatwo\wedge\omegathree\wedge\omegafour),
	\ees
	where $g\in G$ is any isometry that maps $z\in\hermdom$ to the base point $z_0$.
	\end{cor}
	\begin{proof}
	Note that~$\varphi_{\text{\rm KM},2}(\vect{\genvec},z)=g^*\varphi_{\text{\rm KM},2}\big(g(\vect{\genvec}),z_0\big)$.
	It is enough to replace~$\varphi_{\text{\rm KM},2}\big(g(\vect{\genvec}),z_0\big)$ with the formula provided by Proposition~\ref{prop;comexpphiKMdeg2}.
	\end{proof}
	We define additional auxiliary polynomials~$\pol_{\vect{\alpha}}$ on~$V^2$ as
	\ba\label{eq;defpolPabcd}
	\pol_{\vect{\alpha}}(\vect{\genvec})
	&\textcolor{\mycolor}{{}\coloneqq
	4
	\det\big(
	\begin{smallmatrix}
	x_{\alphaone,1} & x_{\alphaone,2}\\
	x_{\betaone,1} & x_{\betaone,2}
	\end{smallmatrix}
	\big)
	\det\big(
	\begin{smallmatrix}
	x_{\alphatwo,1} & x_{\alphatwo,2}\\
	x_{\betatwo,1} & x_{\betatwo,2}
	\end{smallmatrix}
	\big)}
	\\
	&=
	4 \sum_{\sigma,\sigma'\in S_2} \sgn(\sigma)\sgn(\sigma') x_{\sigma(\alphaone),1} \cdot x_{\sigma'(\alphatwo),1} \cdot x_{\sigma(\betaone),2} \cdot x_{\sigma'(\betatwo),2},
	\ea
	for every~$\vect{\genvec}=\big(\sum_{j}x_{j,1}\basevec_j,\sum_{j}x_{j,2}\basevec_j\big)\in V^2$, where~$\sigma$ and~$\sigma'$ are permutations of respectively~$\{\alphaone,\betaone\}$ and~$\{\alphatwo,\betatwo\}$.
	Note that~$\Gpol_{\vect{\alpha}}=\exp\big(-\trace\Delta/8 \pi\big) \pol_{\vect{\alpha}}$ and that if~$\alphaone\neq\alphatwo$ and~$\betaone\neq\betatwo$, then~$\pol_{\vect{\alpha}}$ is harmonic and~$\Gpol_{\vect{\alpha}}=\pol_{\vect{\alpha}}$.
	
	\begin{rem}
	\textcolor{\newcolor}{Let~$S(V^2)$ be the polynomial Fock space in~$\mathcal{S}(V^2)$, and let~$\iota\colon S(V^2)\to\mathcal{P}(\CC^{2(b+2)})$ be the associated intertwining map, where~$\mathcal{P}(\CC^{2(b+2)})$ is the space of complex polynomials in~$2(b+2)$ variables; see~\cite[Appendix~A]{fm;cycleswith} for precise definitions.
	In place of working with functions in the Fock space, it is sometimes easier to consider their images under~$\iota$.
	We remark that even if~$\Gpol_{\vect{\alpha}}\stgatwo$ lies in~$S(V^2)$, there is no simplification in working with its image~$\iota(\Gpol_{\vect{\alpha}}\stgatwo)$.
	In fact, one can easily check that
	\[
	\iota(\Gpol_{\vect{\alpha}}\stgatwo)(\vect{z})
%	=
%	\iota(\pol_{\vect{\alpha}}\stgatwo)(\vect{z})
	=
	\frac{1}{2^6}\det\big(\begin{smallmatrix}
	z_{\alpha_1,1} & z_{\alpha_1,2}
	\\
	z_{\beta_1,1} & z_{\beta_1,2}
	\end{smallmatrix}\big)
	\det\big(\begin{smallmatrix}
	z_{\alpha_2,1} & z_{\alpha_2,2}
	\\
	z_{\beta_2,1} & z_{\beta_2,2}
	\end{smallmatrix}\big),
	\]
	which is of the same form as~\eqref{eq;defpolPabcd}.
	Here we denoted~$\vect{z}=(z_{1,1},\dots,z_{b+2,1},z_{1,2},\dots,z_{b+2,2})$.}
	\end{rem}
	
	The polynomials~$\pol_{\vect{\alpha}}$ will play a key role in the upcoming sections, thanks to the following result.
	We suggest the reader to recall the construction of \emph{very homogeneous polynomials} from Definition~\ref{def;homogforrb22}.
	\begin{lemma}\label{lemma;homogpolabcddeg2}
	The polynomials~$\pol_{\vect{\alpha}}$ are very homogeneous of degree~$(2,0)$.
%	satisfy the homogeneity property
%	\bes
%	\pol_{\vect{\alpha}}(\vect{x}\cdot N)=(\det N)^2\cdot\pol_{\vect{\alpha}}(\vect{x}),
%	\ees
%	for every $\vect{x}\in (\RR^{b,2})^2$ and $N\in\CC^{2\times 2}$.
	\end{lemma}
	\begin{proof}
	\textcolor{\mycolor}{
	The~$2\times 2$-minor of the~$b\times 2$-matrix~$\vect{x}$ given as the determinant of~$\big(\begin{smallmatrix}
	x_{\alphaone,1} & x_{\alphaone,2}
	\\
	x_{\betaone,1} & x_{\betaone,2}
	\end{smallmatrix}\big)$ is a very homogeneous polynomial of degree~$(1,0)$, since~$\alphaone,\betaone\in\{1,\dots,b\}$.
	The polynomial~$\pol_{\vect{\alpha}}$ is, up to a constant, the product of two such minors.
	Hence, it is very homogeneous of degree~$(2,0)$.
	}
	\end{proof}
	
	\subsection{The Kudla--Millson theta function in terms of Siegel theta functions}\label{eq;vvKMthetamD}
	Let~$X$ be an orthogonal Shimura variety arising from~$L$; see Section~\ref{sec;OSV}.
	We fix once and for all the weight~$k=1+b/2$.
%	Given $\tau=x+iy\in\HH_2$, we denote by~$g_\tau\in\Sp_4(\RR)$ a standard element which moves the base point~$i\in\HH_2$ to~$\tau$, that is
%	\be\label{eq;stelgendeg2}
%	g_\tau=\left(\begin{smallmatrix}
%	1 & x\\
%	0 & 1
%	\end{smallmatrix}\right)\left(\begin{smallmatrix}
%	a & 0\\
%	0 & a^{-t}
%	\end{smallmatrix}\right),\qquad\text{for some $a\in\SL_2(\RR)$ such that $y=aa^{-t}$}.
%	\ee
%	Usually, we consider $g_\tau$ to be the standard element with $a=y^{1/2}$ in~\eqref{eq;stelgendeg2}.
%	In fact, the imaginary part $y$ of $\tau$ is a real positive-definite symmetric matrix, hence~$y$ admits a unique square root matrix which is positive-definite.
	
	\begin{defi}
	The \emph{Kudla--Millson theta form of genus~$2$} is defined as
	\be\label{eq;KMthetavv}
	\Theta(\tau,z,\varphi_{\text{KM},2})=\det y^{-k/2}\sum_{\gendisc\in\disc{L}^2}\sum_{\vect{\lambda}\in \gendisc + L^2}\big(\omega_{\infty,2}(\textcolor{\mycolor}{M_\tau})\varphi_{\text{KM},2}\big)(\vect{\lambda},z)
	\mathfrak{e}_\gendisc,
	\ee
	for every $\tau=x+iy\in\HH_2$ and $z\in\Gr(L)$, where $\textcolor{\mycolor}{M_\tau}=\left(\begin{smallmatrix}
	1 & x\\ 0 & 1
	\end{smallmatrix}\right)\left(\begin{smallmatrix}
	y^{1/2} & 0\\ 0 & (y^{1/2})^{-t}
	\end{smallmatrix}\right)$ is the standard element of~$\Sp_4(\RR)$ mapping~${iI_2\in\HH_2}$ to~$\tau$.
	\end{defi}
	
	Let~$A^k_{2,L}$ be the space of~$\CC[\disc{L}^2]$-valued analytic functions on~$\HH_2$ satisfying the weight~$k$ modular transformation property with respect to~$\Mp_4(\ZZ)$ under the Weil representation~$\weil{L}$.
	The Kudla--Millson theta function~\eqref{eq;KMthetavv} is a non-holomorphic modular form with respect to the variable~$\tau\in\HH_2$, and a closed~$4$-form with respect to the variable~${z\in\Gr(L)}$, in short~$\Theta(\tau,z,\varphi_{\text{\rm KM},2})\in A^k_{2,L}\otimes\mathcal{Z}^4(\hermdom)$.
	In fact, the Kudla--Millson theta function is~$\widetilde{\SO}(L)$-invariant, hence it descends to an element of~$A^k_2\otimes\mathcal{Z}^4(X)$.
	
	Recall that~$\halfint_2$ is the set of symmetric half-integral~$2\times 2$-matrices.
	Kudla and Millson~\cite{kumi;intnum} proved that the cohomology class~$[\Theta(\tau,z,\varphi_{\text{\rm KM},2})]$ is a \emph{holomorphic} modular form with values in~$H^{2,2}(X)$.
	Moreover, they proved that the Fourier coefficient of~$[\Theta(\tau,z,\varphi_{\text{\rm KM},2})]$ associated to the indices~$\gendisc\in\disc{L}$ and~$T\in q(\gendisc) + \halfint_2$ with~$T>0$ is a Poincaré dual form for the special cycle of same indices in~$X$.
	The class~$[\Theta(\tau,z,\varphi_{\text{\rm KM},2})]$ is the well-known Kudla's generating series of special cycles; see~\cite[Theorem~$3.1$]{Kudla;speccycl}.
	
	By Corollary~\ref{cor;spreadofgen2KMfunc} we may rewrite the Kudla--Millson theta function as
	\ba\label{eq;genformKMthetagen2}
	\Theta(\tau,z,\varphi_{\text{\rm KM},2})
	=
	\sum_{\substack{\alphaone,\betaone=1\\ \alphaone<\betaone}}^b
	\sum_{\substack{\alphatwo,\betatwo=1\\ \alphatwo<\betatwo}}^b
	&
	\underbrace{(\det y)^{-k/2}
	\sum_{\gendisc\in\disc{L}^2}
	\sum_{\vect{\lambda}\in \gendisc + L^2}
	\Big(\omega_{\infty,2}(\textcolor{\mycolor}{M_\tau})(\Gpol_{\vect{\alpha}}\stgatwo)\Big)\big(g(\vect{\lambda})\big)\mathfrak{e}_\gendisc}_{\eqqcolon F_{\vect{\alpha}}(\tau,g)}
	\\
	& \otimes g^*\big(\omegaone\wedge\omegatwo\wedge\omegathree\wedge\omegafour\big),
	\ea
	where $g\in G$ is any isometry of $V=L\otimes\RR$ mapping~$z$ to the base point~$z_0$ of~$\Gr(L)$, and~$\Gpol_{\vect{\alpha}}$ is the polynomial computed in Proposition~\ref{prop;comexpphiKMdeg2}.
	Recall that for simplicity we write~$\vect{\alpha}$ in place of the vector of indices~$(\alphaone,\alphatwo,\betaone,\betatwo)$.
	Since the Kudla--Millson Schwartz function is the spread to the whole~$\hermdom=\Gr(L)$ of an element of~$\mathcal{S}(V^2)\otimes \bigwedge^4(\mathfrak{p}^*)$ which is~$\textcolor{\newcolor}{K_\infty}$-invariant, the rewriting~\eqref{eq;genformKMthetagen2} does not depend on the choice of~$g$ mapping~$z$ to~$z_0$.
	
	We are going to rewrite the auxiliary functions~$F_{\vect{\alpha}}$ arising as in~\eqref{eq;genformKMthetagen2} in terms of the Siegel theta functions of genus~$2$ introduced in Section~\ref{sec;vvsiegeltheta2}.
	
	The following result is the generalization of~\cite[Lemma~$3.9$]{zuffetti;gen1} in genus~$2$.
	We suggest the reader to recall the construction of the very homogeneous polynomials~$\pol_{\vect{\alpha}}$ from~\eqref{eq;defpolPabcd}.
	\begin{lemma}\label{lemma;KMthetadeg2expl}
	We may rewrite the auxiliary functions~$F_{\vect{\alpha}}$ defined in~\eqref{eq;genformKMthetagen2} as
	\be\label{eq;lemmaKMthetadeg2expl22}
	F_{\vect{\alpha}}(\tau,g)=\Theta_{L,2}(\tau,g,\pol_{\vect{\alpha}}),
	\ee
	where~$\tau=x+iy\in\HH_2$ and~$g\in G$.
	\end{lemma}
	\begin{proof}
	\textcolor{\mycolor}{This is a trivial consequence of Lemma~\ref{lemma;fronv2tov1}, since~$\Gpol_{\vect{\alpha}}= \exp\left(-\trace \Delta/8 \pi\right)\pol_{\vect{\alpha}}$.}
	\end{proof}
	%%%%%%%%%%%%%%%%%%%%%%%%%%%%%%%%%%%%%%%%%%%%%%%%%%%%%%%%%%%%%%
	\section{The unfolding of the Kudla--Millson lift}\label{sec;deg2KMlift}
	In this section we unfold the defining integrals of the genus~$2$ Kudla--Millson theta lift.
	Recall that we denote by~$L$ an even indefinite lattice of signature~$(b,2)$, that~$X$ is an orthogonal Shimura variety arising from~$L$, and~$k=1+b/2$.
	\begin{defi}
	The \emph{Kudla--Millson lift of genus $2$} associated to the lattice~$L$ is the map~$\KMlift\colon S^k_{2,L}\to\mathcal{Z}^4(X)$ defined as
	\be\label{eq;liftingLambda2vv}
	f\longmapsto\KMlift(f)=\int_{\Sp_4(\ZZ)\backslash\HH_2}\det y^k\big\langle f(\tau),\Theta(\tau,z,\varphi_{\text{\rm KM},2})\big\rangle\,\frac{dx\, dy}{\det y^3}.
	\ee
	where $dx\,dy\coloneqq\prod_{k\leq\ell}dx_{k,\ell}\,dy_{k,\ell}$,~$\frac{dx\, dy}{\det y^3}$ is the standard~$\Sp_4(\ZZ)$-invariant volume element of~$\HH_2$, and~$\langle\cdot{,}\cdot\rangle$ is the scalar product of the group algebra~$\CC[\disc{L}^2]$.
	\end{defi}
	By Lemma~\ref{lemma;KMthetadeg2expl} we may rewrite the Kudla--Millson lift in terms of genus~$2$ Siegel theta functions as
	\ba\label{eq;KMdeg2liftmoreexpl}
	\KMlift(f)=
	\sum_{\substack{\alphaone,\betaone=1\\ \alphaone<\betaone}}^b
	\sum_{\substack{\alphatwo,\betatwo=1\\ \alphatwo<\betatwo}}^b
	&\Big(\int_{\Sp_4(\ZZ)\backslash\HH_2}
	\det y^{k} \big\langle f(\tau) , \Theta_{L,2}(\tau,g,\pol_{\vect{\alpha}})\big\rangle \frac{dx\,dy}{\det y^3}\Big)
	\\
	&\times g^*\big(\omegaone\wedge\omegatwo\wedge\omegathree\wedge\omegafour\big),
	\ea
	for every vector-valued Siegel cusp form~$f\in S^k_{2,L}$, and for every~${g\in G=\SO(L\otimes\RR)}$ mapping~$z$ to~$z_0$.
	We refer to the integrals appearing as coefficients in~\eqref{eq;KMdeg2liftmoreexpl}, namely
	\be\label{eq;defintKMliftgen2}
	\intfunctshort(g)
	\coloneqq
	\int_{\Sp_4(\ZZ)\backslash\HH_2}
	\det y^{k} \big\langle f(\tau) , \Theta_{L,2}(\tau,g,\pol_{\vect{\alpha}}) \big\rangle
	\frac{dx\,dy}{\det y^3},
	\ee
	as the \emph{defining integrals} of~$\KMlift(f)$.
	
	The goal of this section is to apply the unfolding trick to such integrals to deduce the Fourier expansion of~$\intfunctshort(g)$.	
	The unfolding trick of genus~$2$ is recalled in Section~\ref{sec;generalonunftrgen2}.
	We apply it to~$\intfunctshort$ in Section~\ref{sec;theunfoldingofKMliftgen2new}, while in Section~\ref{sec;deg2KMliftgo2} we compute the Fourier expansion of such defining integrals.
	
	As we will see, the behavior of the Siegel theta functions~$\Theta_{\brK,2}(\tau,\borw,\polw{\vect{\alpha},\borw}{\hone}{\htwo})$ appearing in the unfolded integrals differs from the one of their counterparts in the genus~$1$ unfolding of the Kudla--Millson lift~\cite[Section~$5$]{zuffetti;gen1}.
	In fact, such theta functions of genus~$2$ are not always modular; see Remark~\ref{rem;nonmodgen2}.
	We will show that a suitable combinations of them may be rewritten in terms of the Jacobi Siegel theta functions introduced in Section~\ref{subsec:JacobiForms}.
	The theory we developed in the latter section will be then employed in Section~\ref{sec;finalunfvvcase} to unfold the integrals a second time, and deduce the injectivity of the lift.
	
	\subsection{The unfolding trick in genus~$\boldsymbol{2}$}\label{sec;generalonunftrgen2}
	We recall here the \emph{unfolding trick} of genus~$2$.	
	This method enables us to simplify an integral of the form
	\be\label{eq;intofFforunftrick}
	\int_{\Sp_4(\ZZ)\backslash\HH_2}H(\tau)\frac{dx\,dy}{\det y^3},
	\ee
	where~$H\colon\HH_2\to\CC$ is a~$\Sp_4(\ZZ)$-invariant function, in the case where~$H$ can be written as an absolutely convergent series of the form
	\be\label{eq;Faspoincviaf}
	H(\tau)=\sum_{M\in\kling\backslash\Sp_4(\ZZ)}\hfunct(M\cdot\tau),
	\ee
%	\\\textcolor{red}{
%	I do not like the choice of notation with~$\mathcal{H}_{\vect{\alpha}}$ of the summand, although it is definitely better than~$h_{\vect{\alpha}}$, since the latter might be confused with the index of the polynomials coming from the decomposition of~$\mathcal{P}_{\vect{\alpha}}$.
%	Other options? (The general summand above has been defined with the command "\texttt{hfunct}".)
%	}\\
	for some $\kling$-invariant map~$\hfunct$, where~$\kling$ is the Klingen parabolic subgroup of~$\Sp_4(\ZZ)$.	
	The unfolding trick aims to rewrite the integral~\eqref{eq;intofFforunftrick} as an integral of~$\hfunct$ over the unfolded domain $\kling\backslash\HH_2$, more precisely as
	\be\label{eq;unfoldingtrick}
	\int_{\Sp_4(\ZZ)\backslash\HH_2}H(\tau)\frac{dx\,dy}{\det y^3}=2\int_{\kling\backslash\HH_2}\hfunct(\tau)\frac{dx\,dy}{\det y^3}.
	\ee
	Let~$\Gamma^J=\SL_2(\ZZ)\ltimes\ZZ^2$ be the Jacobi modular group, and let~$\tau_j=x_j+iy_j$ for every~$1\le j\le 3$.
	Since we can choose
	\be\label{eq;fdomkling}
	\mathcal{S}=\left\{
	\tau\in\HH_2\,:\,\text{$(\tauone,\tautwo)\in\Gamma^J\backslash\HH_1\times\CC$, $\ythree>\ytwo^2/\yone$, $\xthree\in [0,1]$}
	\right\},
	\ee
	as fundamental domain of the action of~$\kling$ on~$\HH_2$, as explained for instance in~\cite[p.~$370$]{boda;petnorm}, the integral on the right-hand side of~\eqref{eq;unfoldingtrick} is easier to compute with respect to the one on the left-hand side.
	
	Let $\mathcal{F}$ be a fundamental domain of the action of~$\Sp_4(\ZZ)$ on~$\HH_2$.
	The equality~\eqref{eq;unfoldingtrick} can be easily checked as
	\bas
	\int_{\Sp_4(\ZZ)\backslash\HH_2} H(\tau)\frac{dx\,dy}{\det y^3}
	& =
	\int_{\mathcal{F}}\sum_{M\in\kling\backslash\Sp_4(\ZZ)}
	\hfunct(M\tau)\frac{dx\,dy}{\det y^3}
%	=
%	\sum_{M\in \kling\backslash\Sp_4(\ZZ)}  \int_{\mathcal{F}}
%	\hfunct(M\tau)\frac{dx\,dy}{\det y^3}
%	\\
%	&
	=\sum_{M\in\kling\backslash\Sp_4(\ZZ)}\int_{M\cdot\mathcal{F}}
	\hfunct(\tau)\frac{dx\,dy}{\det y^3}
	\\
	&=
	2\int_{\kling\backslash\HH_2}
	\hfunct(\tau)\frac{dx\,dy}{\det y^3},
	\eas
	where the factor $2$ arises because the classes of~$\big(\begin{smallmatrix}I_2 & 0\\ 0 & I_2\end{smallmatrix}\big)$ and~$\big(\begin{smallmatrix}-I_2 & 0\\ 0 & -I_2\end{smallmatrix}\big)$ in~$\kling\backslash\Sp_4(\ZZ)$ are different.
	
	\subsection{The unfolding of~$\boldsymbol{\KMlift}$}\label{sec;theunfoldingofKMliftgen2new}
	Suppose that~$L$ splits off a hyperbolic plane.
	As usual, we choose~$\genU$,~$\genUU$ and~$\brK$ as in~\eqref{eq;choiceofuu'gen2}.
	
	To unfold the defining integrals~\eqref{eq;defintKMliftgen2} of the Kudla--Millson lift by means of the procedure illustrated in Section~\ref{sec;generalonunftrgen2}, we construct~$\kling$-invariant functions~$\hfunct_{\vect{\alpha}}(\tau,g)$ such that the integrand of~$\intfunctshort$ may be written as
	\ba\label{eq;splitsumC12habcd}
	\det y^{k}\big\langle f(\tau) , \Theta_{L,2}(\tau,g,\pol_{\vect{\alpha}}) \big\rangle
	= &{}
	\frac{\det y^{k}}{2 u_{z^\perp}^2}
	\big\langle f(\tau) , \Theta_{\brK,2}(\tau,\borw,\polw{\vect{\alpha},\borw}{0}{0})\big\rangle
	\\
	& + \sum_{M\in \kling\backslash\Sp_4(\ZZ)}
	\hfunct_{\vect{\alpha}}(M \tau,g),
	\ea
	for every~$g\in G$ and~$z\in\Gr(L)$ such that~$g$ maps~$z$ to the base point~$z_0$.
	The first summand on the right-hand side of~\eqref{eq;splitsumC12habcd} arises from the error term appearing on the right-hand side of~\eqref{eq;corgenborthmsplitthetaLklin}.
	As we will see, the integral of such a summand is the constant term of the Fourier expansion of~$\intfunctshort$.
	
	Recall that we identify values of modular forms on isomorphic group algebras, and that if~$A$ is a matrix, we denote by~$[A]_{i,j}$ its~$(i,j)$-entry.
	
	\begin{thm}\label{thm;unfongen2}
	Suppose that~$L$ splits off a hyperbolic plane and let~$\genU$, $\genUU$ and~$\brK$ as in~\eqref{eq;choiceofuu'gen2}.
	The function~$\hfunct_{\vect{\alpha}}$ in~\eqref{eq;splitsumC12habcd} may be chosen as
	\ba\label{eq;auxhfunctgen2}
	\hfunct_{\vect{\alpha}}(\tau,g)
	=&{}
	\frac{\det y^{k}}{2\genU_{z^\perp}^2}\sum_{r\ge 1}
	\exp\Big(
	-\frac{\pi r^2}{2\genU_{z^\perp}^2}[y^{-1}]_{2,2}
	\Big)
	\sum_{\hone,\htwo\textcolor{\mycolor}{=0}}^{\textcolor{\mycolor}{2}}
	\left(\frac{r}{2i}\right)^{\hone+\htwo}[y^{-1}]^{\hone}_{2,1} \cdot [y^{-1}]^{\htwo}_{2,2}
	\\
	&\times
	\big\langle f(\tau) , \Theta_{\brK,2}\big(\tau,(0,r\mu),0,\borw,\polw{\vect{\alpha},\borw}{\hone}{\htwo}\big)\big\rangle.
	\ea
	\end{thm}
	To prove Theorem~\ref{thm;unfongen2} we need to introduce the following auxiliary functions.	
	\begin{defi}\label{def;gen2auxfunctvaphih+}
	We define the auxiliary function~$\auxfunplus_{\htot}$ as
	\bas
	\auxfunplus_{\htot}(\tau,\vect{\boralpha},\vect{\borbeta},\borw)\coloneqq \sum_{\substack{\hone,\htwo\\ \hone+\htwo=\htot}}
	[\tau y^{-1}]^{\hone}_{2,1} \cdot [\tau y^{-1}]^{\htwo}_{2,2}\cdot
	\big\langle
	f(\tau) , \Theta_{\brK,2}(\tau,\vect{\boralpha}, \vect{\borbeta},\borw,\polw{\vect{\alpha},\borw}{\hone}{\htwo})\big\rangle,
	\eas
	for every $\tau=x+iy\in\HH_2$, $\vect{\boralpha},\vect{\borbeta}\in(\brK\otimes\RR)^2$, $0\le \htot\le 2$, and $g\in G$.
	\end{defi}
	Since if $y=\left(\begin{smallmatrix}
	\yone & \ytwo\\
	\ytwo & \ythree
	\end{smallmatrix}\right)$, then $y^{-1}=\frac{1}{\det y}\left(\begin{smallmatrix}
	\ythree & -\ytwo\\
	-\ytwo & \yone
	\end{smallmatrix}\right)$ and
	\bes
	\tau y^{-1}=
%	\frac{1}{\det y}\left(\begin{smallmatrix}
%	\tauone & \tautwo\\
%	\tautwo & \tauthree
%	\end{smallmatrix}\right)\cdot \left(\begin{smallmatrix}
%	\ythree & -\ytwo\\
%	-\ytwo & \yone
%	\end{smallmatrix}\right)=
	\frac{1}{\det y}\begin{pmatrix}
	\ythree \tauone - \ytwo \tautwo & \yone \tautwo - \ytwo \tauone\\
	\ythree \tautwo - \ytwo \tauthree & \yone \tauthree - \ytwo \tautwo
	\end{pmatrix},
	\ees
	we deduce that
	\be\label{eq;formtimt-12122}
	[\tau y^{-1}]_{2,1}=\frac{\ythree \tautwo - \ytwo \tauthree}{\det y}
	\qquad
	\text{and}
	\qquad
	[\tau y^{-1}]_{2,2}=\frac{\yone \tauthree - \ytwo \tautwo}{\det y}.
	\ee
	This implies that we may rewrite~$\auxfunplus_{\htot}$ as
	\ba\label{varphigenforgen2finform}
	\auxfunplus_{\htot}(\tau,\vect{\boralpha},\vect{\borbeta},\borw)
	=
	\sum_{\substack{\hone,\htwo\\ \hone+\htwo=\htot}}
	&
	\det y^{-(\hone+\htwo)}
	(\ythree \tautwo - \ytwo \tauthree)^{\hone}
	\cdot
	(\yone \tauthree - \ytwo \tautwo)^{\htwo}
	\\
	&\times\big\langle
	f(\tau) , \Theta_{\brK,2}(\tau,\vect{\boralpha}, \vect{\borbeta},\borw,\polw{\vect{\alpha},\borw}{\hone}{\htwo})\big\rangle.
	\ea
	For future use, we compute here also~$\big[\tau y^{-1} \bar{\tau}\big]_{2,1}$ and~$\big[\tau y^{-1} \bar{\tau}\big]_{2,2}$.
	Since
	\bas
	\tau y^{-1} \bar{\tau} =\frac{1}{\det y}\begin{pmatrix}
	\ythree |\tauone|^2 - \ytwo \tautwo\overline{\tauone}
	+
	\yone |\tautwo|^2 - \ytwo \tauone\overline{\tautwo}
	&
	\ythree \tauone\overline{\tautwo} - \ytwo |\tautwo|^2
	+
	\yone \tautwo\overline{\tauthree} - \ytwo \tauone\overline{\tauthree}
	\\
	\ythree \tautwo\overline{\tauone} - \ytwo \tauthree\overline{\tauone}
	+
	\yone \tauthree\overline{\tautwo} - \ytwo |\tautwo|^2
	&
	\ythree |\tautwo|^2 - \ytwo \tauthree\overline{\tautwo}
	+
	\yone |\tauthree|^2 - \ytwo \tautwo\overline{\tauthree}
	\end{pmatrix},
	\eas
	we have
	\ba\label{eq;formtimt-1bart21}
	\big[\tau y^{-1} \bar{\tau}\big]_{2,1}=\frac{\ythree \tautwo\overline{\tauone} - \ytwo \tauthree\overline{\tauone}
	+
	\yone \tauthree\overline{\tautwo} - \ytwo |\tautwo|^2}{\det y},\\
	\big[\tau y^{-1} \bar{\tau}\big]_{2,2}=
	\frac{\ythree |\tautwo|^2 - \ytwo \tauthree\overline{\tautwo}
	+
	\yone |\tauthree|^2 - \ytwo \tautwo\overline{\tauthree}}{\det y}.
	\ea
	\begin{thm}\label{thm;transfThetaL2genlor}
	\textcolor{\newcolor}{The auxiliary function $\auxfunplus_{\htot}$ is equal to}
	\ba\label{eq;transfThetaL2genlor}
	\auxfunplus_{\htot}(\tau, \vect{\boralpha}, \vect{\borbeta},\borw)=&{}\,
	|\det \tau|^{-2k}\sum_{\substack{\hone,\htwo\\ \hone+\htwo=\htot}}[\tau y	^{-1}\bar{\tau}]_{2,1}^{\hone}\cdot [\tau y^{-1}\bar{\tau}]_{2,2}^{\htwo}
	\\
	&\times\big\langle f(-\tau^{-1}) , \Theta_{\brK,2}(-\tau^{-1},-\vect{\borbeta},\vect{\boralpha},\borw,\polw{\vect{\alpha},\borw}{\hone}{\htwo})\big\rangle.
	\ea
	\end{thm}
	\begin{rem}\label{rem;nonmodgen2}
	Along the proof of Theorem~\ref{thm;transfThetaL2genlor}, we will show that the Siegel theta function~$\Theta_{\brK,2}(\tau,\vect{\boralpha}, \vect{\borbeta},\borw,\polw{\vect{\alpha},\borw}{\hone}{\htwo})$ is \emph{non-modular} whenever either~$\hone$ or~$\htwo$ differs from zero.
	This is a consequence of their behavior with respect to the action of the symplectic matrix~${S=\big(\begin{smallmatrix}
	0 & -I_2\\ I_2 & 0
	\end{smallmatrix}\big)}$ on~$\HH_2$, which is illustrated in~\eqref{eq;ThetaLor2withFtransf02}, \eqref{eq;ThetaLor2withFtransf11} and~\eqref{eq;ThetaLor2withFtransf20}.
	\end{rem}
	\begin{proof}[Proof of Theorem~\ref{thm;transfThetaL2genlor}]
	\textbf{Case~$\boldsymbol{\htot=2}$:}
%	We prove~\eqref{eq;transfThetaL2genlor} applying Lemma~\ref{lemma;somepropFtr} and the Poisson summation formula for unimodular lattices, namely
%	\be\label{eq;Poissumformlor}
%	\sum_{\vect{\lambda}\in \brK^2}f(\vect{\lambda})=\sum_{\vect{\lambda}\in \brK^2}\widehat{f}(\vect{\lambda}),\qquad\text{for every $f\in L^1(V^2)$}.
%	\ee
	We begin by rewriting the~$\gendisc$-component of the Siegel theta function associated to~$\polw{\vect{\alpha},\borw}{0}{2}$, for~$\gendisc\in\disc{\brK}^2$, as
	\be\label{eq;proofgvnnlor}
	\Theta_{\brK,2}^\gendisc(\tau,\vect{\boralpha},\vect{\borbeta},\borw,\polw{\vect{\alpha},\borw}{0}{2})
	=
	\sqrt{\det y} e\big(\vect{\borbeta},\vect{\boralpha}/2\big)
	\sum_{\vect{\lambda}\in \brK^2}
	F_{\tau,\vect{\boralpha},\vect{\borbeta},\borw}^{\gendisc}(\vect{\lambda}),
	\ee
	where $F_{\tau,\vect{\boralpha},\vect{\borbeta},\borw}^{\gendisc}(\vect{\lambda})
	\coloneqq
	f_{\tau,\borw,0,2}(\vect{\lambda} + \gendisc +\vect{\borbeta})
	\cdot
	e\big(-\vect{\lambda}+ \gendisc + \vect{\borbeta},\vect{\boralpha}\big)$, and $f_{\tau,\borw,0,2}$ is the function introduced in Lemma~\ref{lemma;someFtransfofabcdh12lor}.
	The idea is to apply the Poisson summation formula to the right-hand side of~\eqref{eq;proofgvnnlor}.
	To do so, we compute the Fourier transform of $F_{\tau,\vect{\boralpha},\vect{\borbeta},\borw}^\gendisc$ using the properties illustrated in Lemma~\ref{lemma;somepropFtr} as
	\bas
	\widehat{F_{\tau,\vect{\boralpha},\vect{\borbeta},\borw}^\gendisc}(\vect{\lambda})=\widehat{f_{\tau,\borw,0,2}}(\vect{\lambda}-\vect{\boralpha})\cdot e\big(-(\vect{\lambda},\vect{\borbeta}) - (\vect{\lambda},\gendisc)\big).
	\eas
	We now apply the Poisson summation formula and, using the formula of~$\widehat{f_{\tau,\borw,0,2}}$ provided by Lemma~\ref{lemma;someFtransfofabcdh12lor}, deduce that
	\begin{align*}
	\,&\Theta_{\brK,2}^\gendisc(\tau,\vect{\boralpha},\vect{\borbeta},\borw,\polw{\vect{\alpha},\borw}{0}{2})
	\\
	&=
	\sqrt{\det y} \frac{e\big((\vect{\borbeta},\vect{\boralpha}/2)\big)}{|\disc{\brK}|}
	\sum_{\vect{\lambda}\in \brK'^2}
	\widehat{F_{\tau,\vect{\boralpha},\vect{\borbeta},\borw}^\gendisc}(\vect{\lambda})
	\\
	&=
	\frac{\sqrt{\det y}}{|\disc{\brK}|}
	\sum_{\vect{\lambda}\in \brK'^2}
	\widehat{f_{\tau,\borw,0,2}}(\vect{\lambda}-\vect{\vect{\boralpha}})\cdot e\big( -(\vect{\lambda}-\vect{\boralpha}/2,\vect{\borbeta}) - (\vect{\lambda},\gendisc)\big)
	\\
	&=
	\frac{i^{2-b}}{|\disc{\brK}|} \sqrt{\det y}
	\det(\tau)^{-(b-1)/2-2}\det(\bar{\tau})^{-1/2}
	\sum_{\vect{\lambda}\in\brK'^2}
	\Big(\tauthree^2\cdot f_{-\tau^{-1},\borw,0,2}(\vect{\lambda} - \vect{\boralpha})
	\\
	&\quad+
	\tautwo\tauthree\cdot f_{-\tau^{-1},\borw,1,1}(\vect{\lambda}-\vect{\boralpha}) +
	\tautwo^2\cdot f_{-\tau^{-1},\borw,2,0}(\vect{\lambda} - \vect{\boralpha})\Big)
	e\big( -(\vect{\lambda}-\vect{\boralpha}/2,\vect{\borbeta}) - (\vect{\lambda},\gendisc)\big).
	\end{align*}
	In the sum over $\brK^2$ above, we replace~$\vect{\lambda}$ with~$-\vect{\lambda}$.
	Since all polynomials~$\polw{\vect{\alpha},\borw}{\hone}{\htwo}$ satisfying~${\hone+\htwo=2}$ are such that~$\polw{\vect{\alpha},\borw}{\hone}{\htwo}\big(g(-\vect{\xi})\big)=\polw{\vect{\alpha},\borw}{\hone}{\htwo}\big(g(\vect{\xi})\big)$ by~\eqref{eq;nonveryhompropgen2}, and since~$e(-(\vect{\lambda},\gendisc))=e(-(\gendisc',\gendisc))$ for every~$\vect{\lambda}\in\gendisc' + \brK^2$, we deduce that
	\bas
	\Theta_{\brK,2}^\gendisc  (\tau,\vect{\boralpha},\vect{\borbeta},\borw,\polw{\vect{\alpha},\borw}{0}{2})
	&=
	\frac{i^{2-b}}{|\disc{\brK}|}\det(\tau)^{-b/2-2}
	\\
	&\quad\times\Big(
	\tauthree^2 \sum_{\gendisc'\in\disc{\brK}^2}
	e\big(
	-(\gendisc',\gendisc)
	\big)
	\Theta_{\brK,2}^{\gendisc'}(-\tau^{-1},-\vect{\borbeta},\vect{\boralpha},\borw,\polw{\vect{\alpha},\borw}{0}{2})
	\\
	&\quad+
	\tautwo\tauthree
	\sum_{\gendisc'\in\disc{\brK}^2}
	e\big(
	-(\gendisc',\gendisc)
	\big)
	\Theta_{\brK,2}^{\gendisc'}(-\tau^{-1},-\vect{\borbeta},\vect{\boralpha},\borw,\polw{\vect{\alpha},\borw}{1}{1})
	\\
	&\quad+
	\tautwo^2
	\sum_{\gendisc'\in\disc{\brK}^2}
	e\big(
	-(\gendisc',\gendisc)
	\big)
	\Theta_{\brK,2}^{\gendisc'}(-\tau^{-1},-\vect{\borbeta},\vect{\boralpha},\borw,\polw{\vect{\alpha},\borw}{2}{0})
	\Big)
	.
	\eas
	This implies that
	\ba\label{eq;ThetaLor2withFtransf02}
	\Theta_{\brK,2} (\tau,\vect{\boralpha},\vect{\borbeta},\borw,\polw{\vect{\alpha},\borw}{0}{2})
	&=
	\det(\tau)^{-b/2-2}
	\\
	&\quad\times\Big(
	\tauthree^2 \cdot\weil{\brK}(S)
	\Theta_{\brK,2}(-\tau^{-1},-\vect{\borbeta},\vect{\boralpha},\borw,\polw{\vect{\alpha},\borw}{0}{2})
	\\
	&\quad+
	\tautwo\tauthree
	\cdot\weil{\brK}(S)
	\Theta_{\brK,2}(-\tau^{-1},-\vect{\borbeta},\vect{\boralpha},\borw,\polw{\vect{\alpha},\borw}{1}{1})
	\\
	&\quad+
	\tautwo^2
	\cdot\weil{\brK}(S)
	\Theta_{\brK,2}(-\tau^{-1},-\vect{\borbeta},\vect{\boralpha},\borw,\polw{\vect{\alpha},\borw}{2}{0})
	\Big),
	\ea
	where~$S=\big(\begin{smallmatrix}
	0 & -I_2\\
	I_2 & 0
	\end{smallmatrix}\big)$.
	With the same procedure, one can show also that
	\ba\label{eq;ThetaLor2withFtransf11}
	\Theta_{\brK,2}(\tau,\vect{\boralpha},\vect{\borbeta},\borw,\polw{\vect{\alpha},\borw}{1}{1})
	&=
	\det(\tau)^{-b/2-2}
	\\
	&\quad\times
	\Big(
	2\tautwo\tauthree\cdot \weil{\brK}(S)\Theta_{\brK,2}(-\tau^{-1},-\vect{\borbeta},\vect{\boralpha},\borw,\polw{\vect{\alpha},\borw}{0}{2})\\
	&\quad+
	(\tauone\tauthree + \tautwo^2)\cdot
	\weil{\brK}(S)\Theta_{\brK,2}(-\tau^{-1},-\vect{\borbeta},\vect{\boralpha},\borw,\polw{\vect{\alpha},\borw}{1}{1})
	\\
	&\quad+
	2\tauone\tautwo\cdot 
	\weil{\brK}(S)\Theta_{\brK,2}(-\tau^{-1},-\vect{\borbeta},\vect{\boralpha},\borw,\polw{\vect{\alpha},\borw}{2}{0})\Big).
	\ea
	and that
	\ba\label{eq;ThetaLor2withFtransf20}
	\Theta_{\brK,2}(\tau,\vect{\boralpha},\vect{\borbeta},\borw,\polw{\vect{\alpha},\borw}{2}{0})
	&=
	\det(\tau)^{-b/2-2}
	\\
	&\quad\times
	\Big(
	\tautwo^2\cdot \weil{\brK}(S)\Theta_{\brK,2}(-\tau^{-1},-\vect{\borbeta},\vect{\boralpha},\borw,\polw{\vect{\alpha},\borw}{0}{2})\\
	&\quad+
	\tautwo\tauone\cdot\weil{\brK}(S)\Theta_{\brK,2}(-\tau^{-1},-\vect{\borbeta},\vect{\boralpha},\borw,\polw{\vect{\alpha},\borw}{1}{1})
	\\
	&\quad+
	\tauone^2\cdot \weil{\brK}(S)\Theta_{\brK,2}(-\tau^{-1},-\vect{\borbeta},\vect{\boralpha},\borw,\polw{\vect{\alpha},\borw}{2}{0})\Big).
	\ea
	The isomorphic projection~$p\colon\disc{L}^2\to\disc{\brK}^2$ induces an isomorphism~$\CC[\disc{L}^2]\to\CC[\disc{\brK}^2]$.
	We use the latter to identify values of theta functions on isomorphic group algebras,	
	replacing~\eqref{eq;ThetaLor2withFtransf02}, \eqref{eq;ThetaLor2withFtransf11} and~\eqref{eq;ThetaLor2withFtransf20} in~\eqref{varphigenforgen2finform}.
	Since~$f(\tau)=\det \tau^{-k}\cdot\weil{\brK}(S)f(-\tau^{-1})$, we may rewrite~$\auxfunplus_{2}(\tau,\vect{\boralpha},\vect{\borbeta},\borw)$ as 
	\begin{align}\label{eq;rewritingeverlongfgen2bis}
	&\auxfunplus_{2}(\tau,\vect{\boralpha},\vect{\borbeta},\borw)
	=
	\det y^{-2} \cdot |\det\tau|^{-2k}\\
	&\times\Big[\Big( \nonumber
	(\yone \tauthree - \ytwo \tautwo)^2\cdot\overline{\tauthree}^2
	+2
	(\ythree \tautwo - \ytwo \tauthree)(\yone \tauthree - \ytwo \tautwo)\overline{\tautwo}\overline{\tauthree}+(\ythree \tautwo - \ytwo \tauthree)^2 \overline{\tautwo}^2
	\Big)\\
	&\times\big\langle f(-\tau^{-1} , \Theta_{\brK,2}(-\tau^{-1},-\vect{\borbeta},\vect{\boralpha},\borw,\polw{\vect{\alpha},\borw}{0}{2})\big\rangle \nonumber
	\\
	&+\Big( \nonumber
	(\yone \tauthree - \ytwo \tautwo)^2\overline{\tautwo}\overline{\tauthree}
	+(\ythree \tautwo - \ytwo \tauthree)(\yone \tauthree - \ytwo \tautwo)(\overline{\tauone}\overline{\tauthree}+\overline{\tautwo}^2)+(\ythree \tautwo - \ytwo \tauthree)^2\overline{\tautwo}\overline{\tauone}
	\Big)
	\\
	&\times\big\langle f(-\tau^{-1} , \Theta_{\brK,2}(-\tau^{-1},-\vect{\borbeta},\vect{\boralpha},\borw,\polw{\vect{\alpha},\borw}{1}{1})\big\rangle \nonumber
	\\
	&+\Big( \nonumber
	(\yone \tauthree - \ytwo \tautwo)^2\overline{\tautwo}^2
	+2(\ythree \tautwo - \ytwo \tauthree)(\yone \tauthree - \ytwo \tautwo)\overline{\tauone}\overline{\tautwo}+(\ythree \tautwo - \ytwo \tauthree)^2\overline{\tauone}^2
	\Big)
	\\
	&\times\big\langle f(-\tau^{-1} , \Theta_{\brK,2}(-\tau^{-1},-\vect{\borbeta},\vect{\boralpha},\borw,\polw{\vect{\alpha},\borw}{2}{0})\big\rangle\Big]. \nonumber
	\end{align}
	The factor multiplying
	\[
	\det y^{-2}\cdot|\det\tau|^{-2k}\big\langle f(-\tau^{-1} , \Theta_{\brK,2}(-\tau^{-1},-\vect{\borbeta},\vect{\boralpha},\borw,\polw{\vect{\alpha},\borw}{0}{2})\big\rangle
	\]
	\textcolor{\mycolor}{in the first summand} on the right-hand side of~\eqref{eq;rewritingeverlongfgen2bis} equals
	\ba\label{eq;factinfrofTlor202}
	\yone^2|\tauthree|^4 + 2\yone \ythree |\tautwo|^2 |\tauthree|^2 + 2 \ytwo^2 |\tautwo|^2 |\tauthree|^2 + \ythree^2|\tautwo|^4
	+ 2\ytwo^2 \Re(\tautwo^2\overline{\tauthree}^2) -\\
	- 4\yone\ytwo|\tauthree|^2\Re(\tautwo\overline{\tauthree}) -4\ytwo\ythree|\tautwo|^2\Re(\tautwo\overline{\tauthree}).
	\ea
	We verify that it equals~$[\tau y^{-1}\bar{\tau}]_{2,2}^2\cdot\det y^2$.
	We compute the latter via~\eqref{eq;formtimt-1bart21} as
	\ba\label{eq;compty-1bartpower2}
	\yone^2|\tauthree|^4 + 2\yone \ythree |\tautwo|^2 |\tauthree|^2 + \ythree^2|\tautwo|^4
	- 4\yone \ytwo|\tauthree|^2\Re(\tautwo\overline{\tauthree}) -\\
	-4\ytwo \ythree|\tautwo|^2\Re(\tautwo\overline{\tauthree}) + 4\ytwo^2 \big(\Re(\tauthree\overline{\tautwo})\big)^2.
	\ea
	Recall that if $a,b\in\CC$, then $2\big(\Re(a b)\big)^2 = \Re(a^2b^2) + |a|^2 |b|^2$.
	This implies that
	\bes
	4\ytwo^2 \Re(\tauthree\overline{\tautwo}) = 2 \ytwo^2|\tautwo|^2|\tauthree|^2 + 2\ytwo^2\Re(\tautwo^2\overline{\tauthree}^2),
	\ees
	and hence that~\eqref{eq;compty-1bartpower2} equals~\eqref{eq;factinfrofTlor202}.
	
	We skip the computation of the factors in front of the remaining two scalar products computed on the right-hand side of~\eqref{eq;rewritingeverlongfgen2bis},
	%, as well as the check that such quantities are equal to respectively~$[\tau y^{-1}\bar{\tau}]_{2,1}\cdot [\tau y^{-1}\bar{\tau}]_{2,2}\cdot\det y^2$ and~$[\tau y^{-1}\bar{\tau}]_{2,1}^{2}\cdot \det y^2$.
	since the procedure is analogous to the previous one.
	
	\textbf{Case~$\boldsymbol{\htot=1}$:} It is analogous to the case of~$\htot=2$.
	For this reason, we skip it.
	
	\textbf{Case~$\boldsymbol{\htot=0}$:}
	It is enough to check that
	\bas
	\Theta_{\brK,2}(\tau,\vect{\boralpha}, \vect{\borbeta},\borw,\polw{\vect{\alpha},\borw}{0}{0})&=\det \tau^{-b/2-2}
%	\\
%	&\quad\times
	\weil{\brK}(S)\Theta_{\brK,2}(-\tau^{-1},-\vect{\borbeta},\vect{\boralpha},\borw,\polw{\vect{\alpha},\borw}{0}{0}).
	\eas
	This can be done using the Poisson summation formula, as we did above for the case~$\htot=2$, together with Lemma~\ref{lemma;someFtransfofabcdh12lor}.
	In fact, since the polynomial~$\polw{\vect{\alpha},\borw}{0}{0}$ is very homogeneous of degree~$(2,0)$ by Lemma~\ref{lemma;gen2funnypol}, such a theta function is \emph{modular}; cf.\ Remark~\ref{rem;trformsiegthv2}.
	\end{proof}
	We are now ready to prove~Theorem~\ref{thm;unfongen2}.
	\begin{proof}[Proof of Theorem~\ref{thm;unfongen2}]
	By Corollary~\ref{cor;genborthmsplitthetaLklin} it is enough to prove that
	\ba\label{eq;ideagen2unftr}
	\hfunct_{\vect{\alpha}}&{}(M\cdot\tau,g)
	=
	\frac{\det y^{k}}{2 u_{z^\perp}^2}
	\sum_{r\ge 1}
	\sum_{\hone,\htwo\textcolor{\mycolor}{=0}}^{\textcolor{\mycolor}{2}}
	\Big(\frac{r}{2i}\Big)^{\hone+\htwo}\big[\big(c\tau + d\big)y^{-1}\big]_1^{\hone}\big[\big(c\tau+d\big)y^{-1}\big]_2^{\htwo}
	\\
	&\times \exp\Big(-\frac{\pi r^2}{2 u_{z^\perp}^2}\trace \big(c\tau + d\big)^t\big(c\bar{\tau}+d\big)y^{-1}\Big)
	\big\langle f(\tau) , 
	\Theta_{\brK,2}(\tau,r\mu d,-r\mu c,\borw,\polw{\vect{\alpha},\borw}{\hone}{\htwo})\big\rangle,
	\ea
	for every $M=\big(\begin{smallmatrix}* & *\\ C & D\end{smallmatrix}\big)\in\kling\backslash\Sp_4(\ZZ)$, where $c$ (resp.~$d$) is the last row of $C$ (resp.~$D$).
	
	Since~$\polw{\vect{\alpha},\borw}{\hone}{\htwo}$ is very homogeneous only when it is zero or $\hone=\htwo=0$ by Lemma~\ref{lemma;gen2funnypol}, if we compute~$\polw{\vect{\alpha},\borw}{\hone}{\htwo}\big(\borw(\vect{\genvec}\cdot N)\big)$ for some $N\in\CC^{2\times 2}$, in general we do not obtain a multiple of $\polw{\vect{\alpha},\borw}{\hone}{\htwo}\big(\borw(\vect{\genvec})\big)$.
	In fact, the result is a linear combination of polynomials~$\polw{\vect{\alpha},\borw}{\hone'}{\htwo'}\big(\borw(\vect{\genvec})\big)$ such that~$\hone'+\htwo'=\hone+\htwo$, where the linear coefficients depend on the entries of the matrix~$N$; see Lemma~\ref{lemma;nonhomogofderpol}.
	This remark leads us to gather all summands of~$\hfunct_{\vect{\alpha}}$ appearing in~\eqref{eq;auxhfunctgen2} that have the same sum $h_1+h_2$, defining an auxiliary function~$\eta_{\htot}$ as
	\bas
	\eta_{\htot}(\tau,\borw)\textcolor{\newcolor}{\coloneqq}
	\sum_{\substack{\hone,\htwo\\ \hone+\htwo=\htot}}
	[y^{-1}]^{\hone}_{2,1} \cdot [y^{-1}]^{\htwo}_{2,2}
	\cdot
	\big\langle f(\tau),
	\Theta_{\brK,2}(\tau,(0,r\mu),0,\borw,\polw{\vect{\alpha},\borw}{\hone}{\htwo})\big\rangle.
	\eas
	In this way, we may rewrite~$\hfunct_{\vect{\alpha}}$ as
	\bas
	\hfunct_{\vect{\alpha}}(\tau,g)=\frac{\det y^{k}}{2\genU_{z^\perp}^2}\sum_{r\ge 1} \exp\Big(
	-\frac{\pi r^2}{2\genU_{z^\perp}^2}[y^{-1}]_{2,2}
	\Big) \sum_{\htot}
	\left(\frac{r}{2i}\right)^{\htot} \eta_{\htot}(\tau,\borw).
	\eas
	Therefore, we have
	\begin{align*}
	\hfunct_{\vect{\alpha}}(M\cdot\tau,g)
	&=
	\frac{\det\big(\Im(M\cdot\tau)\big)^{k}}{2\genU_{z^\perp}^2}
	\sum_{r\ge1}\exp\Big(-\frac{\pi r^2}{2\genU_{z^\perp}^2}\trace(c\tau+d)^t (c\bar{\tau}+d)y^{-1}\Big)
	\\
	&\quad\times
	\sum_{\htot}\left(\frac{r}{2i}\right)^{\htot}\eta_{\htot}(M\cdot\tau,\borw)
	\\
	&=
	\frac{\det y^{k}}{2u_{z^\perp}^2}\sum_{r\ge1}\exp\Big(-\frac{\pi r^2}{2\genU_{z^\perp}^2}\trace(c\tau+d)^t (c\bar{\tau}+d)y^{-1}\Big)
	\\
	&\quad\times\sum_{\htot}\left(\frac{r}{2i}\right)^{\htot}
	|\det(C\tau+D)|^{-2k}
	\eta_{\htot}(M\cdot \tau,\borw).
	\end{align*}
	We prove~\eqref{eq;ideagen2unftr} by showing that
	\ba\label{eq;proofunfgen2possvarphi}
	\eta_{\htot}(M\cdot\tau,\borw)
	=&\,{}|\det(C\tau+D)|^{2k}
	\sum_{\substack{\hone,\htwo\\ \hone+\htwo=\htot}}[(c\tau+d)y^{-1}]_1^{\hone}\cdot[(c\tau+d)y^{-1}]_2^{\htwo}
	\\
	&\times
	\big\langle
	f(\tau) , \Theta_{\brK,2}(\tau,r\mu d,-r\mu c,\borw,\polw{\vect{\alpha},\borw}{\hone}{\htwo})\big\rangle,
	\ea
	for every $0\le \htot\le 2$.
	Since~$\Sp_4(\ZZ)$ is generated by the matrices of the form
	\bes
	T_B=\big(\begin{smallmatrix}
	I_2 & B\\ 0 & I_2
	\end{smallmatrix}\big),\text{ where $B=B^t\in\ZZ^{2\times 2}$,}
	\qquad\text{and}\qquad
	S=\big(\begin{smallmatrix}
	0 & -I_2\\ I_2 & 0
	\end{smallmatrix}\big),
	\ees
	it is enough to check~\eqref{eq;proofunfgen2possvarphi} for such generators.
	For~$T_B$, this is implied by the trivial identity
	\bes
	\Theta_{\brK,2}(\tau+B,\vect{\boralpha}+\vect{\borbeta} B,\vect{\borbeta},\borw,\polw{\vect{\alpha},\borw}{\hone}{\htwo})=\Theta_{\brK,2}(\tau,\vect{\boralpha},\vect{\borbeta},\borw,\polw{\vect{\alpha},\borw}{\hone}{\htwo}),
	\ees
	which holds for every~$\vect{\boralpha},\vect{\borbeta}\in(\brK\otimes\RR)^2$, even if~$\polw{\vect{\alpha},\borw}{\hone}{\htwo}$ is non-very homogeneous.
	
	We now show~\eqref{eq;proofunfgen2possvarphi} for~$M=S$.
	That equality simplifies to
	\ba\label{eq;wanttoprovegen2forunf}
	\eta_{\htot}(-\tau^{-1},\borw)=|\det\tau|^{2k}\auxfunplus_{\htot}(\tau,0,-(0,r\mu),\borw),
	\ea
	where~$\auxfunplus_{\htot}$ is the auxiliary function of Definition~\ref{def;gen2auxfunctvaphih+}.
	Since the identity
	\[
	(C\bar{\tau}+D)^t\Im(M\tau)(C\tau+D)=\Im(\tau),
	\]
	read with~${M=S}$, may be rewritten as~$\Im(-\tau^{-1})^{-1}=\tau y^{-1}\bar{\tau}$, we may compute~$\eta_{\htot}(-\tau^{-1},\borw)$ as
	\begin{align*}
	\eta_{\htot}(-\tau^{-1},\borw)
	=&{}
	\sum_{\substack{\hone,\htwo\\ \hone+\htwo=\htot}}[\tau y^{-1}\bar{\tau}]_{2,1}^{\hone}\cdot [\tau y^{-1}\bar{\tau}]_{2,2}^{\htwo}
	\\
	&\times
	\big\langle f(-\tau^{-1}) , \Theta_{\brK,2}(-\tau^{-1},(0,r\mu),0,\borw,\polw{\vect{\alpha},\borw}{\hone}{\htwo})\big\rangle.
	\end{align*}
	Hence, the identity we want to prove, namely~\eqref{eq;wanttoprovegen2forunf}, can be now rewritten as
	\bas
	\sum_{\substack{\hone,\htwo\\ \hone+\htwo=\htot}}[\tau y	^{-1}\bar{\tau}]_{2,1}^{\hone}\cdot [\tau y^{-1}\bar{\tau}]_{2,2}^{\htwo}\cdot\big\langle f(-\tau^{-1}) , \Theta_{\brK,2}(-\tau^{-1},(0,r\mu),0,\borw,\polw{\vect{\alpha},\borw}{\hone}{\htwo})\big\rangle
	\\
	=|\det \tau|^{2k} \auxfunplus_{\htot}(\tau,0,-(0,r\mu),\borw).
	\eas
	Theorem~\ref{thm;transfThetaL2genlor} concludes the proof.
	\end{proof}
	
	By Theorem~\ref{thm;unfongen2} we may then unfold the defining integrals~$\intfunctshort$ of the genus~$2$ Kudla--Millson lift as
	\ba\label{eq;splitsumC12habcdunfolded}
	\intfunctshort(g)
	=&{}
	\int_{\Sp_4(\ZZ)\backslash\HH_2}\frac{\det y^{k}}{2 u_{z^\perp}^2} \big\langle f(\tau) , \Theta_{\brK,2}(\tau,\borw,\polw{\vect{\alpha},\borw}{0}{0})\big\rangle
	\,
	\frac{dx\,dy}{\det y^3}
	\\
	& +
	2\int_{\kling\backslash\HH_2} \hfunct_{\vect{\alpha}}(\tau,g)\,\frac{dx\,dy}{\det y^3}.
	\ea
	\subsection{Fourier series of unfolded integrals}\label{sec;deg2KMliftgo2}
	In this section we compute the Fourier expansion of the defining integrals~$\intfunctshort$ of the Kudla--Millson lift~$\KMlift$, for every tuple of indices~${\vect{\alpha}=(\alphaone,\alphatwo,\betaone,\betatwo)}$ with~$\alphaone<\betaone$ and~$\alphatwo<\betatwo$.
	
	By Theorem~\ref{thm;unfongen2}, using the fundamental domain~\eqref{eq;fdomkling} of~$\HH_2$ with respect to the action of the Klingen parabolic subgroup~$\kling$, we may rewrite the last term of the right-hand side of~\eqref{eq;splitsumC12habcdunfolded} as
	\ba\label{eq;rewritunfoldintgen2habcdgo2}
	2&\int_{\kling\backslash\HH_2} \hfunct_{\vect{\alpha}}(\tau,g)\frac{dx\,dy}{\det y^3}
	\\
	&=\int_{(\tauone,\tautwo)\in\Gamma^J\backslash\HH\times\CC}
	\int_{\ythree=\ytwo^2/\yone}^\infty
	\int_{\xthree=0}^1
	\frac{\det y^{k}}{\genU_{z^\perp}^2}\sum_{r\ge 1} \sum_{\hone,\htwo}
	\left(\frac{r}{2i}\right)^{\hone+\htwo}  [y^{-1}]^{\hone}_{2,1} \cdot [y^{-1}]^{\htwo}_{2,2}
	\\
	&\quad\times\exp\Big(
	-\frac{\pi r^2}{2\genU_{z^\perp}^2}[y^{-1}]_{2,2}
	\Big)
	\big\langle f(\tau) , \Theta_{\brK,2}\big(\tau,(0,r\mu),0,\borw,\polw{\vect{\alpha},\borw}{\hone}{\htwo}\big)\big\rangle \,\frac{dx\,dy}{\det y^3},
	\ea
	where $\tau=\left(\begin{smallmatrix}
	\tauone & \tautwo\\
	\tautwo & \tauthree
	\end{smallmatrix}\right)\in\HH_2$, with analogous notation for its real part~$x$ and imaginary part~$y$.
%	Recall that~$\mu$ is the vector in~$(\brK\otimes\RR)\oplus\RR u$ defined as
%	\bes
%	\mu=-\genUU+\genU_{z^\perp}/2\genU_{z^\perp}^2+\genU_z/2\genU_z^2.
%	\ees
	We are going to replace in~\eqref{eq;rewritunfoldintgen2habcdgo2} the vector-valued Siegel cusp form~$f\in S^k_{2,L}$ with its Fourier--Jacobi expansion, and rewriting the genus~$2$ Siegel theta function~$\Theta_{\brK,2}$ in terms of the Jacobi theta functions introduced in Section~\ref{subsec:JacobiForms}.

	For every~$\lambda\in \brK'$, the Jacobi theta function~$\Theta_{\brK, \lambda}$ is valued in~$\CC[\disc{\brK}]$.
	Consider the natural inclusion
	\be\label{eq;inclusionoP}
	\CC[\disc{\brK}] \longrightarrow \CC[\disc{\brK}^2],\qquad \mathfrak{e}_\gendisc \longmapsto\mathfrak{e}_{(0,\gendisc)}.
	\ee
	Under this identification we may regard~$\Theta_{\brK, \lambda}$ as a~$\CC[\disc{\brK}^2]$-valued function.
	Let $\vect{\nu} = 0$. By Lemma~\ref{lem:SublatticePolynomialVeryHomogeneous} we may rewrite
	\ba\label{eq;rewritSiegthet2withjac}
		\,&\sum_{\hone + \htwo = \htot} \Big(-\frac{y_2}{y_1}\Big)^{\hone}
		\Theta_{\brK, 2}(\tau, \vect{\delta}, 0, \borw, \calP_{\vect{\alpha}, \borw, \hone, \htwo})
		\\
		&= \Big(\frac{\det y}{\yone}\Big)^{1/2} \sum_{\gendisc\in\disc{\brK}}
		\sum_{\lambda \in \gendisc + \brK}
		e\Big(-2 i q\big(\lambda_w\big) \frac {\det y}{y_1} - \Big(\lambda, \delta_2\Big)\Big)
		\\
		&\quad\times\Theta_{\brK, \lambda}\Big(\tau_1, \tau_2, \delta_1, 0, \borw, \exp\big(-\frac{y_1}{\det y} \Delta_2\big)
		\calP_{\vect{\alpha}, \borw, 0, \htot}\big( \cdot, \borw(\lambda)\big)\Big)
		\otimes
		\frake_{(\gendisc, 0)}\big(q(\lambda) \tauthree\big),
	\ea
	where~$\Delta_2$ is the Laplacian of the second copy of~$\RR^{b,2}$ in~$(\RR^{b,2})^2$; see Section~\ref{subsec:JacobiForms} for a similar result.
	
	Note that the left-hand side of~\eqref{eq;rewritSiegthet2withjac} is a combination of genus~$2$ Siegel theta functions twisted by a coefficient that is actually a function of~$\Im(\tau)$.
	Moreover, also the polynomial argument of the Jacobi theta functions~$\Theta_{\brK, \lambda}$ depends on~$\Im(\tau)$.
	
	We denote the Fourier--Jacobi expansion of~$f$ by
	\ba\label{eq;FJexpcuspf}
	f(\tau) &= \sum_{\gendisc\in\disc{L}} \sum_{\substack{m \in q(\gendisc) + \IZ \\ m > 0}} \phi_{\gendisc, m}(\tau_1, \tau_2) e(m \tau_3) \frake_{(h, 0)}
	\\
	&=\sum_{\gendisc\in\disc{L}} \sum_{\substack{m \in q(\gendisc) + \IZ \\ m > 0}} \phi_{\gendisc, m}(\tau_1, \tau_2) e(m \xthree) \exp(-2\pi m \ythree) \frake_{(h, 0)},
	\ea
	for some vector-valued Jacobi form~$\phi_{\gendisc,m}$ as introduced in Section~\ref{subsec:JacobiForms}.
	The values of~$\phi_{\gendisc,m}$ are in~$\CC[\disc{L}]$, and are considered as values in~$\CC[\disc{L}^2]$ by means of~\eqref{eq;inclusionoP}.

	We are now ready to illustrate the main result of this section.
	Its counterpart in genus~$1$ is~\cite[Theorem~$5.5$]{zuffetti;gen1}.
	We suggest the reader to recall the definition of~$\intfunctshort$ from~\eqref{eq;defintKMliftgen2}, the construction of the polynomials~$\pol_{\vect{\alpha},\borw,\hone,\htwo}$ and~$\pol_{\vect{\alpha},\borw,\htot}$ from Definition~\ref{def;genborpolgenus2}\textcolor{\mycolor}{, and the vector~$\mu$ from~\eqref{eq;defmu}}. 
	\begin{thm}\label{thm;Fexpgen2}
	Suppose that~$L$ splits off a hyperbolic plane as in~\eqref{eq;choiceofuu'gen2}.
	Then~$\intfunctshort$ admits a Fourier expansion of the form~$\intfunctshort(g)=\sum_{\lambda\in\brK} c_\lambda(g)\cdot e\big( (\lambda,\mu) \big)$.
	Let~$\gendisc\in\disc{\brK}$ and~$\lambda\in\gendisc + \brK$.
	If~$q(\lambda)>0$, then the Fourier coefficient of~$\intfunctshort$ associated to~$\lambda$ is
	\ba\label{eq;Fcoeffgen2KMliftgen2}
	c_\lambda(g)
	&=
	\sum_{t\ge 1, t|\lambda} \int_{(\tauone,\tautwo)\in\Gamma^J\backslash\HH\times\CC} \int_{\ythree=\ytwo^2/\yone}^\infty \nonumber
	\sum_{\htot}\frac{\det y^{k-5/2-\htot}}{\genU_{z^\perp}^2}\Big(\frac{t}{2i}\Big)^{\htot}
	\\
	&\quad\times
	\yone^{\htot-1/2} \nonumber
	\exp\Big(-\frac{\pi t^2\yone}{2\genU_{z^\perp}^2\det y}
		 -\frac{2\pi}{t^2}\Big(
		 \lambda_{w^\perp}^2\ythree + \lambda_w^2\frac{\ytwo^2}{\yone}
		 \Big)
		 \Big)\cdot
		 \Big\langle
		 \phi_{\gendisc/t , q(\lambda)/t^2}(\tauone,\tautwo)
		 ,
		 \\
		 &\qquad\nonumber
		 \Theta_{\brK,\lambda/t}\big(\tauone,\tautwo,\borw,\exp\big(-\yone\Delta_2\det y^{-1}\big)\polw{\vect{\alpha},\borw}{0}{\htot}(\cdot,\borw(\lambda)/t)\big)
		 \Big\rangle
		 \,d\xone\dots d\ythree,
	\ea
	where we say that an integer~$t\ge1$ divides~$\lambda\in\brK'$, in short~$t|\lambda$, if and only if~$\lambda/t$ is still in~$\brK'$, and denote by~$h/t$ the class~$\lambda/t+\brK \in \disc{\brK}$.
	
	The Fourier coefficient of~$\intfunctshort$ associated to~$\lambda=0$, i.e. the constant term of the Fourier series, is
	\ba\label{eq;intconstcoefgen2}
	c_0(g)
	=
	\int_{\Sp_4(\ZZ)\backslash\HH_2}\frac{\det y^{k}}{2 u_{z^\perp}^2}
	\big\langle f(\tau) , \Theta_{\brK,2}(\tau,\borw,\polw{\vect{\alpha},\borw}{0}{0})\big\rangle
	\frac{dx\,dy}{\det y^3}.
	\ea
	
	In all remaining cases, the Fourier coefficients vanish.
	\end{thm}
	The Kudla--Millson lift of a Siegel cusp form is a function of~$g\in G=\SO(L\otimes\RR)$.
	The Fourier expansion of such a function follows from its invariance with respect to translations on the tube domain model of~$\Gr(L)$ given by certain Eichler transformations.
	This method is well-known, and can be employed by choosing an Iwasawa decomposition of~$G$; see~\cite[Section~$4$]{zuffetti;gen1} for details, and~\cite[Theorem~$5.5$]{zuffetti;gen1} for its application to the Kudla--Millson lift in genus~$1$.
	Alternatively, it is possible to use Lemma~\ref{lemma;auxlemmaforchfe} to check the invariance with respect to Eichler transformations of the form~$E(\genU,\lambda')$, with~$\lambda'\in \brK'$; see the end of Section~\ref{sec:JacobiThetaInnerProducts} for details.

	\begin{proof}[Proof of Theorem~\ref{thm;Fexpgen2}]
	We consider the unfolding~\eqref{eq;splitsumC12habcdunfolded} of the defining integrals~$\intfunctshort$.
	The first summand on the right-hand side of~\eqref{eq;splitsumC12habcdunfolded} is the constant term of the Fourier expansion of~$\intfunctshort$.
	This can be easily showed as in the genus~$1$ case; see~\cite[Theorem~$5.5$]{zuffetti;gen1}.
	
	We compute the Fourier expansion of the second summand appearing on the right-hand side of~\eqref{eq;splitsumC12habcdunfolded}.
	First of all, we compute the series expansion of the scalar product of~$f$ with the left-hand side of~\eqref{eq;rewritSiegthet2withjac}, evaluated at~$(\vect{\boralpha},\vect{\borbeta})=((0,r\mu),0)$, with respect to the third entry~${\tau_3=x_3+iy_3}$ of~${\tau\in\HH_2}$.
	By~\eqref{eq;FJexpcuspf} and~\eqref{eq;rewritSiegthet2withjac} such a product is
	\bas
	\Big\langle &{} f(\tau) , \sum_{\hone + \htwo = \htot} \left(-\frac{y_2}{y_1}\right)^{\hone}
		\Theta_{\brK, 2}(\tau, (0,r\mu), 0, \borw, \calP_{\vect{\alpha}, \borw, \hone, \htwo})\Big\rangle
		\\
		 &=
		 \sqrt{\frac{\det y}{\yone}} \sum_{\gendisc\in\disc{\brK}}\sum_{\ell\in\QQ}\bigg(
		 \sum_{\substack{m\in q(\gendisc)+\ZZ,\, m>0\\\lambda\in\gendisc + \brK\\ m-q(\lambda)=\ell}}
		 \exp\Big(
		 -2\pi\big(m+q(\lambda)\big)\ythree
		 \Big)
		 \\
		 &\quad\times
		 e\Big(-\frac{2iq(\lambda_w)\det y}{y_1} + r(\lambda,\mu)\Big)\cdot\Big\langle\phi_{\gendisc,m}(\tauone,\tautwo) ,
		 \\
		 &\qquad
		  \Theta_{\brK,\lambda}\big(\tauone,\tautwo,\borw,\exp\big(-\yone\Delta_2\det y^{-1}\big)\polw{\vect{\alpha},\borw}{0}{\htot}(\cdot,\borw(\lambda))\big)\Big\rangle
		 \bigg)e(\ell\xthree)
	\eas
	We replace the previous formula in the defining formula of~$\hfunct_{\vect{\alpha}}$ provided by Theorem~\ref{thm;unfongen2}, deducing that
	\begin{align}\label{eq;proofintofFJcoef}
	2&\int_{\kling\backslash\HH_2} \hfunct_{\vect{\alpha}}(\tau,g)\frac{dx\,dy}{\det y^3}= 
	\int_{(\tauone,\tautwo)\in\Gamma^J\backslash\HH\times\CC} \int_{\ythree=\ytwo^2/\yone}^\infty
	\sum_{r\ge1}\sum_{\htot}\frac{\det y^{k-3-\htot}}{\genU_{z^\perp}^2}
	\\
	&\quad\times \Big(\frac{r}{2i}\Big)^{\htot}\yone^{\htot} \nonumber
	\exp\Big(-\frac{\pi r^2\yone}{2\genU_{z^\perp}^2\det y}\Big)
	\\
	&\quad\times\int_{\xthree=0}^1
	\Big\langle
	f(\tau)
	,
	\sum_{\hone+\htwo=\htot}\Big(-\frac{\ytwo}{\yone}\Big)^{\hone}
	\Theta_{\brK,2}\big(\tau,(0,r\mu),0,\borw,\polw{\vect{\alpha},\borw}{\hone}{\htwo}\big)\Big\rangle
	dx\,dy\nonumber
	\\
	&=
	\int_{(\tauone,\tautwo)\in\Gamma^J\backslash\HH\times\CC} \int_{\ythree=\ytwo^2/\yone}^\infty \nonumber
	\sum_{r\ge1}\sum_{\htot}\frac{\det y^{k-5/2-\htot}}{\genU_{z^\perp}^2}
	\Big(\frac{r}{2i}\Big)^{\htot}\yone^{\htot-1/2} \nonumber
	\exp\Big(-\frac{\pi r^2\yone}{2\genU_{z^\perp}^2\det y}\Big)
	\\
	&\quad\times\sum_{\gendisc\in\disc{\brK}}
	\sum_{\ell\in\QQ}\bigg(
		 \sum_{\substack{m\in q(\gendisc)+\ZZ,\,m>0\\ \lambda\in\gendisc + \brK\\ m-q(\lambda)=\ell}}\nonumber
		 \exp\Big(
		 -2\pi\big(m+q(\lambda)\big)\ythree
		 \Big)
		 \cdot
		 e\Big(-\frac{2iq(\lambda_w)\det y}{y_1} + r(\lambda,\mu)\Big)
		 \\
		 &\quad\times\nonumber
		 \Big\langle
		 \phi_{\gendisc,m}(\tauone,\tautwo)
		 ,
		 \Theta_{\brK,\lambda}\big(\tauone,\tautwo,\borw,\exp\big(-\yone\Delta_2\det y^{-1}\big)\polw{\vect{\alpha},\borw}{0}{\htot}(\cdot,\borw(\lambda))\big)\Big\rangle
		 \bigg)
		 \\
		 &\quad\times\nonumber
		 \int_{\xthree=0}^1
		 e(\ell\xthree)dx\,dy.
	\end{align}
	The last integral appearing on the right-hand side of~\eqref{eq;proofintofFJcoef} may be computed as
	\bes
	\int_{\xthree=0}^1 e(\ell\xthree) d\xthree=\begin{cases}
	1 & \text{if $\ell=0$,}\\
	0 & \text{otherwise.}
	\end{cases}
	\ees
	We may then simplify~\eqref{eq;proofintofFJcoef} extracting only the terms with~$\ell=0$, obtaining that
	\bas
	2 & \int_{\kling\backslash\HH_2} \hfunct_{\vect{\alpha}}(\tau,g)\frac{dx\,dy}{\det y^3}
	\\
	&=
	\sum_{\gendisc\in\disc{\brK}}
	\sum_{\lambda\in\gendisc + \brK}\int_{(\tauone,\tautwo)\in\Gamma^J\backslash\HH\times\CC} \int_{\ythree=\ytwo^2/\yone}^\infty \nonumber
	\sum_{r\ge1}\sum_{\htot}\frac{\det y^{k-5/2-\htot}}{\genU_{z^\perp}^2}\Big(\frac{r}{2i}\Big)^{\htot}
	\\
	&\quad\times \yone^{\htot-1/2} \nonumber
	\exp\Big(-\frac{\pi r^2\yone}{2\genU_{z^\perp}^2\det y}\Big)
		 \exp\Big(
		 -2\pi\lambda^2\ythree
		 \Big)\cdot
		 e\Big(-\frac{2iq(\lambda_w)\det y}{y_1} \Big)
		 \\
		 &\quad\times\nonumber
		 \Big\langle \phi_{\gendisc,q(\lambda)}(\tauone,\tautwo)
		 ,
		 \Theta_{\brK,\lambda}\big(\tauone,\tautwo,\borw,\exp\big(-\yone\Delta_2\det y^{-1}\big)\polw{\vect{\alpha},\borw}{0}{\htot}(\cdot,\borw(\lambda))\big)
		 \Big\rangle
		 \\
		 &\quad\times d\xone\,d\xtwo\,d\yone\,d\ytwo\,d\ythree\cdot e\big(r(\lambda,\mu)\big).
		 \eas
	We now gather the terms multiplying~$e\big((\lambda,\mu)\big)$, for every~$\lambda\in\brK'$.
	This can be done simply \textcolor{\mycolor}{by} replacing the sum~$\sum_{r\ge 1}$ with~$\sum_{t\ge 1, t|\lambda}$, the vector $\lambda$ with~$\lambda/t$, and~$\gendisc$ with~$\gendisc/t$.
	In this way, we obtain that
	\begin{align*}
	2 & \int_{\kling\backslash\HH_2} \hfunct_{\vect{\alpha}}(\tau,g)\frac{dx\,dy}{\det y^3}
	\\
	&=
	\sum_{\gendisc\in\disc{\brK}}\sum_{\lambda\in\gendisc + \brK}\sum_{t\ge 1, t|\lambda}\int_{(\tauone,\tautwo)\in\Gamma^J\backslash\HH\times\CC} \int_{\ythree=\ytwo^2/\yone}^\infty \nonumber
	\sum_{\htot}\frac{\det y^{k-5/2-\htot}}{\genU_{z^\perp}^2}\Big(\frac{t}{2i}\Big)^{\htot}
	\\
	& \quad \times \yone^{\htot-1/2} \nonumber
	\exp\Big(-\frac{\pi t^2\yone}{2\genU_{z^\perp}^2\det y}
		 -\frac{2\pi}{t^2}\Big(
		 \lambda_{w^\perp}^2\ythree + \lambda_w^2\frac{\ytwo^2}{\yone}
		 \Big)
		 \Big)
		 \cdot
		 \Big\langle\phi_{\gendisc/t,q(\lambda)/t^2}(\tauone,\tautwo)
		 ,
		 \\
		 & \qquad \nonumber
		 \Theta_{\brK,\lambda/t}\big(\tauone,\tautwo,\borw,\exp\big(-\yone\Delta_2\det y^{-1}\big)\polw{\vect{\alpha},\borw}{0}{\htot}(\cdot,\borw(\lambda/t))\big)\Big\rangle
		 \\
		 & \quad \times
		 d\xone\,d\xtwo\,d\yone\,d\ytwo\,d\ythree\cdot e\big((\lambda,\mu)\big).
		 \qedhere
	\end{align*}
%	This is the Fourier expansion of~$2\int_{\kling\backslash\HH_2} h_{\abcdvar}(\tau,g)\frac{dx\,dy}{\det y^3}$.
	\end{proof}
	%%%%%%%%%%%%%%%%%%%%%%%%%%%%%%%%%%%%%%%%%%%%%%%%%%%%%%%%%%%%%%
	%%%%%%%%%%%%%%%%%%%%%%%%%%%%%%%%%%%%%%%%%%%%%%%%%%%%%%%%%%%%%%
	\section{The injectivity of the genus 2 Kudla--Millson lift}\label{sec;finalunfvvcase}
	\renewcommand{\isometrypos}{w^\perp}
	\renewcommand{\isometryneg}{w}
	\renewcommand{\subspaceisometry}{{\isometry_{\sublattice}}}
	\renewcommand{\subspaceisometryprime}{{\isometry'_{\sublattice}}}
	\renewcommand{\subspaceisometrypos}{{\tilde{w}^\perp}}
	\renewcommand{\subspaceisometryneg}{\tilde{w}}
	
	Let $L$ be an even lattice of signature $\left(b, 2\right)$ that splits off two (orthogonal) hyperbolic planes, i.e.~$b \geq 5$	and
%	\bes%\label{eq;Lpslitsofst}
%	L = \underbrace{U}_{\left\langle u, u' \right\rangle} \oplus \underbrace{U}_{\left\langle \tilde{u}, \tilde{u}' \right\rangle} \oplus L^+,\qquad\text{even positive definite lattice $L^+$}.
%	\ees
\bes%\label{eq;Lpslitsofst}
	L = U \oplus \underbrace{U \oplus L^+}_{=\brK}\qquad\text{for some positive definite lattice $L^+$}.
	\ees
	As in the previous sections, we denote by~$\textcolor{\newcolor}{K}$ the orthogonal complement of a hyperbolic plane split off by~$L$.
	Let~$\genU$ and~$\genUU$ be a standard basis of the hyperbolic plane orthogonal to~$M$.
	Similarly, we denote by~$\genUtwo$ and~$\genUUtwo$ the standard basis vectors of the hyperbolic plane orthogonal to~$\Lpos$ in~$\textcolor{\newcolor}{K}$.
	These basis vectors are chosen such that
	\begin{align*}
		u^2 = u'^2 = \tilde{u}^2 = \tilde{u}'^2 = 0
		\qquad\text{and}\qquad
		(u, u') = (\tilde{u}, \tilde{u}') = 1.
	\end{align*}
	Without loss of generality we may assume that the orthonormal basis~$e_1, \ldots, e_{b+2}$ of~$L \otimes \IR$ is such that
	\begin{align*}
		u = \frac{e_b + e_{b+2}}{\sqrt{2}}
		 ,\qquad 
		 %\text{and} \quad 
		 u' = \frac{e_b - e_{b+2}}{\sqrt{2}}
		 ,\qquad 
		%\\
		\tilde{u} = \frac{e_{b-1} + e_{b+1}}{\sqrt{2}}
		\qquad
		\text{and} \qquad 
		\tilde{u}' = \frac{e_{b-1} - e_{b+1}}{\sqrt{2}}.
	\end{align*}
	
	Recall that for every isometry~$g$ we denote by~$\borw$ the linear map defined for every~${\genvec\in L\otimes\RR}$ as~$\borw(v) \coloneqq g(v_{\textcolor{\newcolor}{K} \otimes \IR})$.
	Similarly, we define $\borww(v) \coloneqq \borw(v_{L^+ \otimes \IR})$.
	For every~$z \in \Gr(L)$, let $w$ be the line in $z$ that is orthogonal to $u_z$. The base point~$z_0$ of~$\Gr(L)$ is spanned by~$e_{b+1}$ and~$e_{b+2}$.
	We write~$w_0$ for the line in~$z_0$ orthogonal to~$u_{z_0}$.
		
	We now calculate the polynomials~$\calP_{\boldsymbol{\alpha}, \borw, 0, \htot}$ occurring in the Fourier expansion of the Kudla--Millson lift under certain assumptions on~$\borw$ and the tuple of indices~$\vect{\alpha}$.
	These will eventually simplify the Fourier coefficients computed in Theorem~\ref{thm;Fexpgen2}.
	
	Let $\boldsymbol{\alpha} = \left(\alphaone, \alphatwo, \betaone, \betatwo\right)\in \{1, \ldots, b-2\}^4$ be such that
	\be\label{eq;hypothonalpha}
	\alphaone \neq \alphatwo,\qquad \betaone \neq \betatwo, \qquad \alphaone < \betaone \qquad \text{and} \qquad \alphatwo < \betatwo.
	\ee
	Recall that under these assumptions the homogeneous polynomial~$\pol_{\vect{\alpha}}$ is given by
	$$\calP_{\boldsymbol{\alpha}}\big(\boldsymbol{x}\big) = 4 \big(x_{\alphaone, 1} x_{\betaone, 2} - x_{\alphaone, 2} x_{\betaone, 1}\big) \big(x_{\alphatwo, 1} x_{\betatwo, 2} - x_{\alphatwo, 2} x_{\betatwo, 1}\big)$$
	for every~$\boldsymbol{x} = (x_{i, j})_{i,j} \in (\IR^{b, 2})^2$.
	
	Let $g$ be \emph{an isometry interchanging $e_{\alphaone}$ with $e_b$ and fixing $e_{b+2}$}. To shorten the notation, we will write $v_{i, j} \coloneqq \left(v_j, e_{i}\right)$ if $\boldsymbol{v} = \left(v_1, v_2\right) \in \left(L \otimes \IR\right)^2$. We have
	$$\calP_{\boldsymbol{\alpha}}(g(\boldsymbol{v})) 
	= 
	4 \big(v_{b, 1} g(\boldsymbol{v})_{\betaone, 2} - v_{b, 2} g(\boldsymbol{v})_{\betaone, 1}\big)
	\cdot
	\big(g(\boldsymbol{v})_{\alphatwo, 1} g(\boldsymbol{v})_{\betatwo, 2} - g(\boldsymbol{v})_{\alphatwo, 2} g(\boldsymbol{v})v_{\betatwo, 1}\big).$$
	Moreover, since $u_{z^\perp} = e_b/\sqrt{2}$ and $(v, u_{z^\perp}) = x_b/\sqrt{2}$, we deduce
	\begin{align*}
		\calP_{\boldsymbol{\alpha}}(g(\boldsymbol{v}))
		&= 
		4 \sqrt{2} \Big((v_1, u_{z^\perp}) g(\boldsymbol{v})_{\betaone, 2} - (v_2, u_{z^\perp}) g(\boldsymbol{v})_{\betaone, 1}\Big)
		\\
		&\quad\times
		\Big(g(\boldsymbol{v})_{\alphatwo, 1} g(\boldsymbol{v})_{\betatwo, 2} - g(\boldsymbol{v})_{\alphatwo, 2} g(\boldsymbol{v})_{\betatwo, 1}\Big)
		\\
		&= 4 \sqrt{2} \Big((v_2, u_{z^\perp}) \borw(\boldsymbol{v})_{\betaone, 1} - (v_1, u_{z^\perp}) \borw(\boldsymbol{v})_{\betaone, 2}\Big)
		\\
		&\quad\times
		\Big(\borw(\boldsymbol{v})_{\alphatwo, 2} \borw(\boldsymbol{v})_{\betatwo, 1} - \borw(\boldsymbol{v})_{\alphatwo, 1} \borw(\boldsymbol{v})_{\betatwo, 2}\Big).
	\end{align*}
	By comparing the previous formula with~\eqref{eq;gengen2borfupol}, we deduce that
	\begin{align*}
	&\calP_{\boldsymbol{\alpha}, \borw, 0, \htot}\left(\borw\left(\boldsymbol{v}\right)\right) \\
	&= \begin{cases}
		4 \sqrt{2} g(\boldsymbol{v})_{\betaone, 1} \cdot
		\big(\borw(\boldsymbol{v})_{\betatwo, 1} \borw(\boldsymbol{v})_{\alphatwo, 2} - \borw(\boldsymbol{v})_{\alphatwo, 1} \borw(\boldsymbol{v})_{\betatwo, 2}\big) &\mbox{if } \htot  = 1, \\
		0 &\mbox{otherwise},
	\end{cases}
	\end{align*}
	so that the corresponding polynomial for $g$ is given by
	$$\calP_{\boldsymbol{\alpha}, \borw, 0, 1}(\boldsymbol{x}) = 4 \sqrt{2} x_{\betaone, 1} \big(x_{\betatwo, 1} x_{\alphatwo, 2} - x_{\alphatwo, 1} x_{\betatwo, 2}\big),$$
	which is in fact independent of the choice of $g$ as long as it interchanges~$e_{\alphaone}$ with~$e_b$ and fixes~$e_{b+2}$.
	Since we assumed $\alphaone \neq \alphatwo$ and~$\betaone \neq \betatwo$, the polynomials $\calP_{\boldsymbol{\alpha}, \borw, 0, 1}(\boldsymbol{x})$ are harmonic with respect to $\Delta_1$, $\Delta_2$ and $\trace(\Delta y^{-1})$.
	
	Next, recall that if~$\eta\in\brK$ is of positive norm, then the Fourier coefficient~$c_\jacindexvec(g)$ of the defining integral~$\intfunctshort$ of the lift~$\KMlift(f)$ calculated in Theorem~\ref{thm;Fexpgen2} equals
	\begin{align*}
	 & \sum_{t\ge 1, t|\jacindexvec}\int_{\left(\tauone,\tautwo\right)\in\Gamma^J\backslash\HH\times\CC} \int_{\ythree=\ytwo^2/\yone}^\infty
	 \sum_{\htot}\frac{\det y^{k-5/2-\htot}}{\genU_{z^\perp}^2}\left(\frac{t}{2i}\right)^{\htot}
	 \\
	 &\times
	 \left\langle \phi_{\jacindexvec / t, q\left(\jacindexvec\right)/t^2}\left(\tauone,\tautwo\right), \Theta_{\brK,\jacindexvec/t}\Big(\tauone,\tautwo,\borw,\exp\left(-\yone\Delta_2\det y^{-1}\right)\big(\polw{\vect{\alpha},\borw}{0}{\htot}\left(\cdot,\borw\left(\jacindexvec/t\right)\right)\big)\Big)\right\rangle
	 \\
	 &\times \yone^{\htot-1/2}
	 \exp\bigg(-\frac{\pi t^2\yone}{2\genU_{z^\perp}^2\det y}
	 -\frac{2\pi}{t^2}\bigg(
	 \jacindexvec_{w^\perp}^2\ythree + \jacindexvec_w^2\frac{\ytwo^2}{\yone}
	 \bigg)
	 \bigg)
	 \,
	 d\xone\,d\xtwo\,d\yone\,d\ytwo\,d\ythree.
	\end{align*}
	We apply the change of variables $\ythree \mapsto \ythree + \ytwo^2 / \yone$ to the previous integral and obtain
	\ba\label{eq;injproofrocucv}
%		&\sum_{t\ge 1, t|\jacindexvec}\int_{\left(\tauone,\tautwo\right)\in\Gamma^J\backslash\HH\times\CC} \int_{\ythree=0}^\infty
%		\sum_{\htot}\frac{\left(y_1 y_3\right)^{k-5/2-\htot}}{\genU_{z^\perp}^2}\left(\frac{t}{2i}\right)^{\htot}
%		\\
%		&\quad\times \left\langle \phi_{\jacindexvec / t, q\left(\jacindexvec\right)/t^2}\left(\tauone,\tautwo\right), \Theta_{\brK,\jacindexvec/t}\Big(\tauone,\tautwo,\borw,\exp\left(-\Delta_2y_3^{-1}\right)\big(\polw{\vect{\alpha},\borw}{0}{\htot}\left(\cdot,\borw\left(\jacindexvec/t\right)\right)\big)\Big) \right\rangle
%		\\
%		&\quad\times \yone^{\htot-1/2}
%		\exp\bigg(-\frac{\pi t^2}{2\genU_{z^\perp}^2 y_3}
%		-\frac{2\pi}{t^2}
%		\jacindexvec_{w^\perp}^2\ythree
%		\bigg) \exp\bigg(-\frac{2\pi \ytwo^2 \jacindexvec^2}{\yone t^2}\bigg)
%		d\xone\,d\xtwo\,d\yone\,d\ytwo\,d\ythree \\
%		&=
		&\sum_{t\ge 1, t|\jacindexvec}
		\int_{\left(\tauone,\tautwo\right)\in\Gamma^J\backslash\HH\times\CC} \int_{\ythree=0}^\infty
		\sum_{\htot}\frac{y_3^{k-5/2-\htot}}{\genU_{z^\perp}^2}\left(\frac{t}{2i}\right)^{\htot}
		\\
		&\times \Big\langle \phi_{\jacindexvec / t, q\left(\jacindexvec\right)/t^2}\left(\tauone,\tautwo\right), \Theta_{\brK,\jacindexvec/t}\Big(\tauone,\tautwo,\borw,\exp\left(-\Delta_2y_3^{-1}\right)\big(\polw{\vect{\alpha},\borw}{0}{\htot}(\cdot,\borw(\jacindexvec/t))\big)\Big) \Big\rangle
		\\
		&\times \yone^{k-3}
		\exp\bigg(-\frac{\pi t^2}{2\genU_{z^\perp}^2 y_3}
		-\frac{2\pi}{t^2}
		\jacindexvec_{w^\perp}^2\ythree
		\bigg) \exp\bigg(-\frac{2\pi \ytwo^2 \jacindexvec^2}{\yone t^2}\bigg)
		d\xone\,d\xtwo\,d\yone\,d\ytwo\,d\ythree.
	\ea
	If the Kudla--Millson lift of the Siegel cusp form~$f$ vanishes, then the Fourier coefficients of~$\intfunctshort$ vanish for all~$\boldsymbol{\alpha}$, since the wedge products~$\omegaone\wedge\omegatwo\wedge\omegathree\wedge\omegafour$ are linearly independent in~$\bigwedge^4 T_{z_0}^* \hermdom$.
	This means that the coefficients~$c_\eta$ computed in Theorem~\ref{thm;Fexpgen2} are, as functions on~$G$, identically zero.
	
	In particular~$c_\eta(g)=0$ for the special choice of~$g$ made above, i.e.\ an isometry interchanging $e_{\alphaone}$ with $e_b$ and fixing $e_{b+2}$.
	Under this assumption on~$g$ and~\eqref{eq;hypothonalpha} on~$\vect{\alpha}$, the computation given in~\eqref{eq;injproofrocucv} of the coefficient~$c_\eta(g)$ may be simplified.
	In fact, the polynomials~$\polw{\vect{\alpha},\borw}{0}{\htot}$ are harmonic, the Jacobi theta functions appearing therein is independent of $y_3$, and the integral over~$\ythree$ is positive.
	Hence, if the Kudla--Millson lift of a Siegel cusp form~$f$ vanishes, then
	\ba\label{eq:DivisorSumInnerProducts}
		& \frac{1}{\genU_{z^\perp}^2} \sum_{t\ge1,\,t|\jacindexvec} \frac{t}{2i} \int_{\left(\tauone,\tautwo\right)\in\Gamma^J\backslash\HH\times\CC} \yone^{k-3} 
		\exp\bigg(-\frac{2\pi \ytwo^2 \jacindexvec^2}{\yone t^2}\bigg)
		\\
		&\quad\times 
		\Big\langle \phi_{\jacindexvec / t, q(\jacindexvec)/t^2}(\tauone,\tautwo), \Theta_{\brK, \jacindexvec / t}
		\big(\tau_1, \tau_2,\borw,\polw{\vect{\alpha},\borw}{0}{1}\big(\cdot,\borw(\jacindexvec/t)\big)\big)\Big\rangle
		\,d\xone\,d\xtwo\,d\yone\,d\ytwo \\
		&=\frac{1}{2i \genU_{z^\perp}^2}
		\sum_{t\ge1,\,t|\jacindexvec} t \Big\langle \phi_{\jacindexvec / t, q\left(\jacindexvec\right)/t^2}, \Theta_{\brK, \jacindexvec / t}\big(\cdot,\cdot ,\borw,\polw{\vect{\alpha},\borw}{0}{1}\big(\cdot,\borw(\jacindexvec/t)\big)\big) \Big\rangle_{\Pet}
	\ea
	vanishes for all $\boldsymbol{\alpha}$ satisfying~\eqref{eq;hypothonalpha} and $\jacindexvec \in \brK'$, where $\left\langle \cdot{,} \cdot \right\rangle_{\Pet}$ is the Petersson inner product introduced in Section \ref{sec:JacobiThetaInnerProducts}.
	
	If~$\jacindexvec \in \brK'$ is primitive, then there is only one summand in~\eqref{eq:DivisorSumInnerProducts}, hence the only Petersson inner product appearing therein vanishes.
	By induction on the number of divisors on~$\jacindexvec$ we obtain the vanishing of every Petersson inner product appearing in~$\eqref{eq:DivisorSumInnerProducts}$, independently from the primitivity of~$\jacindexvec \in \brK'$.
	\textcolor{\mycolor}{ This proves Corollary~\ref{cor;introvanpetjac}.}
	
	We now show an injectivity result for the inner products in~\eqref{eq:DivisorSumInnerProducts} involving the Jacobi theta functions that we have introduced in Section \ref{sec:JacobiThetaInnerProducts}.
	
	\renewcommand{\jacindexvec}{\eta}
	\renewcommand{\sublattice}{{L^+}}
	\renewcommand{\isometry}{{g_M}}
	\renewcommand{\subspaceisometry}{{g_{\Lpos}}}
	
	\begin{thm}\label{thm:JacobiThetaLiftInjectivity}
		Let $\jacindexvec \in \sublattice'$ and let $N$ be the largest natural number dividing $\jacindexvec$. Let $\phi$ be a Jacobi cusp form of index $\jacindexvec$ with Fourier expansion
		$$\phi(\jacvarone, \jacvartwo)
		=
		\sum_{\substack{\sigma \in D_\lattice \\ r \in \IZ + q(\sigma) \\ s \in \IZ + (\sigma, \jacindexvec)}}
		a(\sigma, r, s) \cdot \frake_\lambda(r \jacvarone + s \jacvartwo).$$
		If the inner product
		\begin{align}
			\left\langle \phi, \Theta_{\brK, \jacindexvec + l \tilde{\isotropicvec}}\Big(\cdot, \cdot, \borw, \calP_{\boldsymbol{\alpha}, \borw, 0, 1}\big(\cdot, \borw\left(\jacindexvec + l \tilde{\isotropicvec}\right)\big)\Big) \right\rangle_{\Pet} \label{eq:jacinnerproduct}
		\end{align}
		vanishes for all $l = 1, \ldots, N$, for every~$\boldsymbol{\alpha}$ satisfying hypothesis~\eqref{eq;hypothonalpha} and every isometry~$g$ interchanging~$e_{\alpha_1}$ with~$e_b$ and fixing~$e_{b+2}$, then
		$$a\big(\lambda, q(\lambda), (\lambda, \jacindexvec)\big) = 0\qquad \text{for every~$\lambda \in \sublattice'$}.$$
	\end{thm}
	
	\begin{proof}
		Let~$g$ be the isometry that interchanges~$e_{\alpha_1}$ with~$e_b$, $e_{\alpha_2}$ with~$e_{b - 1}$, and fixes the remaining basis vectors.
		Then $z = z_0$ and $w = w_0$.
		It is easy to see that~$\tilde{u}_{w_0^\perp} = e_{b - 1}/\sqrt{2}$ and~$(v, \tilde{u}_{w_0^\perp}) = x_{b - 1}/\sqrt{2}$.
		Hence, we have
		\begin{align*}
			\calP_{\boldsymbol{\alpha}, \borw, 0, 1}(\borw(\boldsymbol{v}))
			&=
			8 (v_1, \tilde{u}_{w_0^\perp})
			\cdot
			\big(\borw(\boldsymbol{v})_{\betatwo, 1} \borw(\boldsymbol{v})_{\alphatwo, 2} - \borw(\boldsymbol{v})_{\alphatwo, 1} \borw(\boldsymbol{v})_{\betatwo, 2}\big)
			\\
			&=
			8 (v_1, \tilde{u}_{w_0^\perp})
			\cdot
			(v_{\betatwo, 1} v_{\alphatwo, 2} - v_{\alphatwo, 1} v_{\betatwo, 2}),
		\end{align*}
		from which by definition of $\calP_{\boldsymbol{\alpha}, \borww, 0, 1, \htot}$ we obtain
		\begin{align*}
			\calP_{\boldsymbol{\alpha}, \borww, 0, 1, \htot}\left(\boldsymbol{x}\right) = \begin{cases}
				8(x_{\betatwo, 1} x_{\alphatwo, 2} - x_{\alphatwo, 1} x_{\betatwo,2}) &\mbox{if } \htot = 1, \\
				0 &\mbox{otherwise}.
			\end{cases}
		\end{align*}
		Hence, the assumptions of Theorem~\ref{thm:JacobiThetaUnfolding} are satisfied and the vanishing of the inner products in \eqref{eq:jacinnerproduct} implies that the Fourier coefficients computed in Theorem~\ref{thm:JacobiThetaUnfolding}, namely
		\begin{align*}
				&\textcolor{\newcolor}{c_{\eta + l \tilde{u},\tilde{\lambda}}(g) = \frac{2}{\lvert \tilde{u}_{\isometrypos} \rvert \lvert \eta_{\tilde{w}^\perp} \rvert} 
				\sum_{n \mid \tilde{\lambda}} \sum_{\subpolposdeg}
				\frac{n^{b^- + h + k - 4}}{(2i)^{-\subpolposdeg}}
				\Bigg(\frac{\lvert \eta_{w^\perp} \rvert}{2 \lvert u_{w^\perp} \rvert \sqrt{\tilde{\lambda}_{\tilde{w}^\perp}^2 \eta_{\tilde{w}^\perp}^2 - (\tilde{\lambda}_{\tilde{w}^\perp}, \eta_{\tilde{w}^\perp})^2}}\Bigg)^{b^- / 2 + k - h - 2}}
				\\
				&\textcolor{\newcolor}{\times 
				\exp\Bigg(-\frac{\pi i (\eta, \tilde{u}_{w^\perp}) (\tilde{\lambda}_{\tilde{w}^\perp}, \eta_{\tilde{w}^\perp})}{\tilde{u}_{w^\perp}^2 \eta_{\tilde{w}^\perp}^2}\Bigg)
				K_{b^-/2 + k - h - 2}\left(2 \pi \frac{\lvert \eta_{w^\perp} \rvert \sqrt{\tilde{\lambda}_{\tilde{w}^\perp}^2 \eta_{\tilde{w}^\perp}^2 - (\tilde{\lambda}_{\tilde{w}^\perp}, \eta_{\tilde{w}^\perp})^2}}{\lvert \tilde{u}_{w^\perp} \rvert \eta_{\tilde{w}^\perp}^2}\right) }\\
				&\textcolor{\newcolor}{\times 
				\overline{\pol}_{\mathbf{\alpha},\subspaceisometry, 0, 1 \subpolposdeg}\big(\subspaceisometry(\tilde{\lambda}), g_K(\eta + l \tilde{u})\big) \sum_{\substack{\lambda \in \sublattice' / \jacindexvec_K \\ \lambda_{\jacindexvec} = \tilde{\lambda}}} e\left(\frac{l\left(\lambda, \jacindexvec\right)}{\jacindexvec^2}\right)
				a\big(\lambda / n, q\left(\lambda\right) / n^2, \left(\lambda, \jacindexvec\right) / n\big),}
		\end{align*}
%		\begin{align*}
%			c_{\jacindexvec + l \tilde{\isotropicvec}, \tilde{\lambda}}\left(g\right)
%			&=
%			\sum_{\htot} \sum_{n \mid \tilde{\lambda}} n^{2\htot - 2}
%			\,\overline{\calP_{\boldsymbol{\alpha}, \borww, 0, 1, \htot}}
%			\Big(
%			\big(\subspaceisometry (\tilde{\lambda}), \borw (\jacindexvec + l \tilde{\isotropicvec})\big)\Big)
%			\\
%			&\quad\times
%			\int_0^\infty \jacvaroneim^{- \htot - 1 + b/2}
%			\exp\bigg(-\frac{\pi n^2}{2 \jacvaroneim \isotropicvec_{w^\perp}^2} - \frac{\pi n^2 (\jacindexvec, \tilde{\isotropicvec}_{w^\perp})^2}{4 q(\jacindexvec) \tilde{\isotropicvec}_{w^\perp}^4 \jacvaroneim}\bigg)
%			d \jacvaroneim
%			\\
%			&\quad\times
%			\sum_{\substack{\lambda \in \sublattice' / \jacindexvec_{\sublattice} \\ \lambda_{\jacindexvec} = \tilde{\lambda}}} e\bigg(\frac{l (\lambda, \jacindexvec)}{\jacindexvec^2}\bigg)
%			\cdot
%			 a\Big(\lambda / n, q(\lambda) / n^2, (\lambda, \jacindexvec) / n\Big),
%		\end{align*}
		vanish for all $l = 1, \ldots, N$ and $\tilde{\lambda} \in (\jacindexvec^\perp \cap L^+)'$.
		Here we added the index $\eta + l \tilde{\isotropicvec}$ to avoid confusion with the Fourier coefficients, written as~$c_\lambda(g)$, of the defining integrals of the Kudla--Millson lift.
		
		Next assume that $\tilde{\lambda}$ is primitive. Then we see that
		\begin{align*}
				&\textcolor{\newcolor}{\sum_{\subpolposdeg} (2i)^{\subpolposdeg} \left(\frac{\lvert \eta_{w^\perp} \rvert}{2 \lvert u_{w^\perp} \rvert \sqrt{\tilde{\lambda}_{\tilde{w}^\perp}^2 \eta_{\tilde{w}^\perp}^2 - (\tilde{\lambda}_{\tilde{w}^\perp}, \eta_{\tilde{w}^\perp})^2}}\right)^{b^- / 2 + k - h - 2}}
				\\
				&\textcolor{\newcolor}{\times 
				\exp\left(-\frac{\pi i (\eta, \tilde{u}_{w^\perp}) (\tilde{\lambda}_{\tilde{w}^\perp}, \eta_{\tilde{w}^\perp})}{\tilde{u}_{w^\perp}^2 \eta_{\tilde{w}^\perp}^2}\right) K_{b^-/2 + k - h - 2}\left(2 \pi \frac{\lvert \eta_{w^\perp} \rvert \sqrt{\tilde{\lambda}_{\tilde{w}^\perp}^2 \eta_{\tilde{w}^\perp}^2 - (\tilde{\lambda}_{\tilde{w}^\perp}, \eta_{\tilde{w}^\perp})^2}}{\lvert \tilde{u}_{w^\perp} \rvert \eta_{\tilde{w}^\perp}^2}\right) }\\
				&\textcolor{\newcolor}{\times 
				\overline{\pol}_{\mathbf{\alpha},\subspaceisometry, 0, 1, \subpolposdeg}\big(\subspaceisometry(\tilde{\lambda}), g_K(\eta + l \tilde{u})\big) \sum_{\substack{\lambda \in \sublattice' / \jacindexvec_K \\ \lambda_{\jacindexvec} = \tilde{\lambda}}} e\left(\frac{l\left(\lambda, \jacindexvec\right)}{\jacindexvec^2}\right)
				a\big(\lambda, q\left(\lambda\right), \left(\lambda, \jacindexvec\right)\big)}
		\end{align*}
%		\begin{align*}
%			&\sum_{\htot} \overline{\calP_{\boldsymbol{\alpha}, \borww, 0, 1, \htot}}
%			\Big(\big(\subspaceisometry(\tilde{\lambda}), \borw(\jacindexvec + l \tilde{\isotropicvec})\big)\Big)
%			\int_0^\infty \jacvaroneim^{- \htot - 1 + b^-/2}
%			\\
%			&\quad\times
%			 \exp\bigg(-\frac{\pi}{2 \jacvaroneim \tilde{\isotropicvec}_{w^\perp}^2} - \frac{\pi \left(\jacindexvec, \tilde{\isotropicvec}_{w^\perp}\right)^2}{4 q\left(\jacindexvec\right) \tilde{\isotropicvec}_{w^\perp}^4 \jacvaroneim}\bigg) d \jacvaroneim
%			\sum_{\substack{\lambda \in \sublattice / \jacindexvec_{\sublattice} \\ \lambda_{\jacindexvec}' = \tilde{\lambda}}} 
%			e\left(\frac{l \left(\lambda, \jacindexvec\right)}{\jacindexvec^2}\right) 
%			\cdot
%			a\Big(\lambda, q(\lambda), (\lambda, \jacindexvec)\Big)
%		\end{align*}
		vanishes. If the moment matrix $q(\tilde{\lambda}, \jacindexvec + l\tilde{u})$ is not positive definite, then the corresponding Fourier coefficient of $\phi$ vanishes since~$\phi$ is a Jacobi cusp form.
		Otherwise, one can choose $\boldsymbol{\alpha}$ such that
		$$\calP_{\boldsymbol{\alpha}, \borww, 0, 1, 1}\Big(\big(\subspaceisometry(\tilde{\lambda}), \isometry(\jacindexvec + l \tilde{\isotropicvec})\big)\Big)
		=
		8 (\lambda_\delta \jacindexvec_\beta - \lambda_\beta \jacindexvec_\delta)$$
		does not vanish. Since \textcolor{\newcolor}{the $K$-Bessel function is strictly positive}, the finite sum
		$$\sum_{\substack{\lambda \in \sublattice' / \jacindexvec_{\sublattice} \\ \lambda_{\jacindexvec} = \tilde{\lambda}}}
		e\bigg(\frac{l (\lambda, \jacindexvec)}{\jacindexvec^2}\bigg)
		\cdot
		a\big(\lambda, q(\lambda), (\lambda, \jacindexvec)\big)$$
		must vanish for all $l = 1, \ldots, N$. Now fix some $\lambda \in \sublattice' / \jacindexvec_{\sublattice}$ with $\lambda_{\jacindexvec} = \tilde{\lambda}$. Then this sum can be rewritten as
		\begin{align*}
			e\bigg(\frac{l \big(\lambda, \jacindexvec\big)}{\jacindexvec^2}\bigg)
			\sum_{m \in \IZ / N \IZ}
			e\Big(\frac{l m}{N}\Big) 
			\cdot
			a\Big(\lambda + \frac{m \jacindexvec}{N}, q\Big(\lambda + \frac{m \jacindexvec}{N}\Big), \Big(\lambda + \frac{m \jacindexvec}{N}, \jacindexvec\Big)\Big)
		\end{align*}
		and the sum must vanish for all $l = 1, \ldots, N$. But this sum can be interpreted as an inner product of the Fourier coefficient with the character $m \mapsto e\left(lm / N\right)$ and since characters form an orthogonal basis of the functions on $\IZ / N \IZ$, this implies
		$$a\Big(\lambda + \frac{m \jacindexvec}{N}, q\Big(\lambda + \frac{m \jacindexvec}{N}\Big), \Big(\lambda + \frac{m \jacindexvec}{N}, \jacindexvec\Big)\Big) = 0$$
		for all $m \in \IZ / N \IZ$. If $\lambda$ is not primitive, then, by induction, for every $1 < n \mid \lambda$, the corresponding Fourier coefficient $a\big(\lambda / n, q(\lambda) / n^2, (\lambda, \jacindexvec) / n\big)$ vanishes. But since the sum over all divisors of $\lambda$ must vanish, we obtain again that $a\big(\lambda + \frac{m \jacindexvec}{N}, q(\lambda + \frac{m \jacindexvec}{N}), (\lambda + \frac{m \jacindexvec}{N}, \jacindexvec)\big)$ must vanish, which proves the theorem.
	\end{proof}

	\begin{thm}\label{thm:KMInjectivity}
		Assume that $L_p^+$ splits off two hyperbolic planes for every prime $p$. Then the Kudla--Millson lift $\KMlift$ is injective.
	\end{thm}
	
	\begin{proof}
		Let $f$ be a cusp form such that its Kudla--Millson lift $\KMlift(f)$ vanishes and let
		$$f(\tau) = \sum_{\sigma \in D_L^2} \sum_{\substack{T \in \Lambda_2 + q(\sigma) \\ T \geq 0}}
		a_f(\sigma, T) \frake_\sigma(T \tau)$$
		be its Fourier expansion. Then the Fourier coefficients $c_\jacindexvec(g)$ of $\KMlift(f)$ vanish and thus~\eqref{eq:DivisorSumInnerProducts} vanishes.
		Denote by~$\phi_{\sigma_2, m}$ the Fourier--Jacobi coefficients of~$f$.
		Then according to~\eqref{eq:FouerCoeffFourierJacExpansion} the Fourier coefficients of $\phi_{\jacindexvec, q\left(\jacindexvec\right)}$ are given by $a_f\big((\lambda, \jacindexvec), q(\lambda, \jacindexvec)\big)$, where $\lambda \in (L^+)'$. If~$\jacindexvec \in {L^+}'$ is primitive, then~\eqref{eq:DivisorSumInnerProducts} is given by
		\begin{align*}
		\frac{1}{2i \genU_{z^\perp}^2} \Big\langle \phi_{\jacindexvec / t, q\left(\jacindexvec\right)/t^2}
		,
		\Theta_{\brK, \jacindexvec / t}\Big(\cdot, \cdot,\borw,\polw{\vect{\alpha},\borw}{0}{1}\big(\cdot,\borw(\jacindexvec/t)\big)\Big) \Big\rangle_{\Pet}
		\end{align*}
		which implies that the inner product vanishes.
		Theorem~\ref{thm:JacobiThetaLiftInjectivity} implies that for all choices of~${\jacindexvec \in \sublattice'}$ the Fourier coefficient~$a\big((\lambda , \jacindexvec), q(\lambda,\jacindexvec)\big)$ of~$f$ vanishes.
		Induction over $1 < t \mid \jacindexvec$ now shows that all Fourier coefficients of the form $a\big((\lambda, \jacindexvec), q(\lambda, \jacindexvec)\big)$ of $f$ vanish.
		\textcolor{\mycolor}{Since~$L_p^+$ splits off two hyperbolic planes for every~$p$}, by Corollary~\ref{cor:Representations} all Fourier coefficients are of the form~$a\big((\lambda, \jacindexvec), q(\lambda, \jacindexvec)\big)$ for some~$\lambda,\jacindexvec \in L^{+'}$.
		Hence, the Siegel cusp form~$f$ vanishes.
	\end{proof}

	\begin{cor}\label{cor;frommaininjres}
	Let $L$ be an even lattice of signature $(b, 2)$ and let $l(L)$ denote the minimal number of generators of the discriminant group~$\disc{L}$. If $b > l(L) + 6$, then The Kudla--Millson lift is injective. In particular, if~$L$ is unimodular or, more generally, maximal and $b > 9$, then the Kudla--Millson lift is injective.
	\end{cor}

	\begin{proof}
		The lattice~$L$ has rank~$b + 2 > l(L) + 8$.
		By~\cite[Corollary~$1.13.5$]{Nikulin} the lattice~$L$ decomposes as~$L \simeq U \oplus U \oplus L^+$ for some even positive definite lattice~$L^+$.
		Since $l(L) = l(L^+)$, the rank of~$L^+$ is $b - 2 > l(\Lpos) + 4$. By the Jordan decomposition~\cite[§91C, §92, §93]{OMeara} we deduce that~$L_p^+$ splits two hyperbolic planes.
		If~$L$ is maximal, then~$l(L) < 4$ and thus~${b \geq 10 > l(L) + 6}$.
		The assertion now follows from Theorem~\ref{thm:KMInjectivity}.
	\end{proof}
	
	%%%%%%%%%%%%%%%%%%%%%%%%%%%%%%%%%%%%%%%%%%%%%%%%%%%%%%%%%%%%%%
	\appendix
	\section{Fourier transforms and very homogeneous polynomials}\label{sec;appendixtot}
	In this appendix we \textcolor{\mycolor}{prove the properties on Fourier transforms stated in Section~\ref{sec;Ftransfgen2} and} illustrate certain technical results regarding the polynomials introduced in Section~\ref{sec;splitgen2theta}.
	These are needed to study the non-modularity of the Siegel theta function~$\Theta_{L,2}$ associated polynomials which are not very homogeneous, as well as to prove Theorem~\ref{sec;theunfoldingofKMliftgen2new}, which provides the first unfolding of the Kudla--Millson lift.
	
	\subsection{Fourier transforms and differential operators}\label{sec;Ftransfgen2}
	Let $W$ be a real vector space endowed with a non-degenerate symmetric bilinear form~$(\cdot{,}\cdot)$, and let $f\colon W^2\to\CC$ be a~$L^1$-function.
	The Fourier transform~$\widehat{f}\colon W^2\to\CC$ of~$f$ is defined as
	\bes
	\widehat{f}(\vect{\xi})=\int_{\vect{\genvec}\in W^2}f(\vect{\genvec})\cdot e\big(\trace (\vect{\xi},\vect{\genvec})\big)d\vect{\genvec}.
	\ees
	The integral defining the Fourier transform can be studied also when the variable~$\vect{\xi}$ takes~\emph{complex values}.
	Depending on~$f$, such an integral might not converge for some~${\vect{\xi}\in W^2\otimes\CC}$.
	In this section, we assume that~$f$ admits an extension of its Fourier transform to the whole complexification of~$W$.
	
	The following results collect all properties of Fourier transforms needed for the purposes of this article.
	\begin{lemma}\label{lemma;somepropFtr}
	Let $\vect{\genvec}_0\in W^2$.
	\begin{enumerate}[label=(\roman*), leftmargin=*]
	\item\label{it;Ftr1} The Fourier transform of $f(\vect{\genvec}-\vect{\genvec}_0)$ is $e\big(\trace(\vect{\genvec}_0,\vect{\genvec})\big)\cdot\widehat{f}(\vect{\genvec})$.
	\item\label{it;Ftr2} The Fourier transform of $f(\vect{\genvec})\cdot e\big(\trace(\vect{\genvec}_0,\vect{\genvec})\big)$ is $\widehat{f}(\vect{\genvec}+\vect{\genvec}_0)$.
	\end{enumerate}
	\end{lemma}
	\begin{proof}
	These properties are well-known.
	\end{proof}
	The next lemma provides a generalization in genus~$2$ of the main results of~\cite[Section~$3$]{bo;grass}.
	\begin{lemma}\label{lemma;onFtransfgenus2}
	\leavevmode
	\begin{enumerate}[label=(\roman*), leftmargin=*]
	\item Let $B\in\CC^{2\times 1}$.\label{item;1onFtransfgenus2}
	The Fourier transform of $f(\vect{\genvec})\cdot e\big(\trace(B\vect{\genvec})\big)$ is~${\widehat{f}(\vect{\genvec}+B^t)}$.
	\item Let $\tau\in\HH_2$, and let $\pol$ be a polynomial on the space $\RR^{m\times 2}$, endowed with the standard bilinear product.\label{item;2onFtransfgenus2}
	The Fourier transform of
	\bes
	\pol(\vect{\genvec})\cdot e\big(\trace(\vect{\genvec}^t\vect{\genvec}\tau)/2)
	\ees
	is
	\bes
	\det (-i\tau)^{-m/2}\cdot
	\exp\Big(\frac{i}{4\pi}\trace(\Delta\tau^{-1})\Big)(\pol)
	(-\vect{\genvec}\tau^{-1})
	\cdot e\Big(-\frac{1}{2}\trace(\vect{\genvec}^t \vect{\genvec} \tau^{-1})\Big).
	\ees
	\item Let $\pol$ be a polynomial on $\RR^{1\times 2}$, where the latter is endowed with the standard bilinear product~$(\vect{x},\vect{y})=x_1y_1 + x_2y_2$, and let $A\in\HH_2$, $B\in\CC^{2\times1}$, $C\in\CC$.\label{item;3onFtransfgenus2}
	The Fourier transform of
	\bes
	\pol(\vect{\genvec})\cdot e\big(\trace(A\vect{\genvec}^t\vect{\genvec}) + \trace(B\vect{\genvec}) + C\big)
	\ees
	is
	\bas
	\det(-2iA)^{-1/2} & \exp\Big(\frac{i}{8\pi}\trace(\Delta A^{-1})\Big)(\pol)\Big(\frac{1}{2}(-\vect{\genvec}-B^t)A^{-1}\Big)\times\\
	\times & e\Big(-\frac{1}{4}\trace(\vect{\genvec}^t\vect{\genvec}A^{-1}) - \frac{1}{2}\trace(B\vect{\genvec}A^{-1}) - \frac{1}{4}\trace(BB^tA^{-1}) + C\Big).
	\eas
	\item\label{item;4onFtransfgenus2}
	Let $\tau\in\HH_2$, and let $\pol$ be a polynomial on $\RR^{m\times 2}$, endowed with the standard bilinear product.
	The Fourier transform of
	\be\label{eq;4onFtransfgenus2}
	\exp\Big( -\frac{1}{8\pi}\trace(\Delta y^{-1}) \Big)(\pol)(\vect{\genvec})\cdot e\Big(\frac{1}{2}\trace(\vect{\genvec}^t\vect{\genvec}\tau)\Big)
	\ee
	is
	\bes
	\det(-i\tau)^{-m/2}\cdot
	\exp\Big(
	-\frac{1}{8\pi}\trace\big(\Delta \tau^{-2}\Im(-\tau^{-1})^{-1}\big)
	\Big)(\pol)(-\vect{\genvec}\tau^{-1})\cdot
	e\Big(
	-\frac{1}{2}\trace(\vect{\genvec}^t\vect{\genvec}\tau^{-1})
	\Big),
	\ees
	which is equal to
	\bes
	\det(-i\tau)^{-m/2}\cdot\det (\tau)^{-s}
	\exp\Big(
	-\frac{1}{8\pi}\trace\big(\Delta\Im(-\tau^{-1})^{-1}\big)
	\Big)(\pol)(\vect{\genvec})\cdot
	e\Big(
	-\frac{1}{2}\trace(\vect{\genvec}^t\vect{\genvec}\tau^{-1})
	\Big),
	\ees
	if $\pol$ is \emph{very homogeneous} of degree $s$.
	\item\label{item;5onFtransfgenus2}
	Suppose that $\pol$ is a polynomial defined on $(\auxspace^+\oplus \auxspace^-)^2$, where $\auxspace^+$ (resp.~$\auxspace^-$) is a positive definite (resp.\ negative definite) subspace of $\RR^{b,2}$.
	Denote by $d^+$ and $d^-$ the dimensions of $\auxspace^+$ and $\auxspace^-$ respectively.
	If the value of $\pol(\vect{\genvec})$ depends only on the projection $\vect{\genvec}_{\auxspace^+}$, that is, $\pol$ is of degree zero on $(\auxspace^-)^2$, then the Fourier transform of
	\bes
	\exp\Big(
	-\frac{1}{8\pi}\trace(\Delta y^{-1})
	\Big)(\pol)(\vect{\genvec})\cdot e\Big(
	\trace\big(q(\vect{\genvec}_{\auxspace^+})\tau\big) + \trace\big(q(\vect{\genvec}_{\auxspace^-})\bar{\tau}\big)
	\Big)
	\ees
	is
	\bas
	\det(-i\tau)^{-d^+/2} \det(i\bar{\tau})^{-d^-/2}
	\exp\Big(
	-\frac{1}{8\pi} \trace\big(\Delta \tau^{-2}\Im(-\tau^{-1})^{-1}\big)
	\Big)(\pol)(-\vect{\genvec} \tau^{-1})\\
	\times
	e\Big(
	- \trace\big(q(\vect{\genvec}_{\auxspace^+})\tau^{-1}\big) - \trace\big(q(\vect{\genvec}_{\auxspace^-})\bar{\tau}^{-1}\big)
	\Big).
	\eas
	\item\label{item;6onFtransfgenus2}
	Let $\pol$ be a \emph{very homogeneous} polynomial of degree $(m^+,m^-)$ on~${(\auxspace^+\oplus \auxspace^-)^2}$, where~$\auxspace^+$ (resp.~$\auxspace^-$) is a positive definite (resp.\ negative definite) subspace of $\RR^{b,2}$.
	Denote by $d^+$ and $d^-$ the dimensions of $\auxspace^+$ and $\auxspace^-$ respectively.
	The Fourier transform of
	\bes
	\exp\Big(
	-\frac{1}{8\pi}\trace(\Delta y^{-1})
	\Big)(\pol)(\vect{\genvec})\cdot e\Big(
	\trace\big(q(\vect{\genvec}_{\auxspace^+})\tau\big) + \trace\big(q(\vect{\genvec}_{\auxspace^-})\bar{\tau}\big)
	\Big)
	\ees
	is
	\bas
	\,&\det(-i\tau)^{-d^+/2} \cdot \det(\tau)^{-m^+} \cdot \det(i\bar{\tau})^{-d^-/2} \cdot \det(\bar{\tau})^{-m^-}
	\\
	&\times\exp\Big(
	-\frac{1}{8\pi} \trace\big(\Delta \Im(-\tau^{-1})^{-1}\big)
	\Big)(\pol)(\vect{\genvec})\cdot
	e\Big(
	- \trace\big(q(\vect{\genvec}_{\auxspace^+})\tau^{-1}\big) - \trace\big(q(\vect{\genvec}_{\auxspace^-})\bar{\tau}^{-1}\big)
	\Big).
	\eas
%	\textcolor{red}{This is a generalization of~\cite[Corollary~3.5]{bo;grass}.}
	\end{enumerate}
	\end{lemma}
	\begin{proof}
	Part~\ref{item;1onFtransfgenus2} is well known.
	Part~\ref{item;2onFtransfgenus2} is~\cite[Lemma~$4.5$]{roehrig}.
	Part~\ref{item;3onFtransfgenus2} follows from~\ref{item;2onFtransfgenus2} applied with~$\tau=2A$, and from~\ref{item;1onFtransfgenus2}.
	
	To prove Part~\ref{item;4onFtransfgenus2}, we apply~\ref{item;2onFtransfgenus2} with~$\exp\big( -\frac{1}{8\pi}\trace(\Delta y^{-1}) \big)(\pol)$ in place of~$\pol$, deducing that the Fourier transform of~\eqref{eq;4onFtransfgenus2} is
	\be\label{eq;prooflemFtransfs}
	\det (-i\tau)^{-m/2}\cdot
	\exp\Big(\frac{i}{4\pi}\trace(\Delta\tau^{-1})-\frac{1}{8\pi}\trace(\Delta y^{-1})\Big)(\pol)
	(-\vect{\genvec}\tau^{-1})
	\cdot e\big(-\trace(\vect{\genvec}^t \vect{\genvec} \tau^{-1})/2\big),
	\ee
	where we decompose $\tau=x+iy\in\HH_2$.
	We rewrite the exponential operator appearing in~\eqref{eq;prooflemFtransfs} as
	\ba\label{eq;simplifofexpopgen2}
	\exp\Big(\frac{i}{4\pi}\trace(\Delta\tau^{-1})-\frac{1}{8\pi}\trace(\Delta y^{-1})\Big)=
	\exp\Big(
	-\frac{1}{8\pi}\trace\big(
	\Delta ( y^{-1} - 2i\tau^{-1})
	\big)
	\Big)
	\\
	=\exp\Big(
	-\frac{1}{8\pi}\trace\big(
	\Delta \tau^{-1}y^{-1}(\tau-2iy)
	\big)
	\Big)=
	\exp\Big(
	-\frac{1}{8\pi}\trace\big(
	\Delta \tau^{-1}y^{-1}\bar{\tau}
	\big)
	\Big).
	\ea
	It is well-known that
	\bes
	(C\bar{\tau}+D)^t\Im(M\cdot \tau) (C\tau+D)=\Im(\tau),\qquad\text{for every $M=\left(\begin{smallmatrix}
	A & B\\ C & D
	\end{smallmatrix}\right)\in\Sp_4(\ZZ)$.}
	\ees
	If we specialize it with $M=\left(\begin{smallmatrix}
	0 & -I_2\\ I_2 & 0
	\end{smallmatrix}\right)$, we may rewrite it as
	\bes
	\Im(-\tau^{-1})^{-1}=\tau \Im(\tau)^{-1} \bar{\tau}.
	\ees
	We use such relation to rewrite the right-hand side of~\eqref{eq;simplifofexpopgen2} as
	\bas
	\exp\Big(
	-\frac{1}{8\pi}\trace\big(
	\Delta \tau^{-1}y^{-1}\bar{\tau}
	\big)
	\Big)=
	\exp\Big(
	-\frac{1}{8\pi}\trace\big(
	\Delta \tau^{-2}\Im(-\tau^{-1})^{-1}
	\big)
	\Big).
	\eas
	If we assume $\pol$ to be \emph{very homogeneous} of degree $m$, then by~\cite[Lemma~$4.4$~($4.5$)]{roehrig} we deduce that
	\bas
	\,&\exp\Big(
	-\frac{1}{8\pi}\trace
	\big(
	\Delta \tau^{-2}\Im(-\tau^{-1})^{-1}
	\big)
	\Big)(\pol)(-\vect{\genvec}\tau^{-1})
	\\
	&\qquad =
	\exp\Big(
	-\frac{1}{8\pi}\trace\big(
	\Delta\Im(-\tau^{-1})^{-1}
	\big)
	\Big)\big(\pol(-\vect{\genvec}\tau^{-1})\big)
	\\
	&\qquad= 
	\exp\Big(
	-\frac{1}{8\pi}\trace\big(
	\Delta\Im(-\tau^{-1})^{-1}
	\big)
	\Big)\big(\det(-\tau)^{-s}\cdot\pol(\vect{\genvec})\big)
	\\
	&\qquad= 
	\det(-\tau)^{-s}\exp\Big(
	-\frac{1}{8\pi}\trace\big(
	\Delta\Im(-\tau^{-1})^{-1}
	\big)
	\Big)(\pol)(\vect{\genvec}).
	\eas
	
	To prove Part~\ref{item;5onFtransfgenus2} and Part~\ref{item;6onFtransfgenus2}, it is enough to apply~\ref{item;4onFtransfgenus2} to~$\auxspace^+$ and~$\auxspace^-$.
	Since the idea is analogous, we provide only the proof of Part~\ref{item;6onFtransfgenus2}.
	Since~$\pol$ is a very homogeneous polynomial of degree $(m^+,m^-)$, there exist two polynomials $\pol_+$ and $\pol_-$ defined respectively on $\auxspace^+$ and $\auxspace^-$, such that~$\pol(\vect{\genvec})=\pol_+(\vect{\genvec}_{\auxspace^+})\cdot\pol_-(\vect{\genvec}_{\auxspace^-})$, and such that
	\bes
	\pol_+(\vect{\genvec}_{\auxspace^+}\cdot N)=(\det N)^{m^+}\cdot \pol_+(\vect{\genvec}_{\auxspace^+})\qquad\text{and}\qquad
	\pol_-(\vect{\genvec}_{\auxspace^-}\cdot N)=(\det N)^{m^-}\cdot \pol_-(\vect{\genvec}_{\auxspace^-}),
	\ees
	for every $N\in\RR^{2\times 2}$ and $\vect{\genvec}\in (\auxspace^+\oplus \auxspace^-)^2$.
	We may then rewrite
	\ba\label{eq;proofgenCor2linesbor}
	&\,\exp\Big(
	-\frac{1}{8\pi}\trace(\Delta y^{-1})
	\Big)(\pol)(\vect{\genvec})\cdot  e\Big(
	\trace\big(q(\vect{\genvec}_{\auxspace^+})\tau\big) + \trace\big(q(\vect{\genvec}_{\auxspace^-})\bar{\tau}\big)
	\Big)
	\\
	&\qquad= 
	\underbrace{\exp\Big(
	-\frac{1}{8\pi}\trace(\Delta y^{-1})
	\Big)(\pol_+)(\vect{\genvec}_{\auxspace^+})\cdot e\Big(\trace\big(q(\vect{\genvec}_{\auxspace^+})\tau\big)\Big)}_{\eqqcolon f^+_\tau(\vect{\genvec}_{\auxspace^+})}
	\\
	&\qquad\quad\times 
	\underbrace{\exp\Big(
	-\frac{1}{8\pi}\trace(\Delta y^{-1})
	\Big)(\pol_-)(\vect{\genvec}_{\auxspace^-})\cdot e\Big(\trace\big(q(\vect{\genvec}_{\auxspace^-})\bar{\tau}\big)
	\Big)}_{\eqqcolon f^-_\tau(\vect{\genvec}_{\auxspace^-})}.
	\ea
	The Fourier transform of the left-hand side of~\eqref{eq;proofgenCor2linesbor} is the product of the Fourier transforms of~$f^+_\tau$ and~$f^-_\tau$, since the latter two functions do not depend on common variables.
	Since the quadratic form~$q|_{\auxspace^*}$ on $\auxspace^+$ is \emph{positive definite}, we may apply~\ref{item;4onFtransfgenus2} to compute the Fourier transform of~$f^+_\tau$ as
	\ba\label{eq;proofgenCor2linesbor2}
	\widehat{f^+_\tau}(\vect{\xi}_{\auxspace^+})
	&=
	\det(\tau/i)^{-d^+/2}\cdot\det (\tau)^{-m^+}\\
	&\quad\times
	\exp\Big(
	-\frac{1}{8\pi}\trace\big(\Delta\Im(-\tau^{-1})^{-1}\big)
	\Big)(\pol_+)(\vect{\xi}_{\auxspace^+})\cdot
	e\big(
	-\trace(q(\vect{\xi}_{\auxspace^+})\tau^{-1})
	\big).
	\ea
	Since the quadratic form~$q|_{\auxspace^-}$ on~$\auxspace^-$ is \emph{negative definite}, before applying~\ref{item;4onFtransfgenus2} we rewrite~$\widehat{f^-_\tau}$ as
	\ba\label{eq;proofgenCor2linesbor3}
	\widehat{f^-_\tau}(\vect{\xi})
	&=
	\int_{\auxspace^-} f^-_\tau(\vect{x})\cdot e\big((\vect{\xi},\vect{x})\big) d\vect{x}
	=
	\int_{\auxspace^-} \exp\Big(
	-\frac{1}{8\pi}\trace(\Delta y^{-1})
	\Big)(\pol_-)(\vect{x})
	\\
	&\quad\times e\Big(\trace\big(-q(\vect{x})\cdot (-\bar{\tau})\big)
	\Big)\cdot
	e\big(-(-\vect{\xi},\vect{x})\big) d\vect{x},
	\ea
	where we denote by~$(\cdot{,}\cdot)$ the bilinear form associated to~$q|_{\auxspace^-}$.
	The right-hand side of~\eqref{eq;proofgenCor2linesbor3} is now the evaluation on $-\vect{\xi}$ of the Fourier transform of the function
	\bes
	\exp\Big(
	-\frac{1}{8\pi}\trace(\Delta y^{-1})
	\Big)(\pol_-)(\vect{\genvec}_{\auxspace^-})\cdot e\Big(\trace\big(-q(\vect{\genvec}_{\auxspace^-})\cdot (-\bar{\tau})\big)
	\Big)
	\ees
	with respect to the \emph{positive definite} quadratic space $(\auxspace^-,-q|_{\auxspace^-})$.
	Since $\Im(\bar{\tau}^{-1})=\Im(-\tau^{-1})$, we may apply~\ref{item;4onFtransfgenus2} and deduce that~\eqref{eq;proofgenCor2linesbor3} equals
	\ba\label{eq;proofgenCor2linesbor4}
	\,&\det(-\bar{\tau}/i)^{-d^-/2}\cdot\det(\bar{\tau})^{-m^-}\cdot \exp\Big(
	-\frac{1}{8\pi}\trace(\Delta \Im(-\tau^{-1})^{-1})
	\Big)(\pol_-)(-\vect{\xi})
	\\
	&\qquad\times e\Big(
	-\trace\big(q(-\vect{\xi})\bar{\tau}^{-1}\big)
	\Big).
	\ea
	Since~$\pol$ is very homogeneous, we deduce that
	\bas
	\exp\Big(
	-\frac{1}{8\pi}\trace(\Delta y^{-1})
	\Big)(\pol_-)(-\vect{\xi})
%	=\sum_{k=0}^\infty \frac{\big(\det(-I_2)\big)^{m^--2k}}{k!}\Big(
%	-\frac{1}{8\pi}\trace(\Delta y^{-1})
%	\Big)^k(\pol_-)(\vect{\xi})=\\
	=\exp\Big(
	-\frac{1}{8\pi}\trace(\Delta y^{-1})
	\Big)
	(\pol_-)(\vect{\xi})
	\eas
	for every positive definite $y\in\RR^{2\times 2}$.
	It is enough to insert~\eqref{eq;proofgenCor2linesbor2} and~\eqref{eq;proofgenCor2linesbor4} in~\eqref{eq;proofgenCor2linesbor} to conclude the proof.
%	
%	Note that since~$\tau\in\HH_2$ is a~$2\times 2$ matrix, then $\det(-\tau)=\det(\tau)$.
%	We may use such relations to rewrite the equalities above as in the statement of Lemma~\ref{lemma;onFtransfgenus2}.
	\end{proof}

	\subsection{Some decompositions of very homogeneous polynomials}\label{sec;appendix}
	
	Let~$\polw{\vect{\alpha},\borw}{\hone}{\htwo}$ be the auxiliary polynomials arising as in Definition~\ref{def;genborpolgenus2} with~$\pol=\pol_{\vect{\alpha}}$.
	
	The following result provides an explicit formula for~$\polw{\vect{\alpha},\borw}{\hone}{\htwo}$ in the usual case that~$L$ splits off a hyperbolic plane~$U$.
	Recall that we choose~$\genU$ and~$\genUU$ to be the standard generator of~$U$; see~\eqref{eq;choiceofuu'gen2}.
	\begin{lemma}\label{lemma;gen2funnypol}
	\textcolor{\mycolor}{
	Let~${z\in\Gr(L)}$ and~${g\in G}$ such that $g$ maps $z$ to~$z_0$.
	Let~$L$ split a hyperbolic plane~$U$, and let~$\genU,\genUU$ be the standard generators of~$U$.
	For every~${\vect{\genvec}=(\genvec_1,\genvec_2)}$ in~$V^2$, the value~${\polw{\vect{\alpha},\borw}{\hone}{\htwo}\big(\borw(\vect{\genvec})\big)}$ may be computed as follows.
	\begin{itemize}[leftmargin=*]
	\item If $\htot_j=0$ and $\htot_{3-j}=2$, where $j=1,2$, then
	\bas
	\polw{\vect{\alpha},\borw}{\hone}{\htwo}\big(\borw(\vect{\genvec})\big)
	&=
	\frac{4}{\genU_{z^\perp}^4}
	\prod_{i=1,2}\det\left(
	\begin{smallmatrix}
	(g(\genU),\basevec_{\alphagen{i}}) & (\borw(\genvec_j),\basevec_{\alphagen{i}})\\
	(g(\genU),\basevec_{\betagen{i}}) & (\borw(\genvec_j),\basevec_{\betagen{i}})
	\end{smallmatrix}
	\right).
%	\\
%	&=
%	\frac{4}{\genU_{z^\perp}^4}
%	\det\left(
%	\begin{smallmatrix}
%	(g(\genU),\basevec_{\alphaone}) & (\borw(\genvec_j),\basevec_{\alphaone})\\
%	(g(\genU),\basevec_{\betaone}) & (\borw(\genvec_j),\basevec_{\betaone})
%	\end{smallmatrix}
%	\right)
%	\det\left(
%	\begin{smallmatrix}
%	(g(\genU),\basevec_{\alphatwo}) & (\borw(\genvec_j),\basevec_{\alphatwo})\\
%	(g(\genU),\basevec_{\betatwo}) & (\borw(\genvec_j),\basevec_{\betatwo})
%	\end{smallmatrix}
%	\right).
	\eas
	\item If $\hone=\htwo=1$, then~$\polw{\vect{\alpha},\borw}{1}{1}\big(\borw(\vect{\genvec})\big)$ is
	\bas
%	&=
%	-\frac{4}{\genU_{z^\perp}^4}
%	\det\left(
%	\begin{smallmatrix}
%	(g(\genU),\basevec_{\alphaone}) & (\borw(\genvec_2),\basevec_{\alphaone}) \\
%	(g(\genU),\basevec_{\betaone}) & (\borw(\genvec_2),\basevec_{\betaone})
%	\end{smallmatrix}
%	\right)
%	\det\left(
%	\begin{smallmatrix}
%	(g(\genU),\basevec_{\alphatwo}) & (\borw(\genvec_1),\basevec_{\alphatwo}) \\
%	(g(\genU),\basevec_{\betatwo}) & (\borw(\genvec_1),\basevec_{\betatwo})
%	\end{smallmatrix}
%	\right)
%	\\
%	&\quad-
%	\frac{4}{\genU_{z^\perp}^4}
%	\det\left(
%	\begin{smallmatrix}
%	(g(\genU),\basevec_{\alphaone}) & (\borw(\genvec_1),\basevec_{\alphaone}) \\
%	(g(\genU),\basevec_{\betaone}) & (\borw(\genvec_1),\basevec_{\betaone})
%	\end{smallmatrix}
%	\right)
%	\det\left(
%	\begin{smallmatrix}
%	(g(\genU),\basevec_{\alphatwo}) & (\borw(\genvec_2),\basevec_{\alphatwo}) \\
%	(g(\genU),\basevec_{\betatwo}) & (\borw(\genvec_2),\basevec_{\betatwo})
%	\end{smallmatrix}
%	\right)
%	\\
%	&=
	-\frac{4}{\genU_{z^\perp}^4}
	\sum_{i=1,2}
	\det\left(
	\begin{smallmatrix}
	(g(\genU),\basevec_{\alphaone}) & (\borw(\genvec_i),\basevec_{\alphaone}) \\
	(g(\genU),\basevec_{\betaone}) & (\borw(\genvec_i),\basevec_{\betaone})
	\end{smallmatrix}
	\right)
	\det\left(
	\begin{smallmatrix}
	(g(\genU),\basevec_{\alphatwo}) & (\borw(\genvec_{3-i}),\basevec_{\alphatwo}) \\
	(g(\genU),\basevec_{\betatwo}) & (\borw(\genvec_{3-i}),\basevec_{\betatwo})
	\end{smallmatrix}
	\right)
	.
	\eas
	\item If $\htot_j=1$ and $\htot_{3-j}=0$, where $j=1,2$, then~$\polw{\vect{\alpha},\borw}{\hone}{\htwo}\big(\borw(\vect{\genvec})\big)$ is
	\bas
%	\polw{\vect{\alpha},\borw}{\hone}{\htwo}\big(\borw(\vect{\genvec})\big)
%	&=
%	\frac{4}{\genU_{z^\perp}^2}
%	\det\left(
%	\begin{smallmatrix}
%	(g(\genU),\basevec_{\alphaone}) & (\borw(\genvec_{3-j}),\basevec_{\alphaone}) \\
%	(g(\genU),\basevec_{\betaone}) & (\borw(\genvec_{3-j}),\basevec_{\betaone})
%	\end{smallmatrix}
%	\right)
%	\det\left(
%	\begin{smallmatrix}
%	(\borw(\genvec_{j}),\basevec_{\alphatwo}) &  (\borw(\genvec_{3-j}),\basevec_{\alphatwo})\\
%	(\borw(\genvec_{j}),\basevec_{\betatwo}) & (\borw(\genvec_{3-j}),\basevec_{\betatwo})
%	\end{smallmatrix}
%	\right)
%	\\
%	&\quad +
%	\frac{4}{\genU_{z^\perp}^2}
%	\det\left(
%	\begin{smallmatrix}
%	(\borw(\genvec_{j}),\basevec_{\alphaone}) & (\borw(\genvec_{3-j}),\basevec_{\alphaone}) \\
%	(\borw(\genvec_{j}),\basevec_{\betaone}) & (\borw(\genvec_{3-j}),\basevec_{\betaone})
%	\end{smallmatrix}
%	\right)
%	\det\left(
%	\begin{smallmatrix}
%	(g(\genU),\basevec_{\alphatwo}) & (\borw(\genvec_{3-j}),\basevec_{\alphatwo}) \\
%	(g(\genU),\basevec_{\betatwo}) &  (\borw(\genvec_{3-j}),\basevec_{\betatwo})
%	\end{smallmatrix}
%	\right)
%	\\
%	&=
	\frac{4}{\genU_{z^\perp}^2}
	\sum_{i=1,2}
	\det\left(
	\begin{smallmatrix}
	(g(\genU),\basevec_{\alphagen{i}}) & (\borw(\genvec_{3-j}),\basevec_{\alphagen{i}}) \\
	(g(\genU),\basevec_{\betagen{i}}) & (\borw(\genvec_{3-j}),\basevec_{\betagen{i}})
	\end{smallmatrix}
	\right)
	\det\left(
	\begin{smallmatrix}
	(\borw(\genvec_{j}),\basevec_{\alphagen{3-i}}) &  (\borw(\genvec_{3-j}),\basevec_{\alphagen{3-i}})\\
	(\borw(\genvec_{j}),\basevec_{\betagen{3-i}}) & (\borw(\genvec_{3-j}),\basevec_{\betagen{3-i}})
	\end{smallmatrix}
	\right)
	.
	\eas
	\item If $\hone=\htwo=0$, then
	\bas
	\polw{\vect{\alpha},\borw}{0}{0}\big(\borw(\vect{\genvec})\big)
%	&=
%	4\det\left(
%	\begin{smallmatrix}
%	(\borw(\genvec_{1}),\basevec_{\alphaone}) & (\borw(\genvec_{2}),\basevec_{\alphaone})\\
%	(\borw(\genvec_{1}),\basevec_{\betaone}) & (\borw(\genvec_{2}),\basevec_{\betaone})
%	\end{smallmatrix}
%	\right)
%	\det
%	\left(
%	\begin{smallmatrix}
%	(\borw(\genvec_{1}),\basevec_{\alphatwo}) & (\borw(\genvec_{2}),\basevec_{\alphatwo}) \\
%	(\borw(\genvec_{1}),\basevec_{\betatwo}) & (\borw(\genvec_{2}),\basevec_{\betatwo})
%	\end{smallmatrix}
%	\right)
%	\\
	&=
	4\prod_{i=1,2}
	\det\left(
	\begin{smallmatrix}
	(\borw(\genvec_{1}),\basevec_{\alphagen{i}}) & (\borw(\genvec_{2}),\basevec_{\alphagen{i}})\\
	(\borw(\genvec_{1}),\basevec_{\betagen{i}}) & (\borw(\genvec_{2}),\basevec_{\betagen{i}})
	\end{smallmatrix}
	\right).
	\eas
	\item In all remaining cases, we have $\polw{\vect{\alpha},\borw}{\hone}{\htwo}=0$.
	\end{itemize}
	If the polynomial $\polw{\vect{\alpha},\borw}{\hone}{\htwo}$ differs from zero, then it is \emph{very homogeneous} only when both~${\hone}$ and~${\htwo}$ are zero.
	}
	\end{lemma}
	\begin{proof}
%	We deduce from~\eqref{eq;defpolPabcd} that
%	\ba\label{eq;decgenpolabcd}
%	\pol_{\vect{\alpha}}\big(g(\vect{\genvec})\big)
%	&=
%	4\sum_{\sigma,\sigma'\in S_2}\sgn(\sigma)\sgn(\sigma')
%	\big(\genvec_1,g^{-1}(\basevec_{\sigma(\alphaone)})\big)
%	\big(\genvec_1,g^{-1}(\basevec_{\sigma'(\alphatwo)})\big)
%	\\
%	&\quad\times
%	\big(\genvec_2,g^{-1}(\basevec_{\sigma(\betaone)})\big)
%	\big(\genvec_2,g^{-1}(\basevec_{\sigma'(\betatwo)})\big),
%	\ea
%	where $\sigma$ (resp.\ $\sigma'$) acts as a permutation of the indexes $\{\alphaone,\betaone\}$ (resp.\ $\{\alphatwo,\betatwo\}$).
%	\\\textcolor{red}{Shorter version with determinants:}\\
	\textcolor{\mycolor}{
	We deduce from~\eqref{eq;defpolPabcd} that
	\ba\label{eq;decgenpolabcd}
	\pol_{\vect{\alpha}}\big(g(\vect{\genvec})\big)
%	&=
%	4
%	\det\left(
%	\begin{smallmatrix}
%	(\genvec_1,g^{-1}(\basevec_{\alphaone})) & (\genvec_2,g^{-1}(\basevec_{\alphaone}))\\
%	(\genvec_1,g^{-1}(\basevec_{\betaone})) & (\genvec_2,g^{-1}(\basevec_{\betaone}))
%	\end{smallmatrix}
%	\right)
%	\det\left(
%	\begin{smallmatrix}
%	(\genvec_1,g^{-1}(\basevec_{\alphatwo})) & (\genvec_2,g^{-1}(\basevec_{\alphatwo}))\\
%	(\genvec_1,g^{-1}(\basevec_{\betatwo})) & (\genvec_2,g^{-1}(\basevec_{\betatwo}))
%	\end{smallmatrix}
%	\right)
%	\\
	&=
	4\prod_{i=1,2}
	\det\left(
	\begin{smallmatrix}
	(\genvec_1,g^{-1}(\basevec_{\alphagen{i}})) & (\genvec_2,g^{-1}(\basevec_{\alphagen{i}}))\\
	(\genvec_1,g^{-1}(\basevec_{\betagen{i}})) & (\genvec_2,g^{-1}(\basevec_{\betagen{i}}))
	\end{smallmatrix}
	\right).
	\ea
	}
%	We decompose~${g^{-1}(v_j)=s_j u_{z^\perp} + v'_j}$, with~$s_j\in\RR$ and~$v'_j=\big( g^{-1}(v_j)\big)_{w^\perp}$, for every~$j$, and replace such decomposition in~\eqref{eq;decgenpolabcd} to deduce that
%	\bas
%	\pol_{\vect{\alpha}}\big(g(\vect{\genvec})\big)
%	&=
%	4\sum_{\sigma,\sigma'\in S_2}\sgn(\sigma)\sgn(\sigma')\Big(
%	(\genvec_1,u_{z^\perp})^2 s_{\sigma(\alphaone)}s_{\sigma'(\alphatwo)} + (\genvec_1,u_{z^\perp})\\
%	&\quad\times\Big(
%	s_{\sigma(\alphaone)}(\genvec_1,v'_{\sigma'(\alphatwo)})+s_{\sigma'(\alphatwo)}(\genvec_1,v'_{\sigma(\alphaone)})
%	\Big) + (\genvec_1,v'_{\sigma(\alphaone)})(\genvec_1,v'_{\sigma'(\alphatwo)})
%	\Big)
%	\\
%	&\quad\times
%	\Big(
%	(\genvec_2,u_{z^\perp})^2 s_{\sigma(\betaone)}s_{\sigma'(\betatwo)}
%	+
%	(\genvec_2,u_{z^\perp})\Big(
%	s_{\sigma(\betaone)}(\genvec_2,v'_{\sigma'(\betatwo)})
%	\\
%	&\quad+s_{\sigma'(\betatwo)}(\genvec_2,v'_{\sigma(\betaone)})
%	\Big) + (\genvec_2,v'_{\sigma(\betaone)})(\genvec_2,v'_{\sigma'(\betatwo)})
%	\Big).
%	\eas
%	\\\textcolor{red}{Shorter version with determinants:}\\
	\textcolor{\mycolor}{We decompose~${g^{-1}(v_j)=s_j u_{z^\perp} + v'_j}$, with~$s_j\in\RR$ and~$v'_j=\big( g^{-1}(v_j)\big)_{w^\perp}$, for every~$j$, and replace such decomposition in~\eqref{eq;decgenpolabcd}, obtaining that~$\pol_{\vect{\alpha}}\big(g(\vect{\genvec})\big)$ equals
	\ba\label{eq;factorPwrtsplit}
	4\prod_{i=1,2}\det\left(
	\left(\begin{smallmatrix}
	s_{\alphagen{i}}(\genvec_1, \genU_{z^\perp})
	&
	s_{\alphagen{i}}(\genvec_2, \genU_{z^\perp})
	\\
	s_{\betagen{i}}(\genvec_1, \genU_{z^\perp})
	&
	s_{\betagen{i}}(\genvec_2, \genU_{z^\perp})
	\end{smallmatrix}
	\right)+\left(
	\begin{smallmatrix}
	(\genvec_1,v'_{\alphagen{i}})
	&
	(\genvec_2,v'_{\alphagen{i}})
	\\
	(\genvec_1,v'_{\betagen{i}})
	&
	(\genvec_2,v'_{\betagen{i}})
	\end{smallmatrix}
	\right)\right)
	.
	\ea
	Since~$\det(M+N)=\det (M) + \det (N) + \trace(M)\trace(N) - \trace(MN)$ for any~$2\times 2$ matrix, we may compute that the~$i$-th factor in~\eqref{eq;factorPwrtsplit} equals
	\bas%\label{eq;compofsuits}
	\det \left(
	\begin{smallmatrix}
	(\genvec_1,v'_{\alphagen{i}})
	&
	(\genvec_2,v'_{\alphagen{i}})
	\\
	(\genvec_1,v'_{\betagen{i}})
	&
	(\genvec_2,v'_{\betagen{i}})
	\end{smallmatrix}
	\right)
	+
	(\genvec_1,\genU_{z^\perp})
	\det\left(
	\begin{smallmatrix}
	s_{\alphagen{i}} & (\genvec_2,v'_{\alphagen{i}})\\
	s_{\betagen{i}} & (\genvec_2,v'_{\betagen{i}})
	\end{smallmatrix}
	\right)
	-
	(\genvec_2, \genU_{z^\perp})
	\det\left(
	\begin{smallmatrix}
	s_{\alphagen{i}} & (\genvec_1,v'_{\alphagen{i}})\\
	s_{\betagen{i}} & (\genvec_1,v'_{\betagen{i}})
	\end{smallmatrix}
	\right).
	\eas
	We then replace this in~\eqref{eq;factorPwrtsplit} together with~$s_j=\big(g(\genU),\basevec_j\big)/\genU_{z^\perp}^2$ and~$(\genvec,\genvec_j')=\big(\borw(\genvec),\basevec_j\big)$, for~$j=\alphagen{i},\betagen{i}$.
	A simple comparison of that new formula for~\eqref{eq;decgenpolabcd} with~\eqref{eq;gengen2borfupol}, extracting the factors multiplying~$(\genvec_1,\genU_{z^\perp})^{h_1}(\genvec_2,\genU_{z^\perp})^{h_2}$ for every~$h_1$ and~$h_2$, verifies Lemma~\ref{lemma;gen2funnypol}.
	}

	To prove that the polynomial $\polw{\vect{\alpha},\borw}{0}{0}\big(\borw(\vect{\genvec})\big)$ is very homogeneous, one can follow the same wording of Lemma~\ref{lemma;homogpolabcddeg2}.
	It is an easy exercise to see that the remaining non-trivial polynomials are non-very homogeneous.
	\end{proof}
	
	Although the auxiliary polynomials~$\polw{\vect{\alpha},\borw}{\hone}{\htwo}$ are in general not very homogeneous, they satisfy the property
	\be\label{eq;nonveryhompropgen2}
	\polw{\vect{\alpha},\borw}{\hone}{\htwo}\big(\borw(\lambda \vect{\genvec})\big)=\lambda^{4-\hone-\htwo}\polw{\vect{\alpha},\borw}{\hone}{\htwo}\big(\borw(\vect{\genvec})\big),
	\ee
	for every $\lambda\in\CC$, or equivalently, they are homogeneous of degree~$4-\hone-\htwo$ in the classical sense.
	
	The following result illustrates the transformation property of~$\polw{\vect{\alpha},\borw}{\hone}{\htwo}$ induced by the right-multiplication of its argument by a matrix of~$\HH_2$.
	\begin{lemma}\label{lemma;nonhomogofderpol}
	\textcolor{\mycolor}{
	Let~${\tau=\left(\begin{smallmatrix}	
	\tauone & \tautwo\\ \tautwo & \tauthree
	\end{smallmatrix}\right)\in\HH_2}$.
	Let~$L$ split a hyperbolic plane~$U$, and let~$\genU,\genUU$ be the standard generators of~$U$.
	\begin{itemize}[leftmargin=*]
	\item If $\hgen{j}=0$ and $\hgen{3-j}=2$, then
	\bas%\label{eq;polh+=202}
	\polw{\vect{\alpha},\borw}{h_1}{h_2}\big(\borw(\vect{\genvec}) \tau \big)
	&=
	\taugen{j}^2 \cdot \polw{\vect{\alpha},\borw}{0}{2}\big(\borw(\vect{\genvec})\big)
	+
	\taugen{j+1}^2 \cdot \polw{\vect{\alpha},\borw}{2}{0}\big(\borw(\vect{\genvec})\big)
	\\
	&\quad
	-\taugen{j}\cdot \taugen{j+1} \cdot \polw{\vect{\alpha},\borw}{1}{1}\big(\borw(\vect{\genvec})\big).
	\eas
	\item If $\hone=\htwo=1$, then
	\bas%\label{eq;polh+=211}
	\polw{\vect{\alpha},\borw}{1}{1}\big(\borw(\vect{\genvec}) \tau \big)
	&=
	-2\tauone \tautwo \cdot \polw{\vect{\alpha},\borw}{0}{2}\big(\borw(\vect{\genvec})\big)
	-
	2\tautwo \tauthree \cdot \polw{\vect{\alpha},\borw}{2}{0}\big(\borw(\vect{\genvec})\big)
	\\
	&\quad+
	(\tauone\tauthree+\tautwo^2) \cdot \polw{\vect{\alpha},\borw}{1}{1}\big(\borw(\vect{\genvec})\big)
	.
	\eas
	\item If $\hgen{j}=0$ and $\hgen{3-j}=1$, then
	\bas%\label{eq;polh+=101}
	\polw{\vect{\alpha},\borw}{h_1}{h_2}\big(\borw(\vect{\genvec}) \tau \big)
	=
	(-1)^{j+1}\det\tau\Big(
	\taugen{j}\cdot\polw{\vect{\alpha},\borw}{0}{1}\big(\borw(\vect{\genvec})\big)
	-
	\taugen{j+1}\cdot\polw{\vect{\alpha},\borw}{1}{0}\big(\borw(\vect{\genvec}) \big)\Big).
	\eas
	\end{itemize}
	}
	\end{lemma}
	\begin{proof}
	\textcolor{\mycolor}{This is an immediate consequence of Lemma~\ref{lemma;gen2funnypol} and the multilinearity of the determinant.}
	\end{proof}
	The following result will be relevant to compute the transformation property of the theta function~$\Theta_{\brK,2}$ attached to~$\polw{\vect{\alpha},\borw}{\hone}{\htwo}$, with respect to the action of~$\Mp_4(\ZZ)$.
	\begin{lemma}\label{lemma;someFtransfofabcdh12lor}
	\textcolor{\mycolor}{
	Let~$L$ split a hyperbolic plane~$U$, and let~$\genU,\genUU$ be the standard generators of~$U$.
	Let
	\bas
	f_{\tau,\borw,\hone,\htwo}(\vect{\genvec})
	&=
	\exp \Big(
	-\frac{1}{8\pi}\trace(\Delta y^{-1})
	\Big)\big(\polw{\vect{\alpha},\borw}{\hone}{\htwo}\big)\big(\borw(\vect{\genvec})\big)
	\\
	&\quad
	\times e\Big(
	\trace\big( q(\vect{\genvec}_{w^\perp})\tau\big)+
	\trace\big( q(\vect{\genvec}_w)\bar{\tau}\big)
	\Big),
	\eas
	where $\vect{\genvec}\in(\brK\otimes\RR)^2$ and~$\tau=\left(\begin{smallmatrix}	
	\tauone & \tautwo\\ \tautwo & \tauthree
	\end{smallmatrix}\right)\in\HH_2$.
	\begin{itemize}[leftmargin=*]
	\item If $\hgen{j}=2$ and $\hgen{3-j}=0$, then the Fourier transform of~$f_{\tau,\borw,h_1,h_2}$ is
	\bas
	\widehat{f_{\tau,\borw,h_1,h_2}}(\vect{\xi})
	&=
	i^{2-b}
	\det(\tau)^{-(b-1)/2-2}\det(\bar{\tau})^{-1/2}\Big(
	\taugen{j+1}^2 f_{-\tau^{-1},\borw,0,2}(\vect{\xi})
	\\
	&\quad+
	\taugen{j}\taugen{j+1} f_{-\tau^{-1},\borw,1,1}(\vect{\xi}) +
	\taugen{j}^2 f_{-\tau^{-1},\borw,2,0}(\vect{\xi})
	\Big).
	\eas
	\item If $\hone=\htwo=1$, then the Fourier transform of~$f_{\tau,\borw,1,1}$ is
	\bas
	\widehat{f_{\tau,\borw,1,1}}(\vect{\xi})
	&=
	i^{2-b}
	\det(\tau)^{-(b-1)/2-2}\det(\bar{\tau})^{-1/2}\Big(
	2\tautwo\tauthree f_{-\tau^{-1},\borw,0,2}(\vect{\xi})
	\\
	&\quad+
	(\tauone\tauthree + \tautwo^2) f_{-\tau^{-1},\borw,1,1}(\vect{\xi}) +
	2\tauone\tautwo f_{-\tau^{-1},\borw,2,0}(\vect{\xi})
	\Big).
	\eas
	\item If $\hgen{j}=1$ and $\hgen{3-j}=0$, then the Fourier transform of~$f_{\tau,\borw,h_1,h_2}$ is
	\bas
	\widehat{f_{\tau,\borw,\hone,\htwo}}(\vect{\xi})
	=
	-i^{2-b}
	\det(\tau)^{-(b-1)/2-2} \det(\bar{\tau})^{-1/2}
	\Big(
	\taugen{j+1} f_{-\tau^{-1},\borw,0,1}(\vect{\xi}) + \taugen{j} f_{-\tau^{-1},\borw,1,0}(\vect{\xi})
	\Big).
	\eas
	\item If $\hone=\htwo=0$, then the Fourier transform of~$f_{\tau,\borw,0,0}$ is
	\bas
	\widehat{f_{\tau,\borw,0,0}}(\vect{\xi})
	=
	i^{2-b} \det(\tau)^{-(b-1)/2-2}  \det(\bar{\tau})^{-1/2} f_{-\tau^{-1},\borw,0,0}(\vect{\xi}).
	\eas
	\end{itemize}
	}
	\end{lemma}
	\begin{proof}
%	The idea is to use Lemma~\ref{lemma;onFtransfgenus2} together with Lemma~\ref{lemma;nonhomogofderpol}.
	\textbf{Case $\boldsymbol{\hone=\htwo=0}$:}
	By Lemma~\ref{lemma;nonhomogofderpol} the polynomial~$\polw{\vect{\alpha},\borw}{0}{0}$ is \emph{very homogeneous} of degree~$(2,0)$, hence we may apply Lemma~\ref{lemma;onFtransfgenus2}~\ref{item;6onFtransfgenus2} to deduce that
	\bas
	\widehat{f_{\tau,\borw,0,0}}(\vect{\xi})
	&=
	i^{2-b}\det(\tau)^{-(b-1)/2-2} \det(\bar{\tau})^{-1/2} e\Big(
	- \trace\big(q(\vect{\xi}_{w^\perp})\tau^{-1}\big) - \trace\big(q(\vect{\xi}_{w})\bar{\tau}^{-1}\big)
	\Big)
	\\
	&\quad\times\exp\Big(
	-\frac{1}{8\pi} \trace\big(\Delta \Im(-\tau^{-1})^{-1}\big)
	\Big)(\polw{\vect{\alpha},\borw}{0}{0})\big(\borw(\vect{\xi})\big).
	\eas
	
	\textbf{Case $\boldsymbol{\hone=0}$ and $\boldsymbol{\htwo=2}$:}
	By Lemma~\ref{lemma;nonhomogofderpol}, the polynomial~$\polw{\vect{\alpha},\borw}{0}{2}$ is non-very homogeneous.
	We apply Lemma~\ref{lemma;onFtransfgenus2}~\ref{item;5onFtransfgenus2} to deduce that~$\widehat{f_{\tau,\borw,0,2}}(\vect{\xi})$ equals
	\ba\label{eq;casegenus202}
	\det(\tau/i)^{-(b-1)/2} 
	\det(i\bar{\tau})^{-1/2} e\Big(
	- \trace\big(q(\vect{\xi}_{w^\perp})\tau^{-1}\big) - \trace\big(q(\vect{\xi}_{w})\bar{\tau}^{-1}\big)
	\Big)\quad
	\\
	\times
	\exp\Big(
	-\frac{1}{8\pi} \trace\big(\Delta \tau^{-2}\Im(-\tau^{-1})^{-1}\big)
	\Big)(\polw{\vect{\alpha},\borw}{0}{2})\big(-\borw(\vect{\xi}) \tau^{-1}\big).
	\ea
	By~\cite[Lemma~$4.4$~($4.5$)]{roehrig}, we rewrite the exponential operator applied to~$\polw{\vect{\alpha},\borw}{0}{2}$ appearing in~\eqref{eq;casegenus202} as
	\ba\label{eq;casegenus202bis}
	\exp\Big(
	-\frac{1}{8\pi} \trace \big(\Delta \tau^{-2}\Im(-\tau^{-1})^{-1}\big)
	\Big)(\polw{\vect{\alpha},\borw}{0}{2})\big(-\borw(\vect{\xi}) \tau^{-1}\big)
	\quad
	\\
	=
	\exp\Big(
	-\frac{1}{8\pi}\trace\big(
	\Delta\Im(-\tau^{-1})^{-1}
	\big)
	\Big)\Big(\polw{\vect{\alpha},\borw}{0}{2}\big(-\borw (\vect{\xi})\tau^{-1}\big)\Big).
	\ea
	Since if $\tau=\left(\begin{smallmatrix}
	\tauone & \tautwo\\
	\tautwo & \tauthree
	\end{smallmatrix}\right)\in\HH_2$, then $-\tau^{-1}=\frac{1}{\det\tau}\left(\begin{smallmatrix}
	-\tauthree & \tautwo\\
	\tautwo & -\tauone
	\end{smallmatrix}\right)$, we deduce by Lemma~\ref{lemma;nonhomogofderpol} that
	\bas
	\polw{\vect{\alpha},\borw}{0}{2} \big(-\borw(\vect{\xi})\tau^{-1}\big)
	&=
	\frac{\tauthree^2}{\det\tau^2} \cdot \polw{\vect{\alpha},\borw}{0}{2}\big( \borw(\vect{\xi})\big)
	+
	\frac{\tautwo\tauthree}{\det\tau^2}
	\\
	&\quad\times
	\polw{\vect{\alpha},\borw}{1}{1}\big(\borw(\vect{\xi})\big)
	+
	\frac{\tautwo^2}{\det\tau^2}\cdot \polw{\vect{\alpha},\borw}{2}{0}\big(\borw(\vect{\xi})\big).
	\eas
	Replacing this in~\eqref{eq;casegenus202bis}, we deduce that
	\begin{align*}
	\widehat{f_{\tau,\borw,0,2}}(\vect{\xi})
	&=
	\det(\tau/i)^{-(b-1)/2}\det(i\bar{\tau})^{-1/2} e\Big(
	- \trace\big(q(\vect{\xi}_{w^\perp})\tau^{-1}\big) - \trace\big(q(\vect{\xi}_{w})\bar{\tau}^{-1}\big)
	\Big)
	\\
	&\quad\times \Big(
	\frac{\tauthree^2}{\det\tau}\exp\Big(
	-\frac{1}{8\pi}\trace\big(
	\Delta\Im(-\tau^{-1})^{-1}
	\big)
	\Big)(\polw{\vect{\alpha},\borw}{0}{2})\big(\borw(\vect{\xi})\big)
	\\
	&\quad+
	\frac{\tautwo\tauthree}{\det\tau^2} \cdot
	\exp\Big(
	-\frac{1}{8\pi}\trace\big(
	\Delta\Im(-\tau^{-1})^{-1}
	\big)
	\Big)(\polw{\vect{\alpha},\borw}{1}{1})\big(\borw(\vect{\xi})\big)
	\\
	&\quad+
	\frac{\tautwo^2}{\det\tau^2} \cdot
	\exp\Big(
	-\frac{1}{8\pi}\trace\big(
	\Delta\Im(-\tau^{-1})^{-1}
	\big)
	\Big)(\polw{\vect{\alpha},\borw}{2}{0})\big(\borw(\vect{\xi})\big)
	\Big)
	\\
	&=
	i^{2-b}\det(\tau)^{-(b-1)/2-2}\det(\bar{\tau})^{-1/2}\Big(
	\tauthree^2\cdot f_{-\tau^{-1},\borw,0,2}(\vect{\xi})
	\\
	&\quad+
	\tautwo\tauthree\cdot f_{-\tau^{-1},\borw,1,1}(\vect{\xi}) +
	\tautwo^2\cdot f_{-\tau^{-1},\borw,2,0}(\vect{\xi})
	\Big).
	\end{align*}
	
	\textbf{All remaining cases:}
	The proof is analogous and left to the reader.
	\end{proof}
	
	\printbibliography

\end{document}